\documentclass{amsbook}
\usepackage{epsfig}
\usepackage{graphicx}
\usepackage{color}
\newtheorem{theorem}{Theorem}[chapter]
\newtheorem{lemma}[theorem]{Lemma}

\newtheorem{fact}[theorem]{Fact}
\theoremstyle{definition}
\newtheorem{definition}[theorem]{Definition}
\theoremstyle{corollary}
\newtheorem{corollary}[theorem]{Corollary}
\theoremstyle{conjecture}
\newtheorem{conjecture}[theorem]{Conjecture}
\theoremstyle{proposition}
\newtheorem{proposition}[theorem]{Proposition}
\newcommand{\supp}{\mathop{\mathrm{supp}}}
\newtheorem{example}[theorem]{Example}

\theoremstyle{remark}
\newtheorem{remark}[theorem]{Remark}

\newcommand{\vp}{\mathop{\mathrm{vp}}}
\newcommand{\vol}{\mathop{\mathrm{vol}}}
\newcommand{\Imm}{\mathop{\mathrm{Im}}}
\newcommand{\Ree}{\mathop{\mathrm{Re}}}
\newcommand{\diag}{\mathop{\mathrm{diag}}}
\newcommand{\SobSDt}[2]{^{\nabla}\|#1\|_{\tau,\,\varepsilon}^{#2}}
\newcommand{\SobSDz}[2]{^{\nabla}\|#1\|_{\zeta,\,\varepsilon}^{#2}}

\newcommand{\SobSdt}[2]{^{\partial}\|#1\|_{\tau,\,\varepsilon\emph{}}^{#2}}
\newcommand{\SobSdz}[2]{^{\partial}\|#1\|_{\zeta,\,\varepsilon\emph{}}^{#2}}
\newcommand{\SobODt}[2]{^{\nabla}\|#1\|_{\Omega_\tau,\,\varepsilon}^{#2}}
\newcommand{\SobOdt}[2]{^{\partial}\|#1\|_{\Omega_\tau,\,\varepsilon\emph{}}^{#2}}

\def\R{{\mathbb R}}
\def\abs#1{\left|#1\right|}

\def\csub{\subset\subset}

\def\eps{\varepsilon}
\def\Gen{\mathcal G}

\def\Ord#1{O#1}

\numberwithin{section}{chapter} \numberwithin{equation}{chapter}

\begin{document}
\thispagestyle{empty}
\pagenumbering{roman}
\begin{flushright}
\includegraphics{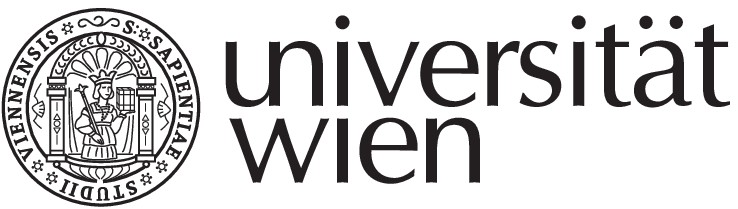}\\[2cm]
\end{flushright}
\begin{center}
\begin{huge}
Dissertation \\[1cm]
\end{huge}
\begin{Huge}
{\bf The wave equation on singular space-times}\\[2cm]
\end{Huge}
\begin{Small}
angestrebter akademischer Grad\\[0.6cm]
\end{Small}
\begin{large}
Doktor der Naturwissenschaften (Dr.\ rer. nat.)\\[4cm]
\end{large}

\end{center}
\begin{Large}
Verfasser: Dipl.-Ing. Eberhard Mayerhofer\\
Matrikel-Nummer: 9540617\\
Dissertationsgebiet: Mathematik\\
Betreuer: Prof.\ Dr. Michael Kunzinger\\[3cm]
Wien, am 1. August 2006
\end{Large}
\newpage \ \\[2cm]
\newpage
\begin{small} \ \\[3cm]
Abstract.\\
The first part of my thesis lays the foundations to generalized
Lorentz geometry. The basic algebraic structure of
finite-dimensional modules over the ring of generalized numbers is
investigated. This includes a new characterization of invertibility
in the ring of generalized numbers as well as a characterization of
free elements inside the $n$-dimensional module $\widetilde{\mathbb
R}^n$. The index of symmetric bilinear forms is introduced; this new
concept enables a (generalized) pointwise characterization of
generalized pseudo Riemannian metrics on smooth manifolds as
introduced by M.\ Kunzinger and R. Steinbauer. It is shown that free
submodules have direct summands, however $\widetilde{\mathbb R}^n$
turns out not to be semisimple. Applications of these new concepts are
a generalized notion of causality, the generalized inverse Cauchy
Schwarz inequality for time-like or null vectors, constructions of
pseudo Riemannian metrics as well as generalized energy tensors. The
motivation for this part of my thesis evolved from the main topic,
the wave equation on singular space-times.

The second and main part of my thesis is devoted to establishing a local
existence and uniqueness theorem for the wave equation on singular
space-times. The singular Lorentz metric subject to our discussion
is modeled within the special algebra on manifolds in the sense of
J.\ F.\ Colombeau. Inspired by an approach to generalized hyperbolicity
of conical-space times due to J.\ Vickers and J.\ Wilson, we succeed
in establishing certain energy estimates, which by a further
elaborated equivalence of energy integrals and Sobolev norms allow
us to prove existence and uniqueness of local generalized solutions
of the wave equation with respect to a wide class of generalized
metrics.

The third part of my thesis treats three different point value
resp.\ uniqueness questions in algebras of generalized functions.
The first one, posed by Michael Kunzinger, reads as follows: Is the
theorem by Albeverio et al., that elements of the so-called p-adic
Colombeau Egorov algebra are determined uniquely on standard points,
a p-adic scenario? We answer this problem by means of a
counterexample which shows that the statement in fact does not hold.
We further show that elements of an Egorov algebra of generalized
functions on a locally compact ultrametric space allow a point-value
characterization if and only if the metric induces the discrete
topology. {\it Secondly}, we prove that the ring of generalized
(real or complex) numbers endowed with the sharp norm does not admit
nested sequences of closed balls to have an empty intersection. As
an application we outline a possible version of the Hahn-Banach
Theorem as well as the ultrametric Banach fixed point theorem. {\it
Finally}, we establish that scaling invariant generalized functions
on the real line are constant and we prove several new
characterizations of locally constant generalized functions.
\end{small}
\newpage
\newpage
\thispagestyle{empty}
\newpage
\newpage
%

\chapter*{Preface}
The present thesis represents my research work 2004-2006 in the
field of generalized functions carried out under the supervision of
Professor Michael Kunzinger at the Faculty of Mathematics,
University of Vienna. All sections in this book have been the basis
for scientific papers. For references concerning publication of this
material I refer to the arxiv, where all of my submitted papers can
be found, along with updated information concerning their
publication status.

Vienna, February 2008 \aufm{Eberhard Mayerhofer}

\mainmatter \tableofcontents
\chapter{Introduction}\label{chapterintro}
Differential algebras of generalized functions in the sense of J. F.
Colombeau provide a rigorous setting for treating numerous problems
for which a general concept of multiplication of distributions is
needed. Popular examples of such include partial
differential equations with singular (in the sense of non-smooth,
say distributional) data or coefficients: A sensible theory must admit singular solutions of the latter,
therefore it is necessary to introduce a product of singular objects
(in our example a singular coefficient times a singular solution);
however, it is advisable to do this in a consistent way, meaning
that on reasonable function subspaces of the distributions, the
usual point-wise product coincides with such a  product of
singular objects, and is associative and commutative.

Many counterexamples support that on $\mathcal D'$ such a product
with values in $\mathcal D'$ cannot exist (cf. \cite{Bible}, chapter 1). Let us consider
the following: assume we were given an associative product $\circ$ on $\mathcal D'$ and let
$\vp(1/x)$ denote the principal value of $1/x$, then we
would have
\[
\delta=\delta\circ(x \circ \vp(1/x))=(\delta\circ x)\circ
\vp(1/x)=0,
\]
which is impossible, since $\delta\neq 0$. Apart from certain "irregular (intrinsic or extrinsic) operations"
(cf.\ \cite{MObook}), there are basically two ways out of this dilemma:
\begin{enumerate}
\item We could restrict ourselves to strict subspaces of $\mathcal D'$ which have a natural algebraic structure, for instance
Sobolev spaces $\mathcal H^s(\mathbb R^n)$ for $s>n/2$,
$L_{loc}^\infty,\, C^k$ etc., and
\item\label{dilemmaexit2} we could try to embed $\mathcal D'$ into a larger space $\mathcal G$ which can be endowed with the structure of a differential algebra.
\end{enumerate}

Since we want to multiply distributions unrestrictedly, we shall settle for (\ref{dilemmaexit2}).

First we formulate the desired properties of a differential algebra
$(\mathcal G, \circ, +)$ containing the distributions. Let
$\Omega\subset \mathbb R^s$ open. We wish to construct an
associative, commutative algebra $({\mathcal G},+,\circ)$ such that:
\begin{enumerate}
\item\label{eig1alg} There exists a linear embedding $\iota:{\mathcal D}' \hookrightarrow {\mathcal G}$ such that
$\iota(1)$ is the unit in ${\mathcal G}$.
\item There exist derivation operators $D_i: {\mathcal G} \to
{\mathcal G}$ ($1\leq i \leq s$), which are linear and satisfy the
Leibniz-rule.
\item $D_i \mid_{{\mathcal D}'} = \frac{\partial}{\partial x_i}$
($1\le i \le s$), that is the derivation operators restricted to
$\mathcal D'$ are the usual partial derivations.
\item \label{propfour} $\circ \mid_{{\mathcal C}^\infty(\Omega)\times
{\mathcal C}^\infty(\Omega)}$ is the point-wise product of
functions.
\end{enumerate}
Item (\ref{propfour}) corresponds to the above requirement that the new
product should coincide with the usual point-wise product on a
"reasonable" subspace of $\mathcal D'$. Schwartz's famous
impossibility result (\cite{Schw1}) states that
such an algebra, does not exist, if the requirement (\ref{propfour}) is replaced by the respective requirement
for continuous functions. Nevertheless J. F. Colombeau
successfully constructed differential algebras $(\mathcal
G,+,\circ)$ satisfying (\ref{eig1alg})--(\ref{propfour}) (\cite{Colombeau,C}). Meanwhile there are
a number of such algebras of generalized functions. For a general construction scheme,
cf. \cite{Bible}. In the following subsection we
explain how the so-called {\it special version} on open sets of
$\mathbb R^n$ is constructed. We then may introduce the special
algebra on manifolds and we shall discuss its relevance for
applications in general relativity. The chapter will end with an
introduction to point-value concepts in algebras of generalized
functions.



\subsection*{Colombeau's special algebra}\label{specialeuclidean}
Let $\Omega$ be an open subset of $\mathbb R^d$. The so-called {\it special Algebra}\footnote{In the literature the special algebra is often denoted by $\mathcal G^s$ (with the aim to distinguish it from other Colomebau algebras), however, since we only work in the special algebra we shall omit the index $s$ throughout.} due to J.\ F.\ Colombeau is given by the quotient
\[
\mathcal G(\Omega):=\mathcal E_M(\Omega) /\mathcal N(\Omega),
\]
where the (ring of) moderate functions $\mathcal E_M(\Omega) $ resp.\ the ring of negligible elements (being an ideal
in $\mathcal E_M(\Omega) $) are given by
\begin{align}\nonumber
\!\!\mathcal E_M(\Omega) \!\!&:=&\!\! \{(u_\varepsilon)_\varepsilon\in C^{\infty}(\Omega)^{(0,1]} | \forall K\subset\subset \Omega \,\forall \alpha
\, \exists\, N 
\ \sup_{x\in K} |\partial^\alpha u_\varepsilon(x)| = O(\varepsilon^{-N})\}\\\nonumber
\!\!\mathcal N(\Omega) \!\!&:=&\!\! \{(u_\varepsilon)_\varepsilon\in C^{\infty}(\Omega)^{(0,1]} | \forall K\subset\subset \Omega \,\forall \alpha
\, \forall\, m 
\ \sup_{x\in K} |\partial^\alpha u_\varepsilon(x)| = O(\varepsilon^{m})\}.
\end{align}
The algebraic operations ($+,\circ$) as well as (partial) differentiation,
composition of functions etc. are meant to be performed
component-wise on the level of representatives; the transfer to the
quotient $\mathcal G(\Omega)$ is then well defined (cf.\ the
comprehensive presentation in the first chapter of \cite{Bible}). Once a
Schwartz mollifier $\rho$ on $\mathbb R^d$ with all moments vanishing has
been chosen, the space of compactly supported distributions may be
embedded into $\mathcal G(\Omega)$ via convolution; an
embedding of all of $\mathcal D'(\Omega)$ into our algebra is
achieved via a partition of unity using sheaf
theoretic arguments, therefore being not canonical. 
\section[Generalized functions and applications to relativity]{Algebras of generalized functions on manifolds and applications in general
relativity} The aim of this section is to review the basics of the special algebra on
manifolds $X$ as well as the definitions of generalized sections of vector bundles
with base space $X$ and we recall the definition of generalized pseudo-Riemannian
metrics. At the end of the section we motivate the use of
differential algebras for applications in relativity, in particular
for the wave equation on singular space-times which is treated in
the present book.
\subsection{The special algebra on manifolds}
Similarly as in section \ref{specialeuclidean} one may define algebras of generalized functions on manifolds. We start first by 
introducing the special algebra on manifolds in a coordinate
independent way as in \cite{K}. However, for two reasons we shall
later translate the definitions into respective definitions in terms of coordinate expressions:
For for the sake of clarity and simplicity, but also for the following
purpose: In chapter \ref{chapterwaveeq} we shall perform estimates
in a coordinate patch in order to derive a (local) existence result
for the Cauchy problem of the wave equation in a generalized
setting. 

The material presented here stems from the original sources
\cite{K, KS}. For a comprehensive presentation we refer to the--meanwhile standard reference on generalized function algebras
-- \cite{Bible}. Moreover, for further works in geometry based
on Colombeau's ideas we refer to (\cite{GlobTh,KKo1,KU2,3MikesV,KS,KSV,GenConKSV}).

For what follows in this section, $X$ shall denote a paracompact,
smooth Hausdorff manifold of dimension $n$ and by $\mathcal P(X)$ we
denote the space of linear differential operators on $X$. The
special algebra of generalized functions on $X$ is constructed as
the quotient $\mathcal G(X):=\mathcal E_M(X)/\mathcal N(X)$, where
the ring of moderate (resp.\ negligible) functions is given by
\begin{align}\nonumber
\mathcal E_M(X):=\{(u_{\varepsilon})_{\varepsilon}\in (C^{\infty}(X))^I\mid\forall\;K\subset\subset X\;\forall\;P\in\mathcal P(X)\;\exists\;N\in\mathbb N:\\\sup_{x\in K}\vert Pu_{\varepsilon}\vert=O(\varepsilon^{-N})\,(\varepsilon\rightarrow 0)
\end{align}
resp.\
\begin{align}\nonumber
\mathcal N(X):=\{(u_{\varepsilon})_{\varepsilon}\in (C^{\infty}(X))^I\mid\forall\;K\subset\subset X\;\forall\;P\in\mathcal P(X)\;\forall\;m\in\mathbb N:\\\sup_{x\in K}\vert Pu_{\varepsilon}\vert=O(\varepsilon^m)\,(\varepsilon\rightarrow 0).
\end{align}
The $C^{\infty}$ sections of a vector bundle $(E,X,\pi)$ with base
space $X$ we denote by $(E,X,\pi)$. Moreover, let $\mathcal
P(X,E)$ be the space of linear partial differential operators acting
on $\Gamma(X,E)$. The $\mathcal G(X)$ module of generalized
sections $\Gamma_{\mathcal G}(X,E)$ of a vector bundle
$(E,X,\pi)$ on $X$ is defined similarly as (the algebra of
generalized functions on $X$) above, in that we use asymptotic
estimates with respect to the norm induced by some arbitrary
Riemannian metric on the respective fibers, that is, we define the
quotient
\[
\Gamma_{\mathcal G}(X,E):=\Gamma_{\mathcal E_M}(X,E)/\Gamma_{\mathcal N}(X,E),
\]
where the ring (resp.\ ideal) of moderate (resp.\ negligible) nets of sections is given by
\begin{align}\nonumber
\Gamma_{\mathcal E_M}(X,E):=\{(u_{\varepsilon})_{\varepsilon}\in
(\Gamma(X,E))^I\mid\forall\;K\subset\subset X\;\forall\;P\in\mathcal
P(X, E)\;\exists\;N\in\mathbb N:\\\sup_{x\in
K}\|Pu_{\varepsilon}\|=O(\varepsilon^N)\,(\varepsilon\rightarrow 0)
\end{align}
resp.\
\begin{align}\nonumber
\Gamma_{\mathcal N}(X,E):=\{(u_{\varepsilon})_{\varepsilon}\in
(\Gamma(X,E))^I\mid\forall\;K\subset\subset X\;\forall\;P\in\mathcal
P(X, E)\;\forall\;m\in\mathbb N:\\\sup_{x\in
K}\|Pu_{\varepsilon}\|=O(\varepsilon^m)\,(\varepsilon\rightarrow 0).
\end{align}
In this book we shall deal with generalized sections of the tensor
bundle $\mathcal T^{r}_{s}(X)$ over $X$, this we denote by
\[
\mathcal G^{r}_{s}(X):=\Gamma_{\mathcal G}(X,\mathcal T^{r}_{s}(X)).
\]
Elements of the latter we call {\it generalized tensors of type
$(r,s)$}. 
We end this section by translating the global description of generalized vector bundles
in terms of coordinate expressions. Following the notation of \cite{KS}, we denote by
$(V,\Psi)$ a vector bundle chart over a chart $(V,\psi)$ of the base $X$. With $\mathbb R^{n'}$, the typical fibre,
we can write:
\[
\Psi:\pi^{-1}(V)\rightarrow\psi(V)\times \mathbb R^{n'},
\]
\[
z\mapsto(\psi(p),\psi^1(z),\dots,\psi^ {n'}(z)).
\]
Let now $s\in\Gamma_{\mathcal G}(X,E)$. Then the local expressions of $s$, $s^i=\Psi^i\circ s\circ \psi^{-1}$ lie
in $\mathcal G(\psi(V))$. 

An equivalent "local definition" of generalized vector bundles can be achieved by 
defining moderate nets $(s_\varepsilon)_\varepsilon$ of smooth sections $s_\varepsilon$
to be such for which the local expressions $s_\varepsilon^i=\Psi^i\circ s_\varepsilon\circ \psi^{-1}$
are moderate, that is $(s_\varepsilon^i)_\varepsilon\in\mathcal E_M(\psi(V))$. The notion negligible is defined completely similar.
The proof of this fact can be achieved by using Peetre's theorem (cf.\ \cite{Bible}, p. 289).
\subsection{Generalized pseudo-Riemannian geometry}\label{introducerepseudoriemannereetconnexione}
We begin with recalling the following characterization of non-degenerateness of symmetric (generalized) tensor fields of type (0,2) on $X$ (\cite{KS1}, Theorem 3.\ 1). For
a characterization of invertibility of generalized functions we refer to
Proposition 2.\ 1 of \cite{KS1} and for a further characterization we refer to the appendix of chapter \ref{chaptercausality} (namely Theorem \ref{downhilliseasier}).                                          
\begin{theorem}\label{chartens02}
Let $g\in \mathcal G^0_2(X)$. The following are equivalent:
\begin{enumerate}
\item \label{chartens021} For each chart $(V_{\alpha},\psi_{\alpha})$ and each $\widetilde x\in (\psi_{\alpha}(V_{\alpha}))^{\sim}_c$ the map
$g_{\alpha}(\widetilde x): \widetilde{\mathbb R}^n\times \widetilde{\mathbb R}^n\rightarrow \widetilde{\mathbb R}$ is symmetric and non-degenerate.
\item $g: \mathcal G^0_1(X)\times \mathcal G^0_1(X)\rightarrow \mathcal G(X)$ is symmetric and $\det (g)$ is invertible
in $\mathcal G(X)$.
\item \label{chartens023} $\det g$ is invertible in $\mathcal G(X)$ and for each relatively compact open set $V\subset X$ there exists a representative $(g_{\varepsilon})_{\varepsilon}$ of $g$ and $\varepsilon_0>0$ such that $g_{\varepsilon}\mid_V$ is a smooth pseudo-Riemannian metric for all $\varepsilon<\varepsilon_0$.
\end{enumerate}
\end{theorem}
Furthermore, the index of $g\in \mathcal G^0_2(X)$ is introduced in the following well defined way
(cf. Definition 3.\ 2 and Proposition 3.\ 3 in \cite{KS1}):
\begin{definition} \label{defpseud}
Let $g\in \mathcal G^0_2(X)$ satisfy one (hence all) of the equivalent conditions in Theorem \ref{chartens02}. If there exists some $j\in\mathbb N$ with the property that for each relatively compact open set $V\subset X$ there exists a representative $(g_{\varepsilon})_{\varepsilon}$ of $g$ as in Theorem \ref{chartens02} (\ref{chartens023}) such for each $\varepsilon<\varepsilon_0$ the index of $g_{\varepsilon}$ is equals $j$ we say $g$ has index $j$. Such symmetric 2-forms we call generalized pseudo-Riemannian metrics on $X$.
\end{definition}
We shall work in generalized space-times. These are pairs $(\mathcal
M, g)$, where $\mathcal M$ is an orientable paracompact four dimensional smooth
manifold and $g$ is a symmetric generalized (0,2) tensor with
invertible $\det g$ (cf.\ Theorem \ref{chartens02}) and index
$\nu=1$. In chapter \ref{chaptercausality} we develop algebraic
foundations of generalized Lorentz geometry; here the emphasis lies
on considering Lorentz metrics from a generalized point of view and
to develop causality notions in the generalized context. In the
subsequent chapter \ref{chapterwaveeq} we use the so found new concepts to
define and work with space-time symmetries, namely (smooth) time-like
Killing vector fields $\xi$ with respect to a generalized metric $g$ (cf.\
Definition \ref{defstaticgen} and the subsequent elaboration).

We end this section with reviewing the notion of generalized
connections and curvature (\cite{KS}, section 5).

A generalized connection $\hat D$ is a mapping $\mathcal G^1_0(\mathcal M)\times \mathcal G^1_0(\mathcal M)\rightarrow \mathcal G^1_0(\mathcal M)$
satisfying (for the notion of generalized Lie derivative, cf. \cite{KS})
\begin{enumerate}
\item $\hat D_\xi\eta$ is $\widetilde{\mathbb R}$--linear in $\eta$,
\item $\hat D_\xi\eta$ is $\mathcal G(\mathcal M)$--linear in $\xi$ and
\item $\hat D_\xi(u\eta)=u\hat D_\xi\eta+\xi(u)\eta$ for all $u$ in $\mathcal G(\mathcal M)$.\\
In analogy with the standard pseudo-Riemannian geometry, the connection is unique provided the following additional conditions are satisfied (cf.\ \cite{KS}, Theorem 5.2). For arbitrary $\xi,\eta,\zeta\in\mathcal G^1_0(\mathcal M)$ we have:
\item $[\xi,\eta]=\hat D_\xi \eta-\hat D_\eta \xi$ and
\item $\xi g(\eta,\zeta)=g(D_\xi\eta,\zeta)+g(\eta, D_\xi\zeta)$.
\end{enumerate}

In terms of coordinate expressions, the connection can be written down by means of "generalized" Christoffel symbols: Assume we are given a chart
$(V_\alpha,\psi_\alpha)$ on $\mathcal M$ with coordinates $x^i$ $(i=1,\dots,4)$. The Christoffel symbols are generalized functions
$\Gamma_{ij}^k\in\mathcal G(V_\alpha)$ defined by
\[
\hat D_{\partial_i}\partial_j=\Gamma_{ij}^k\partial_k,\qquad 1\leq i,j\leq n.
\]
\subsection{Generalized function concepts in general relativity} Even\\
though sufficient motivation to study the Cauchy problem of the
wave equation on a space-time whose metric is of lower
differentiability may emerge from a purely mathematical
interest, our original motivation actually stems from physics. The aim
of this section is to answer the following two questions: "Why do we
intend to solve the wave equation on a singular space-time" and,
"Why do we employ generalized function algebras for this matter?".

The field of general relativity is a non-linear theory, in the sense
that the curvature depends non-linearly on the metric and its
derivatives. This results in several problems when one comes to consider
the concept of singularities in space-times:
\begin{enumerate}
\item \label{clarke1} Firstly, from a
mathematical point of view, an immediate problem when a singular
space--time is modeled by means of a distributional metric, is: In
the coordinate formula for the Christoffel symbols (hence in the
formula for the curvature), products of the metric coefficients and
their derivatives occur, and a (distributional) meaning has to be
given to the latter. As outlined above, this is not always possible
in the framework of distributions, because they form a linear theory
( cf.\ the discussion at the beginning of the chapter).
\item \label{clarke2} The second natural obstacle is the difficulty of
distinguishing "strong" singularities from "weak" singularities.
Singularities were originally defined as endpoints of incomplete
geodesics, which could not be extended such that the
differentiability of the resulting space-time remained $C^{2-}$
(cf.\ Hawking and Ellis, \cite{HE}). The class of singularities
defined in this manner unfortunately includes both genuine
gravitational singularities such as Schwarzschild and "weaker"
singularities as in conical space-times, impulsive gravitational
waves and shell crossing singularities. A recent idea put forward by
C.\ J.\ S.\ Clarke in (\cite{Clarke}) supports a new concept of
"weak" singularities: A singularity in a space-time should only be
considered essential if it disrupts the evolution of linear test
fields. According to this idea, Clarke calls a space-times
generalized hyperbolic, if the Cauchy problem for the scalar wave
equation is well posed, and then shows that space-times with locally
integrable curvature are in this class.
\end{enumerate}

Vickers and Wilson are the first authors who apply Clarke's concepts
by showing that conical space-times are generalized hyperbolic
(cf.\ \cite{VW}; in the context of generalized function algebras this is called
 $\mathcal G$--generalized hyperbolic). To further overcome obstacle (\ref{clarke1}) in a
mathematically rigorous way, they reformulate the Cauchy problem in
the full Colombeau algebra. Finally, they show that the resulting generalized
solution is associated with a distributional solution (this can be
done by considering weak limits with respect to the smoothing parameter $\varepsilon$, cf. the definitions given in section
\ref{scalinginvariance}).

We shall follow Vickers' and Wilson's approach and try to generalize
their result to a wide range of generalized space-times in chapter \ref{chapterwaveeq}. However, it
should be noted that contrary to \cite{VW}, we work in the {\it
special algebra} exclusively. Moreover, the technique we are using
(based on certain energy integrals and Sobolev norms) lies somewhere
between Hawking and Ellis' method (\cite{HE}) and Vickers' and
Wilson's.

For more information on the use of generalized function algebras in
relativity we refer to the recent review \cite{SV} on this topic by
R.\ Steinbauer and J.\ Vickers as well as J. Vickers's article (\cite{VickersESI} pp.\
275--290) and the introduction to \cite{VW}. For relativistic applications 
in the framework of Colombeau's theory, see \cite{CVW, Bible,fivepeople}.

\section[Uniqueness issues]{Uniqueness issues in algebras of generalized functions}
The last chapter of the present work consists of three different
problems which we have summarized under the title "point values and uniqueness questions in algebras of generalized functions".
Even though the problems are quite different, they all have to do with
the basic question: "given two generalized functions $f,g$, how can we decide
if $f=g$?" It is clear that we can reduce this to the problem of determining
whether a generalized function $h$ vanishes identically. Before we come to a possible answer
offered by M.\ Kunzinger and M. Oberguggenberger in \cite{MO1} in form of a "uniqueness test" via evaluation of generalized functions on so-called
compactly supported points, we motivate the problem from the distributional point of view.

By definition, a distribution $w\in \mathcal D'$ is zero if the test with
arbitrary test functions $\phi$ yields $\langle w,\phi\rangle=0$.
The question, reformulated in the context of the special algebra,
reads, "is the embedded object $\iota(w)\in\mathcal G$ identically
zero?". 

However, since the key idea of embedding distributions into $\mathcal G$ is regularization
of the latter, we shall leave aside the embedding and answer this question for regularized nets of
distributions in terms of the following characterization:
\begin{theorem}\label{MOdistcase}
Let $u\in\mathcal D'(\mathbb R^n)$ and let $\rho\in\mathcal
D(\mathbb R^n)$ be a standard mollifier, that is, with $\int
\rho(x)\;dx^n=1$ and let
$\rho_\varepsilon(x):=\frac{1}{\varepsilon^n}\rho(\frac{x}{\varepsilon})$.
The following are equivalent:
\begin{enumerate}
\item \label{chardist1} $u=0$ in $\mathcal D'(\mathbb R^n)$.
\item \label{chardist2} For each compactly supported net $(x_{\varepsilon})_{\varepsilon}\in(\mathbb R^n)^{(0,1]}$ we have
\[
(u\ast\rho_{\varepsilon})(x_{\varepsilon})\rightarrow
0\qquad\mbox{if}\qquad \varepsilon\rightarrow 0.
\]
\end{enumerate}
\end{theorem}
\begin{proof}
The implication (\ref{chardist1})$\Rightarrow$(\ref{chardist2}) is
obvious, since for each $\varepsilon>0$ and each $x\in\mathbb R^n$, $\rho_{\varepsilon}\ast
u(x)=\langle u(y),\rho_{\varepsilon}(\frac{x-y}{\varepsilon})\rangle=0$. To show the converse
direction, assume $u\neq 0$ but that (\ref{chardist2})
holds. Then there exists $\phi\in\mathcal D(\mathbb R^n)$ such that
$\langle u,\phi\rangle\neq 0$. It follows that there exists a
positive constant $C_1$ and an index $\varepsilon_0\in (0,1]$ such
that for each $\varepsilon<\varepsilon_0$ we have
\begin{equation}\label{estsm}
\left|\int (u\ast \rho_{\varepsilon})\, \phi\; dx^n\right|\geq C_1.
\end{equation}
Therefore there exist a sequence $\varepsilon_k\rightarrow 0$ in
$(0,1]$, a compactly supported sequence $x_{\varepsilon_k}\in\mathbb
R^n$ and a positive number $C$ such that for each $k\geq 1$ we have
\begin{equation}\label{subclaim}
\vert u\ast \rho_{\varepsilon_k}(x_{\varepsilon_k})\vert\geq C.
\end{equation}
Indeed, if we assume the contrary, then for each set
$K\subset\subset\mathbb R^n$ we would have $\sup_{x\in K}\vert
u\ast\rho_{\varepsilon}\vert\rightarrow 0$ whenever $\varepsilon\rightarrow
0$. Fix $K$ such that $\supp\phi\subseteq K$. Then we have
\[
\left| \int (u\ast\rho_{\varepsilon})\,\phi\;dx^n\right|\leq
\vol(K)\|\phi\|_{\infty}
\|u\ast\rho_{\varepsilon}\|_{K,\infty}\rightarrow 0
\]
whenever $\varepsilon\rightarrow 0$, a contradiction to
(\ref{estsm}).

Finally define $(x_{\varepsilon})_{\varepsilon}$ as follows:
$x_{\varepsilon}:=x_{\varepsilon_k}$ whenever
$\varepsilon\in(\varepsilon_{k+1},\varepsilon_k]$ ($k\geq 1)$ and
$x_{\varepsilon}:=x_{\varepsilon_1}$ when $\varepsilon\in (x_1,1]$.
By (\ref{subclaim}) we have a contradiction to our assumption.
Therefore $u=0$ and we are done.
\end{proof}
It is further evident that, in the above characterization, (\ref{chardist2}) cannot be replaced
by the condition\\
{\it For each } $x \in\mathbb R^n$ {\it we have}
\[
(u\ast\rho_{\varepsilon})(x)\rightarrow
0\quad\mbox{whenever}\quad \varepsilon\rightarrow 0.
\]
To see this, take a standard mollifier $\rho$ with support $\supp \rho=[0,1]$. Then for each $x$ there exists
an index $\varepsilon_0$ such that $\rho_\varepsilon(x)=0$ for each $\varepsilon<\varepsilon_0$. But for $\varepsilon\rightarrow 0$ we have
\[
\rho_\varepsilon\rightarrow\delta \qquad \mbox{in}\qquad \mathcal
D'.
\]
We go on now by showing how these ideas are elaborated in the
context of the special algebra:
\subsubsection{The generalized point values concept}\label{pointissue}
Generalized functions can be evaluated at standard points. To be more precise, let us introduce
the ring of generalized numbers $\widetilde{\mathbb R}$, defined by the quotient
\[
\widetilde{\mathbb R}:=\mathcal E_M/\mathcal N,
\]
where the ring of moderate numbers
\[
\mathcal E_M := \{(x_\varepsilon)_\varepsilon\in \mathbb R^{(0,1]} :
\exists\, N:\; | x_\varepsilon| = O(\varepsilon^{-N})\}.
\]
 Similarly the ideal of negligible numbers $\mathcal N$ in $\mathcal E_M$ is given by
\[
\mathcal N := \{(x_\varepsilon)_\varepsilon\in \mathbb R^{(0,1]} : \forall\, m:\; | x_\varepsilon| = O(\varepsilon^{m})\}.
\]
Let $\widetilde{\mathbb R}_c$ denote the set of compactly
supported elements of $\widetilde{\mathbb R}$, that is: $x_c$ lies
in $\widetilde{\mathbb R}_c$ if and only if there exists a compact
set $K\subseteq \mathbb R$  such that for one (hence any)
representative $(x_\varepsilon)_\varepsilon$ there exists an index
$\varepsilon_0$ such that for all $\varepsilon<\varepsilon_0$ we
have $x_\varepsilon\in K$. It can easily be shown that evaluation of
generalized functions $f$ on compactly supported generalized points
makes perfect sense in the following way: let $(f_\varepsilon)_\varepsilon$ be a representative of
 $f\in\mathcal G(\mathbb R)$, then
\[
\widetilde f(x_c):=(f_\varepsilon(x_\varepsilon))_\varepsilon+\mathcal
N\in\widetilde{\mathbb R}
\]
yields a well defined generalized number. We denote by $\widetilde
f:\widetilde{\mathbb R}_c\rightarrow\widetilde{\mathbb R}$ the above map induced by the generalized function $f$.

By a standard point $x$ we shall mean an element of
$\widetilde{\mathbb R}$ which admits a constant representative, i.\
e.\ $x=(\alpha)_\varepsilon+\mathcal N$ for a certain real number
$\alpha$. M.\ Kunzinger and M.\ Oberguggenberger show in
(\cite{MO1}) that it does not suffice to know the values of
generalized functions at standard points in order to determine them
uniquely. Furthermore, the following analog of Theorem
\ref{MOdistcase} holds:
\begin{theorem}\label{distingprop}
Let $f\in\mathcal G(\mathbb R)$. The following are equivalent:
\begin{enumerate}
\item $f=0$ in $\mathcal G(\mathbb R)$,
\item $\forall\; x_c\;\in\widetilde{\mathbb R}_c:\; \widetilde f(x_c)=0$.
\end{enumerate}
\end{theorem}
Note that a similar statement holds in Egorov algebras (cf.\ the final remark in \cite{MO1}).
In the first section of chapter \ref{chapterpointvalues} we show that also in $p$-adic Egorov algebras such a
characterization holds and that evaluation at standard points does not suffice to determine elements of such algebras uniquely.
In section \ref{sectionsharp} we elaborate a topological question in the ring of generalized numbers $\widetilde{\mathbb R}$
endowed with the so-called sharp topology. Finally, in the end of chapter \ref{chapterpointvalues}, we apply some new differential calculus on $\widetilde{\mathbb R}$ due to Aragona (\cite{A2}) for showing that the only scaling invariant functions on the real line are the constants.

\chapter{Algebraic foundations of Colombeau Lorentz geometry}\label{chaptercausality}

In the course of chapter \ref{chapterwaveeq} we shall establish a
local existence and uniqueness theorem for the Cauchy problem of the
wave equation in a generalized context. The considerations we had to
undertake to achieve this result showed that a generalized concept
of causality might be useful to describe scenarios in a non-smooth
space-time without always having to deal merely with the standard
concepts component-wise on the level of representatives. However,
also from a purely theoretical point of view, the need of a such a
concept becomes clear: the non-standard aspect in Colombeau theory,
which gives rise to a description of objects not point-wise but on
so-called generalized points (cf.\ \cite{MO1} and chapter
\ref{chapterpointvalues}). This has been taken up in the recent and
initial work by M.\ Kunzinger and R.\ Steinbauer on generalized
pseudo-Riemannian geometry (\cite{KS1}), on which we base our
considerations (cf.\  the assumptions on the metric in section
\ref{settingassumptions}), but it has not yet been investigated to a
wide extent. For instance, invertibility of generalized functions
has been characterized (cf.\ \cite{KS1}, Proposition 2.\ 1) and
allowed a notable characterization of symmetric generalized
non-degenerate $(0,2)$ forms (cf. Theorem \ref{chartens02}). But so far
there has not been given a characterization of generalized
pseudo-Riemannian metrics $h$ in terms of bilinear forms $\widetilde h$ stemming
from evaluation of $h$ at compactly supported points (on the
respective manifold).


The main aim of this chapter, therefore, is to describe and discuss
some elementary questions of generalized pseudo-Riemannian geometry
under the aspect of generalized points. Our program is as follows:
Introducing the index of a symmetric bilinear form on the
$n$-dimensional module $\widetilde{\mathbb R}^n$ over the
generalized numbers $\widetilde{\mathbb R}$ enables us to define the
appropriate notion of a bilinear form of Lorentz signature. We can
therefore propose a notion of causality in this context. The general
statement of the inverse Cauchy-Schwartz inequality is then given. We
further show that a dominant energy condition in the sense of
Hawking and Ellis for generalized energy tensors (such as also
indirectly assumed in \cite{VW}) is satisfied. We also answer the
algebraic question: "Does any submodule in $\widetilde{\mathbb R}^n$
have a direct summand?": For free submodules, the answer is positive
and is basically due to a new characterization of free elements in
$\widetilde{\mathbb R}^n$. In general, however, direct summands do
not exist: $\widetilde{\mathbb R}^n$ is not semisimple. In the end of the chapter we present a
new characterization of invertibility in algebras of generalized functions.
Finally, we want to point out that the positivity
issues on the ring of generalized numbers treated here have links to papers by
M.\ Oberguggenberger et al.\ (\cite{MOHor,PSMO}).

\section{Preliminaries}
Let $I:=(0,1]\subseteq \mathbb R$, and let $\mathbb K$ denote $\mathbb R$ resp.\ $\mathbb C$. The ring of generalized numbers over $\mathbb K$ is constructed in the following way: Given the ring of moderate nets of numbers $\mathcal E(\mathbb K):=\{(x_{\varepsilon})_{\varepsilon}\in\mathbb K^I \mid \exists\; m:\vert x_{\varepsilon}\vert=O(\varepsilon^m)\,(\varepsilon\rightarrow 0)\}$ and, similarly, the ideal of negligible nets in $\mathcal E(\mathbb K)$ which are of the form $\mathcal N(\mathbb K):=\{(x_{\varepsilon})_{\varepsilon}\in\mathbb K^I\mid \forall\; m:\vert x_{\varepsilon}\vert=O(\varepsilon^m)\,(\varepsilon\rightarrow 0)\}$, we may define the generalized numbers as the factor ring $\widetilde{\mathbb K}:=\mathcal E_M(\mathbb K)/\mathcal N(\mathbb K)$. An element $\alpha\in\widetilde{\mathbb K}$
is called strictly positive if it lies in $\widetilde{\mathbb R}$ (this means that for any representative $(\alpha_{\varepsilon})_{\varepsilon}=(\Ree(\alpha_{\varepsilon}))_{\varepsilon}+ i(\Imm(\alpha_{\varepsilon}))_{\varepsilon}$ we have $(\Imm(\alpha_{\varepsilon}))_{\varepsilon}\in \mathcal N(\mathbb R)$)  and  if $\alpha$ has a representative $(\alpha_{\varepsilon})_{\varepsilon}$ such that there exists $m\geq 0$ such that $\Ree(\alpha_{\varepsilon})\geq \varepsilon^m$ for each $\varepsilon\in I=(0,1]$, we shall write $\alpha>0$.
Clearly any strictly positive number is invertible. $\beta\in\widetilde{\mathbb R}$ is called strictly negative, if $-\beta>0$. Note that
a generalized number $u$ is strictly positive precisely when it is invertible (due to \cite{KS1} Proposition 2.\ 2 this means that $u$ is strictly non-zero) and positive (i.\ e., $u$ has a representative $(u_{\varepsilon})_{\varepsilon}$ which is greater or equals zero for each $\varepsilon\in I$).
In the appendix to this chapter a new and somewhat surprising characterization of invertibility and strict positivity in the frame of the
special algebra construction is presented.

Let $A\subset I$, then the characteristic function $\chi_A\in\widetilde{\mathbb R}$ is given by the class of $(\chi_{\varepsilon})_{\varepsilon}$, where
\[
\chi_{\varepsilon}:=\begin{cases} 1,\qquad\mbox{if}\qquad \varepsilon\in A\\ 0, \qquad\mbox{otherwise}\end{cases}.
\]
Whenever $\widetilde{\mathbb R}^n$ is involved, we consider it as an $\widetilde{\mathbb R}$--module of dimension $n\geq 1$.
Clearly the latter can be identified with $\mathcal E_M(\mathbb R^n)/\mathcal N({\mathbb R^n})$, but we will not often use this fact
subsequently.
Finally, we denote by $\widetilde{\mathbb R}^{n^2}:=\mathcal M_n(\widetilde{\mathbb R})$ the ring of
 $n\times n$ matrices over $\widetilde{\mathbb R}$. A matrix $A$ is called orthogonal, if $UU^t=\mathbb I$ in $\widetilde{\mathbb R}^{n^2}$ and $\det U=1$ in $\widetilde{\mathbb R}$.
Clearly, there are two different ways to introduce $\widetilde{\mathbb R}^{n^2}$:
\begin{remark}\rm
Denote by $\mathcal E_M(\mathcal M_n(\mathbb R))$ the ring of moderate nets of
$n\times n$ matrices over $\mathbb R$, a subring of $\mathcal M_n(\mathbb R)^I$. Similarly let $\mathcal N(\mathcal M_n(\mathbb R))$
denote the ideal of negligible nets of real $n\times n$ matrices. There is a ring isomorphism
$\varphi: \widetilde{\mathbb R}^{n^2}\rightarrow \mathcal E_M(\mathcal M_n(\mathbb R))/\mathcal N(\mathcal M_n(\mathbb R))$.
\end{remark}
For the convenience of the reader we repeat Lemma 2.\ 6 from \cite{KS1}:
\begin{lemma}\label{nondeg}
Let $A\in\widetilde{\mathbb R}^{n^2}$. The following are equivalent:
\begin{enumerate}\label {charnondeg}
\item \label{charnondeg1} $A$ is non-degenerate, that is, $\xi \in \widetilde{\mathbb R}^n,\;\xi^t A\eta=0$ for each $\eta\in\widetilde{\mathbb R}^n$ implies $\xi=0$.
\item \label{charnondeg2} $A: \widetilde{\mathbb R}^n\rightarrow \widetilde{\mathbb R}^n$ is injective.
\item \label{charnondeg3} $A: \widetilde{\mathbb R}^n\rightarrow \widetilde{\mathbb R}^n$ is bijective.
\item\label{charnondeg4} $\det A$ is invertible in $\widetilde{\mathbb R}$.
\end{enumerate}
\end{lemma}
Note that the equivalence of (\ref{charnondeg1})--(\ref{charnondeg3}) and (\ref{charnondeg4}) results from the fact that in $\widetilde{\mathbb R}$ any nonzero non-invertible element is a zero-divisor. Since we deal with symmetric matrices throughout, we start by giving a basic
characterization of symmetry of generalized matrices:
\begin{lemma}\label{symmetry}
Let $A\in\widetilde{\mathbb R}^{n^2}$. The following are equivalent:

\begin{enumerate}
\item \label{asymmetry} $A$ is symmetric, that is $A=A^t$ in $\widetilde{\mathbb R}^{n^2}$.
\item \label{bsymmetry} There exists a symmetric representative $(A_{\varepsilon})_{\varepsilon}:=((a_{ij}^{\varepsilon})_{ij})_{\varepsilon}$ of $A$.
\end{enumerate}
\end{lemma}
\begin{proof}
Since (\ref{bsymmetry}) $\Rightarrow$ (\ref{asymmetry}) is clear, we only need to show (\ref{asymmetry}) $\Rightarrow$ (\ref{bsymmetry}). Let $((\bar a_{ij}^{\varepsilon})_{ij})_{\varepsilon}$ a representative of $A$. Symmetrizing yields the desired representative
\[
(a_{ij}^{\varepsilon})_{\varepsilon}:=\frac{ (\bar a_{ij}^{\varepsilon})_{\varepsilon}+(\bar a_{ji}^{\varepsilon})_{\varepsilon} } {2}
\]
of $A$. This follows from the fact that for each pair $(i,j)\in\{1,\dots,n\}^2$ of indices one has $(\bar a_{ij}^{\varepsilon})_{\varepsilon}-(\bar a_{ji}^{\varepsilon})_{\varepsilon} \in \mathcal N(\mathbb R)$ due to the symmetry of $A$.
\end{proof}
Denote by $\|\,\|_F$ the Frobenius norm on $\mathcal M_n(\mathbb C)$.
In order to prepare a notion of eigenvalues for symmetric matrices, we repeat a numeric result given in \cite{SJ} (Theorem 5.\ 2):
\begin{theorem}\label{perturbation}
Let $A\in\mathcal M_n(\mathbb C)$ be a Hermitian matrix with eigenvalues $\lambda_1\geq\dots\geq\lambda_n$. Denote by
$\widetilde A$ a non-Hermitian perturbation of $A$, i.\ e., $E=\widetilde A- A$ is not Hermitian. We further call the eigenvalues of $\widetilde A$ (which might be complex) $\mu_k+i\nu_k\;(1\leq k\leq n)$ where $\mu_1\geq\dots\geq\mu_n$. In this notation, we have
\[
\sqrt{ \sum_{k=1}^n\vert(\mu_k+i\nu_k)-\lambda_k \vert^2 }\leq\sqrt 2\| E\|_F.
\]
\end{theorem}
\begin{definition}\label{eigenvalues}
Let $A\in\widetilde{\mathbb R}^{n^2}$ be a symmetric matrix and let $(A_{\varepsilon})_{\varepsilon}$ be an arbitrary representative of $A$. Let for any $\varepsilon\in I$, $\theta_{k,\varepsilon}:=\mu_{k,\varepsilon}+i\nu_{k,\varepsilon}\;(1\leq k\leq n)$ be the eigenvalues of $A_{\varepsilon}$ ordered by the size of the real parts, i.\ e., $\mu_{1,\varepsilon}\geq\dots\geq\mu_{n,\varepsilon}$.
The generalized eigenvalues $\theta_k\in\widetilde{\mathbb C}\;(1\leq k\leq n)$ of $A$ are defined as the classes $(\theta_{k,\varepsilon})_{\varepsilon}+\mathcal N(\mathbb C)$.
\end{definition}
\begin{lemma}\label{schur}
Let $A\in\widetilde{\mathbb R}^{n^2}$ be a symmetric matrix. Then the eigenvalues $\lambda_k\;(1\leq k\leq n)$ of $A$ as introduced in Definition
\ref{eigenvalues} are well defined elements of $\widetilde{\mathbb R}$. Furthermore, there exists an orthogonal $U\in \widetilde{\mathbb R}^{n^2}$ such that
\begin{equation}\label{eqdecschur}
U A U^t=\diag (\lambda_1,\dots,\lambda_n).
\end{equation}
We call $\lambda_i\;(1\leq i\leq n)$ the eigenvalues of $A$. $A$ is non-degenerate if and only if all generalized eigenvalues are invertible.
\end{lemma}
Before we prove the lemma, we note that throughout the chapter we shall omit the term "generalized" (eigenvalues) and we shall call the generalized numbers constructed in the above way simply "eigenvalues" (of a generalized symmetric matrix).
\begin{proof}
Due to Lemma \ref{symmetry} we may choose a symmetric representative $(A_{\varepsilon})_{\varepsilon}=((a_{ij}^{\varepsilon})_{ij})_{\varepsilon}\in\mathcal E_M(\mathcal M_n(\mathbb R))$ of $A$ . For any $\varepsilon$, denote by $\lambda_{1,\varepsilon}\geq\dots\geq\lambda_{n,\varepsilon}$ the resp.\ (real) eigenvalues
of $(a_{ij}^{\varepsilon})_{ij}$ ordered by size. For any $i\in\{1,\dots,n\}$, define $\lambda_i:=(\lambda_{i,\varepsilon})_\varepsilon+\mathcal N(\mathbb R)\in\widetilde{\mathbb R}$. For the well-definedness of the eigenvalues of $A$, we only need to show that for any other (not necessarily symmetric) representative of $A$, the resp.\ net of eigenvalues lies in the same class of $\mathcal E_M(\mathbb C)$; note that the use of complex numbers is indispensable here. Let $
(\widetilde A_{\varepsilon})_{\varepsilon}=((\widetilde a_{ij}^{\varepsilon})_{ij})_{\varepsilon}$ be another representative of $A$. Denote by $\mu_{k,\varepsilon}+i\nu_{k+\varepsilon}$ the eigenvalues of $\widetilde A_{\varepsilon}$ for any $\varepsilon\in I$ such that
the real parts are ordered by size, i.\ e., $\mu_{1,\varepsilon}\geq\dots\geq \mu_{n,\varepsilon}$.
Denote by $(E_{\varepsilon})_{\varepsilon}:=(\widetilde A_{\varepsilon})_{\varepsilon}-(A_{\varepsilon})_{\varepsilon}$.
Due to Theorem \ref{perturbation} we have for each $\varepsilon\in I$:
\begin{equation}\label{nullboundsE}
\sqrt{\sum_{k=1}^n\vert(\mu_{k,\varepsilon}+i\nu_{k,\varepsilon})-\lambda_{k,\varepsilon}\vert^2 }\leq\sqrt 2\|E_{\varepsilon}\|_F.
\end{equation}
Since $(E_{\varepsilon})_{\varepsilon}\in\mathcal N(\mathcal M_n(\mathbb R))$, (\ref{nullboundsE}) implies for any $k\in\{1,\dots,n\}$ and any $m$,
\[
\vert(\mu_{k,\varepsilon}+i\nu_{k,\varepsilon})-\lambda_{k,\varepsilon}\vert=O(\varepsilon^m)\;(\varepsilon\rightarrow 0)
\]
which means that the resp.\ eigenvalues of $(A_\varepsilon)_\varepsilon$ and of $(\widetilde A_\varepsilon)_\varepsilon$ in the above order belong to the same class in $\mathcal E_M(\mathbb C)$. In particular they yield the same elements of $\widetilde{\mathbb R}$.
The preceding argument and Lemma \ref{symmetry} show that without loss of generality we may construct the eigenvalues of $A$ by means of a symmetric representative $(A_{\varepsilon})_{\varepsilon}=((a_{ij}^{\varepsilon})_{ij})_{\varepsilon}\in\mathcal E_M(\mathcal M_n(\mathbb R))$. For such a choice we have for any $\varepsilon$
an orthogonal matrix $U_{\varepsilon}$ such that
\[
U_{\varepsilon}A_{\varepsilon}U_{\varepsilon}^t=\diag(\lambda_{1,\varepsilon},\dots,\lambda_{n,\varepsilon}),\;\lambda_{1,\varepsilon}\geq\dots\geq\lambda_{n,\varepsilon}.
\]
Declaring $U$ as the class of $(U_{\varepsilon})_{\varepsilon}\in\mathcal E_M(\mathcal M_n(\mathbb R))$ yields the proof of the second claim,
since orthogonality for any $U_{\varepsilon}$  implies orthogonality of $U$ in $\mathcal M_n(\widetilde{\mathbb R})$. Finally, decomposition (\ref{eqdecschur}) gives, by applying the multiplication theorem for determinants and the orthogonality of $U$, $\det A=\prod_{i=1}^n \lambda_i$. This shows in conjunction with Lemma \ref{nondeg} that invertibility of all eigenvalues is a sufficient and necessary condition for the non-degenerateness of $A$ and we are done.
\end{proof}
\begin{remark}\rm
A remark on the notion eigenvalue of a generalized symmetric matrix $A\in\widetilde{\mathbb R}^{n^2}$ is in order: Since for any eigenvalue $\lambda$ of $A$
we have $\det (A-\lambda \mathbb I)=\det (U(A-\lambda \mathbb I)U^t)=\det ((UAU^t)-\lambda \mathbb I)=0$, Lemma \ref{charnondeg}
implies that $A-\lambda \mathbb I: \widetilde{\mathbb R}^n\rightarrow\widetilde{\mathbb R}^n$ is not injective. However, again by the same lemma, $\det (A-\lambda \mathbb I)=0$ is not necessary for $A-\lambda \mathbb I$ to be not injective, and a $\theta\in\widetilde{\mathbb R}$ for which $A-\theta I$ is not injective
need not be an eigenvalue of $A$. More explicitly, we give two examples of possible scenarios here:
\begin{enumerate}
\item  Let $\forall\;i\in\{1,\dots,n\}: \lambda_i\neq 0$ and for some $i$ let $\lambda_i$ be a zero divisor. Then besides
$A-\lambda_i\;(i=1,\dots,n)$, also $ A: \widetilde{\mathbb R}^n\rightarrow\widetilde{\mathbb R}^n$ fails to be injective.
\item  "Mixing" representatives of $\lambda_i,\lambda_j\; (i\neq j)$ might give rise to generalized numbers $\theta\in\widetilde{\mathbb R}, \theta\neq \lambda_j\,\forall j\in\{1,\dots,n\}$ for
which $A-\theta \mathbb I$ is not injective as well. Consider for the sake of simplicity the matrix $D:=\diag (1,-1)\in \mathcal M_2(\mathbb R)$.
A rotation $U_{\varphi}:=\left(\begin{array}{cc} \cos(\varphi)& \sin(\varphi)\\-\sin(\varphi)&\cos(\varphi) \end{array}\right)$ yields
by matrix multiplication
\[
U_{\varphi} D U_{\varphi}^t=\left(\begin{array}{cc} \cos(2\varphi)& -\sin(2\varphi)\\-\sin(2\varphi)&-\cos(2\varphi) \end{array}\right).
\]
\end{enumerate}
The choice of $\varphi=\pi/2$ therefore switches the order of the entries of $D$, i.\ e., $U_{\pi/2}DU_{\pi/2}^t=\diag(-1,1)$.
Define $U,\lambda$ as the classes of $(U_{\varepsilon})_{\varepsilon}, (\lambda_{\varepsilon})_{\varepsilon}$ defined by
\[
U_{\varepsilon}:=\begin{cases} I:\;\varepsilon\in I\cap\mathbb Q\\ U_{\pi/2}:\;\mbox{else} \end{cases},
\]
\[
\lambda_{\varepsilon}:=\begin{cases} 1:\;\varepsilon\in I\cap\mathbb Q\\ -1\;\mbox{else} \end{cases},
\]
further define $\mu\in\widetilde{\mathbb R} $ by $\mu+\lambda=0$.
Then we have for $A:= [(D)_{\varepsilon}]$:
\[
UDU^t=\diag(\lambda,\mu).
\]
\end{remark}
Therefore as shown above, $D-\lambda \mathbb I,\; D-\mu \mathbb I$ are not injective considered as maps $\widetilde{\mathbb R}^n\rightarrow\widetilde{\mathbb R}^n$.
But neither $\lambda$, nor $\mu$ are eigenvalues of $D$.
\begin{definition}\label{indexmatrixdef}
Let $A\in\widetilde{\mathbb R}^{n^2}$. We denote by $\nu_{+}(A)$ (resp.\ $\nu_{-}(A)$) the number of
strictly positive (resp.\ strictly negative) eigenvalues, counting multiplicity. Furthermore, if $\nu_{+}(A)+\nu_{-}(A)=n$,
we simply write $\nu(A):=\nu_{-}(A)$. If $A$ is symmetric and $\nu(A)=0$, we call $A$ a positive definite symmetric matrix. If $A$ is symmetric and $\nu_{+}(A)+\nu_{-}(A)=n$ and $\nu(A)=1$, we say $A$ is a symmetric $L$-matrix.
\end{definition}
The following corollary shows that for a symmetric non-degenerate matrix in $\widetilde{\mathbb R}^{n^2}$ counting $n$ strictly positive resp.\ negative eigenvalues is equivalent to having a (symmetric) representative for which any $\varepsilon$-component has the same number (total $n$) of positive resp.\ negative real eigenvalues. We skip the proof.
\begin{corollary}\label{charindexmatrix}
Let $A\in\widetilde{\mathbb R}^{n^2}$ be symmetric and non-degenerate and $j\in\{1,\dots,n\}$. The following are equivalent:
\begin{enumerate}
\item $\nu_{+}(A)+\nu_{-}(A)=n$, $\nu(A)=j$.
\item \label{charindexmatrix2} For each symmetric representative $(A_{\varepsilon})_{\varepsilon}$ of $A$ there exists some $\varepsilon_0 \in I$ such that for any $\varepsilon<\varepsilon_0$ we have for the eigenvalues
$\lambda_{1,\varepsilon}\geq\dots\geq\lambda_{n,\varepsilon}$  of $A_{\varepsilon}$:
\[
\lambda_{1,\varepsilon},\dots,\lambda_{n-j,\varepsilon}>0,\;\;\lambda_{n-j+1,\varepsilon},\dots,\lambda_{n,\varepsilon}<0.
\]
\end{enumerate}
\end{corollary}

%
\section{Causality and the inverse Cauchy-Schwarz inequality}
In a free module over a commutative ring $R\neq \{0\}$, any two bases have the same cardinality. Therefore, any free module $\mathfrak M_n$ of dimension $n\geq 1$ (i.\ e., with a basis having $n$ elements) is isomorphic to $R^n$ considered as module over $R$ (which is free, since it has the canonical basis). As a consequence we may confine ourselves to considering the module $\widetilde{\mathbb R}^n$ over $\widetilde{\mathbb R}$ and its submodules. We further assume that from now on $n$, the dimension of $\widetilde{\mathbb R}^n$, is greater than $1$. It is quite natural to start with an appropriate version of the Steinitz exchange lemma:
\begin{proposition}\label{steinitzprop}
Let $\mathcal B=\{v_1,\dots,v_n\}$ be a basis for $\widetilde{\mathbb R}^n$. Let $w=\lambda_1 v_1+\dots+\lambda_nv_n\in \widetilde{\mathbb R}^n$ such that for some $j\;(1\leq j\leq n)$,
$\lambda_j$ is not a zero divisor. Then, also $\mathcal B':=\{v_1,\dots,v_{j-1},w,v_{j+1},\dots,v_n\}$ is a basis for $\widetilde{\mathbb R}^n$.
\end{proposition}
\begin{proof}
Without loss of generality we may assume $j=1$, that is $\lambda_1$ is invertible. We we have to show that $\mathcal B':=\{w,v_2,\dots,v_n\}$ is a basis
for $\widetilde{\mathbb R}^n$. Assume we are given a vector $v=\sum_{i=1}^n\mu_iv_i\in \widetilde{\mathbb R}^n$, $\mu_i\in\widetilde{\mathbb R}$. Since $\lambda_1$ is invertible, we may write $v_1=\frac{1}{\lambda_1}w-\frac{\lambda_2}{\lambda_1}v_2-\dots-\frac{\lambda_n}{\lambda_1}v_n$. Thus we find
$ v=\frac{\mu_1}{\lambda_1}w+ \sum_{k=2}^n (\mu_k-\frac{\mu_1\lambda_k}{\lambda_1})v_k$, which proves that $\mathcal B'$ spans $\widetilde{\mathbb R}^n$. It remains to prove linear independence
of $\mathcal B'$:
Assume that for $\mu,\mu_2,\dots,\mu_n\in\widetilde{\mathbb R}$ we have $\mu w+\mu_2v_2+\dots+\mu_n v_n=0$. Inserting
$w=\sum_{i=1}^n\lambda_i v_i$ yields $\mu\lambda_1v_1+(\mu\lambda_2+\mu_2)v_2+\dots+(\mu\lambda_n+\mu_n)v_n=0$ and since $\mathcal B$ is a basis, it follows that $\mu\lambda_1=\mu\lambda_2+\mu_2=\dots=\mu\lambda_n\mu_n=0$. Now, since $\lambda_1$ is invertible, it follows that $\mu=0$. Therefore $\mu_2=\dots=\mu_n=0$ which proves that $w,v_1,\dots,v_n$ are linearly independent, and $\mathcal B'$ is a basis.
\end{proof}
\begin{definition}\label{bilformdef}
Let $b:\widetilde{\mathbb R}^n\times\widetilde{\mathbb R}^n\rightarrow \widetilde{\mathbb R}$ be a symmetric bilinear form on $\widetilde{\mathbb R}^n$. Let $j\in\mathbb N_0$. If for some basis $\mathcal B:=\{e_1,\dots,e_n\}$ of $\widetilde{\mathbb R}^n$ we have $\nu((b(e_i,e_j))_{ij})=j$ we call $j$ the index of $b$. If $j=0$ we say that $b$ is positive definite and if $j=1$ we call $b$ a symmetric bilinear form of Lorentz signature.
\end{definition}
Note that as in the classical setting, there is no notion of 'eigenvalues' of a symmetric bilinear form, since a change of coordinates that is not induced by an orthogonal matrix need not conserve the eigenvalues of the original coefficient matrix.
We are obliged to show that the notion above is well defined. The main argument is Sylvester's inertia law (cf.\ \cite{GF}, pp.\ 306):
\begin{proposition}
The index of a bilinear form $b$ on $\widetilde{\mathbb R^n}$ as introduced in Definition \ref{bilformdef} is well defined.
\end{proposition}
\begin{proof}
Let $\mathcal B$, $\mathcal B'$ be bases of $\widetilde{\mathbb R^n}$ and let $A$ be a matrix describing a linear map
which maps $\mathcal B$ onto $\mathcal B'$ (this map is uniquely determined in the sense that it only depends on the order of the basis vectors of the resp.\ bases). Let $B$ be the coefficient matrix of the given bilinear form $b$ and let
further $k:=\nu(B)$. The change of bases
results in a 'generalized' equivalence transformation of the form
\[
B\mapsto T:=A^tBA,
\]
$T$ being the coefficient matrix of $h$ with respect to $\mathcal B'$. We only need to show that $\nu(B)=\nu(T)$. Since the index of a matrix is well defined (and this again follows from Lemma \ref{schur}, where it is proved that the eigenvalues of a symmetric generalized matrix are well defined), it is sufficient to show that for one (hence any) symmetric representative $(T_{\varepsilon})_{\varepsilon}$ of $T$ there exists an $\varepsilon_0\in  I$ such that for each $\varepsilon<\varepsilon_0$ we have
\[
 \lambda_{1,\varepsilon}>0,\dots,\lambda_{n-k,\varepsilon}>0,\lambda_{n-k+1,\varepsilon}<0,\dots,\lambda_{n-k,\varepsilon}<0,
\]
where $(\lambda_{i,\varepsilon})_{\varepsilon}$ ($i=1,\dots,n$) are the ordered eigenvalues of $(T_{\varepsilon})_{\varepsilon}$. To this end, let $(B_{\varepsilon})_{\varepsilon}$ be a symmetric representative
of $B$, and define by $(T_{\varepsilon})_{\varepsilon}$ a representative of $T$ component-wise via
\[
T_{\varepsilon}:=A_{\varepsilon}^tB_{\varepsilon}A_{\varepsilon}.
\]
Clearly $(T_{\varepsilon})_{\varepsilon}$ is symmetric. For each $\varepsilon$ let
$\lambda_{1,\varepsilon}\geq\dots\geq\lambda_{n,\varepsilon}$ be the ordered eigenvalues of $T_{\varepsilon}$
and let $\mu_{1,\varepsilon}\geq\dots\geq\mu_{n,\varepsilon}$ be the ordered eigenvalues of $B_{\varepsilon}$.
Since $A$ and $B$ are non-degenerate, there exists some $\varepsilon_0\in I$ and an integer $m_0$ such that
for each $\varepsilon<\varepsilon_0$ and for each $i=1,\dots,n$ we have
\[
\vert \lambda_{i,\varepsilon}\vert\geq \varepsilon^{m_0}\qquad\mbox{and}\qquad \vert \mu_{i,\varepsilon}\vert\geq \varepsilon^{m_0}.
\]
Furthermore due to our assumption $k=\nu(B)$, therefore taking into account the component-wise order of the eigenvalues
$\mu_{i,\varepsilon}$, for each $\varepsilon<\varepsilon_0$ we have:
\[
\mu_{i,\varepsilon}\geq \varepsilon^{m_0}\;\;(i=1,\dots,n-k)\qquad \mbox{and}\qquad \mu_{i,\varepsilon}\leq -\varepsilon^{m_0}\;\;(i=n-k+1,\dots,n).
\]
As a consequence of Sylvester's inertia law we therefore have for each $\varepsilon<\varepsilon_0$:
\[
\lambda_{i,\varepsilon}\geq \varepsilon^{m_0}\;\; (i=1,\dots,n-k)\qquad \mbox{and}\quad \lambda_{i,\varepsilon}\leq -\varepsilon^{m_0}\;\; (i=n-k+1,\dots,n),
\]
since for each $\varepsilon<\varepsilon_0$ the number of positive resp.\ negative eigenvalues of $B_{\varepsilon}$
resp.\ $T_{\varepsilon}$ coincides. We have thereby shown that $\nu(T)=k$ and we are done.
\end{proof}
\begin{definition}
Let $b:\widetilde{\mathbb R}^n\times\widetilde{\mathbb R}^n\rightarrow \widetilde{\mathbb R}$ be a symmetric bilinear form on $\widetilde{\mathbb R}^n$. A basis $\mathcal B:=\{e_1,\dots,e_k\}$ of $\widetilde{\mathbb R}^n$ is called an orthogonal basis with respect to $b$ if $b(e_i,e_j)=0$ whenever $i\neq j$.
\end{definition}
\begin{corollary}\label{exorthbasis}
Any symmetric bilinear form $b$ on $\widetilde{\mathbb R}^n$ admits an orthogonal basis.
\end{corollary}
\begin{proof}
Let $\mathcal B:=\{v_1,\dots,v_n\}$ be some basis of $\widetilde{\mathbb R}^n$, then the coefficient matrix $A:=(b(v_i,v_j))_{ij}\in\widetilde{\mathbb R}^{n^2}$ is symmetric. Due to Lemma \ref{schur}, there is an orthogonal matrix $U\in\widetilde{\mathbb R}^{n^2}$ and generalized numbers
$\theta_i\;(1\leq i\leq n)$ (the so-called eigenvalues) such that $UAU^t=\diag( \theta_1,\dots,\theta_n)$. Therefore the (clearly non-degenerate) matrix $U$ induces a mapping $\widetilde{\mathbb R}^n\rightarrow\widetilde{\mathbb R}^n$ which maps $\mathcal B$ onto some basis $\mathcal B'$
which is orthogonal.
\end{proof}
\begin{definition}
Let $\lambda_1,\dots,\lambda_k\in\widetilde{\mathbb R}$ ($k\geq 1$). Then the span of $\lambda_i\; (1\leq i\leq k)$ is denoted by $\langle\{\lambda_1,\dots,\lambda_n\}\rangle$.
\end{definition}
We now introduce a notion of causality in our framework:
\begin{definition}\label{causaldef1}
Let $g$ be a symmetric bilinear form of Lorentzian signature on $\widetilde{\mathbb R}^n$. Then we call $u\in\widetilde{\mathbb R}^n$
\begin{enumerate}
\item time-like, if $g(u,u)<0$,
\item null, if $u=0$ or $u$ is free and $g(u,u)=0$,
\item space-like, if $g(u,u)>0$.
\end{enumerate}
Furthermore, we say two time-like vectors $u,v$ have the same time-orientation whenever $g(u,v)<0$.
\end{definition}
Note that there exist elements in $\widetilde{\mathbb R}^n$ which are neither time-like, nor null, nor space-like.

The next statement provides a characterization of free elements in $\widetilde{\mathbb R}^n$. We shall repeatedly make use of it in the sequel.
\begin{theorem}\label{freechar}
Let $v$ be an element of $\widetilde{\mathbb R}^n$. Then the following are equivalent:
\begin{enumerate}
\item\label{freechar1} For any positive definite symmetric bilinear form $h$ on $\widetilde{\mathbb R}^n$ we have
\[
h(v,v)>0
\]
\item\label{freechar2} The coefficients of $v$ with respect to some (hence any) basis span $\widetilde{\mathbb R}$.
\item \label{freechar3}$v$ is free.
\item \label{freechar4} The coefficients $v^i$ ($i=1,\dots,n$) of $v$ with respect to some (hence any) basis of $\widetilde{\mathbb R}^n$ satisfy the following: For any choice of representatives $(v^i_{\varepsilon})_{\varepsilon}\;(1\leq i\leq n)$ of $v^i$ there exists some $\varepsilon_0\in I$ such that for each
$\varepsilon<\varepsilon_0$ we have
\[
\max_{i=1,\dots,n} \vert v^i_{\varepsilon}\vert>0.
\]
\item \label{freeloader6} For each representative $(v_{\varepsilon})_{\varepsilon}\in\mathcal E_M(\mathbb R^n)$ of $v$ there exists some $\varepsilon_0\in I$
such that for each $\varepsilon<\varepsilon_0$ we have $v_{\varepsilon}\neq 0$ in $\mathbb R^n$.
\item \label{freeloader7} There exists a basis of $\widetilde{\mathbb R}^n$ such that the first coefficient $v^i$ of $v$ is strictly non-zero.
\item \label{freeloader8} $v$ can be extended to a basis of $\widetilde{\mathbb R}^n$.
\item \label{freechar5} Let $v^i$ ($i=1,\dots,n$) denote the coefficients of $v$ with respect to some arbitrary basis of $\widetilde{\mathbb R}^n$. Then we have
\[
\| v\widetilde{\|}:=\left(\sum_{i=1}^n (v^i)^2\right)^{1/2}>0.
\]
\end{enumerate}
\end{theorem}
\begin{proof}
We proceed by establishing the implications (\ref{freechar1}) $\Rightarrow$ (\ref{freechar2}) $\Rightarrow$
(\ref{freechar3}) $\Rightarrow$ (\ref{freechar1}), further the equivalence
(\ref{freechar1}) $\Leftrightarrow$ (\ref{freechar5}) as well as (\ref{freechar4}) $\Leftrightarrow$ (\ref{freechar5})
and (\ref{freechar4}) $\Leftrightarrow$ (\ref{freeloader6}) and end with the proof of
(\ref{freechar4}) $\Rightarrow$ (\ref{freeloader7}) $\Rightarrow$ (\ref{freeloader8}) $\Rightarrow$ (\ref{freechar5})
$\Rightarrow$ (\ref{freechar4}).\\
If $v=0$ the equivalences are trivial. We shall therefore assume $v\neq 0$.\\
(\ref{freechar1}) $\Rightarrow$ (\ref{freechar2}): Let $(h_{ij})_{ij}$ be the coefficient matrix of $h$ with respect to some fixed basis
$\mathcal B$ of $\widetilde{\mathbb R}^n$. Then $\lambda:=\sum_{1\leq i,j\leq n}h_{ij}v^iv^j=h(v,v)>0$, in particular
$\lambda$ is invertible and $\sum_j(\sum_i \frac{h_{ij}v^i}{\lambda})v^j=1$ which shows that
$\langle\{ v^1,\dots,v^n\}\rangle=\widetilde{\mathbb R}$. Since the choice of the basis was arbitrary, (\ref{freechar2}) is shown.
\\(\ref{freechar2}) $\Rightarrow$ (\ref{freechar3}):
We assume $\langle\{ v^1,\dots,v^n\}\rangle=\widetilde{\mathbb R}$ but that there exists some $\lambda\neq 0: \lambda v=0$, that is,
$\forall\;i: 1\leq i\leq n:\,\lambda v^i=0$. Since the coefficients of $v$ span $\widetilde{\mathbb R}$,
there exist $\mu_1,\dots,\mu_n$ such that $\lambda=\sum_{i=1}^n\mu_iv^i$. It follows that
$\lambda^2=\sum_{i=1}^n\mu_i(\lambda v^i)=0$ but this is impossible, since $\widetilde{\mathbb R}$ contains no nilpotent elements.
\\(\ref{freechar3}) $\Rightarrow$ (\ref{freechar1}):
Due to Lemma \ref{schur} we may assume that we have chosen a basis such that the coefficient matrix with respect to the latter
is in diagonal form, i.\ e., $(h_{ij})_{ij}=\diag (\lambda_1,\dots,\lambda_n)$ with $\lambda_i>0\; (1\leq i\leq n)$. We have to show that $h(v,v)=\sum_{i=1}^n\lambda_i (v^i)^2>0$. Since there exists
$\varepsilon_0\in I$ such that for all representatives of $\lambda_1,\dots,\lambda_n, v^1,\dots, v^n$ we have for $\varepsilon<\varepsilon_0$ that $\gamma_{\varepsilon}:=\lambda_{1\varepsilon} (v^1_{\varepsilon})^2+\dots+\lambda_{n\varepsilon} (v^n_{\varepsilon})^2\geq 0$,
$h(v,v)\not> 0$ would imply that there exists a zero sequence $\varepsilon_k\rightarrow 0$ ($k\rightarrow 0$) such that $\gamma_{\varepsilon_k}<\varepsilon^k$. This implies that $h(v,v)$ is a zero divisor
and it means that all summands share a simultaneous zero divisor, i.\ e., $\exists\;\mu\neq 0\,\forall\;i\in\{1,\dots,n\}:\,\mu \lambda_i (v^i)^2=0$. Since $v$ was free, this is a contradiction and we have shown that  (\ref{freechar1}) holds.\\ The equivalence (\ref{freechar1}) $\Leftrightarrow$ (\ref{freechar5}) is evident. We proceed by establishing the equivalence
(\ref{freechar4}) $\Leftrightarrow$ (\ref{freechar5}). First, assume (\ref{freechar5}) holds, and let $(v^i_{\varepsilon})_{\varepsilon}\;(1\leq i\leq n)$ be arbitrary representatives of $v^i\;(i=1,\dots,n)$. Then
\[
\left(\sum_{i=1}^n (v^i_{\varepsilon})^2\right)_{\varepsilon}
\]
is a representative of $(\|v\widetilde{\|})^2$ as well, and since $\|v\widetilde{\|}$ is strictly positive, there exists some $m_0$ and some $\varepsilon_0\in I$ such that
\[
\forall\;\varepsilon<\varepsilon_0:\sum_{i=1}^n (v^i_{\varepsilon})^2>\varepsilon^{m_0}.
\]
This immediately implies (\ref{freechar4}). In order to see the converse direction, we proceed indirectly. Assume (\ref{freechar5}) does not hold, that is, we assume there exist representatives $(v_{\varepsilon}^i)_{\varepsilon}$ of $v^i$ for $i=1,\dots,n$ such that
for some sequence $\varepsilon_k\rightarrow 0$ ($k\rightarrow\infty$) we have for each $k>0$ that
\[
\sum_{i=1}^n (v_{\varepsilon_k}^i)^2<\varepsilon_k^k.
\]
Therefore one may even construct representatives $(\widetilde v_{\varepsilon}^i)_{\varepsilon}$ for $v^i$ ($i=1,\dots,n$)
such that for each $k>0$ and each $i\in\{1,\dots,n\}$ we have $\widetilde v_{\varepsilon_k}^i=0$. It is now evident that $(\widetilde v_{\varepsilon}^i)_{\varepsilon}$ violate condition (\ref{freechar4}) and we are done with (\ref{freechar4}) $\Leftrightarrow$ (\ref{freechar5}).
(\ref{freechar4}) $\Leftrightarrow$ (\ref{freeloader6}) is evident. So we finish the proof by showing (\ref{freechar4}) $\Rightarrow$ (\ref{freeloader7}) $\Rightarrow$ (\ref{freeloader8}) $\Rightarrow$ (\ref{freechar4})
%
%
Clearly (\ref{freeloader8}) $\Rightarrow$ (\ref{freechar4}). To see (\ref{freechar4}) $\Rightarrow$ (\ref{freeloader7}) we first observe that the condition (\ref{freechar4}) implies that there exists some $m_0$
such that for suitable representatives $(v_{\varepsilon}^i)_{\varepsilon}$ of $v^i$ ($i=1,\dots,n$) we have for each $\varepsilon\in I$ $\max_{i=1,\dots,n}\vert v_{\varepsilon}^i\vert>\varepsilon^{m_0}$, i.\ e.,
\[
\forall\;\varepsilon\in I\;\exists\; i(\varepsilon)\in\{1,\dots,n\}:\vert v_{\varepsilon}^{i(\varepsilon)}\vert>\varepsilon^{m_0}.
\]
We may view $(v_{\varepsilon})_{\varepsilon}:=((v_{\varepsilon}^1,\dots,v_{\varepsilon}^n)^t)_{\varepsilon}\in\mathcal E_M(\mathbb R^n)$
as a representative of $v$ in $\mathcal E_M(\mathbb R^n)/\mathcal N(\mathbb R^n)$. Denote for each $\varepsilon\in I$ by $A_{\varepsilon}$ the representing matrix of the linear map $\mathbb R^n\rightarrow\mathbb R^n$ that merely permutes the $i(\varepsilon)$ th.\ canonical coordinate of $\mathbb R^n$ with the first one. Define $A:\widetilde{\mathbb R^n}\rightarrow\widetilde{\mathbb R^n}$ the bijective linear map with representing matrix
\[
A:=(A_{\varepsilon})_{\varepsilon}+\mathcal E_M(\mathcal M_n(\mathbb R)).
\]
What is evident now from our construction, is: The first coefficient of
\[
\widetilde v:=Av=(\mathcal A_{\varepsilon}v_{\varepsilon})_{\varepsilon}+\mathcal E_M(\mathbb R^n)
\]
is strictly nonzero and we have shown (\ref{freeloader7}). Finally we verify (\ref{freeloader7}) $\Rightarrow$ (\ref{freeloader8}). Let $\{e_i\mid 1\leq i\leq n\}$ denote the canonical basis of $\widetilde{\mathbb R}^n$. Point (\ref{freeloader7}) ensures the existence of a bijective linear map $A$ on $\widetilde{\mathbb R}^n$ such that
the first coefficient $\bar v^1$ of $\bar v=(\bar v^1,\dots,\bar v^n)^t:=Av$ is strictly non-zero; applying Proposition \ref{steinitzprop} yields another basis $\{\bar v,e_2,\dots,e_n\}$ of $\widetilde{\mathbb R}^n$. Since $A$ is bijective,
$\{v=A^{-1}\bar v,A^{-1}e_2,\dots,A^{-1}e_n\}$ is a basis of $\widetilde{\mathbb R}^n$ as well and we are done.
\end{proof}
We may add a non-trivial example of a free vector to the above characterization:
\begin{example}\rm
For $n>1$, let $\lambda_i\in\widetilde{\mathbb R}\; (1\leq i\leq n)$ have the following properties
\begin{enumerate}
\item \label{ex1} $\lambda_i^2=\lambda_i\;\forall\; i\in\{1,\dots,n\}$
\item \label{ex2} $\lambda_i\lambda_j=0\;\forall\;i\neq j$
\item \label{ex3} $\langle\{\lambda_1,\dots,\lambda_n\}\rangle=\widetilde{\mathbb R}$
\end{enumerate}
This choice of zero divisors in $\widetilde{\mathbb R}$ is possible (idempotent elements in $\widetilde{\mathbb R}$ are thoroughly discussed in \cite{A1}, pp.\ 2221--2224). Now, let $\mathcal B=\{e_1,\dots,e_n\}$ be the canonical basis of $\widetilde{\mathbb R}^n$. Theorem \ref{freechar} (\ref{freechar3}) implies that $v:=\sum_{i=1}^n (-1)^{(i+1)(n+1)}\lambda_i e_i$ is free. Furthermore let $\gamma\in\Sigma_n$ be the cyclic permutation which sends $\{1,\dots,n\}$ to $\{n,1,\dots,n-1\}$. Clearly the sign of $\gamma$ is positive if and only if $n$ is odd. Define $n$ vectors $v_j\;(1\leq j\leq n)$ by $v_1:=v$, and such that $v_j$ is given by $v_j:=\sum_{k=1}^n \lambda_{\gamma^{j-1}(k)}e_k$ whenever $j>1$. Let $A$ be the matrix having the $v_j$'s as column vectors. Then
\[
\det A=\sum_{l=1}^n\lambda_l^n=\sum_{l=1}^n\lambda_l.
\]
Due to properties (\ref{ex1},\ref{ex3}), $\det A$ is invertible. Therefore, $\mathcal B':=\{v,v_2,\dots,v_n\}$ is a basis of $\widetilde{\mathbb R}^n$, too. The reader is invited to check further equivalent properties of $v$ according to Theorem \ref{freechar}.
\end{example}
Since any symmetric bilinear form admits an orthogonal basis due to Corollary \ref{exorthbasis} we further conclude by means of Theorem
\ref{freechar}:
\begin{corollary}\label{corbilpos}
Let $b$ be a symmetric bilinear form on $\widetilde{\mathbb R}^n$. Then the following are equivalent:
\begin{enumerate}
\item \label{corbilpos1} For any free $v\in\widetilde{\mathbb R}^n$, $b(v,v)>0$.
\item $b$ is positive definite.
\end{enumerate}
\end{corollary}
For showing further algebraic properties of $\widetilde{\mathbb R}^n$ (cf.\ section \ref{semi}), also the following lemma will be crucial:
\begin{lemma}\label{posprop}
Let $h$ be a positive definite symmetric bilinear form. Then we have the following:
\begin{enumerate}
\item \label{h1} $\forall\; v\in\widetilde{\mathbb R}^n: h(v,v)\geq 0$ and $h(v,v)=0\Leftrightarrow v=0$.
\item \label{h2}Let $\mathfrak m$ be a free submodule of $\widetilde{\mathbb R}^n$. Then $h$ is a positive definite symmetric bilinear form on $\mathfrak m$.
\end{enumerate}
\end{lemma}
\begin{proof}
First, we verify (\ref{h1}): Let $v^i\;(1\leq i\leq n)$ be the coefficients of
$v$ with respect to some orthogonal basis $\mathcal B$ for $h$. Then we can write $h(v,v)=\sum_{i=1}^n\lambda_i(v^i)^2$ with $\lambda_i$ strictly positive for each $i\in\{1,\dots,n\}$. Thus $h(v,v)\geq 0$, and $h(v,v)=0$ implies $\forall\; i\in\{1\dots  n\}:v^i=0$, i.\ e., $v=0$. This finishes the proof of part (\ref{h1}).
In order to show (\ref{h2}) we first notice that by definition, any free submodule admits a basis. Let $\mathcal B_{\mathfrak m}:=\{w_1,\dots,w_k\}$ be such for $\mathfrak m$ and denote by $h_{\mathfrak m}$ the restriction of $h$ to $\mathfrak m$. Then, due to Theorem \ref{freechar} (\ref{freechar1}), we have for all $1\leq i\leq k$, $h_{\mathfrak m}(w_i,w_i)>0$. Let $A:=(h_{\mathfrak m}(w_i,w_j))_{ij}$ be the coefficient matrix of $h_{\mathfrak m}$ with respect to $\mathcal B_{\mathfrak m}$. Since $h_{\mathfrak m}$ is symmetric, so is the matrix $A$
and thus, due to Lemma \ref{schur} there is an orthogonal matrix $U\in\widetilde{\mathbb R}^{k^2}$ and there are generalized numbers $\lambda_i\;(1\leq i\leq k)$
such that $UAU^t=\diag(\lambda_1,\dots,\lambda_k)$ which implies that the (orthogonal, thus non-degenerate) $U$ maps
$\mathcal B_{\mathfrak m}$ on an orthogonal basis $\mathcal B:=\{e_1,\dots,e_k\}$ of $\mathfrak m$ with respect to $h_{\mathfrak m}$ and again by Theorem \ref{freechar} (\ref{freechar1}) we have $\lambda_i>0\;(1 \leq i\leq k)$. By Definition \ref{bilformdef}, $h_{\mathfrak m}$ is also positive definite on $\mathfrak m$ and we are done.
\end{proof}
Since any time-like or space-like vector is free, we further have as a consequence of Theorem \ref{freechar}:
\begin{proposition}\label{coraustausch}
Suppose we are given a bilinear form of Lorentzian signature on $\widetilde{\mathbb R}^n$ and let $u\in\widetilde{\mathbb R}^n\setminus \{0\}$ be time-like, null or space-like. Then $u$ can be extended to a basis of $\widetilde{\mathbb R}^n$.
\end{proposition}
In the case of a time-like vector we know a specific basis in which the first coordinate is invertible:
\begin{remark}
Suppose we are given a bilinear form $b$ of Lorentzian signature on $\widetilde{\mathbb R}^n$, let $u$ be a time-like vector.
Due to the definition of $g$ we may suppose that we have a basis so that the scalar product of $u$ takes the form
\[
g(u,u)=-\lambda_1 (u^1)^2+\lambda_2 (u^2)^2\dots+\lambda_n (u^n)^2.
\]
with $\lambda_i$ strictly positive for each $i=1,\dots,n$. Since $g(u,u)<0$, we see that the first coordinate $u^1$
of $u$ must be strictly non-zero.
\end{remark}
It is worth mentioning that an analogue of the well known criterion of positive definiteness of matrices in $\mathcal M_n(\mathbb R)$ holds in our setting:
\begin{lemma}\label{criterion}
Let $A\in\widetilde{\mathbb R}^{n^2}$ be symmetric. If the determinants of all principal subminors of $A$ (that are the submatrices $A^{(k)}:=(a_{ij})_{1\leq i,j\leq k}\;(1\leq k\leq n)$) are strictly positive, then $A$ is positive definite.
\end{lemma}
\begin{proof}
Choose a symmetric representative $(A_{\varepsilon})_{\varepsilon}$ of $A$ (cf.\  Lemma \ref{symmetry}). Clearly the assumption
$\det A^{(k)}>0\; (1\leq k\leq n)$ implies that $\exists\; \varepsilon_0\;\exists\; m\;\forall\; k: 1\leq k\leq n\; \forall\; \varepsilon<\varepsilon_0:\det A^{(k)}_{\varepsilon}\geq \varepsilon^m$, that is, for each sufficiently small $\varepsilon$, $A_{\varepsilon}$ is a positive definite symmetric
matrix due to a well known criterion in linear algebra. Furthermore $\det A^{(n)}=\det A>0$ implies $A$ is non-degenerate which finally shows
that $A$ is positive definite.
\end{proof}
Before we go on we note that type changing of tensors on $\widetilde{\mathbb R}^n$ by means of a non-degenerate symmetric bilinear form $g$ clearly is possible.
Moreover, given a (generalized) metric $g\in\mathcal G^0_2(X)$ on a manifold $X$ (cf.\ section \ref{introducerepseudoriemannereetconnexione}), lowering (resp.\ raising) indices
of generalized tensor fields on $X$ (resp.\ tensors on $\widetilde{\mathbb R}^n$) is compatible with evaluation on compactly supported generalized points (which actually yields the resp.\  object on $\widetilde{\mathbb R}^n$). This basically follows from Proposition 3.9 (\cite{KS1}) combined
with Theorem 3.1 (\cite{KS1}). As usual we write the covector associated to $\xi\in\widetilde{\mathbb R}^n$ in abstract index notation as $\xi_a:=g_{ab}\xi^b$. We call $\xi_i\;(i=1,\dots,n)$ the covariant components of $\xi$.\\
The following technical lemma is required in the sequel:
\begin{lemma}\label{uvfree}
Let $u,v\in\widetilde{\mathbb R}^n$ such that $u$ is free and $u^tv=0$. Then for each representative $(u_\varepsilon)_\varepsilon$ of $u$ there exists a representative $(v_\varepsilon)_\varepsilon$ of $v$
such that for each $\varepsilon \in I$ we have $u^t_\varepsilon v_\varepsilon=0$.
\end{lemma}
\begin{proof}
Let $(u_\varepsilon)_\varepsilon$, $(\hat v_\varepsilon)_\varepsilon$ be representatives of $u,v$ respectively. Then there exists $(n_\varepsilon)_\varepsilon\in\mathcal N$ such that
\[
(u_\varepsilon^t)_\varepsilon(\hat v_\varepsilon)_\varepsilon=(n_\varepsilon)_\varepsilon.
\]
By Theorem \ref{freechar} (\ref{freechar4}) we conclude
\[
\exists\; \varepsilon_0\;\exists\; m_0\;\forall\;\varepsilon<\varepsilon_0\;\exists\; j(\varepsilon):\;\vert u_\varepsilon^{j(\varepsilon)}\vert\geq \varepsilon^{m_0}.
\]
Therefore we may define a new representative $(v_\varepsilon)_\varepsilon$ of $v$ in the following way: For $\varepsilon\geq\varepsilon_0$ we set
$v_\varepsilon:=0$, otherwise we define
\[
v_\varepsilon:=\begin{cases}\hat v_\varepsilon^{j}, \quad j\neq j(\varepsilon)\\\hat v_{\varepsilon}^{j(\varepsilon)}-\frac{n_{\varepsilon}}{u_\varepsilon^{j(\varepsilon)}}\quad \mbox{otherwise} \end{cases}
\]
and clearly we have $u^t_\varepsilon v_\varepsilon=0$ for each $\varepsilon\in I$.
\end{proof}
The following result in the style of \cite{FL1} (Lemma 3.1.1, p.\ 74) prepares the inverse Cauchy-Schwarz inequality in our framework.
We follow the book of Friedlander which helps us to calculate the determinant of the coefficient matrix of a symmetric bilinear form,
which then turns out to be strictly positive, thus invertible. This is equivalent to non-degenerateness of the bilinear form
(cf.\ Lemma \ref{nondeg}):
\begin{proposition}\label{procs}
Let $g$ be a symmetric bilinear form of Lorentzian signature. If $u\in\widetilde{\mathbb R}^n$ is time-like, then $u^{\perp}$
is an $n$$-$$1$ dimensional submodule of $\widetilde{\mathbb R}^n$ and $g\mid_{u^{\perp}\times u^{\perp}}$ is positive definite.
\end{proposition}
\begin{proof}
Due to Proposition
\ref{coraustausch} we can choose a basis of $\widetilde{\mathbb R}^n$ such that $\Pi:=\langle\{u\}\rangle$ is spanned by the first vector, i.\ e.,
\[
\Pi=\{\xi\in\widetilde{\mathbb R}^n\vert \xi^A=0, A=2,\dots,n\}.
\]
Consequently we have
\[
\langle \xi,\xi\rangle\vert_{\Pi\times\Pi}=g_{11}(\xi^1)^2,
\]
and $g_{11}=\langle u,u\rangle<0$. If $\eta\in \Pi':=u^{\perp}$, then $\langle \xi,\eta\rangle=\xi^i\eta_i$, hence the covariant component $\eta_1$ must vanish (set $\xi:=u$, i.\ e., $\langle \xi,\eta\rangle=\langle u,\eta\rangle=\eta_1=0$). Therefore we have
\begin{equation}\label{uperp}
\langle \eta ,\theta\rangle\vert_{\Pi'\times\Pi'}=g^{AB}\eta_A\theta_B.
\end{equation}
Our first observation is that $u^{\perp}$ is a free ($n-1$ dimensional) submodule
with the basis $\xi_{(2)},\dots,\xi_{(n)}$ given in terms of the chosen coordinates above
via
\[
\xi_{(k)}^j:=g^{ij}\delta_i^k,\quad k=2,\dots,n
\]
(cf.\ (\ref{matmulti}) below, these are precisely the $n-1$ row vectors there!)
Due to the matrix multiplication
\begin{equation}\label{matmulti}
\left (\begin{array}{cccc} 1 & 0 &\dots &0 \\ g^{21}& g^{22}&\dots&g^{2n}\\ \dots &\dots&\dots&\dots\\g^{n1}& g^{n2}&\dots&g^{nn}\end{array}\right)(g_{ij})=\left(\begin{array}{cc} g_{11}& *\\ 0&\mathbb I_{n-1}\end{array} \right)
\end{equation}
evaluation of the determinants yields
\[
\det g^{AB}\det g_{ij}=g_{11}.
\]
And it follows from $\det g_{ij}<0, g_{11}<0$ that $\det g^{AB}>0$ which in particular shows that $g^{AB}$ is a non-degenerate symmetric matrix, $g\mid_{u^{\perp}\times u^{\perp}}$ therefore being a non-degenerate symmetric bilinear form on an $n-1$ dimensional free submodule. What is left to prove is positive definiteness of $g^{AB}$.
We claim that for each  $u\in v^{\perp}$, $g(v,v)\geq 0$. In conjunction with the fact that $g\mid_{u^{\perp}}$ is non-degenerate, it follows that $g(v,v)>0$ for any free $v\in u^{\perp}$ (this can be seen by using a suitable basis for $u^{\perp}$ which diagonalizes $g\mid_{u^{\perp}\times u^{\perp}}$, cf.\ Corollary \ref{corbilpos}) and we are done.

To show the subclaim we have to undergo an $\varepsilon$-wise argument. Let $(u_{\varepsilon})_{\varepsilon}\in\mathcal E_M(\mathbb R^n)$ be a representative of $u$ and let
$((g^{\varepsilon}_{ij})_{ij})_{\varepsilon}\in\mathcal E_M(\mathcal M_n(\mathbb R))$ be a symmetric representatives of $(g_{ij})_{ij}$, where $(g_{ij})_{ij}$ is the coefficient matrix of $g$ with respect to the canonical basis of $\widetilde{\mathbb R}^n$. For each $\varepsilon$ we denote by $g_{\varepsilon}$ the symmetric bilinear form
induced by $(g^{\varepsilon}_{ij})_{ij}$, that is, the latter shall be the coefficient matrix of $g_{\varepsilon}$ with respect to the canonical basis of $\mathbb R^n$. First we show that
\begin{equation}\label{identitynormalspaces}
u^{\perp}=\{(v_{\varepsilon})_{\varepsilon}\in\mathcal E_M(\mathbb R^n):\; \forall\; \varepsilon>0: v_{\varepsilon}\in u_{\varepsilon}^{\perp}\}+\mathcal N(\mathbb R^n),
\end{equation}
Since the inclusion relation $\supseteq$ is clear, we only need to show that $\subseteq$ holds. To this end,
pick $v\in u^{\perp}$. Then $g(u,v)=g_{ij}u^iv^j=0$ and the latter implies that for each representative $(\hat v_{\varepsilon})_{\varepsilon}$ of $v$ there exists $(n_{\varepsilon})_{\varepsilon}\in\mathcal N$
such that
\[
(g_{ij}^{\varepsilon}u_{\varepsilon}^i\hat v_{\varepsilon}^j)_{\varepsilon}=(n_{\varepsilon})_{\varepsilon}.
\]
We may interpret $(g_{ij}^{\varepsilon}u_{\varepsilon}^i) (j=1,\dots,n)$ as the representatives of the coefficients of a vector $w$ with coordinates  $w_j:=g_{ij}u^i$, and $w$ is free, since $u$ is free and $g$ is non-degenerate. Therefore we may employ Lemma \ref{uvfree} which yields a representative $(v_{\varepsilon}^j)_{\varepsilon}$ of $v$
such that
\[
(g_{ij}^{\varepsilon}u_{\varepsilon}^i
 v_{\varepsilon}^j)_{\varepsilon}=0.
\]
This precisely means that there exists a representative $(v_{\varepsilon})_\varepsilon$ of $v$ such
that for each $\varepsilon$ we have $v_\varepsilon\in u_\varepsilon^\perp$. We have thus finished the proof of identity (\ref{identitynormalspaces}).

To finish the proof of the claim, that is $g(v,v)\geq0$, we pick a representative $(v_{\varepsilon})_{\varepsilon}$ of $v$ and an $\varepsilon_0\in I$ such that for each
$\varepsilon<\varepsilon_0$ we have
\begin{enumerate}
\item each $g_{\varepsilon}$ is of Lorentzian signature
\item $u_{\varepsilon}$ is time-like
\item $v_{\varepsilon}\in u_{\varepsilon}^\perp$.
\end{enumerate}
Note that this choice is possible due to (\ref{identitynormalspaces}). Further, by the resp.\ classic result of Lorentz geometry (cf.\ \cite{FL1}, Lemma 3.\ 1.\ 1) we have
$g_{\varepsilon}(v_{\varepsilon},v_\varepsilon)\geq 0$ unless $v_{\varepsilon}=0$. Since
$(g_{ij}^{\varepsilon}v_\varepsilon^i v_\varepsilon^j)_\varepsilon$ is a representative of
$g(v,v)$ we have achieved the subclaim.
\end{proof}
\begin{corollary}\label{dirsum1}
Let $u\in\widetilde{\mathbb R}^n$ be time-like. Then $u^{\perp}:=\{v\in \widetilde{\mathbb R}^n:\langle u,v\rangle=0\}$ is a submodule of $\widetilde{\mathbb R}^n$ and  $\widetilde{\mathbb R}^n=\langle \{u\}\rangle \oplus u^{\perp}$.
\end{corollary}
\begin{proof}
The first statement is obvious. For $v\in\widetilde{\mathbb R}^n$, define the orthogonal projection of $v$ onto $\langle\{u\}\rangle$ as $P_u(v):=\frac{\langle u, v\rangle}{\langle u,u\rangle}u$. Then one sees that $v=P_u(v)+(v-P_u(v))\in \langle \{u\}\rangle+u^{\perp}$.
Finally, assume $\widetilde{\mathbb R}^n\neq \langle \{u\}\rangle \oplus u^{\perp}$, i.\ e., $\exists\; \xi\neq 0, \xi\in \langle \{u\}\rangle \cap u^{\perp} $. It follows $\langle \xi,\xi\rangle\leq 0$ and due to the preceding proposition $\xi\in u^{\perp}$ implies $\langle \xi,\xi\rangle \geq 0$. Since we have a partial ordering $\leq $, this is impossible unless $\langle \xi,\xi\rangle=0$. However by Lemma \ref{posprop} (\ref{h1}) we have $\xi=0$. This contradicts our assumption and proves that $\widetilde{\mathbb R}^n$ is the direct sum of $u$ and its orthogonal complement.
\end{proof}
The following statement on the Cauchy--Schwarz inequality is a crucial result in generalized Lorentz Geometry. It slightly differs from the classical result as is shown in Example \ref{csex}. However it seems to coincide with the classical inequality in physically relevant cases, since
algebraic complications which mainly arise from the existence of zero divisor in our scalar ring
of generalized numbers, presumably are not inherent in the latter. Our proof follows the lines of the proof of the analogous classic statement in O'Neill's book (\cite{ON}, chapter 5, Proposition 30, pp.\ 144):
\begin{theorem}(Inverse Cauchy--Schwarz inequality)\label{cs}
Let $u,\;v \in \widetilde{\mathbb R}^n$ be time-like vectors. Then
\begin{enumerate}
\item \label{cs1} $\langle u,v\rangle^2\geq \langle u,u\rangle \langle v,v\rangle$, and
\item \label{cs2} equality in (\ref{cs1}) holds if $u,v$ are linearly dependent over $\widetilde{\mathbb R}^*$, the units
in $\widetilde{\mathbb R}$.
\item \label{cs3} If $u,v$ are linearly independent, then $\langle u,v\rangle^2>\langle u,u\rangle \langle v,v\rangle$.
\end{enumerate}
\end{theorem}
\begin{proof}
In what follows, we keep the notation of the preceding corollary. Due to Corollary \ref{dirsum1}, we may decompose $u$ in a unique way
$v=a u+w$ with $a\in\widetilde{\mathbb R},\, w\in u^{\perp}$. Since $u$ is time-like,
\[
\langle v,v\rangle=a^2 \langle u,u\rangle+\langle w,w\rangle<0.
\]
Then
\begin{equation}\label{eqcs}
\langle u,v\rangle^2=a^2\\\langle u,u\rangle^2=(\langle v,v\rangle-\langle w,w\rangle)\langle u,u\rangle\geq \langle u,u\rangle \langle v,v\rangle
\end{equation}
since $\langle w,w\rangle\geq 0$ and this proves (\ref{cs1}). \\In order to prove (\ref{cs2}), assume $u,v$ are linearly dependent over $\widetilde{\mathbb R}^*$, that is, there exist $\lambda,\,\mu$, both units in $\widetilde{\mathbb R}$ such that
$\lambda u+\mu v=0$. Then $u=-\frac{\mu}{\lambda} v$ and equality in (\ref{cs2}) follows.\\ Proof of (\ref{cs3}): Assume now, that
$u,v$ are linearly independent. We show that this implies that $w$ is free.
For the sake of simplicity we assume without loss of generality that $\langle u,u\rangle=\langle v,v\rangle=-1$ and we choose
a basis $\mathcal B=\{e_1,\dots,e_n\}$ with $e_1=u$ due to Proposition \ref{coraustausch}.
Then with respect to the new basis we can write $u=(1,0,\dots,0)^t$, $v=(v^1,\dots,v^n)^t$, $w=v-P_u(v)=(v^1-(-g(v,e_1)), v^2,\dots,v^n)^t=(0,w^2,\dots,w^n)^t$.
Assume $\exists\; \lambda\neq0: \lambda w=0$, then
\[
(\lambda v^1)u+\lambda v=\lambda v^1 e_1-\lambda g(v,e_1) e_1=\lambda v^1 e_1-\lambda v^1 e_1=0
\]
which implies that $u,v$ are linearly dependent. This contradicts the assumption in (\ref{cs3}). Thus $w$ indeed is free.
Applying Theorem \ref{freechar} yields $\langle w,w\rangle>0$. A glance at
(\ref{eqcs}) shows that the proof of (\ref{cs3}) is finished.
\end{proof}
The following example indicates what happens when in \ref{cs} (\ref{cs2}) linear dependence over the units in $\widetilde{\mathbb R}$ is replaced by
linear dependence over $\widetilde{\mathbb R}$:
\begin{example}\label{csex}
Let $\lambda\in\widetilde{\mathbb R}$ be an idempotent zero divisor, and write $\alpha:=[(\varepsilon)_{\varepsilon}]$. Let $\eta=\diag(-1,1\dots,1)$ be the Minkowski metric.
Define $u=(1,0,\dots,0)^t,v=(1,\lambda\alpha,0,\dots,0)^t$. Clearly $\langle u,u\rangle=-1,\langle v,v\rangle=-1+\lambda^2\alpha^2<0$
But
\[
\langle u,v\rangle^2=1\neq \langle u,u\rangle\langle v,v\rangle=-(-1+\lambda^2\alpha^2)=1-\lambda^2\alpha^2.
\]
However, also the strict relation fails, i.\ e., $\langle u,v\rangle^2 \not> \langle u,u\rangle\langle v,v\rangle$, since $\lambda$
is a zero divisor.
\end{example}
\section[Further algebraic properties of $\widetilde{\mathbb R}^n$]{Further algebraic properties of finite dimensional modules over the ring of generalized
numbers} This section is devoted to a discussion of direct summands
of submodules inside $\widetilde{\mathbb R}^n$. The question first
involves free submodules of arbitrary dimension. However,
we establish a generalization of Theorem \ref{freechar}
(\ref{freeloader8}) not only with respect to the dimension of the
submodule; the direct summand we construct is also an orthogonal
complement with respect to a given positive definite symmetric
bilinear form. Having established this in \ref{semi}, we
subsequently show that $\widetilde{\mathbb R}^n$ is not semisimple,
i.\ e., non-free submodules in our module do not admit direct
summands.
\subsection{Direct summands of free submodules}\label{semi}
The existence of positive bilinear forms on $\widetilde{\mathbb R}^n$ ensures the existence of direct summands
of free submodules of $\widetilde{\mathbb R}^n$:
\begin{theorem}
Any free submodule $\mathfrak m$ of $\widetilde{\mathbb R}^n$ has a direct summand.
\end{theorem}
\begin{proof}
Denote by $\mathfrak m$ the free submodule in question with $\dim \mathfrak m=k$, let $h$ be a positive definite symmetric bilinear form on $\mathfrak m$ and $h_{\mathfrak m}$ its restriction to $\mathfrak m$. Now, due to Lemma \ref{posprop} (\ref{h2}), $h_{\mathfrak m}$ is a positive definite symmetric bilinear form. In particular, there exists an orthogonal basis $\mathcal B_{\mathfrak m}:=\{e_1,\dots,e_k\}$ of $\mathfrak m$ with respect to $h_{\mathfrak m}$. We further may assume that the latter one is orthonormal. Denote by $P_{\mathfrak m}$ the orthogonal projection on $\mathfrak m$ which due to the orthogonality of $\mathcal B_{\mathfrak m}$ may be written in the form
\[
P_{\mathfrak m}:\;\widetilde{\mathbb R}^n\rightarrow \mathfrak m,\; v\mapsto \sum_{i=1}^k\langle v,e_i\rangle e_i.
\]
Finally, we show ${\mathfrak  m}^{\perp}=\ker P_{\mathfrak m}$:
\begin{eqnarray}\nonumber
{\mathfrak m}^{\perp}&=&\{v\in\widetilde{\mathbb R}^n\mid \forall\; u\in \mathfrak m: h(v,u)=0\}=\\\nonumber
&=&\{v\in\widetilde{\mathbb R}^n\mid \forall\; i=1,\dots,k: h(v,e_i)=0\}=\\\nonumber
&=&\{v\in \widetilde{\mathbb R}^n\mid P_{\mathfrak m}(v)=0\}=\ker P_{\mathfrak m}.
\end{eqnarray}
Where both of the last equalities are due to the definition of $P_{\mathfrak m}$ and the fact that $B_{\mathfrak m}$ is a basis of $\mathfrak m$.
As always in modules, ${\mathfrak m}^{\perp}=\ker P_{\mathfrak m}\Leftrightarrow {\mathfrak m}^{\perp}$ is a direct summand and we are done.
An alternative end of this proof is provided by Lemma \ref{posprop}: Since we have $\mathfrak m+\mathfrak m^{\perp}=\widetilde{\mathbb R}^n$,
we only need to show that this sum is a direct one. But Lemma \ref{posprop} (\ref{h1}) shows that $0 \neq u \in \mathfrak m\cap {\mathfrak m}^{\perp}$ is absurd, since $h$ is positive definite.
\end{proof}
We thus have also shown (cf.\ Theorem \ref{freechar}):
\begin{corollary}\label{orthdecriem}
Let $w\in\widetilde{\mathbb R}^n$ be free and let $h$ be a positive definite symmetric bilinear form. Then $\widetilde{\mathbb R}^n=\langle \{w\}\rangle \oplus w^{\perp}$.
\end{corollary}
We therefore have added a further equivalent property to Theorem \ref{freechar}.
\subsection{$\widetilde{\mathbb R}^n$ is not semisimple}\label{secsemi}
In this section we show that $\widetilde{\mathbb R}^n$ is not semisimple. Recall that a module $B$ over a ring $R$ is
called simple, if $RA\neq \{0\}$ and if $A$ contains no non-trivial strict submodules. For the convenience of the reader, we recall the following fact on modules (e.\ g., see \cite{Hungerford}, p.\ 417):
\begin{theorem}\label{charsemisimple}
The following conditions on a nonzero module $A$ over a ring $R$ are equivalent:
\begin{enumerate}
\item \label{charsemisimple1} $A$ is the sum of a family of simple submodules.
\item \label{charsemisimple2} $A$ is the direct sum of a family of simple submodules.
\item For every nonzero element a of $A$, $Ra\neq 0$; and every submodule $B$ of $A$ is a direct summand (that is, $A=B\oplus C$ for some submodule $C$.
\end{enumerate}
\end{theorem}
Such a module is called semisimple. However, property (\ref{charsemisimple1}) is violated in $\widetilde{\mathbb R}^n$ $(n\geq 1)$:
\begin{proposition}
Every submodule $A\neq \{0\}$ in $\widetilde{\mathbb R}^n$ contains a strict submodule.
\end{proposition}
\begin{proof}
Let $u\in A$, $u\neq 0$. We may write $u$ in terms of the canonical basis $e_i\;(i=1,\dots,n)$, $u=\sum_{i=1}^n \lambda_ie_i$ and without loss of generality we may assume $\lambda_1\neq 0$. Denote a representative of $\lambda_1$ by $(\lambda_1^{\varepsilon})_{\varepsilon}$.
$\lambda_1\neq 0$ in particular ensures the existence of a zero sequence  $\varepsilon_k \searrow 0$ in $I$ and an $m>0$ such that
for all $k\geq 1$, $\vert \lambda_1^{\varepsilon_k}\vert\geq \varepsilon_k^m$. Define $D:=\{\varepsilon_k\mid k\geq 1\}\subset I$, let
$\chi_D\in\widetilde{\mathbb R}$ be the characteristic function on $D$. Clearly, $\chi_D u\in A$, furthermore, if the submodule generated by $\chi_D u$ is not a strict submodule of $A$, one may replace $D$ by $\bar D:=\{\varepsilon_{2k}\mid k\geq 1\}$ to achieve one in the same way, which however is a strict submodule of $A$ and we are done.
\end{proof}
The preceding proposition in conjunction with Theorem \ref{charsemisimple} gives rise to the following conclusion:
\begin{corollary}
$\widetilde{\mathbb R}^n$ is not semisimple.
\end{corollary}
\section{Energy tensors and a dominant energy condition}\label{energygeneralizedsection} In this section we
elaborate a dominant energy condition in the spirit of Hawking and
Ellis (\cite{HE}) for generalized energy tensors. The latter will be
constructed as tensor products of generalized Riemann metrics derived
from a (generalized) Lorentzian metric and time-like vector fields.
They shall be helpful for an application of the Stokes theorem to
generalized energy integrals in the course of establishing a (local)
existence and uniqueness theorem for the wave equation on a
generalized space-time (cf. \cite{VW}, however ongoing research
treats a wide range of generalized space-times, cf.\ chapter \ref{chapterwaveeq}). Throughout this
section $g$ denotes a symmetric bilinear form of Lorentz signature
on $\widetilde{\mathbb R}^n$, and for $u,v\in\widetilde{\mathbb
R}^n$ we write $\langle u,v\rangle:=g(u,v)$. We introduce the notion
of a (generalized) Lorentz transformation:
\begin{definition}
We call a linear map $L:\widetilde{\mathbb R}^n\rightarrow\widetilde{\mathbb R}^n$ a Lorentz transformation, if it preserves
the metric, that is
\[
\forall \xi \in \widetilde{\mathbb R}^n:\;\langle L\xi,L\eta\rangle=\langle \xi,\eta\rangle
\]
or equivalently,
\[
L^{\mu}_{\lambda}L^{\nu}_{\rho}g_{\mu\nu}=g_{\lambda\rho}.
\]
\end{definition}
In the original (classical) setting the following lemma is an exercise in a course on relativity \cite{RB}:
\begin{lemma}\label{lorentz}
Let $\xi, \eta\in \widetilde{\mathbb R}^n$ be time-like unit vectors with the same time-orientation. Then
\[
L^{\mu}_{\lambda}:=\delta^{\mu}_{\lambda}-2\eta^{\mu}\xi_{\lambda}+\frac{(\xi^{\mu}+\eta^{\mu})(\xi_{\lambda}+\eta_{\lambda})}{1-\langle \xi,\eta\rangle}
\]
is a Lorentz transformation with the property $L\xi=\eta$.
\end{lemma}
The following proposition is a crucial ingredient in the subsequent proof of the (generalized) dominant energy condition for certain energy tensors of this section:
\begin{proposition}\label{metrconstr}
Let $u,v \in \widetilde{\mathbb R}^n$ be time-like vectors such that $\langle u,v\rangle<0$. Then
\[
h_{\mu\nu}:=u_{(\mu}v_{\nu)}-\frac{1}{2}\langle u,v\rangle g_{\mu\nu}
\]
is a positive definite symmetric bilinear form on $\widetilde{\mathbb R}^n$.
\end{proposition}
\begin{proof}
Symmetry and bilinearity of $h$ are clear. 
What would be left is to show that the coefficient matrix of $h$ with respect to an arbitrary basis is invertible.
However, determining the determinant of $h$ is nontrivial. So we proceed by showing that for any free $w\in\widetilde{\mathbb R}^n$,
$h(w,w)$ is strictly positive (thus also deriving the classic statement). We may assume $\langle u ,u\rangle=\langle v,v\rangle=-1$; this can be achieved by scaling $u,v$ (note that this is due to the fact that for a time-like (resp.\ space-like) vector $u$, $\langle u, u\rangle$ is strictly non-zero, thus invertible in $\widetilde{\mathbb R}$). We may assume we have chosen an orthogonal basis $\mathcal B=\{e_1,\dots,e_n\}$ of $\widetilde{\mathbb R}^n$ with respect to $g$, i.\ e., $g(e_i,e_j)=\varepsilon_{ij}\lambda_i$, where $\lambda_1\leq \dots\leq \lambda_n$  are the eigenvalues of $(g(e_i,e_j))_{ij}$.
Due to Lemma \ref{lorentz} we can treat $u,v$ by means of generalized Lorentz transformations such that both vectors appear in the form $u=(\frac{1}{\lambda_1},0,0,0)$, $v=\gamma(v)(\frac{1}{\lambda_1},\frac{V}{\lambda_2},0,0)$, where $\gamma(v)=\sqrt{-g(v,v)}=\sqrt{1-V^2} >0$ (therefore $\vert V\vert <1$). Let $w=(w^1,w^2,w^3,w^4)\in \widetilde{\mathbb R}^n$ be free (in particular $w\neq 0$). Then
\begin{equation}\label{beig1}
h(w,w):=h_{ab}w^aw^b= \langle u,w\rangle  \langle v,w\rangle-\frac{1}{2}\langle w,w\rangle \langle u,v\rangle.
\end{equation}
Obviously, $\langle u,w \rangle=-w^1,\langle v,w\rangle=\gamma(v)(-w^1+Vw^2),\langle u,v\rangle=-\gamma(v)$. Thus
\begin{multline}\nonumber
h(w,w)=\gamma(v)(-w^1)(-w^1+Vw^2)+\frac{\gamma(v)}{2}(-(w^1)^2+(w^2)^2+(w^3)^2+(w^4)^2)=\\=
-\gamma(v) V w^1w^2+\frac{1}{2}\gamma(v) (+(w^1)^2+(w^2)^2+(w^3)^2+(w^4)^2).
\end{multline}
If $Vw^1w^2\leq 0$, we are done. If not, replace $V$ by $\vert V\vert$ ($-V\geq-\vert V \vert$) and rewrite the last formula in the following form :
\begin{equation}\label{estw}
h(w,w)\geq\frac{\gamma(v)}{2}\left ( (\vert V \vert(w^1-w^2)^2 +(1-\vert V \vert)(w^1)^2+(1-\vert V \vert)(w^2)^2+(w^3)^2+(w^4)^2\right).
\end{equation}
Clearly for the first term on the right side of (\ref{estw}) we have $\vert V\vert(w^1-w^2)^2\geq 0$. From $v$ is time-like we further deduce $1-\vert V \vert=\frac{1-V^2}{1+\vert V \vert}>0$. Since $w$ is free we may apply Theorem \ref{freechar}, which yields $(1-\vert V\vert)(w^1)^2+(1-\vert V\vert)(w^2)^2+(w^3)^2+(w^4)^2>0$
and thus $h(w,w)>0$ due to equation (\ref{estw} and we are done.
\end{proof}
Finally we are prepared to show a dominant energy condition in the style of Hawking and Ellis (\cite{HE}, pp.\ 91--93) for a generalized energy tensor. In what follows, we use abstract index notation.
\begin{theorem}\label{dec}
For $\theta\in\widetilde{\mathbb R}^n$ the energy tensor $E^{ab}(\theta):=(g^{ac}g^{bd}-\frac{1}{2}g^{ab}g^{cd})\theta_c\theta_d$ has the following properties
\begin{enumerate}
\item \label{energy1} If $\xi,\eta\in\widetilde{\mathbb R}^n$ are time-like vectors with the same orientation, then we have for any free
$\theta$, $E^{ab}(\theta)\xi_a\eta_b>0$.
\item \label{energy2} Suppose $\langle \theta,\theta\rangle$ is invertible in $\widetilde{\mathbb R}$. If $\xi\in\widetilde{\mathbb R}^n$ is time-like, then $\eta^b:=E^{ab}(\theta)\xi_a$ is time-like and $\eta^a\xi_a>0$, i.\ e., $\eta$ is past-oriented with respect to $\xi$. Conversely, if $\langle \theta,\theta\rangle$ is a zero divisor, then $\eta$ fails to be time-like.
\end{enumerate}
\end{theorem}
\begin{proof}
(\ref{energy1}): Define a symmetric bilinear form  $h^{ab}:=(g^{(ac}g^{b)d}-\frac{1}{2}g^{ab}g^{cd})\xi_c\eta_d$. Due to our assumptions on
$\xi$ and $\eta$, Proposition \ref{metrconstr} yields that $h^{ab}$ is a positive definite symmetric bilinear form. By Theorem \ref{freechar} we conclude that for any free $\theta\in\widetilde{\mathbb R}^n$, $h_{ab}\theta^a\theta^b>0$. It is not hard to check that $E^{ab}(\theta)\xi_a\eta_b=h^{ab}\theta_a\theta_b$ and therefore we have proved (\ref{energy1}).\\ (\ref{energy2}): To start with, assume $\eta$ is time-like. Then $g(\xi,\eta)=g_{ab}\xi^a\eta^b=g_{ab}\xi^aE(\theta)^{ac}\xi_c=E^{ab}(\theta)\xi_a\xi_b$.
That this expression is strictly greater than zero follows from (\ref{energy1}), i.\ e., $E^{ab}(\theta)\xi_a$ is past-directed with respect to $\xi$
whenever $\langle\theta,\theta\rangle$ is invertible, since the latter implies $\theta$ is free. It remains to prove that $\langle \eta,\eta\rangle<0$. A straightforward calculation yields
\[
\langle \eta,\eta\rangle=\langle E(\theta)\xi,E(\theta)\xi\rangle=\frac{1}{4}\langle \theta,\theta\rangle^2\langle \xi,\xi\rangle.
\]
Since $\langle \theta,\theta\rangle$ is invertible and $\xi$ is time-like, we conclude that $\eta$ is time-like as well. Conversely, if $\langle \theta,\theta\rangle$ is a zero-divisor, also $\langle E(\theta)\xi,E(\theta)\xi\rangle$ clearly is one. Therefore, $\eta=E(\theta)\xi$ cannot be time-like, and we are done.
\end{proof}
A remark on this statement is in order. A comparison with (\cite{HE}, pp.\ 91--93) shows, that
our "dominant energy condition" on $T^{ab}$ is stronger, since the vectors $\xi,\eta$ in (\ref{energy1}) need not coincide. Furthermore,
if in (\ref{energy2}) the condition "$\langle \theta,\theta\rangle$ is invertible" was dropped, then (as in the classical ("smooth") theory) we could conclude that $\eta$ was not space-like, however, unlike in the smooth theory, this does not imply $\eta$ to be time-like or null (cf.\ the short note after Definition \ref{causaldef1}).

\section[Point values and generalized causality]{Generalized point value characterizations of generalized pseudo-Riemannian metrics and of causality of generalized vector fields}
Throughout this section $X$ denotes a paracompact smooth Hausdorff manifold of dimension $n$. Our goal is to give first a point value characterization of generalized pseudo-Riemannian metrics.
Then we describe causality of generalized vector fields on $X$ by means of
causality in $\widetilde{\mathbb R}^n$ with respect to the bilinear form induced
by a generalized Lorentzian metric through evaluation on compactly supported points (cf.\ \cite{MO1}). For a review on the basic definition of generalized sections of vector bundles in the sense of M.\ Kunzinger and R.\ Steinbauer (\cite{KS1}) we refer to the introduction.
We start by establishing a point-value characterization of generalized pseudo-Riemannian metrics with respect to their index:
\begin{theorem}\label{charindexg}
Let $g\in \mathcal G^0_2(X)$ satisfy one (hence all) of the equivalent statements of Theorem \ref{chartens02}, $j\in\mathbb N_0$. The following are equivalent:
\begin{enumerate}
\item \label{charindexg1} $g$ has (constant) index $j$.
\item \label{charindexg2} For each chart $(V_{\alpha},\psi_{\alpha})$ and each $\widetilde x \in (\psi_{\alpha}(V_{\alpha}))_c^{\sim}$, $g_{\alpha}(\widetilde x)$ is a symmetric bilinear form on $\widetilde{\mathbb R}^n$ with index $j$.
\end{enumerate}
\end{theorem}
\begin{proof}
(\ref{charindexg1})$\Rightarrow$(\ref{charindexg2}): Let $\widetilde x\in\psi_{\alpha}(V_{\alpha})_c^{\sim}$ be supported in $K\subset\subset \psi_{\alpha}(V_{\alpha})$ and choose a representative $(g_{\varepsilon})_{\varepsilon}$ of $g$ as in Theorem \ref{chartens02} (\ref{chartens023}) and Definition \ref{defpseud}. According to Theorem \ref{chartens02} (\ref{chartens021}), $g_{\alpha}(\widetilde x):\widetilde{\mathbb R}^n\times\widetilde{\mathbb R}^n\rightarrow\widetilde{\mathbb R}$ is symmetric and non-degenerate. So it merely remains to prove that 
the index of $g_\alpha(\widetilde x)$ coincides with the index of $g$. Since $\widetilde x$ is compactly supported, we may shrink $V_{\alpha}$
to $U_\alpha$ such that the latter is an open relatively compact subset of $X$ and $\widetilde x\in \psi_\alpha(U_\alpha)$. By Definition \ref{defpseud}
there exists a symmetric representative $(g_\varepsilon)_\varepsilon$ of $g$ on $U_\alpha$ and an $\varepsilon_0$ such that for all 
$\varepsilon<\varepsilon_0$, $g_\varepsilon$ is a pseudo-Riemannian metric on $U_\alpha$ with constant index $\nu$. Let $(\widetilde x_\varepsilon)_\varepsilon$
be a representative of $\widetilde x$ lying in $U_\alpha$ for each $\varepsilon<\varepsilon_0$. Let $g_{\alpha,\,ij}^\varepsilon$
 be the coordinate expression of $g_\varepsilon$ with respect to the chart $(U_\alpha,\psi_\alpha)$. Then for each $\varepsilon<\varepsilon_0$, $g_{\alpha,\,ij}^\varepsilon(\widetilde x_\varepsilon)$ has precisely $\nu$ negative and $n-\nu$ positive eigenvalues,
 therefore due to Definition \ref{indexmatrixdef}, the class $g_{ij}:=[(g_{\alpha,\,ij}^\varepsilon(\widetilde x_\varepsilon))_\varepsilon]\in\mathcal M_n(\widetilde {\mathbb R})$
 has index $\nu$. By Definition \ref{bilformdef} it follows that the respective bilinear form $g_{\alpha}(\widetilde x)$ induced by $(g_{ij})_{ij}$ with respect to the canonical basis of $\widetilde{\mathbb R}$ has index $\nu$ and we are done.\\ To show the converse direction, one may proceed
 by an indirect proof. Assume the contrary to (\ref{charindexg1}), that is, $g$ has non-constant index $\nu$. In view of Definition \ref{defpseud}
 there exists an open, relatively compact chart $(V_\alpha, \psi_\alpha)$, a symmetric representative $(g_\varepsilon)_\varepsilon$ of $g$ on $V_\alpha$
and a zero sequence $\varepsilon_k$ in $I$ such that the sequence $(\nu_k)_k$ of indices $\nu_k$ of $g_{\varepsilon_k}\mid_{V_\alpha}$
has at least two accumulation points, say $\alpha\neq\beta$. Let $(x_\varepsilon)_\varepsilon$ lie in $\psi_\alpha(V_\alpha)$ for each $\varepsilon$.
Therefore the number of negative eigenvalues of $(g_{ij})_{ij}:=(g_{\alpha,ij}^\varepsilon(x_\varepsilon))_{ij}$ is not constant for sufficiently small $\varepsilon$, and
therefore for $\widetilde x:=[(x_\varepsilon)_\varepsilon]$, the respective bilinear form $g_{\alpha}(\widetilde x)$ induced by $(g_{ij})_{ij}$ with respect to the canonical basis of $\widetilde{\mathbb R}$ has no index and we are done.
\end{proof}
Before we go on to define the notion of causality of vector fields with respect to a generalized metric of Lorentz signature, we
introduce the notion of strict positivity of functions (in analogy with strict positivity of generalized numbers, cf.\ section \ref{secin}):
\begin{definition}
A function $f\in\mathcal G(X)$ is called strictly positive in $\mathcal G(X)$, if for any compact subset $K\subset X$ there
exists some representative $(f_{\varepsilon})_{\varepsilon}$ of $f$ such that for some $(m,\varepsilon_0)\in\mathbb R\times I$ we have $\forall\; \varepsilon\in (0,\varepsilon_0]:\inf_{x\in K}\vert f_{\varepsilon}(x)\vert>\varepsilon^m$. We write $f>0$. $f\in\mathcal G(X)$ is called strictly negative in $\mathcal G(X)$, if $-f>0$ on $X$.
\end{definition}
If $f>0$ on $X$, it follows that the condition from above holds for any representative. Also, $f>0$ implies that $f$ is invertible (cf.\
Theorem \ref{downhilliseasier} below). Before giving the main result of this section, we have to characterize
strict positivity (or negativity) of generalized functions by strict positivity (or negativity) in $\widetilde{\mathbb R}$. Denote
by $X_c^{\sim}$ the set of compactly supported points on $X$. Suitable modifications of point-wise characterizations of generalized functions (as Theorem 2.\ 4 in \cite{MO1}, pp.\ 150) or of point-wise characterizations of positivity
(e.\ g., Proposition 3.\ 4 in (\cite{PSMO}, p.\ 5) as well, yield:
\begin{proposition}
For any element $f$ in $\mathcal G(X)$ we have:
\[
f>0 \Leftrightarrow \forall\; \widetilde x\in X_c^{\sim}: f(\widetilde x)>0.
\]
\end{proposition}
Now we have the appropriate machinery at hand to characterize causality of generalized vector fields:
\begin{theorem}\label{characterizationcausality}
Let ${\xi}\in\mathcal G^1_0(X)$, $ g\in \mathcal G^0_2(X)$ be a Lorentzian metric. The following are equivalent:
\begin{enumerate}
\item \label{causalitychar1} For each chart $(V_{\alpha},\psi_{\alpha})$ and each $\widetilde x \in (\psi_{\alpha}(V_{\alpha}))_c^{\sim}$, ${\xi}_{\alpha}(\widetilde x)\in\widetilde{\mathbb R}^n$ is time-like (resp.\ space-like, resp.\ null)
with respect to $ g_{\alpha}(\widetilde x)$ (a symmetric bilinear form on $\widetilde{\mathbb R}^n$ of Lorentz signature).
\item \label{causalitychar2} $ g({\xi},\xi)< 0$  (resp.\ $>0$, resp.\ $=0$)
in $\mathcal G(X)$.
\end{enumerate}
\end{theorem}
\begin{proof}
(\ref{causalitychar2})$\Leftrightarrow$ $\forall\;\widetilde x\in X_c^{\sim}: g(\xi,\xi)(\widetilde x)<0$ (due to the preceding proposition) $\Leftrightarrow$ for each chart $(V_{\alpha},\psi_{\alpha})$ and for all $\widetilde x_c\in \psi_{\alpha}(V_{\alpha})_c^{\sim}: g_{\alpha}(\widetilde x)(\xi_{\alpha}(\widetilde x),\xi_{\alpha}(\widetilde x))<0$ in $\widetilde{\mathbb R}$ $\Leftrightarrow$ (\ref{causalitychar1}).
\end{proof}
The preceding theorem gives rise to the following definition:
\begin{definition}\label{defcausalityglobal}
A generalized vector field $ \xi\in \mathcal G^1_0(X)$ is called time-like (resp.\ space-like, resp.\ null) if
it satisfies one of the respective equivalent statements of Theorem \ref{characterizationcausality}. Moreover, two time-like vector fields $\xi,\eta$  are said to have the same time orientation, if
$\langle \xi,\eta\rangle<0$. Due to the above, this notion is consistent with the point-wise one given in
\ref{causaldef1}.
\end{definition}
We conclude this section by harvesting constructions of generalized pseudo-\\Riemannian metrics by means of
point-wise results of the preceding section in conjunction with the point-wise characterizations of the global objects of this chapter:
\begin{theorem}\label{genRiemannmetricconstrglobal}
Let $g$ be a generalized Lorentzian metric and let $\xi,\eta\in \mathcal G ^1_0(X)$ be time-like vector fields with the same time orientation. Then
\[
h_{ab}:=\xi_{(a}\eta_{b)}-\frac{1}{2}\langle \xi,\eta\rangle g_{ab}
\]
is a generalized Riemannian metric.
\end{theorem}
\begin{proof}
Use Proposition \ref{metrconstr} together with Theorem \ref{characterizationcausality} and Theorem \ref{charindexg}.
\end{proof}
\section[Appendix. Invertibility revisited]{Appendix. Invertibility and strict positivity in generalized function algebras revisited}\label{secin}
This section is devoted to elaborating a new characterization of invertibility as well as of strict positivity of generalized numbers resp.\ functions. The first investigation on which many works in this field are based was done by M.\ Kunzinger and R.\ Steinbauer in \cite{KS1}; the authors of the latter work
established the fact that invertible generalized numbers are precisely such for which the modulus of any representative is bounded from below by a fixed power of the smoothing parameter (cf.\ the proposition below). It is, however, remarkable, that (as the following statement shows) component-wise invertibility suffices: We here show that a number is invertible if each component of any representative is invertible for sufficiently small smoothing parameter.
\begin{proposition}\label{genpointinv}
Let $\gamma\in\widetilde{\mathbb R}$. The following are equivalent:
\begin{enumerate}
\item \label{inv1} $\gamma$ is invertible.
\item \label{inv2} $\gamma$ is strictly nonzero, that is: for some (hence any) representative $(\gamma_{\varepsilon})_{\varepsilon}$ of $\gamma$ there exists a $m_0$ and a $\varepsilon_0\in I$ such that for each $\varepsilon<\varepsilon_0$ we have
$\vert \gamma_{\varepsilon}\vert>\varepsilon^{m_0}$.
\item \label{inv3} For each representative $(\gamma_{\varepsilon})_{\varepsilon}$ of $\gamma$ there exists some $\varepsilon_0\in I$ such that for all
$\varepsilon<\varepsilon_0$ we have $\alpha_{\varepsilon}\neq 0$.
\item \label{inv4} $\vert \gamma\vert$ is strictly positive.
\end{enumerate}
\end{proposition}
\begin{proof}
Since (\ref{inv1}) $\Leftrightarrow$ (\ref{inv2}) by (\cite{KS1}, Theorem 1.2.38) and (\ref{inv1}) $\Leftrightarrow $ (\ref{inv4}) follows from the definition of
strict positivity, we only need to establish the equivalence (\ref{inv2}) $\Leftrightarrow$ (\ref{inv3}) in order to complete
proof.
As the reader can easily verify, the definition of strictly non-zero is independent of the representative, that is for each representative $(\gamma_{\varepsilon})_{\varepsilon}$ of $\gamma$ we have
some $m_0$ and some $\varepsilon_0$ such that for all $\varepsilon<\varepsilon_0$ we have $\vert\gamma_{\varepsilon}\vert>\varepsilon^{m_0}$. By this consideration (\ref{inv3}) follows from (\ref{inv2}).
In order to show the converse direction, we proceed by an indirect argument. Assume there exists some representative $(\gamma_{\varepsilon})_{\varepsilon}$
of $\gamma$ such that for some zero sequence $\varepsilon_k\rightarrow 0$ ($k\rightarrow \infty$) we have $\vert\gamma_{\varepsilon_k}\vert<\varepsilon_k^k$ for each $k>0$.
Define a moderate net $(\hat{\gamma}_{\varepsilon})_{\varepsilon}$ in the following way:
\[
\hat{\gamma}_{\varepsilon}:=\begin{cases} 0\qquad\mbox{if}\qquad \varepsilon=\varepsilon_k\\ \gamma_{\varepsilon} \qquad\mbox{otherwise} \end{cases}.
\]
It can then easily be seen that $(\hat{\gamma}_{\varepsilon})_{\varepsilon}-(\gamma_{\varepsilon})_{\varepsilon}\in \mathcal N(\mathbb R)$
which means that $(\hat{\gamma}_{\varepsilon})_{\varepsilon}$ is a representative of $\gamma$ as well. However the latter violates
(\ref{inv3})  and we are done.
\end{proof}
Analogously we can characterize the strict order relation on the generalized real numbers:
\begin{proposition}
Let $\gamma\in\widetilde{\mathbb R}$. The following are equivalent:
\begin{enumerate}
\item \label{inv11} $\gamma$ is strictly positive, that is: for some (hence any) representative $(\gamma_{\varepsilon})_{\varepsilon}$ of $\gamma$ there exists an $m_0$ and an $\varepsilon_0\in I$ such that for each $\varepsilon<\varepsilon_0$ we have
$\gamma_{\varepsilon}>\varepsilon^{m_0}$.

\item \label{inv21} $\gamma$ is strictly nonzero and has a representative $(\gamma_{\varepsilon})_{\varepsilon}$ which is positive for each index $\varepsilon>0$.
\item \label{inv31} For each representative $(\gamma_{\varepsilon})_{\varepsilon}$ of $\gamma$ there exists some $\varepsilon_0\in I$ such that for all
$\varepsilon<\varepsilon_0$ we have $\alpha_{\varepsilon}> 0$.
\end{enumerate}
\end{proposition}
The statement can be shown in a similar manner as the the preceding one.

Next, we may note that the above has an immediate generalization to generalized functions. Here $X$ denotes a paracompact, smooth Hausdorff manifold of dimension $n$.
\begin{theorem}\label{downhilliseasier}
Let $u\in\mathcal G(X)$. The following are equivalent:
\begin{enumerate}
\item \label{invf1} $u$ is invertible (resp.\ strictly positive).
\item \label{invf2} For each representative $(u_{\varepsilon})_{\varepsilon}$ of $u$ and each compact set $K$ in $X$ there exists some
$\varepsilon_0\in I$ and some $m_0$ such that for all $\varepsilon<\varepsilon_0$ we have $\inf_{x\in K} \vert u_{\varepsilon}\vert>\varepsilon^{m_0}$ (resp.\ $\inf_{x\in K} u_{\varepsilon}>\varepsilon^{m_0}$).
\item  \label{invf3} For each representative $(u_{\varepsilon})_{\varepsilon}$ of $u$ and each compact set $K$ in $X$ there exists some
$\varepsilon_0\in I$ such that $\forall\; x\in K\;\forall\; \varepsilon<\varepsilon_0: u_{\varepsilon}\neq 0$ (resp.\ $u_{\varepsilon}>0$).
\end{enumerate}
\end{theorem}
\begin{proof}
We only show that the characterization of invertibility holds, the rest of the statement is then clear. Since (\ref{invf1})$\Leftrightarrow$(\ref{invf2}) due to (\cite{KS1}, Proposition 2.1) we only need to establish the equivalence
of the third statement. Since (\ref{invf2})$\Rightarrow$(\ref{invf3}) is evident, we finish the proof by showing the converse direction.
Assume (\ref{invf2}) does not hold, then there exists a compactly supported sequence $(x_k)_k\in X^{\mathbb N}$ such that for some representative $(u_{\varepsilon})_{\varepsilon}$ of $u$ we have $\vert u_{\varepsilon_k}(x_k)\vert<\varepsilon_k^k$
for each $k$. Similarly to the proof of Proposition \ref{genpointinv} we observe that $(\hat u_{\varepsilon})_{\varepsilon}$ defined as
\[
\hat{u}_{\varepsilon}:=\begin{cases} u_{\varepsilon}-u_{\varepsilon}(x_k)\qquad\mbox{if}\qquad \varepsilon=\varepsilon_k\\ u_{\varepsilon} \qquad\mbox{otherwise}\end{cases}
\]
yields another representative of $u$ which, however, violates (\ref{invf3}) and we are done.
\end{proof}

\chapter{The wave equation on singular
space-times}\label{chapterwaveeq} We are interested in a local
existence and uniqueness result for the scalar wave equation on a
generalized four dimensional space-time $(\mathcal M,g)$, the
Lorentzian metric $g$ being modeled as a symmetric generalized tensor
field $g\in \mathcal G_2^0(\mathcal M)$ with index $\nu=1$. 

As usual the d'Alembertian $\Box$ is defined by
\[
\Box:=\nabla^a\nabla_a:=g^{ab}\nabla_a\nabla_b
\]
where $\nabla$ denotes the covariant derivative induced by $g$. The
appropriate initial value problem for the wave equation shall be
formulated as soon as we have introduced the specific class of
generalized metrics subject to our discussion.
\section{Preliminaries}\label{preliminarywaves}
To start with, we collect some basic material from (smooth)
Lorentzian geometry and fix some notation. Throughout this section,
$(\mathcal M,g)$ denotes a smooth space-time. We follow the
convention that the signature of $g$ is $(-,+,+,+)$.  The (quite
standard) constructions revisited in the subsections
\ref{constrsubs}, \ref{energytensorsubs} below have suitable
generalizations in the Colombeau setting; these are established in
chapter \ref{chaptercausality}.
\subsection{Constructions of Riemannian metrics from Lorentzian
metrics}\label{constrsubs} The final results in the end of this
section involve point-wise arguments. Therefore, we start by
recalling elementary results from four-dimensional Minkowski
space-time $(M,\eta_{\mu\nu})$ (where $\eta=\diag (-1,1,1,1)$ and
$M=\mathbb R^4$). Following the convention concerning the signature
of the Lorentzian metric, we have the following conventions on
causality (using the notation $\langle
\xi,\eta\rangle:=g_{ab}\xi^a\eta^b$): A vector $\xi\in M$ is called
\begin{enumerate}
\item time-like, if $\langle \xi,\xi\rangle<0$,
\item space-like, if $\langle \xi,\xi\rangle>0$ and
\item null, if $\langle \xi,\xi\rangle=0$.
\end{enumerate}
It should be noted that we follow the convention that $\xi=0$ is
defined to be a null vector. To begin with we show:
\begin{lemma}
Let $u, v$ be time-like vectors in $(M,\eta_{\mu\nu})$ such that
$\langle u,u\rangle=\langle v,v\rangle=-1$ and $\langle
u,v\rangle<0$ (that is, $u$ and $v$ have the same time-orientation).
Then the following statements hold:
\[
L^{\mu}_{\nu}:=\delta^{\mu}_{\nu}-2v^{\mu}u_{\nu}+\frac{(u^{\mu}+v^{\mu})(u_{\nu}+v_{\nu})}{1-\langle
u,v\rangle}
\]
is a Lorentz Transformation, meaning \[
L^{\mu}_{\nu}L^{\lambda}_{\rho}\eta_{\mu\lambda}=\eta_{\nu\rho},
\]
and has the property $Lu=v$.
\end{lemma}
\begin{proof}
The first part of the statement is shown by means of simple
algebraic manipulations:
\begin{flushleft}
$L^{\mu}_{\nu}L^{\lambda}_{\rho}\eta_{\mu\lambda}=$\\$\left(\delta^{\mu}_{\nu}-2v^{\mu}u_{\nu}+\frac{(u^{\mu}+v^{\mu})(u_{\nu}+v_{\nu})}{1-\langle
u,v\rangle}\right)\left(\delta^{\lambda}_{\rho}-2v^{\lambda}u_{\rho}+\frac{(u^{\lambda}+v^{\lambda})(u_{\rho}+v_{\rho})}{1-\langle
u,v\rangle}\right)\eta_{\mu\lambda}=$\\$
\left(\delta^{\mu}_{\nu}-2v^{\mu}u_{\nu}+\frac{(u^{\mu}+v^{\mu})(u_{\nu}+v_{\nu})}{1-\langle
u,v\rangle}\right)\left(\eta_{\mu\rho}
-2v_{\mu}u_{\rho}+\frac{(u_{\mu}+v_{\mu})(u_{\rho}+v_{\rho})}{1-\langle
u,v\rangle}\right)=$\\$\eta_{\nu\rho}-2v_{\nu}u_{\rho}+\frac{(u_{\nu}+v_{\nu})(u_{\rho}+v_{\rho})}{1-\langle
u,v\rangle}-2 v_{\rho}u_{\nu}+ 4\langle v,v\rangle
u_{\nu}u_{\rho}+$\\$\frac{(-2 \langle u,v\rangle u_{\nu}-2 \langle
v,v\rangle u_{\nu})(u_{\rho}+v_{\rho})}{1-\langle
u,v\rangle}+\frac{(u_{\rho}+v_{\rho})(u_{\nu}+v_{\nu})}{1-\langle
u,v\rangle}+$\\$\frac{(-2\langle u,v\rangle u_{\rho}-2\langle
v,v\rangle u_{\rho})(u_{\nu}+v_{\nu})}{1-\langle
u,v\rangle}+\frac{(\langle u,u\rangle+\langle u,v\rangle+\langle
u,v\rangle+\langle v,v\rangle
)(u_{\nu}+v_{\nu})(u_{\rho}+v_{\rho})}{(1-\langle
u,v\rangle)^2}=$\\$\eta_{\nu\rho}-2
v_{\nu}u_{\rho}-2v_{\rho}u_{\nu}-4
u_{\nu}u_{\rho}+2u_{\nu}(u_{\rho}+v_{\rho})+2u_{\rho}(u_{\nu}+v_{\nu})=$\\$\eta_{\nu\rho}$.
\end{flushleft}
The other claim is obtained by a further calculation:
\begin{eqnarray}
L_\nu^{\mu}u^{\nu}&=&u^{\mu}-2 v^{\mu}\langle u,u\rangle
+\frac{(u^{\mu}+v^{\mu})(\langle u,u\rangle+\langle
u,v\rangle)}{1-\langle
u,v\rangle}=\\\nonumber&=&u^{\mu}+2v^{\mu}-u^{\mu}-v^{\mu}=v^{\mu},
\end{eqnarray}
that is, $Lu=v$ and we are done.
\end{proof}
Constructions of Riemannian metrics by means of Lorentzian metrics and
time-like vector fields will be used later on. Here is the
result in full generality (we will also use simpler constructions,
where $u=v$, cf.\ the corollary below):
\begin{lemma}\label{posdefmetricfromguv}
Let $u,v$ be time-like vectors in ($M,\eta$) with the same
time-orientation. Then
\[
h_{ab}^{u v}:=u_{(a}v_{b)}-\frac{1}{2}\langle u,v\rangle \eta_{ab}
\]
is a symmetric positive definite bilinear form on $M$.
\end{lemma}
\begin{proof}

{\it Step 1.} \\ By scaling $u,v$ appropriately it can be seen that we
may assume without loss of generality that $u^2=v^2=-1$ and that
$u,v$ lie in the future light cone.\\
 {\it Step 2.} \\By the preceding lemma, the Lorentz group acts
 transitively on the future light cone. Therefore, there exists a
 Lorentz transformation $L_1$ such that $\bar u:=L_1 u=(1,0,0,0)$ and we set
 $\bar v:=L_1v$. By means of a rotation $L_2$ of the space coordinates
it can further be achieved that $\hat u:=L_2\bar u=(1,0,0,0)$ and
$\hat v:=L_2\bar v=L_2L_1v=\gamma(V)(1, V,0,0)$ with
$\gamma(V)=(1-V^2)^{-1/2},\;\vert V\vert <1$.\\ {\it Step 3.}\\ We
denote by $L:=L_2L_1$ the composition of the two Lorentz
transformations $L_1,L_2$. With this notation we have by the above,
$\hat u=Lu,\, \hat v=Lv$. Since $h_{ab}^{u v}$ is evidently a
symmetric bilinear form, we only need to show that for each non-zero
vector $w$, we have $h^{u v}(w,w)>0$. Since for $\hat w:=Lw$, $h^{u
v}(w,w)=h^{\hat u \hat v}(\hat w,\hat w)$, and since $L$ is a linear
isomorphism, it therefore suffices to show that for each non-zero
$w$, $h^{\hat u \hat v}(w,w)>0$. Let $w=(w^1,w^2,w^3,w^4)\in M,
w\neq 0$ and set $h=h^{\hat u\hat v}$. Then we have
\[
h(w,w):=h_{ab}w^aw^b= \langle \hat u,w\rangle  \langle \hat
v,w\rangle-\frac{1}{2}\langle w,w\rangle \langle \hat u,\hat
v\rangle.
\]
Obviously, $\langle \hat u,w \rangle=-w^1,\langle \hat
v,w\rangle=\gamma(V)(-w^1+Vw^2),\langle \hat u,\hat
v\rangle=-\gamma(V)$. Thus
\begin{eqnarray}\nonumber
h(w,w)&=&\gamma(V)(-w^1)(-w^1+Vw^2)+\frac{1}{2}\gamma(V)(-(w^1)^2+(w^2)^2+(w^3)^2+(w^4)^2)\\\nonumber&=&-\gamma(V)
V w^1w^2+\frac{1}{2}\gamma(V) (+(w^1)^2+(w^2)^2+(w^3)^2+(w^4)^2)
\end{eqnarray}
If $Vw^1w^2\leq 0$, then we are done. Otherwise $Vw^1w^2=\vert
V\vert \vert w^1\vert \vert w^2\vert< \vert
w^1w^2\vert\leq\frac{(w^1)^2+(w^2)^2}{2}$, because of $\vert V\vert
<1$ and $Vw^1w^2\neq 0$. Inserting this information into the latter
equation yields
\begin{eqnarray}\nonumber
h(w,w)&=&-\gamma(V)
Vw^1w^2+\frac{1}{2}\gamma(V)(+(w^1)^2+(w^2)^2+(w^3)^2+(w^4)^2)>\\\nonumber
&>& \frac{1}{2}\gamma(V)((w^3)^2+(w^4)^2)\geq 0,
\end{eqnarray}
i.\ e.\ $h(w,w)>0$ and we are done.
\end{proof}
An immediate corollary is:
\begin{corollary}\label{corollaryriemannianmetricconstruction}
Let ($\mathcal M,g$) be a smooth space-time. Let $\xi,\eta$ be
time-like vector fields on ($\mathcal M,g$) with the same time
orientation. Then $h_{ab}:=\xi_{(a}\eta_{b)}-\frac{1}{2}\langle
\xi,\eta\rangle g_{ab}$ is a Riemannian metric on $\mathcal M$. As a
consequence we have: if $\theta$ is a time-like unit vector field,
then also $k_{ab}:=g_{ab}+2\theta_a\theta_b$ is a Riemannian metric.
\end{corollary}
\begin{proof}
Let $p\in M$ and choose a local chart $(U,\xi)\ni p$ such that the
coordinate expression of $g$ is Minkowskian at $p$. Then we are in
the setting of Lemma \ref{posdefmetricfromguv}, according to which $h_{ab}$ is a positive definite bilinear form at $p$.
Furthermore $h_{ab}$ is smooth, since $\xi,\eta$ and $g$ are.

To prove the second assertion, we set $\xi=\eta=\theta$. Due to the first claim,
$k_{ab}=2 (\xi_{(a}\eta_{b)}-\frac{1}{2}\langle
\xi,\eta\rangle g_{ab})=g_{ab}+2\theta_a\theta_b$ is a Riemannian metric, and we are done.
\end{proof}
A remark on the Riemannian metric constructed above is in order. The first observation is, that in general
$h^{ab}:=2(\xi^{(a}\eta^{b)}-\frac{1}{2}\langle
\xi,\eta\rangle g^{ab})$ is not the inverse of $h_{ab}=2(\xi_{(a}\eta_{b)}-\frac{1}{2}\langle
\xi,\eta\rangle g_{ab}$) as defined in the preceding corollary, but just the metric equivalent covariant tensor. However, if
$\xi=\eta$, then it is the case! For the sake of simplicity, we assume $\langle \theta,\theta\rangle=-1$. Then
we have $k_{ab}=2 h_{ab}=g_{ab}+2\theta_a\theta_b$, and similarly, $k^{bc}=2 h^{bc}=g^{bc}+2\theta^b\theta^c$. Therefore we obtain
\begin{equation}
k_{ab}k^{bc}=(g_{ab}+2\theta_a\theta_b)(g^{bc}+2\theta^b\theta^c)=\delta_a^c+2\theta_a\theta^c+2\theta_a\theta^c-4\theta_a\theta^c=\delta_a^c,
\end{equation}
and we have shown the assertion.

We shall make use of such metric constructions in the definition of
certain energy integrals (cf.\ section \ref{energyestimates1}).
However, in order to entirely understand their structure we
investigate further in energy tensors and certain positivity
statements, which in the physics literature are referred to as
"dominant energy condition(s)":
\subsection{Energy tensors and dominant energy
condition}\label{energytensorsubs} Let $(\mathcal M, g)$ be a smooth
space-time. The statement of this section are to be understood
point-wise. We start to revisit a notion of (\cite{HE}, pp.\ 90).
\begin{definition}\label{DECsmooth}
A symmetric tensor $T^{ab}$ is said to satisfy the dominant energy
condition if for every time like vector $\xi^a$,
$\eta^b:=T^{ab}\xi_a$ is not space-like and if further
$T^{ab}\xi_a\xi_b\geq 0$.
\end{definition}
A remark on this is in order: The condition $T^{ab}\xi_a\xi_b\geq 0$
implies that the non-space like vector $-\eta^b=-T^{ab}\xi_a$ has
the same time-orientation as $\xi^a$. This follows from
\[
-\eta^b \xi_b=-T^{ab}\xi_a\xi_b\leq 0,
\]
that is $g(\xi,-\eta)\leq 0$, which is equivalent to saying that
$\xi,-\eta$ have the same time-orientation. \\
A consequence of the
dominant energy condition is the following
\begin{lemma}\label{Declemmaconsequence}
Let $T^{ab}$ be a symmetric tensor satisfying the dominant energy
condition. Then for any time-like vectors $\xi^a,\eta^b$ with the
same time-orientation, we have
$T^{ab}\xi_a\eta_b\geq 0$.
\end{lemma}
\begin{proof}
By the dominant energy condition, $\theta^b:=T^{ab}\xi_a$ is
time-like or null, and $-\theta^a$ has the same time-orientation as
$\xi^a$, that is $g_{ab}\xi^a(-\theta^b)\leq 0$. Therefore, by
assumption, $-\theta^b$ also has the same time-orientation as
$\eta^c$. As a consequence we have
\[
-T^{ab}\xi_a\eta_b=g_{ab}\eta^a(-\theta^b)\leq 0,
\]
and we are done.
\end{proof}
Following J.\ Vickers and J.\ Wilson (\cite{VW}) we define a class of (symmetric) energy tensors $T^{ab,k}$. Let $e_{ab}$ be
a Riemannian metric with $e^{ab}$ its inverse, let $W_{a_1\dots a_k}$
be an arbitrary tensor of type $(0,k)$, $k\geq 0$ and let
$\xi^a,\eta^b$ be time-like vectors with the same time-orientation.
We define for $k=0$
\[
T^{ab,0}(W):=-\frac{1}{2}g^{ab}W^2,
\]
and for $k\geq 1$, we set
\[
T^{ab,k}(W):=(g^{ac}g^{bd}-\frac{1}{2}g^{ab}g^{cd})e^{p_1q_1}\dots
e^{p_{k-1}q_{k-1}}W_{cp_1\dots p_{k-1}}W_{dq_1\dots q_{k-1}}.
\]
Then we have the following:
\begin{proposition}\label{Butter}
For each $k\geq 0$, $T^{ab,k}(W)$ is a symmetric tensor which satisfies the dominant energy
condition.
\end{proposition}
\begin{proof}
The case $k=0$ is trivial. Hence we start with $k=1$. We have
\begin{eqnarray}\nonumber
\eta^b:=(g^{ac}g^{bd}-\frac{1}{2}g^{ab}g^{cd})\xi_a
W_cW_d&=&(\xi^cg^{bd}-\frac{1}{2}\xi^b g^{cd})W_cW_d=\\\nonumber
&=&\xi^cW_cW^b-\frac{1}{2}\xi^b W^d W_d=\\\nonumber
&=&W(\xi)W^b-\frac{1}{2}\xi^b \langle W,W\rangle.
\end{eqnarray}
From this we obtain
\begin{eqnarray}\nonumber
g(\eta,\eta)&=&\eta^b\eta_b=(W(\xi)W^b-\frac{1}{2}\xi^b\langle W,
W\rangle)(W(\xi)W_b-\frac{1}{2}\xi_b\langle W,
W\rangle)=\\\nonumber&=&\frac{1}{4}\langle \xi,\xi\rangle \langle
W,W\rangle^2\leq 0,
\end{eqnarray}
where the last inequality holds because $\xi^a$ is time-like. We
have therefore shown that $\eta^b=T^{ab,1}\xi_a$ is time-like or
null. It remains to show that the time-orientation of $-\eta^b$ is
the same as the one of $\xi^a$:
\[
T^{ab,1}(W)\xi_a\xi_b=\{(g^{ac}g^{bd}-\frac{1}{2}g^{ab}g^{cd})\xi_a\xi_b\}W_cW_d=\xi^c
W_c\xi^d W_d-\frac{1}{2} \xi^a\xi_a W^b W_b.
\]
Due to Corollary \ref{corollaryriemannianmetricconstruction},
\[
\{(g^{ac}g^{bd}-\frac{1}{2}g^{ab}g^{cd})\xi_a\xi_b\}=\xi^c\xi^d-\frac{1}{2}\langle \xi,\xi\rangle g^{cd}
\]
is a Riemannian metric, therefore,
\[
T^{ab,1}(W)\xi_a\xi_b\geq 0
\]
and we are done with the case $k=1$.

We reduce the proof for higher orders $k>1$ to the case $k=1$. To
this end, fix $p\in \mathcal M$ and let $\mathcal B:=\{b_1,\dots,
b_4\}$ be an orthonormal basis of $(T_p M)^*$ with respect to
$e^{ab}$. With respect to this basis $T^{ab,k}(W)$ reads
\begin{eqnarray}\nonumber
T^{ab,k}(W):&=&(g^{ac}g^{bd}-\frac{1}{2}g^{ab}g^{cd})\delta
^{p_1q_1}\dots \delta^{p_{k-1}q_{k-1}}W_{cp_1\dots
p_{k-1}}W_{dq_1\dots q_{k-1}}=\\\nonumber &=&\sum_{p_1\dots
p_{k-1}}(g^{ac}g^{bd}-\frac{1}{2}g^{ab}g^{cd}) W_{cp_1\dots
p_{k-1}}W_{d p_1\dots p_{k-1}}.
\end{eqnarray}
Now for each tupel $(p_1,\dots,p_{k-1})$ we have as in the case
$k=1$,
\[
(g^{ac}g^{bd}-\frac{1}{2}g^{ab}g^{cd}) W_{cp_1\dots p_{k-1}}W_{d
p_1\dots p_{k-1}}\xi_a\xi_b\geq 0.
\]
Therefore, by summing over all these indices, we have
\[
T^{ab,k}(W)\xi_a\xi_b\geq 0.
\]
It remains to show that $T^{ab,k}(W)\xi_a$ is time-like or null, supposing
that $\xi_a$ is time-like. To show this, we use the following
property of the light cone: For each $\lambda,\mu\geq 0, \lambda+\mu>0$ and each
$v^a,w^a$ in the future (resp.\ past) light cone, also $\lambda v^a+\mu w^a$
lies in the future (resp.\ past) light cone.

Again, we may reduce to the case $k=1$, and see that for each tuple
$(p_1,\dots,p_{k-1})$,
\[
-\theta^b_{p_1,\dots,p_{k-1}}:=-(g^{ac}g^{bd}-\frac{1}{2}g^{ab}g^{cd}) W_{cp_1\dots p_{k-1}}W_{d p_1\dots p_{k-1}}\xi_a
\]
lies in the same light cone as $\xi_a$. Therefore, by the convexity
property of the light cone, also the sum over all such indices does,
that is,
\[
-T^{ab,k}(W)\xi_a=\sum_{p_1\dots p_{k-1}}-\theta^b_{p_1,\dots,p_{k-1}}
\]
is time-like or null, and we are done.
\end{proof}
As a consequence of Lemma \ref{Declemmaconsequence} and Proposition
\ref{Butter}, we have for all time-like vectors with the same
time-orientation,
\[
T^{ab,k}(W)\xi_a\eta_b\geq 0.
\]
This also may be concluded by  directly applying Corollary
\ref{corollaryriemannianmetricconstruction} by means of which we have the even stronger result:
\begin{corollary}
For each non-zero tensor $W_{a_1,\dots,a_k}$, and for all time-like
vectors $\xi^a,\eta^b$ with the same time-orientation, we have
\begin{equation}
T^{ab,k}(W)\xi_a\eta_b>0
\end{equation}
\end{corollary}
\begin{proof}
By corollary \ref{corollaryriemannianmetricconstruction},
\[
h^{cd}:=(g^{a(c}g^{d)b}-\frac{1}{2}g^{ab}g^{cd})\xi_a\eta_b
\]
is a Riemannian metric. Therefore, $h^{cd}e^{p_1q_1}\dots
e^{p_{k-1}q_{k-1}}$ is a Riemannian metric on $\otimes_{i=1}^k (TM)^*$as
well, and since $W\neq 0$, we have
\[
T^{ab,k}(W)\xi_a\eta_b=h^{cd}e^{p_1q_1}\dots
e^{p_{k-1}q_{k-1}}W_{cp_1\dots p_{k-1}}W_{dq_1\dots q_{k-1}}>0
\]
and we have shown the claim.
\end{proof}
Finally, we mention that the dominant energy condition has
recently been generalized to a so-called super energy condition on
super-energy tensors (cf.\ \cite{Senovilla}). 
\subsection{The d'Alembertian in
local coordinates} The aim of this section is to justify the
coordinate form of the d'Alembertian.
\begin{lemma}\label{Camenbert}
Let $g$ be a smooth Lorentzian metric. In local coordinates $(x^i)$
$(i=1,\dots,4)$, the d'Alembertian takes the form
\begin{equation}\label{coordformdAlembertian}
\Box u=\vert g\vert ^{-\frac{1}{2}}\partial_i(\vert g\vert
^{\frac{1}{2}}g^{ij}\partial_ju).
\end{equation}
\end{lemma}
\begin{proof}
Let $U$ be the domain of the coordinate chart system
$\xi=(x^1,\dots,x^4)$. By (\cite{ON}, Lemma 19, p.\ 195), there
exists a volume Element $\omega$ on $U$ such that
\begin{equation}\label{lemma19}
\omega(\partial_1,\dots,\partial_4)=\vert g\vert^{\frac{1}{2}}
\end{equation}
(the proof essentially uses local orthogonal frame fields). A
further fact (\cite{ON}, Lemma 21, p.\ 195) is that for any local
volume element $\omega$ on $\mathcal M$ we have
\begin{equation}\label{lemma21}
\mathcal (L_\xi \omega )_{bcde}=(\nabla_a \xi^a) \omega_{bcde}
\end{equation}
We claim that the divergence of $\xi$ can be decomposed in the
following way:
\begin{equation}\label{subclaimnabla}
\nabla_a\xi^a=\vert g\vert ^{-\frac{1}{2}}\;\partial_a(\vert g\vert
^{\frac{1}{2}}\xi^a).
\end{equation}
Assuming that this identity holds, we may set $\xi^a:=\nabla^a u$
and derive
\begin{eqnarray}\nonumber
\Box u=\nabla_a(\nabla^a u)&=& \vert g\vert
^{-\frac{1}{2}}\;\partial_a(\vert g\vert ^{\frac{1}{2}}\nabla^a u
)=\\\nonumber &=&\vert g\vert ^{-\frac{1}{2}}\;\partial_a(\vert
g\vert ^{\frac{1}{2}}g^{ab}\nabla_b u )=\\\nonumber&=& \vert g\vert
^{-\frac{1}{2}}\;\partial_a(\vert g\vert
^{\frac{1}{2}}g^{ab}\partial_b u )
\end{eqnarray}
and we have proved the lemma. In order to show the subclaim, we
calculate the left and right hand side of (\ref{subclaimnabla})
separately. We make use of (\ref{lemma21}) and the fact that, since
 we are dealing with a $4$-form $\omega$, it is sufficient to
evaluate the formula at $(\partial_1,\dots,\partial_4)$ only: the
right side of (\ref{lemma21}) yields by means of (\ref{lemma19})
\begin{equation}\label{rightsideoflemma21}
(\nabla_a\xi^a)\;\omega(\partial_1,\dots,\partial_4)=\vert
g\vert^{\frac{1}{2}}\nabla_a\xi^a.
\end{equation}
The left side of (\ref{lemma21}) yields:
\begin{eqnarray}\label{leftsideoflemma21eq1}
\mathcal L_\xi \omega (\partial_1,\dots,\partial_4)&=&\\\nonumber
\mathcal L_\xi(\omega (\partial_1,\dots,\partial_4))-\sum_i
\omega(\partial_1,\dots,\mathcal L_\xi\partial_i,\dots,\partial_4).
\end{eqnarray}
Now we have
\begin{equation}\label{lieklammerV} \mathcal
L_\xi\partial_i=[\xi,\partial_i]=\sum_j
[\xi^j\partial_j,\partial_i]=\sum_j(\xi^i\partial_j\partial_i-\partial_i(\xi^j\partial_j))=-\sum_j(\partial_i\xi^j)\partial_j.
\end{equation}
By (\ref{lemma19}) and (\ref{lieklammerV}) we therefore obtain
\begin{eqnarray}\label{leftsideoflemma21}
\mathcal L_\xi \omega (\partial_1,\dots,\partial_4)&=&\mathcal
L_\xi(\vert g\vert^{\frac{1}{2}})+\sum_{i,j} \frac{\partial \xi
^j}{\partial
x^i}\omega(\partial_1,\dots,\partial_j,\dots,\partial_4)=\\\nonumber
&=&\sum_{i,j} \xi^i\frac{\partial (\sqrt{\vert g\vert})}{\partial
x^i}+\sum_i\frac{\partial \xi^i}{\partial
x^i}(\delta_{ij}\sqrt{\vert g\vert })=\\\nonumber
&=&\sum_i\frac{\partial}{\partial x^i}(\sqrt{\vert g\vert }\xi^i).
\end{eqnarray}
Since (\ref{leftsideoflemma21})$\equiv$(\ref{rightsideoflemma21})
because of (\ref{lemma21}) we have succeeded to show
(\ref{subclaimnabla}) and we are done with the subclaim.
\end{proof}
\subsection{General Lorentzian metrics in suitable coordinates} For
computational purposes it is advisable to find coordinates in which
the metric has a special form, such that calculations can be carried out
more easily. In this section we first recall what a metric looks
like in Gaussian normal coordinates, and we finish by showing that
in suitable coordinates a static metric can be written without $(t,x^\mu)$--cross terms. At the end of section (\ref{settingsection}) we shall
return to this topic from a generalized point of view.
\begin{theorem}\label{theoremgaussian}
Let $\Sigma$ be a three dimensional space-like manifold. Any point
$p\in\Sigma$ has a neighborhood such that in Gaussian normal
coordinates, the Lorentzian metric $g$ on $\mathcal M$ locally takes
the form
\begin{equation}\label{metrgaussiancoord}
ds^2=-V^2(t,x^\gamma)dt^2+g_{\alpha\beta}(t,x^\gamma)dx^\alpha
dx^\beta,
\end{equation}
that is, without $(t,x^\mu)$--cross terms (here the variables in
Greek letters are ranging between $1$ and $3$, therefore $x^\alpha$
denote the space-variables, whereas $x^0=t$ is the time variable).
It can further be achieved that $V^2\equiv 1$.
\end{theorem}
\begin{proof}
For the proof of this statement we follow the lines of (\cite{Wald},
pp.\ 42-43). A proof for the respective statement in a more general
context can be found in (\cite{ON}, pp.\ 199-200, Lemma 25). Since
$\Sigma$ is space-like, the normal $n^a$ is time-like at each point
of $\Sigma$. Fix $p\in\Sigma$ and assume $n^a$ (initially only
defined on $\Sigma$) is extended to a geodesically convex
neighborhood $U$ of $p$. Through each point $q\in U$ we construct the
unique geodesic $\gamma_q(t)$ with $\dot\gamma_q(t=0)=n^a(q)$. We
may now label each $q\in U\cap\Sigma$ by coordinates $x^\mu
\;(\mu=1,2,3)$, and choose $t$ as the parameter along the geodesic
$\gamma_q(t)$. Then $(U, (t(q),x^\mu(q))$ is a local chart at $p$,
and $\partial_t|_{t=0}=n^a|_{\Sigma\cap U}$. From $n^a \perp_g
\Sigma$ it follows that the $(t,x^\mu)$ cross-terms $g_{0\mu}$ of the metric
vanish at $t=0$, since $g_{0\mu}(t=0,x^\mu)=g(\partial_t,\partial_\mu)|_{t=0}$. Moreover,
since parallel transport is an isometry, we have that
$g(\partial_t,\partial_\mu)\equiv 0$ on all of $U$. We have thus proved
(\ref{metrgaussiancoord}). Since $n^a$ is time-like, we can
normalize
 it by the condition $g_{ab}n^an^b=-1$, and therefore it
can even be achieved that $V^2\equiv 1$. This completes the proof of
the theorem.
\end{proof}
Next, we define certain space-time symmetries:
\begin{definition}
A space-time $(\mathcal M,g)$ is called stationary, if there exists
a time-like vector field $\xi^a$ such that $\nabla_{(a}\xi_{a)}=0$.
This is equivalent to $\mathcal L_\xi g=0$. $\xi^a$ is called a time-like Killing vector.

A stationary space-time $(\mathcal M, g)$ with time-like Killing
vector $\xi^a$ is called static, if $\xi^a$ is
hypersurface-orthogonal, that is, through each point $p$ there is a
three dimensional space-like hypersurface $\Sigma$ such that $\xi^a$
is orthogonal to $\Sigma$.
\end{definition}

 In general, the coefficients $-V^2,\, g_{\alpha\beta}$ in (\ref{metrgaussiancoord}) which determine the metric via
 Theorem \ref{theoremgaussian}, are not
independent of the time $t$. However, if $g$ is a static space time,
we have (for a proof cf.\  the respective statement in the generalized setting, \ref{Lemmastaticgeneralized}):
\begin{theorem}\label{statmetrformtheorem}
A static space-time $(\mathcal M,g)$ can locally be written as
\begin{equation}\label{staticformofmetric}
ds^2=-V^2(x^\gamma)dt^2+g_{\alpha\beta}(x^\gamma)dx^\alpha dx^\beta.
\end{equation}
\end{theorem}
Such coordinates we call {\it static} coordinates throughout. As a
consequence of the preceding theorem, we see that the d'Alembertian
takes a quite simple form in static coordinates:
\begin{proposition}\label{staticform}
Let $(\mathcal M,g)$ be a static space-time. Let $V,g_{\alpha\beta}$
be the coefficients of $g$ in static coordinates as given in Theorem
\ref{statmetrformtheorem}. Then the d'Alembertian
takes the following form:
\begin{equation}
\Box u= -V^{-2}\partial_t^2 u+\vert g\vert
^{-1/2}\partial_\alpha\left(\vert g\vert ^{1/2}
g^{\alpha\beta}\partial_\beta\right)u.
\end{equation}
\end{proposition}
\begin{proof}
This follows basically from Lemma \ref{Camenbert} and Theorem
\ref{statmetrformtheorem}: in static coordinates the time
derivatives $\partial_t g_{ab}$ vanish, and the $(t,x^\mu)$ cross
terms of the metric vanish as well. As a consequence, we have
$\partial_t V^{-2}\equiv 0,\; \partial_t g^{\alpha\beta}\equiv 0,\;\partial_t
\vert g\vert\equiv 0$, and we are done.
\end{proof}
\subsection{The wave equation on a smooth space-time}
We begin with recalling causality notions. Let $(\mathcal M,g)$ be a
smooth time-orientable space time. For a point $q$ in $\mathcal M$,
we call $D^+(q)$ the future dependence domain of $q$, that is the
set of all points $p$ which can be reached by future
directed time-like geodesics through $p$. Furthermore, for a set $S$,
$D^+(S):=\bigcup_{q\in S} D^+(q)$ is the future emission of $S$. The
closure of the latter is denoted by $J^+(S):=\overline{D^+(S)}$.
Reversing the time-orientation, we may similarly define $D^-(q)$,
$D^-(S)$ and $J^-(S)$.

A set $S$ is called past-compact if the intersection $S\cap J^-(q)$
is compact for each $q\in S$.

 Let $S$ be a relatively compact three dimensional space-like submanifold
and let $\xi$ be a time like vector field. In the smooth setting,
local smooth solutions for the initial value problem
\begin{eqnarray}\nonumber
\Box u=f\\\label{smoothwaves} u|_{S}=v\\\nonumber \nabla^a\xi_a u|_{S}=w
\end{eqnarray}
are guaranteed to exist by the following theorem (\cite{FL1}, Theorem 5.3.2):
\begin{theorem}\label{exuniquesmoothsolweq}
Let S be a past-compact space-like hypersurface, such that $\partial
J^+(S)=S$. Suppose that $f$ is $C^{\infty}$ and that $C^{\infty}$
Cauchy data $v,w$ are given on $S$. Then the Cauchy problem
(\ref{smoothwaves}) has a unique solution in $J^+(S)$ such that
$u\in C^\infty(J^+(S))$.
\end{theorem}
\subsection{Leray forms}\label{leray}
This section is dedicated to recalling how to decompose volume
integrals inside a foliated domain.

Suppose $(\mathcal M,g)$ is a smooth space-time. Denote by $\mu$ the
volume form induced by $g$ (as mentioned above in the proof of Lemma 
\ref{Camenbert}); in coordinates we may write $\mu$ as
$\mu=\vert g\vert^{\frac{1}{2}} dt\wedge dx^1\dots\wedge dx^3$, with
$\vert g\vert$, the absolute value of the determinant of $g$ (that is $\vert g\vert=-g$).

Let $\Omega$ be an open domain in $\mathcal M$, and let $S$ be in
$C^\infty (\Omega)$ with $dS\neq 0$ on $\Omega$. Choose coordinates
$x^i (i=0,\dots,3)$ such that $S=t:=x^0$. By (\cite{FL1}, Lemma
2.9.2), we may decompose $\mu$ as
\[
\mu=dS\wedge \mu_S,
\]
with a $3-$form $\mu_S$ and the restriction of $\mu_S$ on
$S_\tau:=\{(t,x^\mu)| t=\tau\}$ is unique. We shall write
$\mu_S|_{S_\tau}=:\mu_\tau$. More explicitly, we have
\[
\mu_\tau=\vert g\vert^{\frac{1}{2}} dx^1\wedge dx^2\wedge dx^3.
\]

A consequence of Fubini's theorem in this setting is (\cite{FL1},
Lemma 2.9.3): Any locally integrable function $\psi$ with compact
support in $\Omega$ may be integrated as follows
\begin{equation}\label{fubini}
\int \psi\mu=\int d\tau\int_{S_\tau}\psi \mu_\tau.
\end{equation}
\subsection{Foliations and
integration}\label{smoothsettingfoliation} In this section we show
in which way we shall integrate energy integrals subsequently.

In particular we discuss aspects of integration and local foliations
of compact subregions of space-time which will be tailored to our
needs in such a way that Stokes' theorem can be applied in a
convenient way. This will be needed later on when we derive
estimates for an infinite hierarchy of (generalized) energy
integrals. We point out that the setting of this section is still
the smooth one; this, however, is sufficient for displaying the
concepts which will finally be used in the generalized setting .

From now on we shall suppose that the given space-time $(\mathcal M,g)$
has the following feature: Each point $p$ on a a given initial
space-like surface $\Sigma$ admits a region $\Omega$ with $p\in \Omega$
space-like boundary $S$ and $S_0$, with $S_0:=\Sigma\cap \Omega$ (cf. figure 1. Note that $\Omega$ is not a neighborhood
of $p$ in the usual topology). We call such a region semi-neighborhood of $p$.
Furthermore, we assume that $\Omega$ lies entirely in a region of
space-time which can be foliated by three dimensional space-like
hypersurfaces $\Sigma_\tau$ meaning that there exists a coordinate
system $(t,x^\mu)$ such that
\[
\Sigma_\tau:= \{(t,x ^\mu)\,|\, t=\tau)\}.
\]
Furthermore, $\Sigma=\Sigma_{\tau=0}$ and we define
$S_\tau:=\Sigma_\tau\cap \Omega$.

Let $\gamma>0$. We shall integrate over the compact region
$\Omega_\gamma$ which is the part of $\Omega$ which lies between
$\Sigma_0$ and $\Sigma_\gamma$.

Therefore, the boundary of $\Omega_\gamma$ is given by $S_0$,
$S_\gamma$ and $S_{\Omega,\gamma}:=S\cap \Omega_\gamma$ (cf.\ Figure
\ref{figure1}; note that the boundary is space-like throughout).

At the end of the present section we shall prove that in static
space-times $(\mathcal M, g)$ any point $p$ in $\Sigma$, (the local
space-like manifold through $p$ orthogonal to the given symmetry
$\xi^a$) admits such a semi-neighborhood $\Omega$, and in a
subsequent section we establish an analogous result for generalized
static space-times.
\begin{figure}\label{figure1}
   \begin{center}
   \fbox{\includegraphics[width=5.21in]{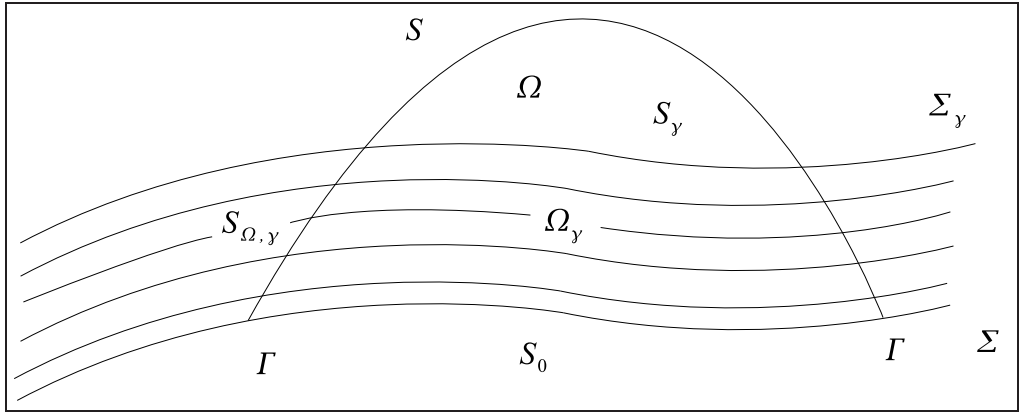}}
   \end{center}
   \caption{Local foliation of space-time}
\end{figure}
Finally, we show how to use this to integrate energies. \newline Assume
$T^{ab}$, a symmetric tensor-field of type (2,0) is given, which
satisfies the dominant energy condition. Let $\xi^a$ be a time-like Killing
vector field, and let $\Sigma_t$ be orthogonal to $\xi^a$. We denote
by $n^a$ the unit normal vector field to $S_{\Omega,\gamma}$. Let
$\mu$ be the volume element induced by the metric. We seek to
calculate
the following integral on $\Omega_\gamma$:
\begin{equation}\label{VolumeKn}
\int_{\Omega_\gamma} \xi_b\nabla_a T^{ab}\mu.
\end{equation}
First, we apply Stokes's theorem in the following fashion (cf.\
Wald, pp.\ 432--434):
\begin{theorem}
Let $N$ be an $n$--dimensional compact oriented manifold with
boundary $\partial N$, $\mu$ the natural volume element induced by the
metric $g$, and $\mu_{\partial N}$ the respective surface form on $\partial N$. Assume $\partial N$ is nowhere null. Let
further $v^a\in\mathfrak X(\mathcal M)$ and denote by $n_a$ the unit normal to $\partial N$ (that is $g^{ab}n_an_b=\pm 1$). Then we have:
\[
\int_N \nabla_a v^a \mu=\int_{\partial N} n_a v^a \mu_{\partial N}
\]
\end{theorem}
In the present setting, the boundaries of $\Omega_\gamma$ are $S_0,
S_\gamma$ with time-like normal $\xi^a$, the Killing vector, and
$S_{\Omega,\gamma}$ with normal $n^a$. In general, $\xi$
is not a unit vector field. Denote therefore by $\hat \xi:=\frac{\xi}{\sqrt{-g(\xi,\xi)}}$ the respective unit vector field.
Since $\xi^a$ is a Killing vector and $T^{ab}$ is symmetric, we have:
\[
\nabla_b(T^{ab}\xi_a)=\xi_a(\nabla_b T^{ab})+T^{ab}\nabla_b
\xi_a=\xi_b(\nabla_a T^{ab})+T^{(ab)}\nabla_{[b}
\xi_{a]}=\xi_b(\nabla_a T^{ab})+0.
\]
The integral (\ref{VolumeKn}) can therefore be decomposed in the
following way by Stokes's Theorem:
\begin{equation}\label{VolumeKn1}
\int_{\Omega_\gamma} \nabla_b(T^{ab}\xi_a)\mu=\int_{\Omega_\gamma}
\xi_b\nabla_aT^{ab}\mu=\int_{S_\gamma}T^{ab}\xi_a
\hat\xi_b\mu_\gamma-\int_{S_{0}}T^{ab}\xi_a
\hat\xi_b\mu_0+\int_{S_{\Omega,\gamma}}T^{ab}\xi_a n_b\mu_{S_{\Omega,\gamma}}.
\end{equation}
However, since $T^{ab}$ satisfies the dominant energy condition, we
have by Lemma \ref{Declemmaconsequence}:
\[
\int_{S_{\Omega,\gamma}}T^{ab}\xi_a n_b\mu_{S_{\Omega,\gamma}}\geq
0.
\]
Using this fact we conclude by means of (\ref{VolumeKn1}) that
\begin{equation}\label{VolumeKn2}
\int_{S_\gamma}T^{ab}\xi_a
\hat\xi_b\mu_\gamma\leq\int_{S_{0}}T^{ab}\xi_a
\hat\xi_b\mu_0+\int_{\Omega_\gamma} \xi_b\nabla_aT^{ab}\mu.
\end{equation}
\section{Description of the method}
We are going to prove an existence and uniqueness theorem for the scalar wave equation
in $\mathcal G(\mathcal M)$ following the method of J.\ Vickers and J.\ Wilson
(\cite{VW}) developed in the context of conical space times. Hence we generalize the result in
(\cite{VW}) from conical space times to generalized static space times. The program is as follows:
\begin{enumerate}
\item We start with specifying the ingredients of the theorem; these are in particular the
\begin{enumerate}
\item assumptions on the generalized Lorentzian metric in terms of a certain asymptotic growth behavior of the representatives. The
metric is designed for admitting local foliations of space-time by
space-like hypersurfaces. 
\item Energy integrals and Sobolev norms are introduced.
\end{enumerate}
\item Part A of the proof establishes that energy integrals (on the
three--dimensional submanifolds
$S_\tau$) and the three-dimensional Sobolev norms as defined below are equivalent. This enables us to
work with energies of arbitrary order instead of Sobolev norms.
\item Part B is devoted to providing moderate bounds on initial energies via moderate bounds on the initial data.
\item In part C we plug in the information from the wave equation into the energy integrals in order to derive an energy
inequality.
\item Part D employs Gronwall's Lemma and shows that, if the initial energies of all orders
are moderate nets of real numbers, then the same holds for all
energies for all times $0\leq\tau\leq \gamma$.
 \item Part E  employs the Sobolev embedding theorem to show that the desired asymptotic growth properties
of the solutions and their derivatives follow from the respective
growth of energies of all orders.
\item In Part F, an existence and uniqueness result is achieved by putting the pieces $A,B,C,D$ and $E$ of the puzzle together.
\item In Part G we show that the solution is independent of the choice of (symmetric) representatives of the
metric.
\end{enumerate}
It should be mentioned that Part A of the method is the crucial part
(the appropriate statement is lemma 1 in \cite{VW}); the rest of the
proof of the main theorem basically follows the lines of \cite{VW},
however, with a few modifications. Instead of using a
pseudo-foliation as Vickers and Wilson (the three dimensional
submanifolds intersect in a two dimensional submanifold of
space-time) we use the natural foliation $\Sigma_\tau:=\{t=\tau\}$
stemming from the static coordinates. Furthermore, for the purpose
of integration, we make use of the fact that the tensor-fields
$T_{ab,\varepsilon}^k(u)$ satisfy the dominant energy condition. As
a consequence of the chosen foliation, we do not need to deal with
improper integrals, as has been done in \cite{VW}.
\section{The assumptions}\label{assumptions}
\subsection{Introduction. Generalized static space-times.}
We begin with introducing a generalized static space-time.
\begin{definition}\label{defstaticgen}
Let $g\in\mathcal G^0_2(\mathcal M)$ be a generalized Lorentz
metric on $\mathcal M$. We say $(\mathcal M,g)$ is static if the
following two conditions are satisfied:
\begin{enumerate}
\item $(\mathcal M,g)$ is stationary, that is, there exists a smooth time-like vector field $\mathcal \xi$
such that $\nabla_{(a}\xi_{b)}=0$; this vector field we call Killing
as in the smooth setting and the one parameter group of isometries \footnote{To see this, note that due to identity (\ref{equivliekill})
we have $\mathcal L_\xi g \equiv 0$ in $\mathcal G$. Therefore $\frac{d}{dt}((Fl_t^\xi)^*g)(x)=(\mathcal L_\xi g)(Fl^\xi_t(x))\equiv 0$
in $\mathcal G$. This implies that $(Fl_t^\xi)^*g=((Fl_0^\xi)^*g)(x)=g$ holds in $\mathcal G$, and
we have proven that $\phi_t$ is a generalized group of isometries of $g$.}
generated by the flow of $\xi$ we denote by $\phi_t$. Following the
new concept of causality in this generalized setting (Definition
\ref{causaldef1}), $\xi$ time-like means that $g(\xi,\xi)$ is a
strictly negative generalized function on $\mathcal M$ (cf.\ Definition
\ref{defcausalityglobal}).
\item There is a three dimensional space-like hypersurface $\Sigma$
through each point of $\mathcal M$ which
is orthogonal to the orbits of the symmetry.
\end{enumerate}
\end{definition}
An important observation is the following:
\begin{theorem}\label{Lemmastaticgeneralized}
Let $(\mathcal M,g)$ be a generalized static space time. Then for
each point $p\in \mathcal U$ there exist a relatively compact open local coordinate chart
$(U,(t,x^\mu))$, $p\in U$, such that for each $\varepsilon>0$ the generalized line
element takes the form
\begin{equation}\label{metricformgenstatic}
ds_\varepsilon^2=-V^2_\varepsilon(x^1,x^2,x^3)dt^2+h_{\mu\nu}^\varepsilon(x^1,x^2,x^3)dx^\mu
dx^\nu=g^\varepsilon_{ab}dx^a dx^b
\end{equation}
where $(g_\varepsilon)_\varepsilon$ is a suitable symmetric
representative of $g$. Also in this setting we call the respective
coordinates static. Further, $V^2(x^1,x^2,x^3)$ is a strictly
positive function, and $h_{\mu\nu}(x^1,x^2,x^3)$ is a generalized
Riemannian metric on $U$.
\end{theorem}
\begin{proof}
On a relatively compact open neighborhood of $p$ we pick a symmetric representative $(g_\varepsilon)_\varepsilon$ of $g$
such that on for each $\varepsilon>0$, $g_\varepsilon$ is Lorentz (cf.\ Definition \ref{defpseud} and Theorem
\ref{chartens02} (\ref{chartens023})). Further, denote by $\nabla^\varepsilon$ the covariant
derivative induced by the metric $g_\varepsilon$.
 To show the claim we proceed in two steps.\\
 {\it Step 1.}
  As in the standard setting, an algebraic manipulation shows the
 equivalence
 \begin{equation}\label{equivliekill}
\mathcal L_\xi g_{ab}=0\Leftrightarrow \nabla_{(a}\xi_{b)}=0.
\end{equation}
Let $p\in \mathcal
 M$ lie in a relatively compact neighborhood $\Omega$ of $\Sigma$
 which can be reached by unique orbits of $\xi^a$ through $\Sigma$. Choose arbitrary
coordinates $x^\mu$ labeling $\Sigma$ and let $t$ be the Killing
parameter. Then $(t,x^\mu)$ are local coordinates near $p$ \footnote{To see this, assume the contrary, that is
$\xi_p=\xi|_p\in T_p\Sigma$. Since $\Sigma$ is space-like also $\xi_p$ is space-like, but this contradicts the assumption 
that $\xi_p$ is time-like.}. In
view of the above equivalence (\ref{equivliekill}) we have a negligible symmetric tensor
field $(n_{ab}^\varepsilon)_\varepsilon $ on $\Omega$ such that
\[
(\partial_t
g_{ab}^\varepsilon(t,x^\mu))_\varepsilon=(n_{ab}^\varepsilon(t,x
^\mu))_\varepsilon.
\]
Since $\Omega$ is relatively compact, we may replace
$g_{ab}^\varepsilon(t,x^\mu))_\varepsilon$ by $\hat
g_{ab}^\varepsilon(t,x^\mu))_\varepsilon$ which is again a local
representation of a suitable representative of the metric, given for
each $\varepsilon$ by:
\[
\hat
g_{ab}^\varepsilon(t,x^\mu):=g_{ab}^\varepsilon(t,x^\mu)-\int_0^t
n_{ab}^\varepsilon(\tau,x ^\mu)d\tau.
\]
For this representative we have in static coordinates by definition:
\[
(\partial_t \hat g_{ab}^\varepsilon(t,x^\mu))_\varepsilon=0.
\]
{\it Step 2.} Finally, we have to show that for a suitable
representative $(\widetilde g_\varepsilon)_\varepsilon$, the $(t,x)$
cross terms vanish. This is easily seen: By the hypersurface
orthogonality we know that $\langle \frac{\partial}{\partial
t},\frac{\partial}{\partial x^\mu}\rangle=g_{0\mu}=0$ in $\mathcal
G(\varphi(\Omega))$ for $\mu=1,2,3$ (($\Omega,\varphi$) denoting the
local chart) Therefore we have negligible nets
$(m_{0,\mu}^\varepsilon)_\varepsilon$ such that
\[
\hat g_{\mu,0}^\varepsilon=\hat
g_{0,\mu}^{\varepsilon}=m_{0,\mu}^\varepsilon.
\]
Since $\Omega$ was chosen to be relatively compact, we may even set
the $(t,x^\mu)$ cross terms zero and still have a local
representation of a suitable representative of $g$. We have shown
that the line element of the metric takes the form
(\ref{metricformgenstatic}).

A simple observation is, that $-V^2=g(\xi,\xi)$, therefore $V^2$ is
a strictly positive function, and $h_{\mu\nu}$ is a generalized
Riemannian metric.
\end{proof}
This concludes the general discussion of generalized space-times.
From a theoretical point of view, however, it is interesting to
further investigate characterizations of generalized space-times
$(\mathcal M,g)$ via standard space-times. We finish this section with
the following conjecture
\begin{conjecture}
On relatively compact open sets, a generalized stationary
space-time $(\mathcal M,g)$ admits a (symmetric) representative
$(g_\varepsilon)_\varepsilon$ of the metric $g$ such that $(\mathcal
M,g_\varepsilon)$ is stationary (with Killing vector $\xi^a$) for each $\varepsilon>0$.
\end{conjecture}

We are now prepared to present the setting of this note:
\subsection{The setting. Assumptions on the
metric}\label{settingassumptions} Throughout the rest of the chapter
we suppose $(\mathcal M,g)$ is a generalized static space-time.
Furthermore we shall work on $(U, ((t, x^\mu))$, ($p\in U$), an open
relatively compact chart such that according to Theorem
\ref{Lemmastaticgeneralized}, $(t, x^\mu)$ are static coordinates at
$p$.

$\xi^a$ shall denote the Killing vector field on $U$ and $\Sigma$ is
the three dimensional space-like hypersurface through $p\in U$, in
static coordinates given by $t=0$.

Let $m_{ab}$ be a background Riemannian metric on $U$ and denote by
$\|\;\|_m$ the norm induced on the fibres of the
respective tensor bundle on $U$. We further impose the following
assumptions on the metric $g$ and the Killing vector $\xi$:
\begin{enumerate}\label{settingmetric}
\item \label{setting2} $\forall\; K\subset\subset U$ and for one (hence any) symmetric representative
$(g_\varepsilon)_\varepsilon$ we have:
\[
\sup_{p\in K}\| g^\varepsilon_{ab}(p)\|_m=O(1),\qquad \sup_{p\in
K}\| g^{ab}_\varepsilon(p)\|_m=O(1)\quad\quad
(\varepsilon\rightarrow 0).
\]
\item \label{setting3} $\forall\;K\subset\subset U\;\forall\; k\in\mathbb N_0\;\forall\; \xi_1,\dots,\xi_k\in \mathfrak
X(U)$ and for one (hence any) symmetric representative
$(g_\varepsilon)_\varepsilon$ we have:
\[
\sup_{p\in K}\| \mathcal L_{\xi_1}\dots\mathcal L_{\xi_k}
g^\varepsilon_{ab}\|_m=O(\varepsilon^{-k})\quad\quad
(\varepsilon\rightarrow 0).
\]
\item \label{setting4} $\forall\;K\subset\subset  U\;\forall\; \eta\in \mathfrak X(U):$
\[
\sup_{p\in K}\| \mathcal L_\eta \hat\xi_\varepsilon
\|_m=O(1),\quad\quad (\varepsilon\rightarrow 0).
\]
where $(\hat\xi_\varepsilon)_\varepsilon:=\frac{\xi}{\sqrt{-g_\varepsilon(\xi,\xi)}}$ is a
representative of the (generalized) observer field $\hat\xi$
given by $\hat \xi:=\frac{\xi}{\sqrt{-\langle
\xi,\xi\rangle}}$. This is well defined by the fact that
$-g(\xi,\xi)=-\langle \xi,\xi\rangle$ is a strictly positive
function on $U$, the square root of the latter is strictly positive
as well, and this means $\sqrt{-\langle \xi,\xi\rangle}$ is
invertible. Hence $\hat\xi$ in fact is a generalized unit
vector field on $U$, i.\ e., $g(\hat\xi,\hat\xi)=-1$
in $\mathcal G(U)$.
\item \label{setting6} For each symmetric representative
$(g_\varepsilon)_\varepsilon$ of the metric $g$ on $U$, for
sufficiently small $\varepsilon$, $\Sigma$ is a past-compact
space-like hypersurface such that $\partial
J_\varepsilon^+(\Sigma)=\Sigma$. Here $J_\varepsilon^+(\Sigma)$
denotes the topological closure (with respect to the topology
inherited by $U$) of the future emission
$D^+_\varepsilon(\Sigma)\subset U$ of $\Sigma$ with respect to
$g_\varepsilon$. Moreover, there exists an open set $A\subseteq \mathcal M$
and an $\varepsilon_0$ such that
\[
A\subseteq\bigcap_{\varepsilon<\varepsilon_0} J_\varepsilon^+(\Sigma).
\]
\end{enumerate}

Note, that (\ref{setting6}) is necessary to ensure existence
of smooth solutions on the level of representatives (cf.\ Theorem
\ref{exuniquesmoothsolweq}): For each sufficiently small $\varepsilon$ there exists
  a unique smooth function $u_\varepsilon$ on at least $A\subseteq\bigcap_{\varepsilon<\varepsilon_0} J_\varepsilon^+(\Sigma)$. 
  Furthermore the conditions (\ref{setting2})--(\ref{setting4}) are independent of
the Riemannian metric $m$.

Property (\ref{setting6}) is an assumption on {\it each} symmetric
representative. A conjecture, however, is the following:
\begin{conjecture}
If for one symmetric representative $(g_\varepsilon)_\varepsilon$ of
the metric $g$, for sufficiently small $\varepsilon$, $\Sigma$ is a
past-compact space-like hypersurface such that $\partial
J_\varepsilon^+(S)=S$, so it is for every symmetric representative
of the metric.
\end{conjecture}

In the remainder of this section we interpret the setting of
Definition \ref{settingmetric} in terms of the static coordinates
$(t,x^\mu)$ of Theorem \ref{Lemmastaticgeneralized}. With respect to
these coordinates, condition (\ref{setting2}) means that all the coefficients
of $g_\varepsilon$ are bounded by  a positive constant $M_0$ for
sufficiently small $\varepsilon$, and so are the coefficients of the
inverse of the metric. Finally, condition (\ref{setting3}) reads in
static coordinates $(t,x^\mu)$: For each $k>0$ there exists a
positive constant $M_k$ such that for sufficiently small
$\varepsilon$ we have
\[
\vert\partial_{\rho_1}\dots
\partial_{\rho_k} g_{ab}^\varepsilon\vert\leq
\frac{M_k}{\varepsilon^k},\quad\vert\partial_{\rho_1}\dots
\partial_{\rho_k}
g^{ab}_\varepsilon\vert\leq \frac{M_k}{\varepsilon^k},
\]
where $\partial_{\rho_i}(i=1,2,3)$ are partial derivatives with respect to
the space variables $x^\mu$ ($\mu=1,2,3$); differentiation with
respect to time is not interesting, since in these coordinates time
dependent contributions to the metric coefficients are negligible,
anyway (cf.\ Theorem \ref{Lemmastaticgeneralized}).

Moreover, condition (\ref{setting2}) implies that there is a positive constant $M$ such that for
sufficiently small $\varepsilon$ we have for the scalar product of the
Killing vector $\xi$:
\begin{equation}\label{setting1}
g_\varepsilon(\xi,\xi)=g_{00}^\varepsilon=-V_\varepsilon^2\leq -M<0.
\end{equation}
\subsection{The setting. Formulation of the initial value
problem}\label{settingsection} Let $v,w\in\mathcal G(\Sigma)$. The
initial value problem we are interested in is the wave equation for
$u\in\mathcal G(\mathcal M)$ subject to the initial conditions:
\begin{eqnarray}\label{weqofsetting}
\Box u&=&0\\\nonumber u|_\Sigma&=&v\\\nonumber \xi^a \nabla_a
u|_\Sigma&=&w.\nonumber
\end{eqnarray}
An immediate consequence is that in static coordinates $(t,x^\mu)$
(cf.\ Theorem \ref{Lemmastaticgeneralized}) which employ the Killing
parameter $t$, on the level of representatives the initial value
problem (\ref{weqofsetting}) simply reads:
\begin{eqnarray}\label{weqofsettingeps}
\Box^\varepsilon u_\varepsilon&=&f_\varepsilon\\\nonumber
u_\varepsilon(t=0,x^\mu)&=&v_\varepsilon(x^\mu)\\\nonumber
\partial_tu_\varepsilon(t=0,x^\mu)&=&w_\varepsilon(x^\mu)\nonumber,
\end{eqnarray}
since $\Sigma$ is locally parameterized as $t=0$. Here
$(f_\varepsilon)_\varepsilon\in\mathcal N(\varphi(\Omega))$, and
$(v_\varepsilon)_\varepsilon, (w_\varepsilon)_\varepsilon\in\mathcal
E_M(\varphi(\Omega\cap \Sigma))$ are local representations of
arbitrary representatives of $v,w$ and $\Box^\varepsilon$ is the
d'Alembertian with respect to an arbitrary {\it symmetric}
representative of $g$.

However, from now on we pick a representative of the metric which in
local coordinates takes the form of Theorem
\ref{Lemmastaticgeneralized}. Based on this choice we establish an
existence and uniqueness result in the sense of Colombeau. Only in
the last section we justify this choice in the sense that we show
that choosing any other symmetric representative would have lead to
the same generalized solution. Except for Part A we also use the
fact that $(u_\varepsilon)_\varepsilon$ is a solution of the initial
value problem on the level of representatives, i.\ e.,
$u_\varepsilon$ satisfies (\ref{weqofsettingeps}) for each
$\varepsilon$.

A remark on the setting is in order. We have chosen the static
setting basically for the reason that the initial value problem
(\ref{weqofsetting}) can be translated to (\ref{weqofsettingeps})
for each $\varepsilon>0$. In particular, this means that we can
treat all equations in one and the same coordinate patch; in
particular local asymptotic estimates, which are required for a
proof of existence and uniqueness of the wave equation, can be
achieved nicely in coordinates. However, in general, a convenient
coordinate form of the metric representative
$(g_\varepsilon)_\varepsilon$ cannot be achieved jointly for each
$\varepsilon>0$. For instance, suppose the mere assumption that we
are given a generalized metric for which a three-dimensional
submanifold $\Sigma$ is space-like (in the sense of chapter 2,
Definition \ref{causaldef1}). Assume $(g_\varepsilon)_\varepsilon$
is a symmetric representative. Let $p\in \Sigma$. Then for each
$\varepsilon>0$ it is possible to introduce Gaussian normal
coordinates at $p$ such that the metric can be written without
$(t,x^\mu)$ cross-terms (cf.\ Theorem \ref{theoremgaussian}).
However, the metric $g_\varepsilon$ will in general depend on
$\varepsilon$, the construction given in the mentioned theorem will
therefore depend on the resulting geodesics initially perpendicular
to $\Sigma$; for different $\varepsilon$ they will not coincide in
general. That means, for each $\varepsilon>0$ there could emerge
different coordinate charts, and the domain of these charts might
even shrink when $\varepsilon\rightarrow 0$.
\subsection{Locally foliated semi-neighborhoods.}
This section is devoted to showing that in the chosen setting, for
any point $p\in \Sigma$ there is a compact semi-neighborhood $\Omega_\gamma$
which can be foliated by space-like (in the generalized sense)
hypersurfaces $\Sigma_t$ (cf.\ figure 2). Throughout, we
follow the notation as has been set out in section
\ref{smoothsettingfoliation}. However, since the problem is a local
one, it suffices to construct the compact region $\Omega_\gamma$
(with space-like boundary throughout) in a coordinate chart. For the
sake of simplicity we will not distinguish notationally between the
image of the foliated region inside the coordinate chart and the
foliated region on the manifold.

Let $p\in\mathcal M$ and let $\Sigma$ be the initial surface through
$p$, perpendicular to $\xi^a$, the (smooth) Killing vector. Due to
Theorem \ref{Lemmastaticgeneralized} we have an open relatively
compact coordinate chart $(U, (t,x^\mu))$ at $p$ such that
$x^\mu(p)=0$, $\Sigma$ is parameterized by $t=0$ and $U$ is foliated
by the space-like hypersurfaces $\Sigma_\tau:\; t=\tau$ orthogonal
to $\xi=\partial/\partial_t$. Due to Theorem
\ref{Lemmastaticgeneralized}  we may find a representative
$(g_\varepsilon)_\varepsilon$ such that the line element associated
to $g_\varepsilon$  reads in these coordinates for sufficiently
small $\varepsilon$
\[
ds_\varepsilon^2=-V_\varepsilon^2(x^\alpha) dt^2+h^\varepsilon_{\mu\nu}(x^\alpha) dx^\mu
dx^\nu.
\]
Furthermore, we have positive constants such that on all of $U$,
$M^{-1}\leq V_\varepsilon^{-2}\leq M_0^{-1}$ and $\vert h^{\mu\nu}_\varepsilon\vert\leq M_0^{-1}$
for sufficiently small $\varepsilon$.

Let $h>0, \rho>0$. We take a paraboloid with boundary $t=0$ and
$S(t,x^\mu)=0, t\geq 0$, the zero level set of the function $S$
given by
\[
S:
=t-h\left(1-\frac{\sum_{\mu}(x^\mu)^2}{\rho^2}\right)=:t-h(1-\frac{\|
x\|^2}{\rho^2}),
\]
where height $h$ and maximal radius $\rho$ of the paraboloid shall
be determined in such a way that the boundary $S$ is space-like with
respect to the generalized metric (cf.\ below). $\Omega$ is the
compact region with boundaries $S$ and $S_0$, the subregion of
$\Sigma$, in coordinates given by $t=0$, $\|x\|\leq \rho$. $\Omega$
is therefore foliated by the three dimensional hypersurfaces
$S_{\zeta}$, the intersection of $\Sigma_\zeta: t=\zeta$ with
$\Omega$, all with normal vector $\xi^a$.

We fix for all that follows $\gamma$ with $0<\gamma<h$, and call
$\Omega_\gamma$ the part of $\Omega$ lying between $t=0$ and
$t=\gamma$. $S_{\Omega,\gamma}$ denotes the part of the boundary $S$
of $\Omega$ which lies between $t=0$ and $t=\gamma$. Therefore,
$\Omega_\gamma$ has boundaries $S_0, S_\gamma$ and
$S_{\Omega,\gamma}$.

Similarly, for $0\leq \tau\leq \gamma$ we use the notation
$S_{\Omega,\tau}, S_0, S_{\tau}$ for the boundaries of
$\Omega_{\tau}$.

Finally, we show that $n^a_\varepsilon:=g^{ab}_\varepsilon n_b$, the
normal to $S$ (hence to the subset $S_{\Omega,\gamma}$) given by the
($g_\varepsilon$-) metric equivalent covector $dS$, is time-like,
if the ratio $h/\rho\leq\frac{1}{2}\sqrt{\frac{M_0}{6M}}$. 
In local coordinates we have
\[
dS=dt+\frac{2h}{\rho^2}\delta_{ij}x^i dx^j.
\]
Therefore
\begin{equation}\label{showfol}
\langle
n^a_\varepsilon,n^a_\varepsilon\rangle_\varepsilon=-V_\varepsilon
^{-2}+\left(\frac{2h}{\rho^2}\right)^2
h^{ij}_\varepsilon\delta_{ik}\delta_{jl} x^k x^l\leq
-V_\varepsilon^{-2}+\left(\frac{2h}{\rho^2}\right)^2 (3 M_0^{-1}\|x\|^2).
\end{equation}
With $\sum_i (x^i)^2=\|x\|^2\leq \rho^2$ we obtain by means of
(\ref{showfol}) the estimate
\[
\langle n^a_\varepsilon,n^a_\varepsilon\rangle_\varepsilon\leq -\frac{1}{M}+12
(\frac{h}{\rho})^2\leq -\frac{1}{2M}
\]
for sufficiently small $\varepsilon$. We have shown that
$n^a_\varepsilon$ is time-like for each $\varepsilon$. In the
generalized sense of causality which is established in chapter
\ref{chaptercausality}, this means that $n^a:=g^{ab} n_b$ is
(generalized) space-like.
\begin{figure}\label{figure2}
   \begin{center}
   \fbox{\includegraphics[width=5.21in]{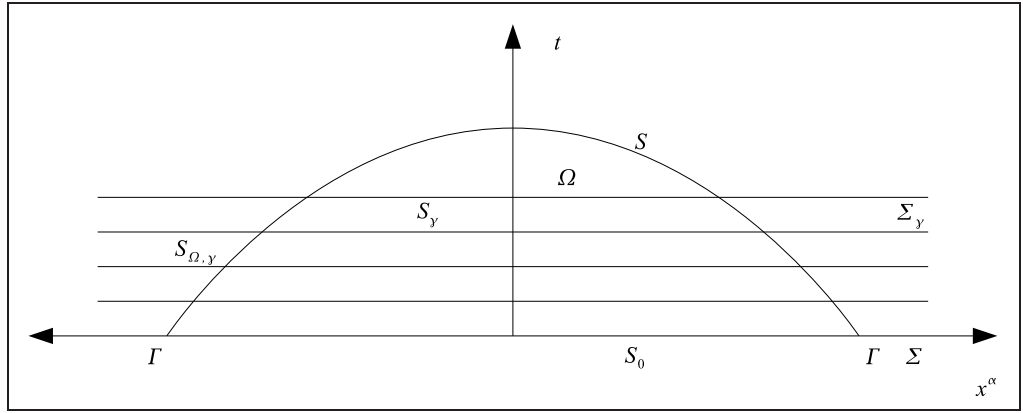}}
   \end{center}
   \caption{Local foliation of space-time}
\end{figure}
\subsection{Energy integrals and Sobolev norms}
Throughout this and all subsequent sections, we may assume that we
have picked a point $p\in \Sigma$ together with a semi-neighborhood
$\Omega_\gamma$ which entirely lies in an open relatively compact
coordinate patch $(U,(t,x^\mu))$, where  $(t,x^\mu)$ denote the
static coordinates at $p$, in which the metric $g$ takes the form
(\ref{metricformgenstatic}) on the level of representatives. All the
results will be proved on the level of representatives inside the
chosen coordinate patch. Since the Killing vector $\xi$ is a
standard vector field, we may always take the constant net
$(\xi)_\varepsilon$ as a representative of $\xi$.

We have revisited constructions of Riemannian metrics by means of
Lorentzian metrics in the preliminary section \ref{preliminarywaves}
and we have further mentioned that there are analogous constructions
in the generalized setting (cf.\ chapter \ref{chaptercausality},
section \ref{energygeneralizedsection}); these we apply now in order
to define Sobolev norms and energy integrals.

We shall deal with two different specific constructions of Riemann
metrics. For the {\it first}, we take $g$ and $\xi$, the given
Killing vector, and define the Riemannian metric
$e^{ab}:=[(e^{ab}_\varepsilon)_\varepsilon]$ on the level of
representatives by
\begin{equation}\label{riemannmetr1}
e_\varepsilon^{ab}:=g^{ab}_\varepsilon-\frac{2}{g_\varepsilon(\xi,\xi)}\xi^a\xi^b=g^{ab}_\varepsilon+\frac{2}{V_\varepsilon
^2}\xi^a\xi^b.
\end{equation}
For sufficiently small $\varepsilon>0$, $e_\varepsilon^{ab}$ is a Riemann
metric on $U$ due to Corollary
\ref{corollaryriemannianmetricconstruction}. 
Furthermore $e^{ab}$ is even a generalized Riemannian metric on $U$:
this follows, for instance, from the respective statement in the
generalized setting (cf. chapter \ref{chaptercausality}, Theorem
\ref{genRiemannmetricconstrglobal}). However, since $g^{ab}$ has
block diagonal form in static coordinates, and the metric
construction (\ref{riemannmetr1}) is quite simple, we can even
directly confirm that $e^{ab}$ is a generalized Riemannian metric.
Indeed, due to the assumptions of the setting, the metric $g$ has
the line element
\[
ds^2=-V^2 dt^2+h_{\mu\nu}dx^\mu dx^\nu,
\]
where $h_{\mu\nu}$ is a generalized Riemannian metric on $\Sigma_t\cap
U$ and $g(\xi,\xi)=-V^2$ is an invertible element of $\mathcal G(U)$
(which follows from the fact that $g$ is assumed to be
non-degenerate). Therefore, the line element of $e$ takes the form
\[
ds^2=+V^2 dt^2+h_{\mu\nu}dx^\mu dx^\nu.
\]
It follows that $ds^2$ is the line-element of a generalized Riemann
metric on $U$.

In the {\it second} construction, $g^{ab},\xi^a$ and $n_a$ are
involved; $\xi^a$ and $n^a:=g^{ab}n_b$ play the role of time-like
vector fields in the construction (cf.\ Corollary
\ref{corollaryriemannianmetricconstruction}, however in the
generalized setting: $\xi^a$ is the Killing vector (restricted to
$S_{\Omega,\gamma}=\Omega_\gamma\cap S$) and $n_a$ is the normal to
$S_{\Omega,\gamma}$. We define a Riemann
 metric on $S_{\Omega,\gamma}$ by
\begin{equation}
G^{cd}:=(g^{a(c}g^{d)b}-\frac{1}{2}g^{ab}g^{cd})\xi_a n_b.
\end{equation}
Since both $\xi_a$ and $n_b$ are time-like with the same
time-orientation, $G^{cd}$ is a generalized Riemannian metric on
$S_{\Omega,\gamma}$. This again follows from Theorem
\ref{genRiemannmetricconstrglobal}. We have omitted to explicitly
denote the restrictions of $g^{ab}$ and $\xi^a$ to
$S_{\Omega,\gamma}$. On the level of representatives $G^{cd}$ reads:
\begin{equation}
G^{cd}_\varepsilon:=(g^{a(c}_\varepsilon g^{d)b}_\varepsilon
-\frac{1}{2}g^{ab}_\varepsilon g^{cd}_\varepsilon )\xi^\varepsilon_a
n_b.
\end{equation}

We proceed now to energy tensors and energy integrals.\\

Let $u$ now be a smooth function defined on the coordinate patch
$U$, and let $\nabla^\varepsilon$ denote the covariant derivative
with respect to $g_\varepsilon$ for each $\varepsilon>0$. For each
non-negative integer $k$, we define energy tensors
$T^{ab,k}_{\varepsilon}(u)$ on $\Omega_\gamma$ as well as energies
$E^{k}_{\tau, \varepsilon}(u)$ on $S_\tau$ $(0\leq\tau\leq\gamma)$
of order $k$ as follows. For $k=0$ we set
\begin{equation}
T^{ab,0}_{\varepsilon}(u):=-\frac{1}{2}g^{ab}_\varepsilon u^2.
\end{equation}
For $k>0$ we define energy tensors
\begin{eqnarray}\label{tabkdef}
T^{ab,k}_{\varepsilon}(u)&:=&(g^{ac}_\varepsilon g^{bd}_\varepsilon
-\frac{1}{2}g^{ab}_\varepsilon g^{cd}_\varepsilon
)e^{p_1q_1}_\varepsilon\dots e^{p_{k-1}
q_{k-1}}_\varepsilon\\\nonumber &\times &(\nabla_c^\varepsilon\nabla_{p_1}^\varepsilon\dots\nabla_{p_{k-1}}^\varepsilon
u)(\nabla_d^\varepsilon\nabla_{q_1}^\varepsilon\dots\nabla_{q_{k-1}}^\varepsilon
u).
\end{eqnarray}
We are now prepared to define the energy integrals $E^k_{\tau,
\varepsilon}(u)$ via the energy tensors $T^{ab,k}_{\varepsilon}(u)$
of any order $k$. The energy integral of the $k$-th hierarchy is
given by
\begin{equation}
E^k_{\tau,\varepsilon}(u):=\sum_{j=0}^k\int_{S_\tau}
T^{ab,j}_\varepsilon(u)\xi_a^\varepsilon\hat\xi_b^\varepsilon\mu_\tau^\varepsilon.
\end{equation}
Here $\mu_\tau^\varepsilon$ is the unique three-form induced on
$S_\tau$ by $\mu^\varepsilon$ such that $d\tau \wedge
\mu_\tau^\varepsilon=\mu^\varepsilon$ holds on $S_\tau$. (cf.\ \cite{FL1}, p. 66, Lemma
2.9.2) . Furthermore, $\hat\xi_a^\varepsilon:=\frac{\xi_a}{\sqrt{-g_\varepsilon(\xi,\xi)}}$. Moreover, it should be noted that the tensors
$T^{ab,k}_\varepsilon(u)$ are symmetric tensors satisfying the
dominant energy condition. This holds due to Proposition
\ref{Butter}.

Since $\xi^a$ is a Killing vector field and
$T^{ab,k}_\varepsilon(u)$ is a symmetric vector satisfying the
dominant energy condition for each $\varepsilon$ (cf. Proposition
\ref{Butter}), we have as an application of Stokes's theorem (cf.\
(\ref{VolumeKn2})),
\begin{equation}\label{energyhierarchystokes}
E^k_{\tau,\varepsilon}(u)\leq
E^k_{\tau=0,\varepsilon}(u)+\sum_{j=0}^k\int_{\Omega_\tau}\xi_b^\varepsilon\nabla_a^\varepsilon
T^{ab,j}_\varepsilon(u) \mu_\varepsilon.
\end{equation}
The inequality is due to the fact that as a consequence of the
dominant energy condition, the integrand of the surface integral
over $S_{\Omega,\gamma}$ is non negative, hence can be neglected.

This inequality clearly holds for each $\varepsilon>0$ and each
$0\leq \tau\leq \gamma$.

In the remainder of the section we introduce Sobolev norms on the
coordinate patch $U$. Let $\varepsilon>0$ and $0\leq\tau\leq
\gamma$. The three dimensional Sobolev-norms are integrals
of the covariant derivative over $S_\tau$:

\begin{equation}\label{SobSDg}
 \SobSDt{u}{k}:=
\left(\sum_{j=0}^k\int_{S_\tau}\vert\nabla_\varepsilon^{(j)}(u)\vert^2\mu_\tau^\varepsilon\right)^{\frac{1}{2}}
\end{equation}
where, as usual, the integrand is expressed by contraction of the
covariant derivative of $j$th order of $u$ with the Riemannian metric
$e^{ab}$:
\begin{equation}
\vert\nabla_\varepsilon^{(j)}(u)\vert^2:=e^{p_1q_1}_\varepsilon\dots
e^{p_jq_j}_\varepsilon
\nabla^\varepsilon_{p_1}\dots\nabla^\varepsilon_{p_j}u\nabla^\varepsilon_{q_1}\dots\nabla^\varepsilon_{q_j}u.
\end{equation}
Similarly, the three dimensional Sobolev norms involving partial
derivatives only, are:
\begin{equation}\label{SobSdg}
 \SobSdt{u}{k}:=\left(
\sum_{{p_1\dots p_j \atop 0\leq j\leq
k}}\int_{S_\tau}\vert\partial_{p_1}\dots\partial_{p_j}u\vert^2\mu_\tau^\varepsilon\right)^{\frac{1}{2}}.
\end{equation}
On $\Omega_\tau$ we have the respective (usual) Sobolev norms given
by
\begin{equation}\label{SobD}
 \SobODt{u}{k}:=
\left(\sum_{j=0}^k\int_{\Omega_\tau}\vert\nabla_\varepsilon^{(j)}(u)\vert^2\mu^\varepsilon\right)^{\frac{1}{2}}
\end{equation}
as well as
\begin{equation}\label{Sobd}
 \SobOdt{u}{k}:=\left( \sum_{{p_1\dots p_j \atop 0\leq j\leq
k}}\int_{\Omega_\tau}\vert\partial_{p_1}\dots\partial_{p_j}u\vert^2\mu^\varepsilon\right)^{\frac{1}{2}}.
\end{equation}
\section[Equivalence: Energies and Sobolev norms]{Equivalence of energy integrals and Sobolev
norms (Part A)}\label{partA} We start by establishing that the three
dimensional Sobolev norms and the energy integrals are equivalent.
In this section, inequalities are meant to hold for sufficiently
small $\varepsilon$ and for each $0\leq\tau\leq\gamma$ and for each
smooth function $u$ given inside the coordinate patch $U$. In the
end of the section we shall give an interpretation of these
inequalities.

The main statement of this section is the following (the respective
statement in conical space-times is (\cite{VW}, Lemma 1)):
\begin{proposition}\label{lemma1}
For each $k\geq 0$, there exist positive constants $A, A'$ such
that for sufficiently small $\varepsilon$ we have
\begin{eqnarray}\label{ineqEXSD}
E^k_{\tau,\varepsilon}(u)&\leq& A(\SobSDt{u}{k})^2\\\label{ineqSDXE}
A'(\SobSDt{u}{k})^2&\leq&E^k_{\tau,\varepsilon}(u)
\end{eqnarray}
For each $k\geq 1$, there exist positive constants $B_k, B_k'$ such
that for sufficiently small $\varepsilon$ we have
\begin{eqnarray}\label{ineqSDXSd}
(\SobSDt{u}{k})^2&\leq&B_k'\sum_{j=1}^k\frac{1}{\varepsilon^{2(k-j)}}(\SobSdt{u}{j})^2\\\label{ineqSdXSD}
(\SobSdt{u}{k})^2&\leq&B_k\sum_{j=1}^k\frac{1}{\varepsilon^{2(k-j)}}(\SobSDt{u}{j})^2
\end{eqnarray}
Moreover, for $k=0$ we clearly have $(\SobSDt{u}{0})^2=(\SobSdt{u}{0})^2$.
\end{proposition}
Before we present the proof of the statement, we notice:
\begin{enumerate}
\item Note, that the term "sufficiently small $\varepsilon$" in the
statement in particular means that the index $\varepsilon_0$ from
which on the inequalities above hold, depends on the order $k$: For
the latter two inequalities this may happen; the two first
inequalities possess a uniform $\varepsilon_0$ from which on they
hold.
\item The four inequalities hold for each $0 \leq \tau\leq
\gamma$.
\item Note that the scalar product
\[
e^{p_1 q_1}_\varepsilon\dots e^{p_kq_k}_\varepsilon\eta_{p_1\dots
p_k}\eta_{q_1\dots q_k}
\]
and the euclidean scalar product defined on the coordinate patch
only,
\[
\delta^{p_1q_1}\dots \delta^{p_k q_k}\eta_{p_1\dots
p_k}\eta_{q_1\dots q_k}
\]
are equivalent on $U$ for small $\varepsilon$ in the sense that the
respective norms are, that is: there exist positive constants
$C_{k,1}, C_{k,2}$ such that for sufficiently small $\varepsilon$ we have
\begin{eqnarray}\label{ineqk1k2}
C_{k,1}\delta^{p_1q_1}\dots \delta^{p_k q_k}\eta_{p_1\dots
p_k}\eta_{q_1\dots q_k}&\leq& e^{p_1 q_1}_\varepsilon\dots
e^{p_kq_k}_\varepsilon\eta_{p_1\dots p_k}\eta_{q_1\dots
q_k}\\\nonumber &\leq& C_{k,2}\delta^{p_1q_1}\dots \delta^{p_k
q_k}\eta_{p_1\dots p_k}\eta_{q_1\dots q_k}.
\end{eqnarray}
Similarly there exist positive constants $D_{k,1}, D_{k,2}$ such
that for sufficiently small $\varepsilon$ we have
\begin{eqnarray}\label{ineqk1k2contra}
D_{k,1}\delta_{p_1q_1}\dots \delta_{p_k q_k}\eta^{p_1\dots p_k}\eta
^{q_1\dots q_k}&\leq& e_{p_1 q_1}^\varepsilon\dots
e_{p_kq_k}^\varepsilon\eta^{p_1\dots p_k}\eta^{q_1\dots
q_k}\\\nonumber &\leq& D_{k,2}\delta_{p_1q_1}\dots \delta_{p_k
q_k}\eta^{p_1\dots p_k}\eta^{q_1\dots q_k}.
\end{eqnarray}

It is sufficient to show this for $k=1$. Then (\ref{ineqk1k2contra})
may be reformulated as follows: There exist positive constants $C_1,
C_2$ such that for sufficiently small $\varepsilon$ we have
\begin{equation}\label{ineqkishoit1}
C_1\delta_{ab}\eta^a\eta^b\leq e_{ab}^\varepsilon \eta^a\eta^b\leq
C_1\delta_{ab}\eta^a\eta^b.
\end{equation}
Since there are positive constants $M,M_0$ such that for
sufficiently small $\varepsilon$ we have $M\leq V_\varepsilon^2\leq
M_0$, we may reduce the problem to the three space dimensions (the
greek letters therefore ranging between $2$ and $4$): We claim that on compact
subregions of $U$ we have for sufficiently small $\varepsilon$:
\begin{equation}\label{euclviagenmetr}
C_1\delta_{\mu\nu}\eta^\mu\eta^\nu\leq h_{\mu\nu}^\varepsilon \eta^\mu\eta^\nu\leq C_2\delta_{\mu\nu}\eta^\mu\eta^\nu.
\end{equation}
Since we locally have $\vert h_{\mu\nu}^\varepsilon\vert=O(1)\,(\varepsilon\rightarrow 0)$, the
right hand inequality of (\ref{euclviagenmetr}) is trivial. The proof of
the left hand inequality requires a little work: Let $x$ range in a
compact subset $K$ of $U$. We may assume that for small
$\varepsilon$, $h_{\mu\nu}^\varepsilon(x)
h^{\nu\rho}_\varepsilon(x)=\delta_\mu^\rho(x)+n_{\mu,\varepsilon}^\rho(x)$
with negligible $(n_{\mu,\varepsilon}^\rho)_\varepsilon$, therefore
for a negligible $(n_\varepsilon(x))_\varepsilon$
\begin{equation}
\det (h_{\nu\rho}^\varepsilon(x))=\frac{1+n_\varepsilon(x)}{\det(h^{\nu\rho}_\varepsilon(x))},
\end{equation}
since $\det(h^{\nu\rho}_\varepsilon(x))$ is invertible for sufficiently small $\varepsilon$.
Moreover, since $\vert h^{\nu\rho}_\varepsilon(x)\vert=O(1)$, we have $\vert
\det(h^{\nu\rho}_\varepsilon(x))\vert=O(1)$ holds on $K$. In view of
this and the fact that $(n_\varepsilon(x))_\varepsilon$ is
negligible (in particular we may assume that $\vert
n_\varepsilon(x)\vert<1/2$ for all $x$ in $K$ and for small
$\varepsilon$), there exists a positive constant $M'$ such that we have for all $x$ and sufficiently small
$\varepsilon$
\begin{equation}\label{eigenwertefuerdieabsch}
\vert \det (h_{\nu\rho}^\varepsilon(x))\vert\geq\frac{1}{2M'}.
\end{equation}
Furthermore we know that
\[
\det
(h_{\nu\rho}^\varepsilon(x))=\lambda_2^\varepsilon(x)\cdot\dots\cdot\lambda_4^\varepsilon(x),
\]
where $\lambda_i^\varepsilon(x)$ ($i=2,3,4$) are the eigenvalues of
$h_{\nu\rho}^\varepsilon(x)$ at $x$. Therefore, by using
(\ref{eigenwertefuerdieabsch}) and the fact that for each $i$  we have $\lambda_i^\varepsilon=O(1)$, we see that there is a positive
constant $C_1$ such that for each $i=2,\dots,4$ we have
\[
\vert \lambda_i^\varepsilon(x)\vert\geq C_1,
\]
whenever $\varepsilon$ is small enough.
Since for sufficiently small $\varepsilon$, $h_{\mu\nu}^\varepsilon$
is symmetric positive definite, we have
\begin{equation}
\inf_\eta\frac{h_{\mu\nu}^\varepsilon(x)\eta^\mu\eta^\nu}{\delta_{\mu\nu}(x)\eta^\mu\eta^\nu}=\min
_{i=1,\dots,n}\lambda_i^\varepsilon\geq C_1,
\end{equation}
which proves the left inequality of (\ref{euclviagenmetr}).
Therefore we have shown that (\ref{ineqkishoit1}) holds. In a
similar manner one can show the estimates (\ref{ineqk1k2}),
(\ref{ineqk1k2contra}).\end{enumerate}
We are ready to present a proof of Proposition \ref{lemma1}:
\begin{proof}
{\it Part 1: Inequalities (\ref{ineqEXSD}) and (\ref{ineqSDXE}).}\\
To establish these inequalities, we consider the cases $k=0$ and
$k>0$ separately.

For $k=0$, the situation is relatively simple. We have
\begin{equation}
T^{ab,0}_\varepsilon(u)\xi_a^{\varepsilon}\hat\xi_b^{\varepsilon}=-\frac{1}{2}g^{ab}_\varepsilon\xi_a^{\varepsilon}\hat\xi_b^{\varepsilon}u^2
=-\frac{1}{2}g_{ab}^\varepsilon\xi^a\hat\xi^bu^2=-\frac{1}{2}\sqrt{-g_\varepsilon(\xi,\xi)}u^2=\frac{V_\varepsilon}{2}u^2.
\end{equation}
By the assumption on the metric, there exist positive
constants $M,\,M_0$ such that for sufficiently small $\varepsilon$
\begin{equation}\label{Vepsupanddown}
M\leq V_\varepsilon\leq M_0
\end{equation}
It follows that for $A:=M_0/2$ and $A':=M/2$ we have
\begin{equation}\label{Tab0}
A'u ^2\leq
T^{ab,0}_\varepsilon(u)\xi_a^{\varepsilon}\hat\xi_b^{\varepsilon}\leq A
u^2.
\end{equation}
Integrating over $S_\tau$ yields
\begin{equation}
A'(\SobSDt{u}{0})^2\leq E^0_{\tau,\varepsilon}(u)\leq
A(\SobSDt{u}{0})^2
\end{equation}
and we are done with $k=0$.\\Next, we investigate the case $k>0$. To
start with, note that
\begin{eqnarray}\label{makeslemma1totriviality}
(g^{ac}_\varepsilon g^{bd}_\varepsilon-\frac{1}{2}g^{ab}_\varepsilon
g^{cd}_\varepsilon)\xi_a^\varepsilon\hat\xi_b^\varepsilon&=&\left(\xi^c\xi^d-\frac{1}{2}\langle
\xi,\xi\rangle_\varepsilon
g^{cd}_\varepsilon\right)\frac{1}{V_\varepsilon}\\\nonumber&=&\frac{1}{2}V_\varepsilon\left(g^{cd}_\varepsilon+\frac{2}{V_\varepsilon^2}\xi^c\xi^d\right)=\\\nonumber
&=& \frac{1}{2}V_\varepsilon e^{cd}_\varepsilon.
\end{eqnarray}
By the definition (\ref{tabkdef}) of $T^{ab,k}_\varepsilon(u)$ and
by (\ref {makeslemma1totriviality}), we therefore have
\begin{eqnarray}\label{dasistderletztestreich}
T^{ab,j}_{\varepsilon}(u)\xi_a^\varepsilon\hat\xi_b^\varepsilon:&=&\left((g^{ac}_\varepsilon
g^{bd}_\varepsilon -\frac{1}{2}g^{ab}_\varepsilon g^{cd}_\varepsilon
)\xi_a^\varepsilon\hat\xi_b^\varepsilon\right)
e^{p_1q_1}_\varepsilon\dots e^{p_{j-1}
q_{j-1}}_\varepsilon\times\\\nonumber&\times&(\nabla_c^\varepsilon\nabla_{p_1}^\varepsilon\dots\nabla_{p_{j-1}}^\varepsilon
u)(\nabla_d^\varepsilon\nabla_{q_1}^\varepsilon\dots\nabla_{q_{j-1}}^\varepsilon
u)=\\\nonumber &=& \frac{1}{2}V_\varepsilon e^{cd}_\varepsilon
e^{p_1q_1}_\varepsilon\dots e^{p_{j-1}
q_{j-1}}_\varepsilon\times\\\nonumber&\times&(\nabla_c^\varepsilon\nabla_{p_1}^\varepsilon\dots\nabla_{p_{j-1}}^\varepsilon
u)(\nabla_d^\varepsilon\nabla_{q_1}^\varepsilon\dots\nabla_{q_{j-1}}^\varepsilon
u).
\end{eqnarray}
By inequality (\ref{Vepsupanddown}) and
(\ref{dasistderletztestreich}), for all $1\leq j\leq k$ we have for
sufficiently small $\varepsilon$
\begin{equation}
\frac{M}{2}\vert \nabla^{(j)}_\varepsilon(u)\vert^2\leq
T^{ab,j}_\varepsilon(u)\xi_a^\varepsilon \hat\xi_b^\varepsilon\leq
\frac{M_0}{2}\vert \nabla^{(j)}_\varepsilon(u)\vert^2.
\end{equation}
Since $A=\frac{M_0}{2}, A'=\frac{M}{2}$, for all $1\leq j\leq k$ we
have for sufficiently small $\varepsilon$
\begin{equation}\label{Tab1k}
A'\vert \nabla^{(j)}_\varepsilon(u)\vert^2\leq
T^{ab,j}_\varepsilon(u)\xi_a^\varepsilon \hat\xi_b^\varepsilon\leq
A\vert \nabla^{(j)}_\varepsilon(u)\vert^2.
\end{equation}
Therefore we have by summing up (\ref{Tab1k}) and (\ref{Tab0}) the
following estimate:
\begin{equation}\label{Tab0plusTab1}
A'\sum_{j=0}^k\left( \vert
\nabla^{(j)}_\varepsilon(u)\vert^2\right)\leq \sum_{j=0}^k
T^{ab,j}_\varepsilon(u)\xi_a^\varepsilon \hat\xi_b^\varepsilon\leq
A\sum_{j=0}^k\left( \vert \nabla^{(j)}_\varepsilon(u)\vert^2\right).
\end{equation}
Integration over $S_\tau$ therefore yields
\begin{equation}
A'(\SobSDt{u}{k})^2\leq E^k_{\tau,\varepsilon}(u)\leq
A(\SobSDt{u}{k})^2
\end{equation}
and we are done with the proof for $k>0$.\\
{\it Part 2: Inequality (\ref{ineqSDXSd}).}\\
To prove this inequality, we use the asymptotic growth behavior of
partial derivatives of the metric as well as the formula which
expresses the covariant derivative of $u$ in terms of partial
derivatives of $u$ and Christoffel symbols (see identity (\ref{identityloccovpart}) below). The case $k=0$ is a
triviality. Also, the case $k=1$ is quite
simple. Independently of $\varepsilon$ we have
\[
\nabla^\varepsilon_a u=\partial_a u.
\]
There exists a constant $M_0'$ such that for sufficiently small $\varepsilon$ we have
$\vert e^{ab}_\varepsilon\vert\leq M_0'$. Therefore (\ref{ineqSDXSd}) holds for $B_1'':=2M_0'$, since
\begin{equation}\label{N1}
\vert\nabla_\varepsilon^{(1)}u\vert^2=e^{ab}_\varepsilon \partial_a
u\partial_b u\leq 2M_0' \sum_p (\partial_p u)^2=B_1''\sum_p
(\partial_p u)^2.
\end{equation}
With $B_1':=\max(1,B_1'')$ we obtain
\begin{equation}\label{N11}
u^2+\vert \nabla_\varepsilon^{(1)}u\vert^2\leq B_1' (u^2+\sum_p
(\partial_p u)^2
\end{equation}
and integration over $S_\tau$ yields
\[
(\SobSDt{u}{1})^2\leq B_1'(\SobSdt{u}{1})^2.
\]
This is the claim for $k=1$. So let $k=2$. Then
\begin{equation}\label{inductivebasisk2}
\nabla^\varepsilon_a\nabla^\varepsilon_b u=\nabla^\varepsilon
_a(\partial_bu)=\partial_a\partial_b u-\Gamma_{ab,\varepsilon}^c
\partial_c u.
\end{equation}
Since $e_\varepsilon^{ab}=O(1)$ on $\Omega_\gamma\supseteq
\Omega_\tau$ and
$\Gamma_{ab,\varepsilon}^c=O(\frac{1}{\varepsilon})$, there is a
positive constant $B_2'''$ such that
\begin{equation}\label{Npre2}
\sum_{p_1p_2}\vert \nabla^\varepsilon_{p_1}\nabla^\varepsilon_{p_2}u\vert^2 \leq B_2''' \left(
\frac{1}{\varepsilon^2}\sum_p(\partial_p u)^2+\sum_{p_1
p_2}(\partial_{p_1}\partial_{p_2} u)^2\right).
\end{equation}
Using the right hand side of (\ref{ineqk1k2}) we conclude that for the positive constant
$B_2'':=C_{2,2}B_2'''$ and for sufficiently small $\varepsilon$ we have
\begin{equation}\label{N2}
\vert \nabla_\varepsilon^{(2)}u\vert^2 \leq B_2'' \left(
\frac{1}{\varepsilon^2}\sum_p(\partial_p u)^2+\sum_{p_1
p_2}(\partial_{p_1}\partial_{p_2} u)^2\right).
\end{equation}
As a consequence of (\ref{N11}) and (\ref{N2}), there exists a
positive constant $B_2'$ such that for sufficiently small
$\varepsilon>0$ we have:
\begin{equation}
u^2+\vert \nabla_\varepsilon^{(1)}u\vert^2+\vert
\nabla_\varepsilon^{(2)}u\vert^2 \leq B_2'
\left(\frac{1}{\varepsilon^2}(u^2+\sum_p(\partial_p
u)^2)+(u^2+\sum_p(\partial_p u)^2+\sum_{p_1p_2}(\partial_{p_1p_2}
u)^2) \right)
\end{equation} and integrating this inequality over $S_\tau$ yields the claim for $k=2$.
The proof for arbitrary $k$ is inductive.

We claim that for each $2\leq j<k$ we may write the $j$th covariant
derivative of $u$ as
\begin{eqnarray}\label{basisinductcov}
\nabla_{a_1}^\varepsilon\dots\nabla_{a_j}^\varepsilon
u&=&\partial_{a_1}\dots\partial_{a_j}u+\\\nonumber
&+&\sum_{b_1,\dots,b_{j-1}}B_{a_1\dots a_j,\,\varepsilon}^{b_1\dots
b_{j-1}}\partial_{b_1}\dots\partial_{b_{j-1}}u+\\\nonumber
&+&\sum_{b_1,\dots,b_{j-2}}B_{a_1 \dots a_j,\,\varepsilon}^{b_1\dots
b_{j-2}}\partial_{b_1}\dots\partial_{b_{j-2}}u+\\\nonumber &+&
\dots\\\nonumber
&+&\sum_{b_1}B_{a_1 \dots a_j,\,\varepsilon}^{b_1}\partial_{b_1}u,
\end{eqnarray}
with functions (defined in the coordinate patch $(U,(t,x^i))$):
\[
B_{a_1 \dots a_j,\,\varepsilon}^{b_1\dots b_{j-r}}, 1\leq r\leq j-1,
\]
 where for each non-negative integer $m$, we have the following
 growth estimate on compact sets:
\begin{equation}\label{growthchristoferus}
B_{a_1 \dots a_j,\,\varepsilon}^{b_1\dots
b_{j-r}}=O\left(\frac{1}{\varepsilon^{r}}\right).
\end{equation}
Of course, some of the coefficient functions
$B_{a_1\dots a_j,\,\varepsilon}^{b_1\dots b_{j-r}}$ might even vanish.

We use the induction principle for the proof of this subclaim.
The inductive basis $k=2$  holds due to \ formula (\ref{inductivebasisk2}).
For the inductive step, basically two ingredients are needed: First,
the asymptotic growth of the Christoffel symbols on compact subsets
of $U$ when $\varepsilon\rightarrow 0$, which for every non-negative
integer $m$ is
\begin{equation}\label{ingredientno1}
\partial_{\rho_1}\dots\partial_{\rho_m}
\Gamma_{ab,\varepsilon}^c=O\left(\frac{1}{\varepsilon^{m+1}}\right).
\end{equation}
This formula follows by induction directly from the asymptotic
growth of the metric coefficients, the coefficients of the inverse
of the metric and their derivatives.

The second ingredient is the coordinate formula for the covariant derivative
of a tensor of type $(0,n)$, which is:
\begin{equation}\label{identityloccovpart}
\nabla_a \omega_{b_1\dots b_{n}}=\partial_a\omega_{b_1\dots
b_{n}}-\sum_{j=1}^n\Gamma_{ab_j}^c\omega_{b_1\dots b_{j-1} c b_{j+1}\dots
b_{n}}.
\end{equation}
By using the two ingredients (\ref{ingredientno1}) and (\ref{identityloccovpart}) the proof of claim
(\ref{basisinductcov}) is easily proven for $j=k$. 

Having showed decomposition (\ref{basisinductcov}) for each non-negative integer $k$, the proof
of inequality (\ref{ineqSDXSd}) lies at hand. One only needs the
right hand side of estimate (\ref{ineqk1k2}). Then it follows by
(\ref{basisinductcov}) that for some positive constant $A_k''$
\[
\vert \nabla^{(k)}_\varepsilon\vert^2\leq B_k''\left(\sum_{0\leq
j\leq k}\frac{1}{\varepsilon^{2(k-j)}}\sum_{p_1\dots p_j}\vert
\partial_{p_1\dots p_j}u\vert ^2\vert\right)
\]
(cf.\ inequality (\ref{N2}) in the case $k=2$) and integration of
the respective inequality for $j=0,\dots,k$ over $S_\tau$ yields
inequality (\ref{ineqSDXSd}) for any $k\geq 2$. We are done with
Part 2.
\\{\it Part 3. Inequality (\ref{ineqSdXSD}).}\\
This problem is analogous to inequality (\ref{ineqSDXSd}). One starts by
expressing partial derivatives in terms of covariant
derivatives using identity (\ref{identityloccovpart}). Then one may find estimates of squares of partial derivatives via squares of covariant
derivatives of the respective orders. Finally we may use the left hand
inequality of (\ref{ineqk1k2}) given in the preceding remark
and integrate the achieved inequalities over $S_\tau$ and we are done. 
\end{proof}
\section[Bounds on initial energies]{Bounds on initial energies via bounds on initial data (Part B)}\label{partb}
To start with, we establish asymptotic estimates of derivatives of
arbitrary order of the smooth net $(u_\varepsilon)_\varepsilon$ on
(the compact set) $S_0$. This may be used later to establish the
asymptotic growth behavior of the initial energies
$E^{k}_{\tau=0,\varepsilon}(u_\varepsilon)$. The following notation
is useful:
\begin{definition}
Let $O$ be an open subset of $\mathbb R^n$ and let $K\subset\subset
O$ be a compact subset. A net $(g_\varepsilon)_\varepsilon$ of
smooth functions on $O$ is said to satisfy moderate bounds on $K$,
if there exists a number $N$ such that
\[
\sup_{x\in K}\vert g_\varepsilon(x)\vert=O(\varepsilon^N)\qquad
(\varepsilon\rightarrow 0).
\]
\end{definition}
In our main reference \cite{VW} (cf.\ pp.\ 1341-1344), the set of
such functions is denoted by $\mathcal E_M(K)$, however, since this
notation is misleading, we shall not use it.



 We shall establish moderate (resp.\ negligible) bounds in all derivatives of the net of solutions
$(u_\varepsilon)_\varepsilon$ on a fixed compact set only, namely
$\Omega_\gamma$.
The first step is to establish moderate (resp.\ negligible) bounds of $(u_\varepsilon)_\varepsilon$ on S;
this is the subject of this section.

We may go on now by recalling that due to Proposition
\ref{staticform} the d'Alembertian takes the following form in
static coordinates:
\begin{equation}\label{eqstaticform}
\Box^\varepsilon u_\varepsilon= -V_\varepsilon^{-2}\partial_t^2
u_\varepsilon+\vert g_\varepsilon\vert
^{-1/2}\partial_\alpha\left(\vert g_\varepsilon\vert ^{1/2}
g_\varepsilon^{\alpha\beta}\partial_\beta u_\varepsilon\right).
\end{equation}
We may manipulate equation (\ref{eqstaticform}) by using
$\Box^\varepsilon u_\varepsilon=f_\varepsilon$ and receive a formula
for the second derivative of $u_\varepsilon$:
\begin{equation}\label{secondtimeder}
\partial_t^2u_\varepsilon=-V_\varepsilon^2\left ( f_\varepsilon- \vert g_\varepsilon\vert ^{-1/2}\partial_\alpha\left(\vert g_\varepsilon\vert ^{1/2} g_\varepsilon^{\alpha\beta}\partial_\beta u_\varepsilon \right)\right).
\end{equation}
In order to derive asymptotic bounds on initial energies we shall
need the following statement:
\begin{proposition}
If $(v_\varepsilon)_\varepsilon$, $(w_\varepsilon)_\varepsilon$ (as introduced in
(\ref{weqofsettingeps})) satisfy moderate (resp.\ negligible) bounds on $S_0$ in all derivatives, then for all $j,k\geq 0$ the derivative
\[
(\partial_t^j\partial_{\rho_1}\dots\partial_{\rho_k}u_\varepsilon)_\varepsilon
\]
satisfies moderate (resp.\ negligible bounds) on $S_0$.
\end{proposition}
\begin{proof}
{\it Part 1}\\
Recall that $(f_\varepsilon)_\varepsilon$ is negligible. We start by
proving the estimates on $S_0$ for time derivatives of
$(u_\varepsilon)_\varepsilon$ only. The inductive hypothesis is: If
$(v_\varepsilon)_\varepsilon, (w_\varepsilon)_\varepsilon$ satisfy
moderate (resp.\ negligible) bounds on $S_0$, so does for each
$j\geq0$, the net $\partial_t^j u_\varepsilon(0,x^\alpha)$. The
inductive basis may be $j=0$ or $j=1$: In these cases, the statement
holds trivially: Due to the initial value formulation
(\ref{weqofsettingeps}), we have
\[
\partial_t^0 u_\varepsilon(0,x^\alpha)=u_\varepsilon(0,x^\alpha)=v_\varepsilon(x^\alpha)
\]
and
\[
\partial_t u_\varepsilon(0,x^\alpha)=w_\varepsilon(x^\alpha)
\]
which are both moderate (resp.\ negligible) due to our assumption.
Employing the fact that $(u_\varepsilon)_\varepsilon$ solves
(\ref{weqofsettingeps}) as well as the identity
(\ref{secondtimeder}), we have:
\begin{equation}\label{secondtimederestimate}
\partial_t^2u_\varepsilon(0,x^\alpha)=-V_\varepsilon^2\left ( f_\varepsilon(0,x^\alpha)- \vert g_\varepsilon\vert ^{-1/2}\partial_\alpha\left(\vert g_\varepsilon\vert ^{1/2} g_\varepsilon^{\alpha\beta}\partial_\beta\right)v_\varepsilon\right).
\end{equation}
Here we have only explicitly written down the independent variables,
if the resp. functions are not functions of the space-variables
only.

To confirm that the claimed estimates hold for the second derivative
with respect to time, we only need to know that the product of a net
having moderate (resp.\ negligible) bounds with a net having
moderate bounds, has moderate (resp.\ negligible) bounds. Therefore,
the hypothesis holds for order $j=2$ as well, since
$(v_\varepsilon)_\varepsilon$ satisfies moderate (resp.\ negligible)
bounds (and of course, the representatives of the metric
coefficients are moderate by definition, and so are the determinant
and its inverse).

For the inductive step, assume that for $2\leq j<m$ $(m\geq 2)$ the
desired asymptotic growth is known on $S_0$. Differentiating
equation (\ref{secondtimeder}) $m$$-$$2$ times with respect to time
yields:
\begin{equation}\label{secondtimederschweinsbraten}
\partial_t^m u_\varepsilon=-V_\varepsilon^2\left ( \partial_t^{m-2}f_\varepsilon- \vert g_\varepsilon\vert ^{-1/2}\partial_\alpha\left(\vert g_\varepsilon\vert ^{1/2} g_\varepsilon^{\alpha\beta}\partial_\beta \partial_t^{m-2} u_\varepsilon \right)\right).
\end{equation}
Here again we have used the fact that $V_\varepsilon$ and the metric
coefficients are independent of the time variable $t$. Due to the
inductive hypothesis, $(\partial_t^{m-2} u_\varepsilon)_\varepsilon$
is moderate (resp.\ negligible), whereas
$(\partial_t^{m-2}f_\varepsilon)_\varepsilon$ is negligible by
assumption (since $(f_\varepsilon)_\varepsilon$ is). By a similar
reasoning as for the second derivative, we find that $(\partial_t^m
u_\varepsilon)_\varepsilon$ satisfies moderate (resp.\ negligible)
bounds on $S_0$ and we are done.\\
{\it Part 2}\\
Estimates for the derivatives of $(u_\varepsilon)_\varepsilon$ with
respect to space-variables are easily achieved, since
$(u_\varepsilon(0,x^\alpha))_\varepsilon=(v_\varepsilon)_\varepsilon$
is moderate (resp.\ negligible) due to our assumption, and derivation with respect to space-variables commutes with evaluation at $t=0$.\\
{\it Part 3}\\
It is left to be shown that mixed derivatives of
$(u_\varepsilon)_\varepsilon$ of any order have moderate (resp.\
negligible) bounds on $S_0$. Here again an inductive argument as in
the Part 1 applies. We first rewrite
(\ref{secondtimederschweinsbraten}) by using the Leibniz rule:
\begin{align}\label{secondtimederschweinsbratenknoedel}
\partial_t^m u_\varepsilon=-V_\varepsilon^2\left ( \partial_t^{m-2}f_\varepsilon- \vert g_\varepsilon\vert ^{-1/2}\partial_\alpha\left(\vert g_\varepsilon\vert ^{1/2} g_\varepsilon^{\alpha\beta}\right)\partial_\beta\partial_t^{m-2} u_\varepsilon-g_\varepsilon^{\alpha\beta}\partial_\alpha\partial_\beta\partial_t^{m-2}u_\varepsilon \right).
\end{align}
We may define the net
\[
G_\varepsilon^\beta(x^\mu):=\vert g_\varepsilon\vert
^{-1/2}\partial_\alpha\left(\vert g_\varepsilon\vert ^{1/2}
g_\varepsilon^{\alpha\beta}\right).
\]
It is worth mentioning that
$(G_\varepsilon^\beta(x^\mu))_\varepsilon$ is a moderate net in the coordinate patch for
each $\beta=1,2,3$. With this notation,
(\ref{secondtimederschweinsbratenknoedel}) reads
\begin{align}\label{secondtimederschweinsbratenknoedelnachmservieren}
\partial_t^m u_\varepsilon=-V_\varepsilon^2\left ( \partial_t^{m-2}f_\varepsilon- G_\varepsilon^\beta(x^\mu)\partial_\beta\partial_t^{m-2} u_\varepsilon-g_\varepsilon^{\alpha\beta}\partial_\alpha\partial_\beta\partial_t^{m-2}u_\varepsilon \right).
\end{align}
The inductive hypothesis now states that for each order $m$ we have
for each order $k$ that
\[
(\partial_{\rho_1}\dots\partial_{\rho_k}\partial_t^m
u_\varepsilon)(t=0,x^\alpha)
\]
is a moderate (resp.\ negligible) function. The basis of induction
is $m=0$ holds according to Case 2. Assume
therefore that the claim holds for $0\leq j\leq m$ of order of
time-derivatives of $u_\varepsilon$. Differentiating
(\ref{secondtimederschweinsbratenknoedelnachmservieren}) with
respect to time yields
\begin{align}\label{secondtimederschweinsbratenknoedelnachmservierenessen}
\partial_t^{m+1} u_\varepsilon=-V_\varepsilon^2\left ( \partial_t^{m-1}f_\varepsilon- G_\varepsilon^\beta(x^\alpha)\partial_\beta\partial_t^{m-1} u_\varepsilon-g_\varepsilon^{\alpha\beta}\partial_\alpha\partial_\beta\partial_t^{m-1}u_\varepsilon \right).
\end{align}
Now we may set $t=0$ and differentiate $k$ times with respect to
the space-variables. By assumption, for each $l\geq 0$, and for each
$n\leq m$
\[
(\partial_{\rho_1}\dots\partial_{\rho_l}\partial_t^{n}
u_\varepsilon)(t=0,x^\alpha)
\]
is a moderate (resp.\ negligible) function. Plugging this
information into the right hand side of
(\ref{secondtimederschweinsbratenknoedelnachmservierenessen}), we
see that
\[
(\partial_{\rho_1}\dots\partial_{\rho_k}\partial_t^{m+1}
u_\varepsilon)(t=0,x^\alpha)
\]
has moderate (resp.\ negligible) bounds for each $k\geq 0$, and we
are done.
\end{proof}
As a consequence of the preceding statement, we have
\begin{proposition}\label{initialenergiesviainitialdata}
If $(v_\varepsilon)_\varepsilon$, $(w_\varepsilon)_\varepsilon$ are
moderate (resp.\ negligible), then for each $k$ the initial energies
$(E^k_{0,\varepsilon})_\varepsilon$ are moderate (resp.\ negligible)
nets of real numbers.
\end{proposition}
\begin{proof}
This statement is a direct consequence of the form of
the energy integrals (rewritten in terms of partial derivatives using formula
(\ref{identityloccovpart})).
\end{proof}
\section[Energy inequalities]{Energy inequalities (Part C)}\label{energyestimates1}
From now on, we will use the fact that $(u_\varepsilon)_\varepsilon$
is a solution of (\ref{weqofsettingeps}) on $\Omega_\gamma$. 

We start with the simplest case $k=1$. Then we have the following
inequality:
\begin{proposition}\label{energyinequalitylevelk1}
There exist positive constants $C_1'$ and $C_1''$ such that we have
for each $0\leq\tau\leq\gamma$ and for sufficiently small
$\varepsilon$:
\begin{equation}\label{energyinequalitylevelk1formula}
E^1_{\tau,\varepsilon}(u_\varepsilon)\leq
E^1_{0,\varepsilon}(u_\varepsilon)+C_1'
(\SobODt{f_\varepsilon}{0})^2+C_1''\int_{\zeta=0}^\tau
E_{\zeta,\varepsilon}^1(u_\varepsilon) d\zeta.
\end{equation}
\end{proposition}
\begin{proof}
We start with (\ref{energyhierarchystokes}), which for $k=1$ reads
\begin{equation}\label{energyhierarchystokesk1}
E^1_{\tau,\varepsilon}(u_\varepsilon)\leq
E^1_{\tau=0,\varepsilon}(u_\varepsilon)+\int_{\Omega_\tau}\xi_b^\varepsilon\nabla_a^\varepsilon
T^{ab,0}_\varepsilon(u_\varepsilon)
\mu_\varepsilon+\int_{\Omega_\tau}\xi_b^\varepsilon\nabla_a^\varepsilon
T^{ab,1}_\varepsilon(u_\varepsilon) \mu_\varepsilon
\end{equation}
We calculate the integrals on the right hand side of the inequality
(\ref{energyhierarchystokesk1}). For $k=0$ the energy tensor is
defined by
\begin{equation}
T^{ab,0}_\varepsilon(u_\varepsilon)=-\frac{1}{2}g^{ab}_\varepsilon
u_\varepsilon^2.
\end{equation}
The covariant derivative is:
\begin{eqnarray}\nonumber
\nabla_a^\varepsilon T^{ab,0}_\varepsilon(u_\varepsilon)&=&
-\frac{1}{2}\nabla_a^\varepsilon g^{ab}_\varepsilon
u_\varepsilon^2-(\frac{1}{2}g_\varepsilon^{ab})(2u_\varepsilon
\nabla_a^\varepsilon u_\varepsilon)=\\\nonumber &=& 0-u_\varepsilon
\nabla_\varepsilon^b u_\varepsilon\\\label{Tab0wave} &=&
-u_\varepsilon \nabla_\varepsilon^b u_\varepsilon
\end{eqnarray}
Moreover, for $k=1$ the energy tensor reads
\[
T^{ab,1}_\varepsilon(u_\varepsilon)=(g^{ac}_\varepsilon
g^{bd}_\varepsilon-\frac{1}{2}g^{ab}_\varepsilon
g^{cd}_\varepsilon)\nabla_c^\varepsilon u_\varepsilon
\nabla_d^\varepsilon u_\varepsilon.
\]
Therefore we obtain for the covariant derivative
\begin{eqnarray}\nonumber
\nabla_a^\varepsilon T^{ab,1}_\varepsilon(u_\varepsilon)
&=&(g^{ac}_\varepsilon
g^{bd}_\varepsilon-\frac{1}{2}g^{ab}_\varepsilon
g^{cd}_\varepsilon)(\nabla_a^\varepsilon\nabla_c^\varepsilon
u_\varepsilon\nabla_d^\varepsilon u_\varepsilon+\nabla_c^\varepsilon
u_\varepsilon\nabla_a^\varepsilon\nabla_d^\varepsilon
u_\varepsilon)=\\\nonumber&=&(g^{ac}_\varepsilon
g^{bd}_\varepsilon-\frac{1}{2}g^{ab}_\varepsilon
g^{cd}_\varepsilon)(\nabla_a^\varepsilon\nabla_c^\varepsilon
u_\varepsilon\nabla_d^\varepsilon u_\varepsilon+\nabla_c^\varepsilon
u_\varepsilon\nabla_d^\varepsilon\nabla_a^\varepsilon
u_\varepsilon)=\\\nonumber&=&\nabla^c_\varepsilon\nabla_c^\varepsilon
u_\varepsilon\nabla^b_\varepsilon
u_\varepsilon+(\nabla^a_\varepsilon u_\varepsilon
\nabla^b_\varepsilon\nabla_a^\varepsilon
u_\varepsilon-\\\nonumber&-&\frac{1}{2}\nabla^b_\varepsilon\nabla_c^\varepsilon
u_\varepsilon\nabla^c_\varepsilon
u_\varepsilon-\frac{1}{2}\nabla^d_\varepsilon u_\varepsilon
\nabla^b_\varepsilon\nabla_d^\varepsilon u_\varepsilon)=\\\nonumber
&=&\nabla^c_\varepsilon\nabla_c^\varepsilon
u_\varepsilon\nabla^b_\varepsilon u_\varepsilon =(\Box^\varepsilon
u_\varepsilon) \nabla^b_\varepsilon u_\varepsilon=\\\label{Tab1wave}
&=&f_\varepsilon \nabla^b_\varepsilon u_\varepsilon.
\end{eqnarray}
We may now insert (\ref{Tab0wave}) and (\ref{Tab1wave}) into
(\ref{energyhierarchystokesk1}). This yields
\begin{eqnarray}\nonumber
E^1_{\tau,\varepsilon}(u_\varepsilon)&\leq&E^1_{\tau=0,\varepsilon}(u_\varepsilon)+\int_{\Omega_\tau}\xi_b^\varepsilon\nabla_a^\varepsilon
\left(T^{ab,0}_\varepsilon(u_\varepsilon)+T^{ab,1}_\varepsilon(u_\varepsilon)\right)\mu_\varepsilon=\\\nonumber
&=&
E^1_{0,\varepsilon}(u_\varepsilon)+\int_{\Omega_\tau}\xi_b^\varepsilon\nabla^b_\varepsilon
u_\varepsilon
(f_\varepsilon-u_\varepsilon)\mu_\varepsilon=\\\nonumber &=&
E^1_{0,\varepsilon}(u_\varepsilon)+\int_{\Omega_\tau}\xi^a\nabla_a^\varepsilon
u_\varepsilon (f_\varepsilon-u_\varepsilon)\mu_\varepsilon.
\end{eqnarray}
Using the Cauchy Schwarz inequality we further obtain
\begin{equation}\label{energyhierarchystokesk1endversion}
E^1_{\tau,\varepsilon}(u_\varepsilon)\leq
E^1_{0,\varepsilon}(u_\varepsilon)+
\left(\int_{\Omega_\tau}(\xi^a\nabla_a^\varepsilon
u_\varepsilon)^2\mu_\varepsilon\right)^{\frac{1}{2}}\left(\int_{\Omega_\tau}(f_\varepsilon-u_\varepsilon)^2\mu_\varepsilon\right)^{\frac{1}{2}}
\end{equation}
We may now estimate again by means of the Cauchy Schwarz inequality
for the scalar product induced in each tangent space by
$e^\varepsilon_{ab}$,
\begin{equation}
\xi^a\nabla_a^\varepsilon
u_\varepsilon=g_{ab}^\varepsilon\xi^a\nabla^b_\varepsilon
u_\varepsilon\leq e_{ab}^\varepsilon \xi^a\nabla^b_\varepsilon
u_\varepsilon\leq
\sqrt{e_\varepsilon(\xi,\xi)}\vert\nabla_\varepsilon^{(1)}u_\varepsilon\vert.
\end{equation}
Note that the first inequality holds due to the fact that
the difference between the line elements of $g_{ab}^\varepsilon$ and $e_{ab}^\varepsilon$
merely lies in the switch of signs in the first summand from
$-V_\varepsilon^2$ to $+V_\varepsilon^2$ (therefore this inequality is trivial).

Furthermore, there exists a positive constant $C_1$ such that
$\sqrt{e_\varepsilon(\xi,\xi)}\leq C_1$ on $\Omega_\gamma$ for
sufficiently small $\varepsilon$, because $\xi$ is smooth,
$e_{ab}^\varepsilon$ is locally bounded (because
$g_{ab}^\varepsilon$ is in our setting) and
$\Omega\subset\subset U$. It follows that
\[
\int_{\Omega_\tau}(\xi^a\nabla_a^\varepsilon
u_\varepsilon)^2\mu_\varepsilon\leq
C_1^2\int_{\Omega_\tau}\vert\nabla_\varepsilon^{(1)}
u_\varepsilon\vert^2\mu_\varepsilon.
\]
This information we plug into (\ref{energyhierarchystokesk1endversion}) and achieve
\begin{eqnarray}\nonumber
E^1_{\tau,\varepsilon}(u_\varepsilon)&\leq&
E^1_{0,\varepsilon}(u_\varepsilon)+ C_1\left(\int_{\Omega_\tau}
\vert \nabla_\varepsilon^{(1)}
u_\varepsilon \vert^2\mu_\varepsilon\right)^{\frac{1}{2}}\times\\\nonumber
&\times&\left(\left(\int_{\Omega_\tau}f_\varepsilon^2\mu_\varepsilon\right)^{\frac{1}{2}}+\left(\int_{\Omega_\tau}u_\varepsilon^2\mu_\varepsilon\right)^{\frac{1}{2}}
\right)=\\\label{wasweissmanschon} &=&
E^1_{0,\varepsilon}(u_\varepsilon) +
C_1\,\left(\SobODt{u_\varepsilon}{1}\right)^2+\frac{C_1}{2}\,\left(\SobODt{f_\varepsilon}{0}\right)^2,
\end{eqnarray}
where for the second integrand of the right hand side of
(\ref{energyhierarchystokesk1endversion}) we have used the triangle
inequality for the Sobolev norm (and further that $a(b+c)\leq
(a^2+c^2 )+\frac{b^2}{2}$). Next we employ inequality
(\ref{ineqSDXE}) of Proposition \ref{lemma1}: We have
\begin{equation}\label{plugin1}
\left(\SobODt{u_\varepsilon}{1}\right)^2=\int_{\zeta=0}^\tau
(\SobSDz{u_\varepsilon}{1})^2 d\zeta\leq\frac{1}{
A'}\int_{\zeta=0}^\tau E^1_{\zeta,\varepsilon}(u_\varepsilon)d\zeta.
\end{equation}
We may set $C_1':=\frac{C_1}{2},\,C_1'':=C_1/A'$. Plugging
(\ref{plugin1}) into (\ref{wasweissmanschon}) yields the claim
(\ref{energyinequalitylevelk1formula}).
\end{proof}

 For energies hierarchies larger than one, we similarly have:
\begin{proposition}\label{energyinequalitylevelk}
For each $k>1$ there exist positive constants $C_k',C_k'',C_k'''$
such that for each $0\leq\tau\leq\gamma$ and sufficiently small
$\varepsilon$ we have,
\begin{eqnarray}\label{energyinequalitylevelkformula}
E^k_{\tau,\varepsilon}(u_\varepsilon)&\leq&
E^k_{0,\varepsilon}(u_\varepsilon)+C_k'
(\SobODt{f_\varepsilon}{k-1})^2+C_k''\int_{\zeta=0}^\tau
E_{\zeta,\varepsilon}^k(u_\varepsilon) d\zeta+\\\nonumber
&+&C_k'''\sum_{j=1}^{k-1}\frac{1}{\varepsilon^{2(1+k-j)}}\int_{\zeta=0}^\tau
E_{\zeta,\varepsilon}^j(u_\varepsilon)d\zeta.
\end{eqnarray}
\end{proposition}
Before we prove this proposition, we establish a couple of technical
lemmas. The first one gives a formula for the covariant derivative
of the energy tensor $T^{ab,k}_\varepsilon(u)$. For the sake of
simplicity, we omit the smoothing parameter in the technical
lemmas. Moreover, we write $\nabla_I
u:=\nabla_{p_1}\dots\nabla_{p_{k-1}}u$ and for the tensor product
$e^{IJ}:=e^{p_1q_1}\dots e^{p_{k-1}q_{k-1}}$.
\begin{lemma}\label{technicality1}
For each $k\geq 2$, the divergence of
\[
T^{ab,k}(u)=(g^{ac}g^{bd}-\frac{1}{2}g^{ab}g^{cd})e^{IJ}\nabla_c\nabla_I
u\nabla_d\nabla_J u
\]
can be written in the following form:
\begin{eqnarray}\nonumber
\nabla_a T^{ab,k}(u)&=&\\\label{term1}& &(\nabla_a
e^{IJ})(g^{ac}g^{bd}-\frac{1}{2}g^{ab}g^{cd})\nabla_c\nabla_I u
\nabla_d\nabla_J u \\\label{term2}&+& e^{IJ}(g^{bd}\nabla_d\nabla_J
u)(g^{ac}\nabla_a\nabla_c\nabla_I u)\\\label{term3}&-&2
e^{IJ}(\nabla_d\nabla_J u)
(g^{ab}g^{cd}\nabla_{[a}\nabla_{c]}\nabla_I u).
\end{eqnarray}
\end{lemma}
\begin{proof}
We have
\begin{eqnarray}\nonumber
\nabla_a T^{ab,k}(u)&=&\\\nonumber & &(\nabla_a
e^{IJ})(g^{ac}g^{bd}-\frac{1}{2}g^{ab}g^{cd})\nabla_c\nabla_I u
\nabla_d\nabla_J u
\\&+&g^{ac}g^{bd}e^{IJ}(\nabla_a\nabla_c\nabla_Iu\nabla_d\nabla_Ju+\nabla_c\nabla_Iu\nabla_a\nabla_d\nabla_Ju)\\\nonumber
&-&
\frac{1}{2}g^{ab}g^{cd}e^{IJ}(\nabla_a\nabla_c\nabla_Iu\nabla_d\nabla_Ju+\nabla_c\nabla_Iu\nabla_a\nabla_d\nabla_Ju)=\\\nonumber
&=&(\nabla_a
e^{IJ})(g^{ac}g^{bd}-\frac{1}{2}g^{ab}g^{cd})\nabla_c\nabla_I u
\nabla_d\nabla_J
u\\\nonumber&+&g^{bd}e^{IJ}(g^{ac}\nabla_a\nabla_c\nabla_I
u)(\nabla_d\nabla_J u)+\nabla_c\nabla_I
u\nabla^c\nabla^b\nabla^Iu\\\nonumber&-&\frac{1}{2}(\nabla_d\nabla_Iu)(\nabla^b\nabla^d\nabla^Iu)
-\frac{1}{2}(\nabla_c\nabla_Iu)(\nabla^b\nabla^c\nabla^Iu)=\\\nonumber
&=& (\nabla_a
e^{IJ})(g^{ac}g^{bd}-\frac{1}{2}g^{ab}g^{cd})\nabla_c\nabla_I u
\nabla_d\nabla_J
u\\\nonumber&+&g^{bd}e^{IJ}(g^{ac}\nabla_a\nabla_c\nabla_I
u)(\nabla_d\nabla_J u)+\nabla_c\nabla_I
u\nabla^c\nabla^b\nabla^Iu\\\nonumber&-&\nabla_c\nabla_I
u\nabla^b\nabla^c\nabla^Iu\\\nonumber
&=& (\nabla_a
e^{IJ})(g^{ac}g^{bd}-\frac{1}{2}g^{ab}g^{cd})\nabla_c\nabla_I u
\nabla_d\nabla_J
u\\\nonumber&+&g^{bd}e^{IJ}(g^{ac}\nabla_a\nabla_c\nabla_I
u)(\nabla_d\nabla_J u)-2\nabla_c\nabla_I
u\nabla^{[b}\nabla^{c]}\nabla^Iu\\\nonumber
&=& (\nabla_a
e^{IJ})(g^{ac}g^{bd}-\frac{1}{2}g^{ab}g^{cd})\nabla_c\nabla_I u
\nabla_d\nabla_J
u\\\nonumber&+&g^{bd}e^{IJ}(g^{ac}\nabla_a\nabla_c\nabla_I
u)(\nabla_d\nabla_J u)-2\nabla_d\nabla_I
u\nabla^{[b}\nabla^{d]}\nabla^I u\\\nonumber
& =&(\nabla_a
e^{IJ})(g^{ac}g^{bd}-\frac{1}{2}g^{ab}g^{cd})\nabla_c\nabla_I u
\nabla_d\nabla_J u \\\nonumber&+& e^{IJ}(g^{bd}\nabla_d\nabla_J
u)(g^{ac}\nabla_a\nabla_c\nabla_I u)\\\nonumber&-&2
e^{IJ}(\nabla_d\nabla_J u)
(g^{ab}g^{cd}\nabla_{[a}\nabla_{c]}\nabla_I u).
\end{eqnarray}
\end{proof}
We shall consider (\ref{term1}), (\ref{term2}),(\ref{term3})
separately in the following lemmas. We start with (\ref{term1}):
\begin{lemma}\label{thesimplestterm1}
On $U$ we have
\[
\|\nabla_a^\varepsilon e^{IJ}_\varepsilon\|_m=O(1),\quad
(\varepsilon\rightarrow 0).
\]
\end{lemma}
\begin{proof}
This follows directly from the assumptions on the Killing vector
field $\xi$ (cf.\ (\ref{setting1}), and the assumption (\ref{setting4})
on the metric in section \ref {settingassumptions}) and the Leibniz rule:
\begin{eqnarray}\nonumber
\nabla_a^\varepsilon e^{bc}_\varepsilon&=&\nabla_a^\varepsilon
(g^{ab}_\varepsilon-\frac{2}{\langle
\xi,\xi\rangle_\varepsilon}\xi^b\xi^c)=-2\nabla_a^\varepsilon\left(\frac{\xi^
b}{\sqrt{-\langle \xi,\xi\rangle_\varepsilon}}\frac{\xi^
c}{\sqrt{-\langle
\xi,\xi\rangle_\varepsilon}}\right)=\\\label{term4}& &
-2\nabla_a^\varepsilon\left(\frac{\xi^ b}{\sqrt{-\langle
\xi,\xi\rangle_\varepsilon}}\right) \frac{\xi^ c}{\sqrt{-\langle
\xi,\xi\rangle_\varepsilon}}-2\nabla_a^\varepsilon\left(\frac{\xi^c}{\sqrt{-\langle
\xi,\xi\rangle_\varepsilon}}\right)\frac{\xi^ b}{\sqrt{-\langle
\xi,\xi\rangle_\varepsilon}}=O(1)
\end{eqnarray}
\end{proof}
Next we investigate (\ref{term3}):
\begin{lemma}\label{Mr. Ricci.applied}
\begin{eqnarray}\nonumber
-2\nabla_{[a}\nabla_{c]}\nabla_{p_1}\dots\nabla_{p_{k-1}} u&=&
R^d_{p_1 a c}\nabla_d\nabla_{p_2}\dots\nabla_{p_{k-1}}u\\&+&R^d_{p_2
a c}\nabla_{p_1}\nabla_d\dots\nabla_{p_{k-1}}u+\dots+\\\nonumber
&+&R^d_{p_{k-1} a c}\nabla_{p_1}\dots\nabla_{p_2}\dots\nabla_d u
\end{eqnarray}
\end{lemma}
\begin{proof}
The proof is a direct consequence of the Ricci Identities
\begin{equation}\label{MisterRicci}
-2\nabla_{[a}\nabla_{c]}X_{p_1\dots p_k}= R^d_{p_1 a c}X_{d p_2\dots
p_k}+R^d_{p_2 a c}X_{p_1 d\dots p_k}+\dots+R^d_{p_k a c}X_{p_1
p_2\dots d}.
\end{equation}
\end{proof}
This concludes the algebraic treatment of (\ref{term3}). What is left is to
give an asymptotic estimate on compact sets. First we need
information on the asymptotic growth of the Riemann tensor:
\begin{lemma}\label{asympMr.Riemann}
For each compact set $K$ in $U$ and for each $k\geq 0$, there exist
positive constants $F_k>0$ such that for sufficiently small
$\varepsilon$ the following holds on $K$:
\begin{equation}\label{term6}
\vert\partial_{\rho_1}\dots\partial_{\rho_k}R_{abc}^{d,\varepsilon}\vert\leq
\frac{F_k}{\varepsilon^{2+k}}
\end{equation}
and
\begin{equation}\label{term7}
\vert\nabla_{a_1}^\varepsilon\dots\nabla_{a_k}^\varepsilon
R_{abc}^{d,\varepsilon}\vert\leq \frac{F_k}{\varepsilon^{2+k}}.
\end{equation}
\end{lemma}
\begin{proof}
(\ref{term6}) is an immediate consequence of the formula for the
coefficients of the Riemann tensor in terms of Christoffel symbols
(hence in terms of partial derivatives of the metric coefficients)
and their asymptotic growth. For (\ref{term7}) one needs in addition
the formula expressing the covariant derivative in terms of partial
derivatives and Christoffel symbols.
\end{proof}
Lemma \ref{Mr. Ricci.applied} and Lemma \ref{asympMr.Riemann} in
conjunction yield:
\begin{lemma}\label{term3est}
For each compact set $K$ in $U$ and for each $k\geq 2$, there exist
positive constants $G_k>0$ such that for sufficiently small
$\varepsilon$ the following holds on $K$:
\begin{eqnarray}\nonumber
\vert
2\nabla_{[a}^\varepsilon\nabla_{c]}^\varepsilon\nabla_{p_1}^\varepsilon\dots\nabla_{p_{k-1}}^\varepsilon
u\vert^2\leq \frac{G_k}{\varepsilon^4} \sum_{p_1,\dots,p_{k-1}}
\vert\nabla_{p_1}^\varepsilon\nabla_{p_2}^\varepsilon\dots\nabla_{p_{k-1}}^\varepsilon
u\vert^2.
\end{eqnarray}
\end{lemma}
This is an immediate conclusion and therefore we omit the proof.
Finally, we investigate term (\ref{term2}). The next calculation is
a purely algebraic manipulation. Again, we omit to write down the
smoothing parameter $\varepsilon$ explicitly.
\begin{lemma}\label{wavetermtechnicality}
For each $k\geq 2$, we have
\begin{equation}\label{inductivehypothesisJ}
g^{ac}\nabla_a\nabla_c\nabla_{p_1}\dots\nabla_{p_{k-1}}
\,u=\nabla_{p_1}\dots\nabla_{p_{k-1}}\,\Box
u+\sum_{j=1}^{k-1}\mathcal R^{(k-1,j)} u,
\end{equation}
where $\mathcal R^{(k,j)} u$ represents a linear combination of
contractions of the $(k-j)$th covariant derivative of the Riemann
tensor with the $j$th covariant derivative of $u$, $0\leq j\leq k$.
\end{lemma}
\begin{proof}
Before we start, we note that we shall write
\[
\mathcal R^{(k,j)} u+\mathcal R^{(k,j)} u=\mathcal R^{(k,j)} u,
\]
to indicate that the sum of such linear combinations is a linear
combination of the same type (containing the same order of covariant
derivatives of the Riemann tensor and the function $u$). In this
sense, by the Leibniz rule we have:
\begin{equation}\label{diffMrRicci}
\nabla_{p_k}\sum_{j=1}^{k-1}\mathcal R^{(k-1,j)}
u=\sum_{j=1}^k\mathcal R^{(k,j)} u.
\end{equation}
We start by calculating the basis of induction, namely $k=2$. Since
the connection is torsion free, we have:
\begin{equation}
g^{ac}\nabla_a\nabla_c\nabla_{p_1}u=g^{ac}\nabla_a(\nabla_{p_1}\nabla_cu-2\nabla_{[p_1}\nabla_{c]}u)=g^{ac}\nabla_{p_1}\nabla_a\nabla_cu=\nabla_{p_1}\,\Box
u.
\end{equation}
So the claim holds in the case $k=2$ (since the linear combination $\mathcal R^{(1,j)}$ is allowed to vanish).\\
For the inductive step, assume (\ref{inductivehypothesisJ}) holds.
To manage the step $k-1\rightarrow k$, we have to repeatedly use the
Ricci identities (\ref{MisterRicci}) in order to shuffle the
covariant derivative indices of $u$. First, we shuffle the indices
$c,p_1$:
\begin{eqnarray}\nonumber
g^{ac} \nabla_a\nabla_c\nabla_{p_1}\dots\nabla_{p_k}u&=&
=g^{ac}\nabla_a\nabla_{p_1}\nabla_c\nabla_{p_2}\dots\nabla_{p_k}u-\\\nonumber
&-&2g^{ac}\nabla_a\nabla_{[p_1}\nabla_{c]}\nabla_{p_2}\dots\nabla_{p_k}u=\\\nonumber
&=&g^{ac}\nabla_a\nabla_{p_1}\nabla_c\nabla_{p_2}\dots\nabla_{p_k}u+\\\nonumber
&+&g^{ac}\nabla_a \left( \sum_{i=2}^k
R_{p_ip_1c}^d\nabla_{p_2}\dots\nabla_{p_{i-1}}\nabla_d\nabla_{p_{i+1}}\dots\nabla_{p_k}u\right)\\\nonumber
&=&g^{ac}\nabla_a\nabla_{p_1}\nabla_c\nabla_{p_2}\dots\nabla_{p_k}u+
\sum_{j=1}^k\mathcal R^{(k,j)}u.
\end{eqnarray}
Repeating the same procedure a second time by shuffling $p_1$ and
$a$, we receive
\begin{equation}\label{readyformydefense}
g^{ac}\nabla_a\nabla_c\nabla_{p_1}\dots\nabla_{p_k}u=\nabla_{p_1}(g^{ac}\nabla_a\nabla_c\nabla_{p_2}\dots
\nabla_{p_{k}}u)+\sum_{j=1}^k\mathcal R^{(k,j)}u.
\end{equation}
We may now use the induction hypothesis
(\ref{inductivehypothesisJ}). Inserting into
(\ref{readyformydefense}) yields by means of (\ref{diffMrRicci}),
\begin{eqnarray}\nonumber
g^{ac}\nabla_a\nabla_c\nabla_{p_1}\dots\nabla_{p_k}u&=&\nabla_{p_1}\dots\nabla_{p_k}\,
\Box u+ \nabla_{p_1}\left(\sum_{j=1}^{k-1}\mathcal
R^{(k-1,j)}u\right)+\sum_{j=1}^k\mathcal R^{(k,j)}u=\\\nonumber
&=&\nabla_{p_1}\dots\nabla_{p_k}\, \Box u+\sum_{j=1}^k\mathcal
R^{(k,j)}u.
\end{eqnarray}
and we are done.
\end{proof}
The last helpful estimate we establish before proving Proposition
\ref{energyinequalitylevelkformula} is the following:
\begin{lemma}\label{lasttechnicality}
For each compact set $K$ in $U$ and for each $k\geq 2$, there exist
positive constants $G_k>0$ such that for sufficiently small
$\varepsilon$ the following holds on $K$:
\begin{equation}
\vert \mathcal R^{(k-1,j)}_\varepsilon
u\vert^2\leq\frac{G_k}{\varepsilon^{2(k-j+1)}}\sum_{{q_1\dots q_j \atop 1\leq j\leq k-1}}\vert
\nabla_{q_1}^\varepsilon\dots\nabla_{q_j}^\varepsilon u\vert ^2.
\end{equation}
\end{lemma}
\begin{proof}
The proof follows directly from Lemma \ref{asympMr.Riemann} and the
definition of $ R^{(k-1,j)}_\varepsilon$ (a linear combination of
covariant derivatives of the Riemann tensor of $k-1-j$ order and
covariant derivatives of $u$ of order $j$).
\end{proof}

Finally, we are prepared to prove the main statement, Proposition
\ref{energyinequalitylevelk}: \\
\begin{proof}
We start with (\ref{energyhierarchystokes}), where we insert the
solution $(u_\varepsilon)_\varepsilon$ of the wave equation
(\ref{weqofsettingeps}):
\begin{equation}\label{energyhierarchystokesx}
E^k_{\tau,\varepsilon}(u_\varepsilon)\leq
E^k_{\tau=0,\varepsilon}(u_\varepsilon)+\sum_{j=0}^k\int_{\Omega_\tau}\xi_b^\varepsilon\nabla_a^\varepsilon
T^{ab,j}_\varepsilon(u_\varepsilon) \mu_\varepsilon.
\end{equation}
Hence, for each energy hierarchy $m$, we have to estimate the
divergence of $T^{ab,k}_\varepsilon(u_\varepsilon)$ for
each $2\leq k\leq m$ (the case $k=1$ has been proved in
Proposition \ref{energyinequalitylevelk1} and the case $k=0$ can be easily be derived from the information given in
the proof of Proposition \ref{energyinequalitylevelk1}). So let $k\geq 2$. By
Lemma \ref{technicality1}, we have
\begin{eqnarray}\nonumber
\nabla_a^\varepsilon
T^{ab,k}_\varepsilon(u_\varepsilon)&=&\\\label{term1eps}&
&(\nabla_a^\varepsilon e^{IJ}_\varepsilon)(g^{ac}_\varepsilon
g^{bd}_\varepsilon-\frac{1}{2}g^{ab}_\varepsilon
g^{cd}_\varepsilon)\nabla_c^\varepsilon\nabla_I^\varepsilon
u_\varepsilon \nabla_d^\varepsilon\nabla_J^\varepsilon u_\varepsilon
\\\label{term2eps}&+& e^{IJ}_\varepsilon(g^{bd}_\varepsilon\nabla_d^\varepsilon\nabla_J^\varepsilon
u_\varepsilon)(g^{ac}_\varepsilon\nabla_a^
\varepsilon\nabla_c^\varepsilon\nabla_I^\varepsilon
u_\varepsilon)\\\label{term3eps}&-&2
e^{IJ}_\varepsilon(\nabla_d^\varepsilon\nabla_J^\varepsilon
u_\varepsilon) (g^{ab}_\varepsilon
g^{cd}_\varepsilon\nabla_{[a}^\varepsilon\nabla_{c]}
^\varepsilon\nabla_I^\varepsilon u_\varepsilon).
\end{eqnarray}
We estimate the asymptotic growth of all the three terms
(\ref{term1eps}), (\ref{term2eps}), (\ref{term3eps}) by means of the
preceding lemmas. The first term (\ref{term1eps}) can be estimate
by means of Lemma \ref{thesimplestterm1} as follows. For each $k$
there exists a constant $T_k$ such that for sufficiently small
$\varepsilon$ we have
\begin{equation}\label{eingemachtes1}
\vert(\nabla_a^\varepsilon e^{IJ}_\varepsilon)(g^{ac}_\varepsilon
g^{bd}_\varepsilon-\frac{1}{2}g^{ab}_\varepsilon
g^{cd}_\varepsilon)\nabla_c^\varepsilon\nabla_I^\varepsilon
u_\varepsilon \nabla_d^\varepsilon\nabla_J^\varepsilon
u_\varepsilon\vert^2\leq
T_k\sum_{p_1,\dots,p_k}\vert
\nabla_{p_1}^\varepsilon\dots\nabla_{p_k}^\varepsilon u_\varepsilon\vert^2.
\end{equation}
So we are done with the first term. By Lemma
\ref{wavetermtechnicality}, we have
\begin{equation}\label{inductivehypothesisJcitation}
g^{ac}_\varepsilon\nabla_a^\varepsilon\nabla_c^\varepsilon\nabla_{p_1}^\varepsilon\dots\nabla_{p_{k-1}}^\varepsilon
\,u_\varepsilon=\nabla_{p_1}^\varepsilon\dots\nabla_{p_{k-1}}^\varepsilon\,\Box^\varepsilon
u_\varepsilon+\sum_{j=1}^{k-1}\mathcal R^{(k-1,j)}_\varepsilon
u_\varepsilon,
\end{equation}
and Lemma \ref{lasttechnicality} provides the asymptotic growth
behavior of the quantities \\ $\sum_{j=1}^{k-1}\mathcal
R^{(k-1,j)}_\varepsilon u_\varepsilon$. We further may use that
$(u_\varepsilon)_\varepsilon$ solves the initial value problem
(\ref{weqofsettingeps}) on the level of representatives; taking the
covariant derivative $k$ times this implies
\[
\nabla_{p_1}^\varepsilon\dots\nabla_{p_{k-1}}^\varepsilon\,\Box^\varepsilon
u_\varepsilon=\nabla_{p_1}^\varepsilon\dots\nabla_{p_{k-1}}^\varepsilon
f_\varepsilon.
\]
Hence, there exists a positive constant $T_k'$ such that the left
side of (\ref{inductivehypothesisJcitation}) is bounded for small
$\varepsilon$ by
\begin{eqnarray}\nonumber
\vert
g^{ac}_\varepsilon\nabla_a^\varepsilon\nabla_c^\varepsilon\nabla_{p_1}^\varepsilon\dots\nabla_{p_{k-1}}^\varepsilon
\,u_\varepsilon\vert^2&\leq&
T_k'\vert\nabla_{p_1}^\varepsilon\dots\nabla_{p_{k-1}}^\varepsilon
f_\varepsilon\vert ^2+\\\label{eingemachtes2} &+&T_k'\sum_{{q_1\dots
q_j\atop 1\leq j\leq k-1}}\frac{1}{\varepsilon^{2(1+k-j)}}\vert
\nabla_{q_1}^\varepsilon\dots\nabla_{q_j}^\varepsilon
u_\varepsilon\vert^2.
\end{eqnarray}
For term (\ref{term3eps}) we obtain by Lemma \ref{term3est} that locally there exists a constant $G_k>0$ such that for sufficiently small $\varepsilon$
we have:
\begin{equation}\label{eingemachtes3}
\vert
2\nabla_{[a}^\varepsilon\nabla_{c]}^\varepsilon\nabla_{p_1}^\varepsilon\dots\nabla_{p_{k-1}}^\varepsilon
u_\varepsilon\vert^2\leq \frac{G_k}{\varepsilon^4} \sum_{p_1,\dots,p_{k-1}}
\vert\nabla_{p_1}^\varepsilon\nabla_{p_2}^\varepsilon\dots\nabla_{p_{k-1}}^\varepsilon
u_\varepsilon\vert^2.
\end{equation}
We finally may use the estimates (\ref{eingemachtes1}), (\ref{eingemachtes2}) and (\ref{eingemachtes3})
to estimate the energies $(T^{a,b,k}_\varepsilon(u_\varepsilon)_\varepsilon)$. This yields
\begin{eqnarray}\nonumber
\vert \nabla_a^\varepsilon T^{ab,k}_\varepsilon(u_\varepsilon)\vert &\leq&\label{eingemachtes4} S_k\sum_{p_1\dots p_k}
\vert\nabla_{p_1}^\varepsilon\dots\nabla_{p_k}^\varepsilon u_\varepsilon\vert^2+\\\nonumber&+& S_k\sum_{p_1\dots p_{k-1}
}\vert\nabla_{p_1}^\varepsilon\dots\nabla_{p_{k-1}}^\varepsilon f_\varepsilon\vert^2\\\nonumber&+& S_k\frac{1}{\varepsilon^{2(1+k-j)}} \sum_{{q_1,\dots,q_j \atop 1\leq j\leq k-1}}
\vert\nabla_{q_1}^\varepsilon\nabla_{q_2}^\varepsilon\dots\nabla_{q_j}^\varepsilon
u_\varepsilon\vert^2.
\end{eqnarray}
Summation over $k=1\dots m$ and integration yields for positive constants $C_m'$
\begin{eqnarray}\nonumber
E^m_{\tau,\varepsilon}(u_\varepsilon)&\leq& E^
m_{0,\varepsilon}(u_\varepsilon)\\\label{lastmanstanding}
&+&C_m'\left((\SobODt{u_\varepsilon}{m})^2+(\SobODt{f_\varepsilon}{m-1})^2
+\sum_{j=1}^ {m-1}\frac{1}{\varepsilon^
{2(1+m-j)}}(\SobODt{u_\varepsilon}{j})^2\right).
\end{eqnarray}
This may be turned into an energy inequality by Proposition \ref{lemma1} (\ref{ineqSDXE}) and the information from section
\ref{leray}. Indeed, for each $j$, we have a positive constant $A_j'$ such that for small $\varepsilon$
\begin{equation}\label{lerayxxy}
(\SobODt{u_\varepsilon}{j})^2=\int_{\zeta=0}^\tau
(\SobSDt{u_\varepsilon}{j})^2\, d\zeta\leq A_j'\int_{\zeta=0}^\tau
E^j_{\zeta,\varepsilon}(u_\varepsilon)d \zeta.
\end{equation}
Inserting (\ref{lerayxxy}) into (\ref{lastmanstanding}) therefore yields:
\begin{eqnarray}\nonumber
E^m_{\tau,\varepsilon}(u_\varepsilon)&\leq& E^ m_{0,\varepsilon}(u_\varepsilon)\\\nonumber
&+&C_m'(\SobODt{f_\varepsilon}{m-1})^2 +C_m''\int_{\zeta=0}^\tau E^m_{\zeta,\varepsilon}(u_\varepsilon)d\zeta\\\nonumber
 &+&C_m'''\sum_{j=1}^ {m-1}\frac{1}{\varepsilon^ {2(1+m-j)}}\int_{\zeta=0}^\tau E^j_{\zeta,\varepsilon}(u_\varepsilon)d\zeta.
\end{eqnarray}
and the proof is finished.
\end{proof}
\section[Bounds on energies]{Bounds on energies via bounds on initial energies (Part D)}\label{energyestimates2}
If we apply Gronwall's Lemma to
(\ref{energyinequalitylevelkformula}) we obtain:
\begin{proposition}\label{applicationenergygronwall}
For each $k\geq 1$ there exist positive constants
$C_k',C_k'',C_k'''$ such that we have for each $\varepsilon>0$ and
for each $0\leq\tau\leq\gamma$,
\begin{equation}\label{applicationenergygronwallformula}
E^k_{\tau,\varepsilon}(u_\varepsilon)\leq
\left(E^k_{0,\varepsilon}(u_\varepsilon)+C_k'
(\SobODt{f_\varepsilon}{k-1})^2+C_k'''\sum_{j=1}^{k-1}\frac{1}{\varepsilon^{2(1+k-j)}}\int_{\zeta=0}^\tau
E_{\zeta,\varepsilon}^j(u_\varepsilon) d\zeta\right) e^{C_k''\tau}
\end{equation}
Note that $C_1'''=0$ (this refers to the empty sum when $k=1$)
\end{proposition}
A direct consequence of the preceding proposition is the following
statement:
\begin{proposition}\label{energiesviainitialenergies}
Let $0\leq\tau\leq\gamma$. If for each $k$, the initial energy
$(E^k_{0,\,\varepsilon}(u_\varepsilon))_\varepsilon$ determines a
moderate (resp.\ negligible) net of real numbers, then also \[(\sup
_{0\leq\zeta\leq\tau}
E^k_{\zeta,\,\varepsilon}(u_\varepsilon))_\varepsilon\] is moderate
(resp.\ negligible) for each $k$. \footnote{ In the statement of
\cite{VW}, a typing error occurs, and instead of $k$, $k-1$ is
written. Furthermore, for to prove Proposition \ref{finalprop},
it is not sufficient to have moderate resp.\ negligible nets
$(E^k_{\tau,\,\varepsilon}(u_\varepsilon))_\varepsilon$, but the
supremum of the energies over all $0\leq\zeta\leq\tau$ must be
moderate resp.\ negligible. }
\end{proposition}
\begin{proof}
The proof is inductive. The basis of induction is $k=1$: In this
case, the sum in the brackets of inequality
(\ref{applicationenergygronwallformula}) is empty, and since
$(f_\varepsilon)_\varepsilon$ is a negligible function (since it is
the representative of zero), the net of real numbers
$(\SobODt{f_\varepsilon}{0})_\varepsilon$ is negligible. By the
assumption, also
$(E^{j=0}_{0,\,\varepsilon}(u_\varepsilon))_\varepsilon$ is moderate
(resp.\ negligible). As a consequence of inequality
(\ref{applicationenergygronwallformula}), $(\sup
_{0\leq\zeta\leq\tau}
E^1_{\zeta,\,\varepsilon}(u_\varepsilon))_\varepsilon$ is moderate
(resp.\ negligible).

The inductive step is similar: \\ Assume for $0\leq j< k$ we know
that $(\sup _{0\leq\zeta\leq\tau}
E^j_{\zeta,\,\varepsilon}(u_\varepsilon))_\varepsilon$ is moderate
(resp.\ negligible). By assumption,
$(E^{k}_{0,\,\varepsilon}(u_\varepsilon))_\varepsilon$ is moderate
(resp.\ negligible) as well. Furthermore,
$(f_\varepsilon)_\varepsilon$ is a negligible function (since it is
the representative of zero), hence the net of real numbers
$(\SobODt{f_\varepsilon}{0})_\varepsilon$ is negligible. By applying
inequality (\ref{applicationenergygronwallformula}) we achieve that
$(\sup
_{0\leq\zeta\leq\tau}E^{k}_{\zeta,\,\varepsilon}(u_\varepsilon))_\varepsilon$
is moderate (resp.\ negligible) and we are done.
\end{proof}

\section[Sobolev estimates]{Estimates via a Sobolev embedding theorem (Part E)}
In order to translate the bounds on the energies $(E_{\zeta,\varepsilon}^j(u_\varepsilon))$ back to bounds on
the nets $(u_\varepsilon)_\varepsilon$ and its derivatives, we shall need the following
"generalized" Sobolev lemma expressed in terms of the energies $(E_{\zeta,\varepsilon}^j(u_\varepsilon))$.
\begin{lemma}
For $m>3/2$, there exists a constant $K$, a number $N$ and an
$\varepsilon_0$ such that for all $\phi\in C^{\infty}(\Omega_\tau)$
and for all $\zeta\in[0,\tau]$ and for all $\varepsilon<\varepsilon_0$ we have
\begin{equation}\label{lemmaadamsapplicanda}
\sup_{x\in S_\zeta}\vert \phi(x)\vert\leq
K\varepsilon^{-N}\sup_{0\leq\zeta\leq\tau}
E_{\zeta,\varepsilon}^m(\phi).
\end{equation}
\end{lemma}
Before we prove the statement, we note that since the right hand side of (\ref{lemmaadamsapplicanda})
is independent of $\zeta$, the statement is equivalent to
\begin{equation}\label{lemmaadamsapplicandasuperseda}
\sup_{x\in \Omega_\tau}\vert \phi(x)\vert\leq
K\varepsilon^{-N}\sup_{0\leq\zeta\leq\tau}
E_{\zeta,\varepsilon}^m(\phi).
\end{equation}
\begin{proof}
By (\cite{Adams}, Lemma 5.17), there exists\footnote{this follows from the fact that
boundary of the paraboloid $\Omega$ is Lipschitz} a constant $K$ such that for each
$0\leq\zeta\leq\tau$ we have for $m>3/2$
\begin{equation}\label{sobex2}
\sup_{x\in S_\zeta}\vert \phi(x)\vert\leq K\|\phi\|_{m, S_{\zeta}}.
\end{equation}
with $\|\phi\|_{m, S_{\zeta}}$, the three dimensional Sobolev norm
on $S_\zeta$ with the Volume form of $\mathbb R^3$, that is,
\[
\|\phi\|_{m,S_\zeta}=\int_{S_\zeta}\sum_{{\rho_1,\dots,\rho_j \atop 0\leq
j\leq m}}\vert\partial_{\rho_1}\dots\partial_{\rho_j}\phi\vert ^2
dx^1 dx^2 dx^3,
\]
where partial derivatives are only taken with respect to space-variables,
that is tangential to $S_\zeta$ for each $0\leq\zeta\leq \tau$. Note
that the expression is not invariant for two reasons. The first is
that partial derivatives are involved and not covariant derivatives.
Secondly, the volume element of $\mathbb R^3$ is taken. We
shall, however, derive an estimate by invariant expressions, namely,
the energies.

Next, we introduce the determinant of the metric into the Sobolev
norms. Note that on $\Omega_\gamma$, which is a compact set, the absolute value of the
determinant of the metric $\vert g_\varepsilon\vert $ for
sufficiently small $\varepsilon$ is bounded from below by a fixed
power of $\varepsilon$. This follows from invertibility of the
metric. In our case, however, where the metric and its inverse locally
are $O(1)$, there exists a positive constant $C$ and a $\varepsilon_0\in I$ such that for all
$\varepsilon<\varepsilon_0$ we have
\begin{equation}
\vert g_\varepsilon\vert^{\frac{1}{2}} \geq C
\end{equation}
holds on $\Omega_\gamma$. Therefore, for small $\varepsilon$ and
for all $\zeta$, $0\leq\zeta\leq\tau$, we have the estimate
\begin{equation}\label{sobex3}
\|\phi\|_{m, S_{\zeta}}\leq
C^{-1}\int_{S_\zeta}\sum_{{\rho_1,\dots,\rho_j \atop 0\leq
j\leq m}}\vert\partial_{\rho_1}\dots\partial_{\rho_j}\phi\vert ^2
\vert g_\varepsilon\vert^\frac{1}{2} dx^1 dx^2 dx^3.
\end{equation}
Clearly, this can further be estimated by the cruder three
dimensional Sobolev norm $\SobSdz{\phi}{m}$, which respects also
time-derivatives. Therefore, we may estimate (\ref{sobex3}) by
\begin{equation}\label{sobex4}
\forall\; \zeta \in[0,\tau]\;\forall\;\varepsilon<\varepsilon_0:\|\phi\|_{m, S_{\zeta}}\leq C^{-1}(\SobSdz{\phi}{m}).
\end{equation}
Inserting (\ref{sobex4}) into (\ref{sobex2}) yields the estimate
\begin{equation}\label{sobex5}
\forall\; \zeta \in[0,\tau]\;\forall\;\varepsilon<\varepsilon_0:\sup_{x\in S_\zeta}\vert \phi(x)\vert\leq K
C^{-1}(\SobSdz{\phi}{m}).
\end{equation}
Finally we apply Proposition \ref{lemma1} twice,
namely the estimates (\ref{ineqSdXSD}) and (\ref{ineqSDXE}). This yields
a number $N'$ such that for sufficiently small $\varepsilon$ and for
all $0\leq\zeta\leq\tau$ we have
\begin{equation}\label{sobex6}
\sup_{x\in S_\zeta}\vert \phi(x)\vert\leq
\varepsilon^{-N'}E^m_{\zeta,\varepsilon}(\phi).
\end{equation}
On the right side of (\ref{sobex6}) we may now take the supremum
over $\zeta\in[0,\tau]$ and achieve
\begin{equation}\label{sobex5}
\sup_{x\in S_\zeta}\vert \phi(x)\vert\leq
\varepsilon^{-N'}(\sup_{0\leq\zeta\leq
\tau}E^m_{\zeta,\varepsilon}(\phi)).
\end{equation}
\end{proof}
The main statement of this section is the
following:
\begin{proposition}\label{finalprop}
Let $0\leq\tau\leq\gamma$. If for each $k$,
$(\sup_{0\leq\zeta\leq\tau}
E^k_{\zeta,\varepsilon}(u_\varepsilon))_\varepsilon$ is moderate
(resp.\ negligible), then $(u_\varepsilon)_\varepsilon$ satisfies
moderate bounds (negligible bounds) on $\Omega_\tau$.
\end{proposition}
\begin{proof}
Inserting $(u_\varepsilon)_\varepsilon$ into
(\ref{lemmaadamsapplicanda}) yields
\begin{equation}\label{lemmaadamsapplicandaappl1}
\sup_{x\in S_\tau}\vert u_\varepsilon(x)\vert\leq
K\varepsilon^{-N}\sup_{0\leq\zeta\leq\tau}
E_{\zeta,\varepsilon}^m(u_\varepsilon).
\end{equation}
Similarly, for higher derivatives of $(u_\varepsilon)_\varepsilon$,
one achieves bounds via higher energies:
\begin{equation}\label{lemmaadamsapplicandaappl2}
\sup_{x\in \Omega_\tau}\vert
\partial_{\rho_1}\dots\partial_{\rho_k}\partial_t^lu_\varepsilon(x)\vert\leq
K\varepsilon^{-N}\sup_{0\leq\zeta\leq\tau}
E_{\zeta,\varepsilon}^{m+k+l}(u_\varepsilon).
\end{equation}
\end{proof}

\section[Existence and uniqueness]{Existence and uniqueness (Part F)} In this section we
collect all the preceding material and prove a local existence and
uniqueness result for the wave equation; this, however, is based on the
specific choice of representative of the metric $g_{ab}$. In the
next section we show that the generalized solution does indeed not
depend on the (symmetric) choice of the metric representative.

To begin with, we note that the wave equation for the static representative
written down in coordinates is time reversible, meaning: the differential equation
(\ref{weqofsettingeps}) is invariant under a transformation of the form $t\mapsto -t$.
In other words: If $(u_\varepsilon(t,x^i))_\varepsilon$ is a solution of (\ref{weqofsettingeps}) for $t\leq 0$, also
$(u_\varepsilon(-t,x^i))_\varepsilon$ solves (\ref{weqofsettingeps}), however for $t\geq 0$.

Therefore, similarly as in the above we may achieve estimates for $(u_\varepsilon(t,x^i))$ for $t\leq 0$.
The compact region on which the estimates are established we call $\Omega_{-\tau}$, $0\leq\tau\leq\gamma$
which is the (time--)reflected $\Omega_\tau$ (cf. figure 3). It is, however, also possible to
define the $\Omega_\tau$ as in section \ref{smoothsettingfoliation} and apply Stokes' theorem, thus
repeating the whole procedure on estimating of Part A to Part E, just with $\Omega_\tau$ replaced by
$\Omega_{-\tau}$, $0\leq\tau\leq \gamma$.

\begin{figure}\label{figure3}
   \begin{center}
   \fbox{\includegraphics[width=5.21in]{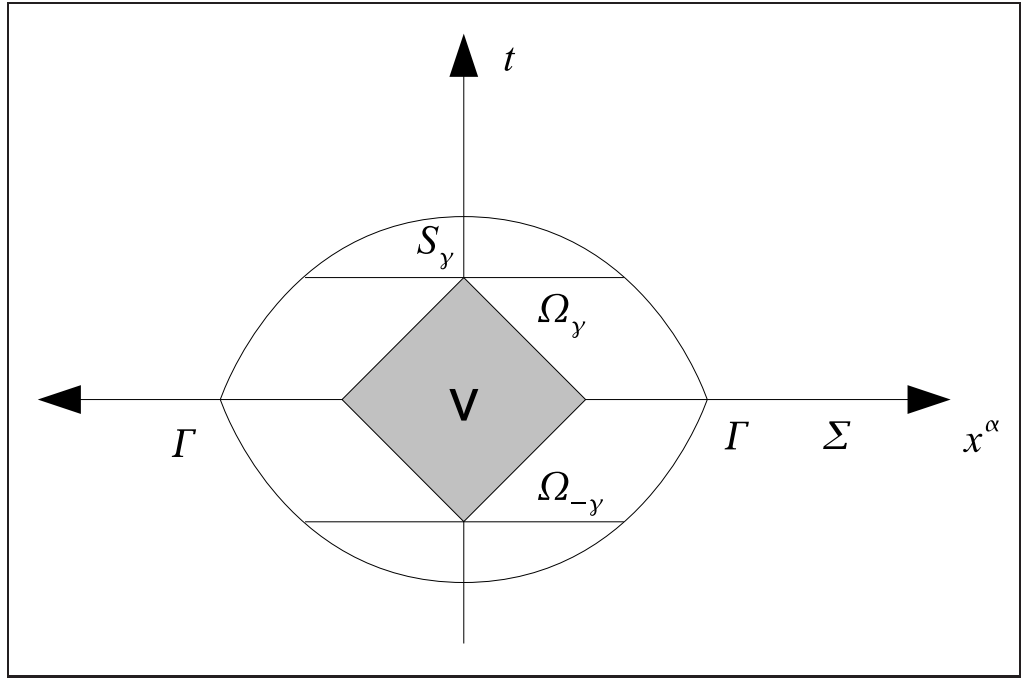}}
   \end{center}
   \caption{Choice of the open set $V$ for the existence result}
\end{figure}

We are now prepared to present the existence and uniqueness theorem for the Cauchy problem of the wave equation
in our setting:
\begin{theorem}\label{THETHEOREM}
For each point $p$ in $\Sigma$ there exists an open neighborhood $V\subset U$ on which a unique generalized
solution $u\in\mathcal G(V)$ of the initial value problem
(\ref{weqofsetting}) exists.
\end{theorem}
Even though there will be redundancies, we shall present a detailed
proof of the theorem.
\begin{proof}
Let $(U,(t,x^\mu))$ be an open relatively compact static coordinate chart at $p$. By Theorem
\ref{Lemmastaticgeneralized}, we choose a representative
$(g_{ab}^\varepsilon)_\varepsilon$ of the metric which is static for
each $\varepsilon$ and which (for small $\varepsilon$) satisfies the
respective bounds according to the setting. Furthermore, a
representative $(e_{ab}^\varepsilon)_\varepsilon$ of $e_{ab}$ may be
directly constructed from the representative
$(g_{ab}^\varepsilon)_\varepsilon$ of the metric.

Under these conditions Proposition \ref{lemma1} may be applied.

On the level of representatives the initial value problem
(\ref{weqofsetting}) takes the form (\ref{weqofsettingeps}) with
$(f_\varepsilon)_\varepsilon$ negligible, and
$(v_\varepsilon)_\varepsilon$, $(w_\varepsilon)_\varepsilon$
moderate.
\\{\it Part 1. Existence of a local moderate net of solutions}\\
The smooth theory then provides smooth solutions
$(u_\varepsilon)_\varepsilon$ on $U$.

We first show that the net $(u_\varepsilon)_\varepsilon$ satisfies moderate bounds on $\Omega_{\gamma}$: Moderate data
$(v_\varepsilon)_\varepsilon$, $(w_\varepsilon)_\varepsilon$
translate by means of Proposition
\ref{initialenergiesviainitialdata} to moderate initial energies
$(E^k_{0,\varepsilon}(u_\varepsilon))_\varepsilon$ for each
hierarchy $k$. Moreover, by means of Proposition
\ref{applicationenergygronwall}, moderate initial energies
$(E^k_{0,\varepsilon}(u_\varepsilon))_\varepsilon$ ($k\geq 1$)
translate to moderate energies
$(E^k_{\tau,\varepsilon}(u_\varepsilon))_\varepsilon$ ($k\geq 1$),
where $0\leq \tau\leq \gamma$, this is the statement of Proposition
\ref{energiesviainitialenergies}. Finally Proposition
\ref{finalprop} states that moderate energies
$(E^k_{\tau,\varepsilon}(u_\varepsilon))_\varepsilon$ ($k\geq
1,\,0\leq \tau\leq \gamma$) translate to moderate bounds of $(u_\varepsilon)_\varepsilon$ and of its derivatives of all orders
on $\Omega_{\gamma}$. Due to the preceding introductory remark, estimates of the same kind hold on $\Omega_{-\gamma}$. We pick an open subset $V$ of $\Omega_{-\gamma,\gamma}:=\Omega_{-\gamma}\cup\Omega_{\gamma}$ (see figure \ref{figure3}). Due to our considerations in the beginning of Part B (section \ref{partb}), we have therefore established
that $(u_\varepsilon)_\varepsilon$ is a moderate net on $V$.
\\{\it Part 2. Uniqueness of solutions}\\
We may now define a local generalized solution $u$ on $V$ by
\[
u:=[(u_\varepsilon)_\varepsilon],
\]
the class of $(u_\varepsilon)_\varepsilon$ from Part 1. What is left
to be shown is that the solution $u$ does not depend on the choice
of representatives of $(f_\varepsilon)_\varepsilon,\,
(v_\varepsilon)_\varepsilon,\, (w_\varepsilon)_\varepsilon$ of
$f\equiv 0, v, w $.

Let therefore $(\hat f_\varepsilon)_\varepsilon,\, (\hat
v_\varepsilon)_\varepsilon,\, (\hat w_\varepsilon)_\varepsilon$ be
further representatives of $f\equiv 0, v, w $, and let $(\hat
u_\varepsilon)_\varepsilon$ be the respective net of smooth
solutions.

Setting
\[
\widetilde u_\varepsilon:=u_\varepsilon-\hat u_\varepsilon,\;
\widetilde f_\varepsilon:=f_\varepsilon-\hat f_\varepsilon,\;
\widetilde v_\varepsilon:=v_\varepsilon-\hat v_\varepsilon,\;
\widetilde w_\varepsilon:=w_\varepsilon-\hat w_\varepsilon,
\]
we see that for each $\varepsilon>0$  $\widetilde u_\varepsilon$ is
a solution of the initial value problem
\begin{eqnarray}\nonumber
\Box^\varepsilon \widetilde u_\varepsilon&=&\widetilde
f_\varepsilon\\\nonumber \widetilde
u_\varepsilon(t=0,x^\mu)&=&\widetilde
v_\varepsilon(x^\mu)\\\nonumber
\partial_t\widetilde u_\varepsilon(t=0,x^\mu)&=&\widetilde w_\varepsilon(x^\mu)\nonumber,
\end{eqnarray}
Note, that here all the nets $(\widetilde
f_\varepsilon)_\varepsilon,(\widetilde
v_\varepsilon)_\varepsilon,(\widetilde w_\varepsilon)_\varepsilon$
are negligible. What is left to show is that the net $(\widetilde
u_\varepsilon)_\varepsilon$ is negligible, as well; uniqueness of
the above defined solution $u$ is then obvious, since $[(\widetilde
u_\varepsilon)_\varepsilon]=[(u_\varepsilon)_\varepsilon]=u$.

Negligible data
$(v_\varepsilon)_\varepsilon$, $(w_\varepsilon)_\varepsilon$
translate by means of Proposition
\ref{initialenergiesviainitialdata} to negligible initial energies
$(E^k_{0,\varepsilon}(u_\varepsilon))_\varepsilon$ for each
hierarchy $k$. Moreover, by means of Proposition
\ref{applicationenergygronwall}, negligible initial energies
$(E^k_{0,\varepsilon}(u_\varepsilon))_\varepsilon$ ($k\geq 1$)
translate to negligible energies
$(E^k_{\tau,\varepsilon}(u_\varepsilon))_\varepsilon$ ($k\geq 1$),
where $0\leq \tau\leq \gamma$, this is the statement of Proposition
\ref{energiesviainitialenergies}. Finally Proposition
\ref{finalprop} states that negligible energies
$(E^k_{\tau,\varepsilon}(u_\varepsilon))_\varepsilon$ ($k\geq
1,\,0\leq \tau\leq \gamma$) translate to a negligible bounds of $(u_\varepsilon)_\varepsilon$ and of its derivatives of all orders
on $\Omega_{\gamma}$. Due to the preceding introductory remark, estimates of the same kind hold on $\Omega_{-\gamma}$. Due to our considerations in the beginning of Part B (section \ref{partb}), we have therefore established
that $(u_\varepsilon)_\varepsilon$ is a negligible net on $V$. This proves uniqueness of the solution $u$ on $V$.
\end{proof}
\section[Dependence on the metric representative]{Dependence on the representative of the metric (Part G)}\label{independenceday}
So far, we have proved that on $V\subset\Omega_\gamma\cup\Omega_{-\gamma}$, a unique solution to
the initial value problem exists. We had, however, picked a
specific symmetric representative $(g_{ab}^\varepsilon)_\varepsilon$
of the metric $g_{ab}$ (to be more precise, these are coordinate
expressions of the metric components) and worked with one and the
same all the time. It is, therefore, advisable, to show that the
generalized solution $u$ of the wave equation is independent of the
choice of the representative of the metric. This is the aim of this
section.

There is only one further assumption we impose on the
representatives $(g_{ab}^\varepsilon)_\varepsilon$ of the metric:
they shall be symmetric (cf.\ the note in the end of the section).

The initial value problem with respect to
$(g_{ab}^\varepsilon)_\varepsilon$ is the following:
\begin{eqnarray}\label{weqhatless}
\Box^\varepsilon u_\varepsilon&=&f_\varepsilon\\\nonumber
u_\varepsilon(t=0,x^\alpha)&=&v_\varepsilon(x^\alpha)\\\nonumber
\partial_tu_\varepsilon(t=0,x^\alpha)&=&w_\varepsilon(x^\alpha)\nonumber
\end{eqnarray}
Now, let $(\hat g_{ab}^\varepsilon)_\varepsilon$ be another
symmetric representative of $g_{ab}$. We call $\hat
\Box^\varepsilon$ the \newline d'Alembertian operator induced by $(\hat
g_{ab}^\varepsilon)_\varepsilon$. The initial value problem with respect
to the latter reads quite similarly
\begin{eqnarray}\label{weqhat}
\hat \Box^\varepsilon \hat u_\varepsilon&=&f_\varepsilon\\\nonumber
\hat
u_\varepsilon(t=0,x^\alpha)&=&v_\varepsilon(x^\alpha)\\\nonumber
\partial_t\hat u_\varepsilon(t=0,x^\alpha)&=&w_\varepsilon(x^\alpha)\nonumber.
\end{eqnarray}

We may pause here for a moment and consider why Proposition \ref{lemma1} (and therefore all subsequent statements
based on the latter) also holds true for the alternative choice $(\hat g_{ab}^\varepsilon)_\varepsilon$ of metric representative:
First, the difference between $(\hat g_{ab}^\varepsilon)_\varepsilon$ and the static representative
$(g_{ab}^\varepsilon)_\varepsilon$ (according to Theorem \ref{Lemmastaticgeneralized}) is negligible by definition. As a consequence
the difference between estimates established on
compact sets and with respect to these different representative is negligible. Since we only work on the compact region $\Omega_\gamma$, the estimates according to
Proposition \ref{lemma1} hold as well for other (symmetric) representatives of the metric and
for small $\varepsilon$; however, presumably with modified positive constants $A, A', B_k, B_k'$.

The proof of Theorem \ref{THETHEOREM} (Part 1) provides moderate
solutions $(u_\varepsilon)_\varepsilon$ and $(\hat
u_\varepsilon)_\varepsilon$ of (\ref{weqhatless}) and
(\ref{weqhat}). It is only left to show that the difference
$(\widetilde
u_\varepsilon)_\varepsilon:=(u_\varepsilon)_\varepsilon-(\hat
u_\varepsilon)_\varepsilon$ is negligible on $\Omega_\tau$.
For this difference we have
\begin{eqnarray}\label{weqhatdiff}
\hat \Box^\varepsilon \widetilde
u_\varepsilon&=&f_\varepsilon-\hat\Box^\varepsilon
u_\varepsilon\\\nonumber \widetilde
u_\varepsilon(t=0,x^\alpha)&=&0\\\nonumber
\partial_t \widetilde u_\varepsilon(t=0,x^\alpha)&=&0\nonumber.
\end{eqnarray}
In view of the proof of Theorem \ref{THETHEOREM} (Part 2) we only
need to show that $f_\varepsilon-\hat\Box^\varepsilon u_\varepsilon$
is negligible. To this end we first manipulate the right hand side of
line 1 of (\ref{weqhatdiff}) as follows:
\begin{equation}\label{simplecalculationsareboringbutihavetodothem}
f_\varepsilon-\hat\Box^\varepsilon
u_\varepsilon=(f_\varepsilon-\Box^\varepsilon
u_\varepsilon)+(\Box^\varepsilon u_\varepsilon-\hat\Box^\varepsilon
u_\varepsilon)= \Box^\varepsilon u_\varepsilon-\hat\Box^\varepsilon
u_\varepsilon,
\end{equation} because $(u_\varepsilon)_\varepsilon$ solves
(\ref{weqhatless}). Therefore the problem is reduced to showing that
$(\Box^\varepsilon u_\varepsilon-\hat\Box^\varepsilon
u_\varepsilon)_\varepsilon$ is negligible. We calculate the
difference in local coordinates. We use $\vert \det
g_{ij}^\varepsilon\vert:=\vert g_\varepsilon\vert =-g_\varepsilon$
and for the sake of simplicity we further omit the index
$\varepsilon$. The difference then reads:
\begin{align} \label{diffmetr}
\Box u-\hat\Box u=(-g)^{-\frac{1}{2}}\partial_a(
(-g)^{\frac{1}{2}}g^{ab}\partial_b u)- (-\hat
g)^{-\frac{1}{2}}\partial_a( (-\hat g)^{\frac{1}{2}}\hat
g^{ab}\partial_b u)=\\\nonumber \left(
(-g)^{-\frac{1}{2}}\partial_a( (-g)^{\frac{1}{2}}g^{ab}\partial_b
u)-(-\hat
g)^{-\frac{1}{2}}\partial_a((-g)^{\frac{1}{2}}g^{ab}\partial_b
u)\right)+\\\nonumber+\left((-\hat
g)^{-\frac{1}{2}}\partial_a((-g)^{\frac{1}{2}}g^{ab}\partial_b
u)-(-\hat g)^{-\frac{1}{2}}\partial_a( (-\hat g)^{\frac{1}{2}}\hat
g^{ab}\partial_b u)\right)=\\\nonumber((-g)^{-\frac{1}{2}}-(-\hat
g)^{-\frac{1}{2}})\partial_a( (-g)^{\frac{1}{2}}g^{ab}\partial_b u)
+(-\hat
g)^{-\frac{1}{2}}\partial_a\left((-g)^{\frac{1}{2}}g^{ab}-(-\hat
g)^{\frac{1}{2}}\hat g^{ab}\right)\partial_b u
\end{align}
The differences within the brackets of the last line of
(\ref{diffmetr}) can easily be shown to be negligible. Indeed, since
$(g_{ab}^\varepsilon-\hat g_{ab}^\varepsilon)_\varepsilon$ is
negligible, also $g_\varepsilon-\hat g_\varepsilon$ is negligible,
therefore, as can be seen by the following elementary algebraic
manipulation, the difference
\begin{equation}\label{Plugyiing1}
(-g_\varepsilon)^{-\frac{1}{2}}-(-\hat
g_\varepsilon)^{-\frac{1}{2}}=\frac{
g_\varepsilon-\hat g_\varepsilon}{\sqrt{g_\varepsilon\hat
g_\varepsilon}(\sqrt{-\hat g_\varepsilon}+\sqrt{-g_\varepsilon})}
\end{equation}
is negligible. Also
\begin{equation}\label{Plugyiing2}
\sqrt{-g_\varepsilon}g^{ab}-\sqrt{-\hat g_\varepsilon} \hat
g^{ab}_\varepsilon=\sqrt {-g_\varepsilon} (g_\varepsilon^{ab}-\hat
g_\varepsilon^{ab})+\hat g_\varepsilon^{ab}\frac{\hat g_\varepsilon-
g_\varepsilon}{\sqrt {-g_\varepsilon}+\sqrt{-\hat g_\varepsilon}}
\end{equation}
is negligible. Plugging (\ref{Plugyiing1}) and (\ref{Plugyiing2})
into (\ref{diffmetr}), we derive that $(\Box^\varepsilon
u_\varepsilon-\hat\Box^\varepsilon u_\varepsilon)_\varepsilon$ is
negligible, and by identity
(\ref{simplecalculationsareboringbutihavetodothem}),
$(f_\varepsilon-\hat\Box^\varepsilon u_\varepsilon)_\varepsilon$ is
a negligible net of smooth functions as well. This is the right hand
side of the differential equation (\ref{weqhatdiff}). Therefore,
Part 2 of the proof of Theorem \ref{THETHEOREM}) ensures that
$(\widetilde u_\varepsilon)_\varepsilon=(u_\varepsilon-\hat
u_\varepsilon)_\varepsilon$ is negligible and we are done.

It goes without saying that non-symmetric perturbations of the
metric are not relevant. Another formulation of the latter would be
the following: The present method for solving the initial value
problem (\ref{weqofsetting}) basically lies in showing the existence
result on the level of representatives given an arbitrary choice of
representatives of the initial data as a well as a symmetric
representative of the metric. The resulting generalized solution
does not depend on the choice of symmetric representatives of the
metric and neither does it depend on the choice of representatives
of the initial data.
\section{Possible generalizations}
We finish this chapter by pointing out possible improvements of
Theorem \ref{THETHEOREM} concerning generality of the statement as well as reducing
the list of necessary assumption on the metric as given in section \ref{settingassumptions}.

First we conjecture that condition (\ref{setting6}) in section \ref{settingassumptions}, which guarantees existence of
smooth solutions on the level of representatives (that is with respect to each sufficiently small $\varepsilon$--component
of the representative of the metric), presumably follows from condition (\ref{setting2}).

Moreover, we believe that the Cauchy problem (\ref{weqofsetting}) also admits unique solutions in the special algebra of generalized functions
even if the condition (\ref{setting2}) are weakened to logarithmic growth properties of the metric coefficients. In this case,
the constants $A, A', B_k, B_k'$ of Proposition \ref{lemma1} might depend on $\varepsilon$, say $A(\varepsilon)=A \log(\varepsilon)$ with a positive constant $A$ etc.\ . Therefore, a later application of Grownwall's Lemma would yield
moderate growth of energies of arbitrary order, since
\[
(e^{A\log \varepsilon})_\varepsilon=(\varepsilon^A)_\varepsilon
\]
is moderate.

\chapter[Point values \& uniqueness questions]{Point values and uniqueness questions in algebras of generalized functions}\label{chapterpointvalues}
\section[Point values in Egorov algebras]{Point value characterizations of ultrametric Egorov
algebras}\label{chapteregorov}
As already mentioned in the introduction, a distinguishing feature (compared to spaces of distributions in the sense of
Schwartz) of Colombeau- and Egorov type algebras is the availability
of a generalized point value characterization for elements of such spaces (see \cite{MO1},
resp.\ \cite{KS} for the manifold setting). Such a characterization may be
viewed as a nonstandard aspect of the theory: for uniquely determining
an element of a Colombeau- or Egorov algebra, its values on classical ('standard')
points do not suffice: there exist elements which vanish on each classical
point yet are nonzero in the quotient algebra underlying the respective
construction. A unique determination can only be attained by taking into
account values on generalized points, themselves given as equivalence
classes of standard points. This characteristic feature is re-encountered
in practically all known variants of such algebras of generalized functions.

It therefore came as a surprise when in a series of papers (\cite{AKS2, AKS})
it was claimed that, contrary to the above general situation, in $p$-adic Colombeau-Egorov algebras
a general point value characterization using only standard points was available.
This chapter is dedicated to a thorough study of (generalized) point value characterizations
of $p$-adic Colombeau-Egorov algebras and to showing that in fact also in the $p$-adic
setting classical point values do not suffice to uniquely determine elements of such.

In the remainder of this section we recall some material from (\cite{AKS2, AKS}), using notation
from \cite{Bible}.
Let $\mathbb N$ be the natural numbers starting with $n=1$. For a fixed prime $p$, let $\mathbb Q_p$ denote the field of rational $p$-adic numbers.
Let $\mathcal D(\mathbb Q_p^n)$ denote the linear space of locally constant complex valued functions on
$\mathbb Q_p^n$ ($n\geq 1$) with compact support. Let further $\mathcal P(\mathbb Q_p^n)
:=\mathcal D(\mathbb Q_p^n)^{\mathbb N}$. $\mathcal P(\mathbb Q_p^n)$ is endowed with an algebra-structure by defining addition and multiplication of sequences component-wise. Let $\mathcal N(\mathbb Q_p^n)$ be the subalgebra of elements $\{(f_k)_k\}\in\mathcal P(\mathbb Q_p^n)$ such that for any compact set $K\subseteq \mathbb Q_p^n$ there exists an $N\in\mathbb N$ such that $\forall\;
x\in K\;\forall\; k\geq N: f_k(x)=0$. This is an ideal in $\mathcal P(\mathbb Q_p^n)$. The quotient algebra
$\mathcal G(\mathbb Q_p^n):=\mathcal P( \mathbb Q_p^n)/\mathcal N(\mathbb Q_p^n)$ is called the $p$-adic Colombeau-Egorov algebra.
Finally, so called Colombeau-Egorov generalized numbers $\widetilde {\mathcal C}$ are introduced in the following way:
Let $\bar {\mathbb C}$ be the one-point compactification of $\mathbb C\cup\{\infty\}$.\\ Factorizing $\mathcal A=\bar{\mathbb C}^{\mathbb N}$  by the ideal $\mathcal I:=\{u=(u_k)_k\in\mathcal A\mid \,\exists N\in\mathbb N\,\forall \;k\geq N: u_k=0\}$ yields then the ring $\widetilde{\mathcal C}$ of Colombeau-Egorov generalized numbers. We replace $\bar{\mathbb C}$ by $\mathbb C$ and construct similarly $\mathcal C$, the ring of generalized numbers: Clearly, $\bar{\mathbb C}$ is not needed in this context, since representatives of elements $f\in\mathcal G(\mathbb Q_p^n)$ merely take on values in $\mathbb C^{\mathbb N}$.
Let $f=[(f_k)_k]\in\mathcal G(\mathbb Q_p^n)$. It is clear that for a fixed $x\in\mathbb Q_p^n$,  the {\it point value of f at x}, $[(f_k(x))_k]$ is a well defined element of
${\mathcal C}$, i.e., we may consider $f$ as a map
\begin{equation}\label{pv}
f: \;\mathbb Q_p^n\rightarrow \mathcal C:\;\;x\mapsto f(x):=(f_k(x))_k+\mathcal I.
\end{equation}
Note that the above constitutes a slight abuse of notation:
The letter $f$ denotes both a generalized function (an element of $\mathcal G(\mathbb Q_p^n)$)
and a mapping on $\mathbb Q_p^n$. \\Finally, let $A$ be a set and let $R$ be a ring. For $B\subset A,\;\theta\in R$ we call the characteristic function of $B$ the map $\chi_{B,\theta}:\; A\rightarrow R$  which is identically $\theta$ on $B$
and which vanishes on $A\setminus B$. Furthermore, if $\theta=1\in R$ we simply write $\chi_B=\chi_{B,1}$.
\subsection[A counterexample]{Uniqueness via point values and a counterexample}
The following statement is proved in Theorem 4.4 of \cite{AKS}:
{\it Let $f\in\mathcal G(\mathbb Q_p^n)$, then:
\[
f=0 \;\;\mbox{in}\;\;\mathcal G(\mathbb Q_p^n)\Leftrightarrow\forall\; x\in\mathbb Q_p^n: f(x)=0\;\;\mbox{in}\;\; \mathcal C.
\]
}
However, inspired by (\cite{zAR}, p.\ 218) we construct the following counterexample to this claim,
which shows that point values cannot uniquely determine elements in $\mathcal G(\mathbb Q_p^n)$ uniquely.
For the sake of simplicity we assume that $n=1$.
\begin{example}\label{example}\rm
For any $l\in\mathbb N$, set
$$
B_l:=\{x\in\mathbb Z_p: \vert x-p^l\vert<\vert p^{2l}\vert\}\subset
\{x\in\mathbb Z_p: \vert x\vert=\vert p^l\vert\}\,.
$$
For any $i \in\mathbb N$, we set $f_i:=\chi_{B_i}$. Clearly $B_i\cap B_j=\emptyset$ whenever $i\neq j$ and since $f_i\in\mathcal D(\mathbb Q_p)$ for all natural numbers $i$,
$(f_i)_i$ is a representative of some $f\in\mathcal G(\mathbb Q_p)$. Now, for any $\alpha\in \mathbb Q_p$, $f(\alpha)=0$ in $\mathcal C$, since either $\alpha\in B_i$ for some $i\in\mathbb N$ (which implies that $f_j(\alpha)=0\;\forall\; j>i$)   or $\alpha\in\mathbb Q_p\setminus \bigcup B_i$, where each $f_i$ ($i\in\mathbb N$) is identically zero. Consider now the sequence $(\beta_i)_{i\geq1}\in\mathbb N^{\mathbb N}\subseteq \mathbb Z_p^{\mathbb N}$, where $\beta_i=p^i\;\forall\;i\in\mathbb N$. It follows that $f_i(\beta_i)=1\;\forall\; i\in\mathbb N$. In particular, for $K=\mathbb Z_p$ or any dressed ball containing $0$, there is no representative $(g_j)_j$ of $f$ such that  for some $N>0$, $g_j=0\; \forall\; j\geq N$. Hence $f\neq 0$ in $\mathcal G(\mathbb Q_p)$ although all standard point values of $f$ vanish.
\end{example}
\begin{remark}
By means of the above example we may analyze the proof of Theorem 4.4 in \cite{AKS}. Let $f$ be the generalized function from \ref{example}. As a compact set choose $K:=B_{\leq p^{-2}}(0)=p^2\mathbb Z_p$. For the representative $(f_k)_k$ constructed in \ref{example} and $x=0$ we have $N(0)=1$, which in the notation of \cite{AKS} means that for any $k\geq 1=N(0)$, $f_k(0)=0$. Also, recall that $B_{\gamma}(a)$ is the dressed ball
$B_{\leq p^{\gamma}}(a)$. The ``parameter of constancy'' (\cite{AKS}, p.\ 6) of $f_1$ at $x=0$, which is the maximal $\gamma$ such that $f_1$ is identically zero on $B_{\gamma}(0)$, is $l_0(0)=-2$. Now, there exists a covering of $K$ consisting of a single set, namely $B_{l_0(0)}(0)$. Thus we may replace the application of the Heine-Borel Lemma in \cite{AKS} by our singleton-covering. But then the claim that (4.1) and (4.2) imply that for all $k\geq N(0)=1$ we have $f_k(0+x')=f_k(0)=0\;\forall\, x'\in K$ does not hold. This indeed follows from the definition of the sequence $(f_k)_k$ of locally constant functions from above, since for any $k\in\mathbb N$ we have $f_k(p^k)=1$.
\end{remark}
\subsection[Ultrametric Egorov algebras]{Egorov algebras on locally compact ultrametric spaces}
In this section we investigate the problem of point value characterization in Egorov algebras in full generality: to this end we consider a
general locally compact ultrametric space $(M,d)$ instead of $\mathbb Q_p^n$, where $M$ need not have a field structure. Our aim is to show that even in such a general setting, the respective algebra cannot have a point value characterization, unless $M$ carries the discrete topology. Denote by $\mathcal E_d(M)$ the algebra of sequences of locally constant functions with compact support, taking values in a commutative ring $R\neq\{0\}$. Let $\mathcal N_d(M)$ be the set of negligible functions $\{(f_k)_k\}\in\mathcal E_d(M)$
such that for any compact set $K\subset M$ there exists an $N\in\mathbb N$ such that $\forall \;x\in K\;\forall\; k\geq N: f_k(x)=0$. The subset $\mathcal N_d(M)$ is an ideal in $\mathcal E_d(M)$ and the quotient algebra $\mathcal G(M, R):=\mathcal E_d(M)/\mathcal N_d(M)$ is called the ultrametric Egorov algebra associated with $(M,d)$. Furthermore, the ring of generalized numbers is defined by $\mathcal R:=R^{\mathbb N}/\sim$, where $\sim ~$ is the equivalence relation on $R^{\mathbb N}$ given by
\[
u\sim v\;\mbox{in}\;R^{\mathbb N}\Leftrightarrow \exists\,N\in\mathbb N\;\forall\;k\geq N: u_k-v_k=0.
\]
We call $\mathcal I(R):=\{w\in R^{\mathbb N}:w\sim 0\}$ the ideal of negligible sequences in $R$. Analogous to (\ref{pv}), for $f\in\mathcal G(M,R)$ evaluation on standard point values is introduced by means of the mapping:
\begin{equation}\label{pv1}
f: \;M\rightarrow \mathcal R:\;\;x\mapsto f(x):=(f_k(x))_k+\mathcal I(R).
\end{equation}
\begin{definition}
An ultrametric Egorov algebra $\mathcal G(M,R)$ is said to admit a standard point value characterization if for each $u\in\mathcal G(M,R)$ we have
\[
u=0\Leftrightarrow \forall\;x\in M: u(x)=0\;\mbox{in}\;\mathcal R.
\]
\end{definition}
Using this terminology, Example \ref{example} shows that $\mathcal G(\mathbb Q_p^n)$ does not admit a standard point value characterization. The main result of this section is the following generalization:
\begin{theorem}
Let $(M,d)$ be a locally compact ultrametric space and let $R\neq \{0\}$. Then $\mathcal G(M, R)$ does not admit a standard point value characterization unless $(M,d)$ is discrete.
\end{theorem}
\begin{proof}
The result follows by generalizing the construction of Example \ref{example}. Assume $(M,d)$ is not discrete, then there exists a point $x\in M$ and a sequence $(x_n)_n$ of distinct points in $M$ converging to $x$. We may assume that $d(x,x_i)>d(x,x_j)$ whenever $i<j$. Define stripped balls $(B_n)_{n\geq 1}$ with centers $(x_n)_{n \geq 1} $ by $B_n:=\{y\in M\mid d(x_n,y)<\frac{d(x_n,x)}{2}\}$. Due to the ultrametric property ``the strongest one wins'' we have $B_n\subset \{z\mid d(x,z)=d(x_n,x)\}$, which further implies that for all $i\neq j\; (i,\,j)\in\mathbb N$, the balls $B_i$, $B_j$ are disjoint sets in $M$. Since $R$ is a non-trivial ring, we may choose some $\theta\in R\setminus\{0\}$. Now we define a sequence $(f_k)_k$ of locally constant functions in the following way: For any $i\geq 1$ set $f_i=\chi_{B_i,\theta}$. Clearly, $f:=[(f_i)_i]\in\mathcal G(M, R)$, and similarly to Example \ref{example}, for any $\alpha\in M$, $f(\alpha)=0$  in $\mathcal R$. Nevertheless for the sequence $(x_n)_n$, which without loss of generality may be assumed to lie in a compact neighborhood of $x$, one has $f_i(x_i)=\theta\; \forall\; i\geq 1$ which implies that $f\neq 0$ in $\mathcal G(M,R)$.
\end{proof}
Recall that a discrete topological space $X$ has the following properties:
\begin{enumerate}
\item $X$ is locally compact.
\item Any compact set in $X$ contains finitely many points only.
\end{enumerate}
Therefore we know that for a set $D$ endowed with the discrete metric and for any commutative ring $R$, the respective ultrametric Egorov algebra $\mathcal G(D, R)$ admits a point wise characterization. We therefore conclude:
\begin{corollary}
For a locally compact ultrametric space $(M,d)$ and a non-trivial ring $R$, the following statements are equivalent:
\begin{enumerate}
\item $\mathcal G(M, R)$ admits a standard point value characterization.
\item The topology of $(M,d)$ is discrete.
\end{enumerate}
\end{corollary}
\subsection{Generalized point values}
In this section we give an appropriate generalized point value characterization in the style of (\cite{MO1}, pp.\ Theorem 2.\ 4) of
$\mathcal G(M, R)$, where $M$ is endowed with a non-discrete ultrametric $d$ for which $M$ is locally compact, and $R\neq \{0\}$. First, we have to introduce a set $\widetilde M_c$ of compactly supported generalized points over $M$. Let $\mathcal E=M^{\mathbb N}$, the ring of sequences in $M$, and identify
two sequences, if for some index $N\in\mathbb N$ one has $d(x_n,y_n)=0$ for each positive integer $n$, that is, $x_n=y_n\; \forall\; n\geq N$; we write $x\sim y$. We call $\widetilde M=\mathcal E/\sim$ the ring of generalized numbers. Finally, $\widetilde M_c$ is the subset
of such elements $x\in \widetilde M$ for which there exists a compact subset $K$  and some representative $(x_n)_n$ of $x$ such that for some $N>0$ we have $x_n\in K$ for all $n\geq N$. It follows that evaluating a function $u\in\mathcal G(M,R)$ at a compactly supported generalized point $x$ is possible, i.e., for representatives $(x_k)_k$, $(u_k)_k$ of $x$ resp.\  $u$, $[(u_k(x_k))_k]$ is a well defined element
of $\mathcal R$.
\begin{proposition}
In $\mathcal G(M,R)$, there is a generalized point value characterization, i.e.,
\[
u=0 \;\mbox{in}\;\mathcal G(M,R)\;\Leftrightarrow\; \forall\; x\in \widetilde M_c: u(x)=0\;\; \mbox{in}\;\; \mathcal R.
\]
\end{proposition}
\begin{proof}
The condition on the right side obviously is necessary. Conversely, let $u\in\mathcal G(M,R)$, $u\neq 0$. This means that there is a representative $(u_k)_k$ of $u$ and a compact set $K\subset\subset M$ such that $u_k$ does not vanish on $K$ for infinitely many $k\in\mathbb N$. In particular this means we have a sequence $(x_k)_k$ in $K$ such that for infinitely many $k\in\mathbb N$, $u_k(x_k)\neq 0$. Clearly this means that $u(x)\neq 0$ in $\mathcal R$, where we have set $x:=[(x_k)_k]$.
\end{proof}
\subsection{The $\delta$-distribution}
In \cite{AKS},  Theorem 4.4 is illustrated by some examples, to highlight the advantage of a point value concept in $\mathcal G(\mathbb Q_p^n)$. In this section we discuss the $\delta$-distribution (Example 4.5 on p.\ 12 in \cite{AKS}) and construct a generalized function $f\in\mathcal G(\mathbb Q_p)$ different from $\delta$ which however coincides with $\delta$ on all standard points in $\mathbb Q_p$. We first embed the $\delta$-distribution in $\mathcal G(\mathbb Q_p)$ as in \cite{AKS} (p.\ 9, Theorem 3.3) which yields $\iota(\delta)=(\delta_k)_k+\mathcal N_p(\mathbb Q_p)$, where $\delta_k(x):=p^k\Omega(p^k\vert x\vert_p)$ for each $k$, and
$\Omega$ is the bump function on $\mathbb R^+_0$ given by
\[
\Omega(t):=\begin{cases} 1,\;\;0\leq t\leq 1\\ 0,\;\;t>1\end{cases}.
\]
Evaluation of $\iota(\delta)$ on standard points is shown in Example 4.5 of \cite{AKS}. With $\tilde c:=(p^k)_k+\mathcal I\in\mathcal C$
one has:
\[
\iota(\delta)(x)=\begin{cases}\widetilde c,\;\;x=0\\ 0,\;\;x\neq 0\end{cases}\qquad (x\in\mathbb Q_p).
\]
Let $\varphi: \mathbb N\rightarrow \mathbb Z$ be a monotonous function such that $\lim_{k\rightarrow\infty}\varphi(k)=\infty$, and
such that the cardinality of $U_{\varphi}:= \{k:\varphi(k)>k\}$ is infinite. Consider an element $f\in\mathcal G(\mathbb Q_p)$ given by $f:=(f_k)_k+\mathcal N(\mathbb Q_p)$ where for any
$k\geq 1$, $f_k(x):=p^k\Omega(p^{\varphi(k)}\vert x\vert_p)$. Then the standard point values of $\iota(\delta)$ and $f$ coincide. Furthermore, they coincide on compactly supported generalized points $x\in\widetilde{\mathbb Q}_{p,c}$ with the property that for any representative $(x_k)_k$ of $x$ there exists an $N\in\mathbb N$ such that $\forall\;k\geq N:\;\vert x_k\vert_p>p^{-\min\{k,\varphi(k)\}}$, since in this case we have
$\delta_k(x_k)=f_k(x_k)=0$. However, there are compactly supported generalized points
violating this condition which yield different generalized point values of $\iota(\delta)$ resp. $f$: for instance, take the generalized point $x_0:=[(p^k)_k]\in\widetilde{\mathbb Q}_{p,c}$. Then $f(x_0)\neq \widetilde c$, since $\theta_k:=f_k(x_k)=0$ for any $k\in U_{\varphi}$ and thus $\theta_k=0$ for infinitely many $k\in\mathbb N$. But $\iota (\delta)(x_0)=\widetilde c$.
\section[Spherical completeness]{Spherical completeness of the ring of generalized numbers}\label{sectionsharp}
Let $(M,d)$ be an ultrametric space. For given $x\in M, r\in \mathbb
R^+$, we call $B_{\leq r}(x):=\{y\in M\mid d(x,y)\leq r\}$ the
dressed ball with center $x$ and radius $r$. Throughout $\mathbb N:=
\{1,2,\dots\}$ denote the {\it positive} integers. Let $(x_i)_i\in
M^{\mathbb N}$ and $(r_i)_i$ be a sequence of positive reals. We
call $(B_i)_i, \;B_i:=B_{\leq r_i}(x_i)\;(i\geq 1)$ a nested
sequence of dressed balls, if $r_1\geq r_2\geq r_3\dots$ and
$B_1\supseteq B_2\supseteq\dots$ . Following standard ultrametric
literature (cf.\ \cite{zAR}), nested sequences of dressed balls might have
empty intersection. The converse property is defined as follows:
\begin{definition}\label{defsph}
$(M,d)$ is called spherically complete, if every nested sequence of
dressed balls has a non-empty intersection.
\end{definition}
It is evident that any spherically complete ultrametric space is
complete with respect to the topology induced by its metric (using
the well known fact that topological completeness of $(M,d)$ is
equivalent to the property of Definition \ref{defsph} with radii
$r_i \searrow 0$) . However, there are popular non-trivial examples
in the literature, for which the converse is not true. As an example
we mention the field of complex $p$-adic numbers together with its
$p$-adic valuation considered as the completion of the algebraic
closure of the field $\mathbb C_p$ of rational $p$-adic numbers. Due
to Krasner, this field has nice algebraic properties (as it is
algebraically closed, and even isomorphic to the complex numbers
cf.\ \cite{zAR}, pp.\ 134--145), but it also has been shown, that
$\mathbb C_p$ is not spherically complete. This is mainly due to the
fact that the complex $p$-adic numbers are a separable, complete
ultrametric space with dense valuation (cf.\ \cite{zAR}, pp.\
143--144). However, for an ultrametric field $K$, spherical
completeness is necessary in order to ensure $K$ has the Hahn Banach
extension property (to which we refer as HBEP), that is, any normed
$K$-vector space $E$ admits continuous linear functionals previously
defined on a strict subspace $V$ of $E$ to be extended to the whole
space under conservation of their norm (cf. W. Ingleton's proof
\cite{Ingleton}). Since spherical completeness fails, it is natural
to ask if the $p$-adic numbers could at least be spherically
completed, i.e., if there existed a spherically complete ultrametric
field $\Omega$ into which $\mathbb C_p$ can be embedded. This
question has a positive answer (cf.\ \cite{zAR}). The necessity of
spherical completeness for the HBEP of $K=\mathbb C_p$ is evident:
even the identity map
\[
\varphi:\;\; \mathbb C_p\rightarrow\mathbb C_p,\quad \varphi(x):=x
\]
cannot be extended to a functional $\psi:\Omega\rightarrow\mathbb
C_p$ under conservation of its norm $\|\varphi\|=1$
 (here we consider $\Omega$ as a $\mathbb C_p$- vector space).\footnote{To check this, let $B_i:=B_{\leq r_i}(x_i)$ be a nested sequence of dressed balls in $\mathbb C_p$ with empty intersection. Then $\hat B_i:=B_{\leq r_i}(x_i)\subseteq\Omega$ have nonempty intersection, say $\Omega\ni\alpha\in\bigcap_{i=1}^{\infty}\hat B_i$. Assume further, the identity $\varphi$ on $\mathbb C_p$ can be extended to some linear map $\psi:\Omega\rightarrow\mathbb C_p$ under conservation of its norm. Then
\[
\vert \psi(\alpha)-x_i\vert_{\Omega}=\vert
\psi(\alpha)-\phi(x_i)\vert_{\Omega}\leq
\|\psi\|\vert\alpha-x_i\vert_{\mathbb
C_p}=\vert\alpha-x_i\vert_{\mathbb C_p},
\]
therefore $\psi(\alpha)\in \bigcap_{i=1}^{\infty}B_i$ which is a
contradiction and we are done.}

The present work is motivated by the question if some version of
Hahn-Banach's Theorem holds on differential algebras in the sense of
Colombeau considered as ultra pseudo normed modules over the ring of generalized
numbers  $\widetilde{\mathbb R}$ (resp.\ $\widetilde{\mathbb C}$).
Even though topological questions on topological $\widetilde{\mathbb C}$
modules have been recently investigated to a wide extent (cf. C. Garetto's recent
papers \cite{Garetto2,Garetto1} as well as \cite{DHPV}), a HBEP has not yet been established in the
literature.

The analogy with the $p$-adic case lies at hand, since the ring of generalized numbers can naturally
be endowed with an ultrametric pseudo-norm.
However, the presence of zero-divisor in $\widetilde{\mathbb R}$ as
well as the failing multiplicativity of the pseudo-norm turns the
question into a non-trivial one and Ingleton's ultrametric version
of the Hahn Banach Theorem cannot be carried over to our setting
unrestrictedly.

On our first step tackling this question we discuss spherical
completeness of the ring of generalized numbers endowed with the
given ultrametric (induced by the respective pseudo-norm, cf.\ the preliminary
section).

$\widetilde{\mathbb R}$ first was introduced as the set of values of
generalized functions at standard points; however, a subring
consisting of compactly supported generalized numbers turned out to
be the set of points for which evaluation determines uniqueness,
whereas standard points do not suffice do determine generalized
functions uniquely (cf. \cite{Eyb4, MO1}). A hint, that
$\widetilde{\mathbb R}$ (or $\widetilde{\mathbb C}$ as well), the
ring of generalized real (or complex) numbers is spherically
complete, is, that contrary to the above outlined situation on
$\mathbb C_p$, the generalized numbers endowed with the topology
induced by the sharp ultra-pseudo norm are not separable. This, for
instance, follows from the fact that the restriction of the sharp
valuation to the real (or complex) numbers is discrete.

Having motivated our work by now, we may formulate the aim of this section,
which is to prove the following:
\begin{theorem}\label{theoremsph}
The ring of generalized numbers is spherically complete.
\end{theorem}
We therefore have an independent proof of the fact (cf.\
\cite{Garetto1}, Proposition 1.\ 30):
\begin{corollary}
The ring of generalized numbers is topologically complete.
\end{corollary}
In the last section of this note we present a modified version of
Hahn-Banach's Theorem which bases on spherically completeness of
$\widetilde{\mathbb R}$ (resp.\ $\widetilde{\mathbb C}$). Finally, a remark on
the applicability of the ultra metric version of Banach fixed point theorem can be found in the Appendix.
\subsection{Preliminaries}
In what follows we repeat the definitions of the ring of (real or complex) generalized numbers along with its non-archimedean valuation function. The material is taken from different sources; as references we may recommend the recent works due to C.\ Garetto (\cite{Garetto2,Garetto1}) and A.\ Delcroix et al.\ (\cite{DHPV}) as well as one of the original sources of this topic due to D.\ Scarpalezos (cf.\ \cite{Scarpalezos1}).\\
Let $I:=(0,1]\subseteq \mathbb R$, and let $\mathbb K$ denote
$\mathbb R$ resp.\ $\mathbb C$. The ring of generalized numbers over
$\mathbb K$ is constructed in the following way: Given the ring of
moderate (nets of) numbers
\[
\mathcal E:=\{(x_{\varepsilon})_{\varepsilon}\in\mathbb K^I\mid
\exists\; m:\vert
x_{\varepsilon}\vert=O(\varepsilon^m)\;(\varepsilon\rightarrow 0)\}
\]
and, similarly, the ideal of negligible nets in $\mathcal E(\mathbb
K)$ which are of the form
\[
\mathcal N:=\{(x_{\varepsilon})_{\varepsilon}\in\mathbb K^I\mid
\forall\; m:\vert
x_{\varepsilon}\vert=O(\varepsilon^m)\;(\varepsilon\rightarrow 0)\},
\]
we may define the generalized numbers as the factor ring
$\widetilde{\mathbb K}:=\mathcal E_M/\mathcal N$. We define a (real
valued) valuation function $\nu:$ on $\mathcal E_M(\mathbb K)$ in
the following way:
\[
\nu((u_{\varepsilon})_{\varepsilon}):=\sup\,\{b\in\mathbb R\mid
\vert
u_{\varepsilon}\vert=O(\varepsilon^b)\;\;(\varepsilon\rightarrow
0)\}.
\]
This valuation can be carried over to the ring of generalized
numbers in a well defined way, since for two representatives of a
generalized number, the valuations above coincide (cf.\
\cite{Garetto1}, section 1). We then may endow $\widetilde{\mathbb
K}$ with an ultra-pseudo-norm ('pseudo' refers to
non-multiplicativity) $\vert \;\; \vert_e$ in the following way:
$\vert 0\vert_e:=0$, and whenever $x\neq 0$, $\vert
x\vert_e:=e^{-\nu(x)}$. With the metric $d_e$ induced by the above
norm, $\widetilde{\mathbb K}$ turns out to be a non-discrete
ultrametric space, with the following topological properties:
\begin{enumerate}
\item $(\widetilde{\mathbb K},d_e)$ is topologically complete (cf. \cite{Garetto1}),
\item $(\widetilde{\mathbb K},d_e)$ is not separable, since the restriction of $d_e$ onto $\mathbb K$ is discrete.
\end{enumerate}
The latter property holds, since on metric spaces second
countability and separability are equivalent and the well known fact
that the property of second countability is inherited by subspaces
(whereas separability is not in general).

In order to avoid confusion we henceforth denote closed balls in
$\mathbb K$ by $B_{\leq r}(x)$ in distinction with dressed balls in
$\widetilde{\mathbb K}$ which we denote by $\widetilde B_{\leq
r}(x)$. Similarly stripped balls and the sphere in the ring of
generalized numbers are denoted by $\widetilde B_{< r}(x)$ resp.\
$\widetilde S_r(x)$. \subsection{Euclidean models of sharp
neighborhoods} Throughout, a net of real numbers
$(C_\varepsilon)_\varepsilon$ is said to {\it increase monotonously
with} $\varepsilon\rightarrow 0$, if the following holds:
\[
\forall \eta,\eta'\in I:\;(\eta\leq\eta'\Rightarrow C_\eta\geq
C_{\eta'}).
\]
To begin with we formulate the following condition:\\
{\it Condition (E).}\\
A net $(C_\varepsilon)_\varepsilon$ of real numbers is said to
satisfy condition (E), if it is
\begin{enumerate}
\item positive for each $\varepsilon$ and
\item monotonically increasing with $\varepsilon\rightarrow 0$, and finally, if
\item the sharp norm is $\vert(C_\varepsilon)_\varepsilon\vert_e=1$.
\end{enumerate}
Next, we introduce the notion of euclidean models of sharp
neighborhoods of generalized points:
\begin{definition}\rm
Let $x\in\widetilde{\mathbb K}$, $\rho\in\mathbb R,\;
r:=\exp(-\rho)$. Let further
$(C_{\varepsilon})_{\varepsilon}\in\mathbb R^I$ be a net of real
numbers satisfying condition (E) and let
$(x_{\varepsilon})_{\varepsilon}$ be a representative of $x$. Then
we call the net of closed balls
$(B_{\varepsilon})_{\varepsilon}\subseteq \mathbb K^I$ given by
\[
B_{\varepsilon}:=B_{\leq
C_{\varepsilon}\varepsilon^{\rho}}(x_{\varepsilon})
\]
for each $\varepsilon\in I$ an euclidean model of $\widetilde
B(x,r)$.
\end{definition}
Note, that every dressed ball admits an euclidean model: let
$(x_\varepsilon)_\varepsilon$ be a representative of $x$ and define
$(C_\varepsilon)_\varepsilon$ by $C_{\varepsilon}:=1$ for each
$\varepsilon\in I$; then $B_{\leq
C_{\varepsilon}\varepsilon^{\rho}}(x_{\varepsilon})$ yields
determines an euclidean model of $\widetilde B_{\leq r}(x)$.\\We
need to mention that whenever we write
$(B^{(1)}_{\varepsilon})_{\varepsilon}\subseteq
(B^{(2)}_{\varepsilon})_{\varepsilon}$, we mean the inclusion
relation $\subseteq$ holds component wise (that is for each
$\varepsilon\in I$), and we say
$(B^{(2)}_{\varepsilon})_{\varepsilon}$ contains
$(B^{(1)}_{\varepsilon})_{\varepsilon}$.\\ The following lemma is
basic; however, in order to get familiar with the concept of
euclidean neighborhoods, we include a detailed proof:
\begin{lemma}\label{capture}
For $x\in\widetilde{\mathbb K}, r>0$, let
$(B_{\varepsilon})_{\varepsilon}$ be an euclidean model for
$\widetilde B_{\leq r}(x)$ and set$\rho=-\log r$. Then we have:
\begin{enumerate}
\item \label{capture1} Any $y\in \widetilde B_{<r}(x)$ has a representative $(y_{\varepsilon})_{\varepsilon}$ such that $y_{\varepsilon}\in B_{\varepsilon}$ for each $\varepsilon \in I$.
\item \label{capture2} There exist $y\in \widetilde S_r(x):=\{x'\in\widetilde{\mathbb K}:\vert x'-x\vert=r\}$ which cannot be caught by representatives lying in $(B_{\varepsilon})_{\varepsilon}$. However
one may blow $(B_{\varepsilon})_{\varepsilon}$ always up to a new
model $(\hat B_{\varepsilon})_{\varepsilon}$ which contains some
representative of $y$, i.e., there exists a net
$(D_{\varepsilon})_{\varepsilon}$ satisfying Condition (E) such that
for $\hat C_{\varepsilon}:=C_{\varepsilon}D_{\varepsilon}$, $\hat
B_{\varepsilon}:=B_{\leq \hat
C_{\varepsilon}\varepsilon^{\rho}}(x_{\varepsilon})$ yields a model
containing some representative of $y$
\item \label{capture3} In any case, it can be arranged, that $d(\partial \hat B_{\varepsilon},y_{\varepsilon})\geq \frac{C_{\varepsilon}}{2}\varepsilon^{\rho}$ for each $\varepsilon\in I$, for some model  $(\hat B_{\varepsilon})_{\varepsilon}$ of $\widetilde B_{\leq r}(x)$ containing $(B_{\varepsilon})_{\varepsilon}$.
\end{enumerate}
\end{lemma}
\begin{proof}
(\ref{capture1}): By definition of the sharp norm, $\vert
y-x\vert_e<r$ is equivalent to the situation, that for each
representative $(y_{\varepsilon})_{\varepsilon}$ of $y$ and for each
representative $(x_{\varepsilon})_{\varepsilon}$ of $x$, we have
\[
\sup\{b\in\mathbb R\mid \vert y_{\varepsilon}-
x_{\varepsilon}\vert=O(\varepsilon^b) (\varepsilon\rightarrow
0)\}>\rho,
\]
and this implies that there exists some $\rho'>\rho$ such that for
any representative $(y_{\varepsilon})_{\varepsilon}$ of $y$ and any
representative $(x_{\varepsilon})_{\varepsilon}$ of $x$ we have
\[
\vert
y_{\varepsilon}-x_{\varepsilon}\vert=o({\varepsilon}^{\rho'}),\quad
\varepsilon\rightarrow 0.
\]
This further implies that for any choice of representatives of $x$
resp.\ of $y$, there exists some $\eta\in I$ with
\begin{equation}\label{tinyradius}
\vert y_{\varepsilon}-x_{\varepsilon}\vert\leq \varepsilon^{\rho'}
\end{equation}
for each  $\varepsilon<\eta$. Since $C_{\varepsilon}>0$ for each $\varepsilon\in I$ and $C_\varepsilon$ is monotonously increasing with $\varepsilon\rightarrow0$, we have $\varepsilon^{\rho'}\leq C_{\varepsilon}\varepsilon^{\rho}$ for sufficiently small $\varepsilon$, therefore, a suitable 
choice of $y_{\varepsilon}$, for $\varepsilon\geq \eta$, yields the
first claim (for instance, one may set
$y_{\varepsilon}:=x_{\varepsilon}$ whenever $\varepsilon\geq \eta$).\\
We go on by proving (\ref{capture2}): For the first part, set
\[
y_{\varepsilon}:=2C_{\varepsilon}\varepsilon^{\rho}+x_{\varepsilon}
\]
Let $y$ denote the class of $(y_{\varepsilon})_{\varepsilon}$. It is
evident, that $y\in \widetilde B_{\leq r}(x)$. However,
$(y_{\varepsilon})\notin B_{\varepsilon}$ for each $\varepsilon\in
I$. Indeed,
\[
\forall\;\varepsilon\in I:\vert
y_{\varepsilon}-x_{\varepsilon}\vert= 2
C_{\varepsilon}\varepsilon^{\rho}>C_{\varepsilon}\varepsilon^{\rho},
\]
since $C_\varepsilon>0$ for each $\varepsilon$. We further show,
that the same holds for any representative $(\bar
y_{\varepsilon})_{\varepsilon}$ of $y$ for sufficiently small index
$\varepsilon$. Indeed, the difference of two representatives being
negligible implies that for any $N>0$ we have
\[
y_{\varepsilon}-\hat y_{\varepsilon}=o(\varepsilon^N)\;\;
(\varepsilon\rightarrow 0).
\]
Therefore, for $N>\rho$ and sufficiently small $\varepsilon$, we
have:
\[
\vert \hat y_{\varepsilon}-y_{\varepsilon}\vert\geq \vert \vert\hat
y_{\varepsilon}-y_{\varepsilon}\vert - \vert
y_{\varepsilon}-x_{\varepsilon}\vert\vert\geq
2C_{\varepsilon}\varepsilon^{\rho}-\varepsilon^N\geq
\frac{3}{2}C_{\varepsilon}\varepsilon^{\rho}>C_{\varepsilon}\varepsilon^{\rho}.
\]
Therefore we have shown the first part of (\ref{capture2}). Let
$y\in \widetilde S_r(x)$. We demonstrate how to blow up
$(B_{\varepsilon})_{\varepsilon}$ to catch some fixed representative
$(y_{\varepsilon})_{\varepsilon}$ of $y$. Since $\vert
y-x\vert=e^{-\rho}=r$, there is a net $C'_{\varepsilon}\geq 0$
($\vert (C'_{\varepsilon})_{\varepsilon}\vert_e=1$) such that
\[
\forall\varepsilon\in I:\;\vert
y_{\varepsilon}-x_{\varepsilon}\vert=C_{\varepsilon}'\varepsilon^{\rho}
\]
Set $C''_{\varepsilon}=\max_{\eta\geq\varepsilon}\{1,C'_{\eta}\}$.
This ensures that $(C''_{\varepsilon})$ is a monotonously increasing
with $\varepsilon\rightarrow 0$, above $1$ for each $\varepsilon\in
I$, and $\vert (C''_{\varepsilon})\vert_e=1$ is preserved. Define
$B'_{\varepsilon}:=B_{\leq
C_{\varepsilon}C''_{\varepsilon}\varepsilon^{\rho}}(x_{\varepsilon})$.
Then $(B'_{\varepsilon})_{\varepsilon}$ is a new model for
$\widetilde B_{\leq r}(x)$ containing the old model and
$(y_{\varepsilon})_{\varepsilon}$ as well, since the product
$C_{\varepsilon}C''_{\varepsilon}$ has the required properties, and
\[
\vert y_{\varepsilon}-x_{\varepsilon}\vert\leq
C''_{\varepsilon}\varepsilon^{\rho}\leq
C_{\varepsilon}''C_{\varepsilon}\varepsilon^{\rho}
\]
and we are done with (\ref{capture2}). \\ Proof of (\ref{capture3}):
So far, we have shown that for each $y\in \widetilde B_{\leq r}(x)$,
there exists an euclidean model  $(B_{\leq
C_{\varepsilon}\varepsilon^{\rho}}(x_{\varepsilon}))$ of $B_{\leq
r}(x)$ such that for some representative
$(y_{\varepsilon})_{\varepsilon}$ of $y\in \widetilde B_{\leq r}(x)$
we have
\[
\forall\;\varepsilon\in I:y_{\varepsilon}\in B_{\varepsilon}.
\]
Therefore, by replacing $C_{\varepsilon}$ by $2C_{\varepsilon}$
above, again a model for $\widetilde B_{\leq r}(x)$ is achieved,
however with the further property that $\vert
y_{\varepsilon}-x_{\varepsilon}\vert\leq
C_{\varepsilon}/2\varepsilon^\rho$ for each $\varepsilon\in I$ which
proves our claim.\end{proof} Before going on by establishing the
crucial statement which will allow us to translate decreasing
sequences of closed balls in the given ultrametric space
$\widetilde{\mathbb K}$ to decreasing sequences of their
(appropriately chosen) euclidean models, we introduce a useful term:
\begin{definition}\rm
Suppose, we have a nested sequence $(\widetilde B_i)_{i=1}^{\infty}$
of closed balls with centers $x_i$  and radius $r_i$ in
$\widetilde{\mathbb K}$ and for each $i\in\mathbb N$
we have an euclidean model $(B^{(i)}_{\varepsilon})_{\varepsilon}$. We say, this associated sequence of euclidean models is proper, 
if
$\left((B^{(i)}_{\varepsilon})_{\varepsilon}\right)_{i=1}^{\infty}$
is nested as well, that is, if we have:
\[
(B^{(1)}_{\varepsilon})_{\varepsilon}\supseteq
(B^{(2)}_{\varepsilon})_{\varepsilon}\supseteq(B^{(3)}_{\varepsilon})_{\varepsilon}\supseteq
\dots.
\]
\end{definition}
\subsection[The main theorem]{Proof of the main theorem} In order to prove the main
statement, we proceed by establishing two important preliminary
statements. First, a remark on the notation in the sequel: If
$(x_i)_i$, a sequence of points in the ring of generalized numbers,
is considered, then $(x_\varepsilon^{(i)})_\varepsilon$ denote
(certain) representatives of the $x_i$'s. Furthermore, for
subsequent choices of nets of real numbers
$(C^{(i)}_\varepsilon)_\varepsilon$, and positive radii $r_i$, we
denote by $\rho_i$ the negative logarithms of the $r_i$'s
($i=1,2,\dots,$) and the euclidean models of the balls $\widetilde
B_{\leq r_i}(x_i)$ with radii $r_\varepsilon
^{i}:=C_\varepsilon^{(i)}\varepsilon^{\rho_i}$ to be constructed are
denoted by
\[
B_\varepsilon^{(i)}:=B_{\leq
r_\varepsilon^{(i)}}(x_\varepsilon^{(i)}).
\]
We start with the fundamental proposition:
\begin{proposition}\label{lemmatique}
Let $x_1,\;x_2\in\widetilde{\mathbb K}$, and $r_1,\;r_2$ be positive
numbers such that $\widetilde B_{\leq r_1}(x_1)\supseteq \widetilde
B_{\leq r_2}(x_2)$.
Let $(x^{(1)}_\varepsilon)_\varepsilon$ be a representative of
$x_1$. Then the following holds:
\begin{enumerate}
\item \label{lemmatique1} There exists a net $(C^{(1)}_\varepsilon)_\varepsilon$ satisfying condition (E) such that for each $\varepsilon\in I$
\begin{equation}\label{eqcontained}
x_\varepsilon^{(2)}\in B_{\leq
\frac{C_\varepsilon^{(1)}\varepsilon^{\rho_1}}{2}}(x_\varepsilon^{(1)}).
\end{equation}
\item \label{lemmatique2} Furthermore, for each net $(C^{(2)}_\varepsilon)_\varepsilon$ satisfying condition (E)
there exists an $\varepsilon_0^{(1)}\in I$ such that for each
$\varepsilon<\varepsilon_0^{(1)}\in I$ we have
$B^{(2)}_{\varepsilon}\subseteq B^{(1)}_{\varepsilon}$.
\end{enumerate}
\end{proposition}
\begin{proof}
Proof of (\ref{lemmatique1}): Let
$(x_{\varepsilon}^{(2)})_{\varepsilon}$ be a representative of
$x_1$. We distinguish the following two cases:
\begin{enumerate}
\item $x_2\in \widetilde S_{r_1}(x_1)$, that is, $\vert x_2-x_1\vert_e=r_1$. Let $(x_{\varepsilon}^{(2)})_{\varepsilon}$ be a representative of $x_2$. Define $\hat C_{\varepsilon}^{(1)}:=\vert x_{\varepsilon}^{(1)}-x_{\varepsilon}^{(2)}\vert$. Now, set $ C_{\varepsilon}^{(1)}:=2\max(\{\hat C_{\eta}^{(1)} \vert \eta>\varepsilon\},1)$. Then not only $C_{\varepsilon}^{(1)}>0$ for each parameter $\varepsilon$, but also
the net $C_{\varepsilon}^{(1)}>0$ is monotonically increasing with
$\varepsilon\rightarrow 0$, furthermore (\ref{eqcontained}) holds,
and we are done with this case.
\item $x_2 \notin \widetilde S_{r_1}(x_1)$, that is, $\vert x_2-x_1\vert_e<r_1$. Set, for instance, $C_{\varepsilon}^{(1)}=1$.
 For each representative $(x_{\varepsilon}^{(2)})_{\varepsilon}$ of $x_2$ it follows that
 \[
 \vert x_\varepsilon^{(2)}-x_\varepsilon^{(1)}\vert=o(\varepsilon^{\rho_1})
 \]
 and a representative satisfying the desired properties is easily found.
\end{enumerate}
Proof of (\ref{lemmatique2}):\\
To show this we consider the asymptotic growth of
$(C_{\varepsilon}^{(1)})_{\varepsilon},(C_{\varepsilon}^{(2)})_{\varepsilon},\varepsilon^{\rho_1},\varepsilon^{\rho_2}$
as well as the monotonicity of $C_{\varepsilon}^{(1)}$: let $y\in
B_{\leq
C_{\varepsilon}^{(2)}\varepsilon^{\rho_2}}(x_{\varepsilon}^{(2)})$.
Then we have by the triangle inequality for each $\varepsilon\in I$:
\begin{equation}\label{est0}
\vert y-x_{\varepsilon}^{(1)}\vert\leq \vert
y-x_{\varepsilon}^{(2)}\vert +\vert
x_{\varepsilon}^{(2)}-x_{\varepsilon}^{(1)}\vert\leq
C_{\varepsilon}^{(2)}\varepsilon^{\rho_2}+\frac
{C_{\varepsilon}^{(1)}\varepsilon^{\rho_1}}{2}.
\end{equation}
We know further that by the monotonicity $\forall \varepsilon\in I:
C_{\varepsilon}^{(1)}\geq C_0^{(1)}:=C_0$ so that
\begin{equation}\label{est1}
\frac{C_{\varepsilon}^{(2)}}{C_{\varepsilon}^{(1)}}\varepsilon^{\rho_2-\rho_1}\leq
C_0C_{\varepsilon}^{(2)}\varepsilon^{\rho_2-\rho_1}.
\end{equation}
Moreover, since the sharp norm of $C_{\varepsilon}^{(2)}$ equals
$1$, for any $\alpha>0$ we have
\[
C_{\varepsilon}^{(2)}=o(\varepsilon^{-\alpha}),\;
(\varepsilon\rightarrow 0).
\]
which in conjunction with the fact that $\rho_2>\rho_1$ allows us to
further estimate the right hand side of (\ref{est1}): Obtaining
\[
\frac{C_{\varepsilon}^{(2)}}{C_{\varepsilon}^{(1)}}\varepsilon^{\rho_2-\rho_1}=o(1),\;(\varepsilon\rightarrow
0),
\]
we plug this information into (\ref{est0}). This yields for
sufficiently small $\varepsilon$, say
$\varepsilon<\varepsilon_0^{(1)}$:
\begin{equation}
\vert y-x_{\varepsilon}^{(1)}\vert\leq \frac
{C_{\varepsilon}^{(1)}\varepsilon^{\rho_1}}{2}+ \frac
{C_{\varepsilon}^{(1)}\varepsilon^{\rho_1}}{2}=C_{\varepsilon}^{(1)}\varepsilon^{\rho_1};
\end{equation}
the proof is finished.
\end{proof}
\begin{proposition}\label{propseq}
Any nested sequence of closed balls in $\widetilde{\mathbb K}$
admits a proper sequence of associated euclidean models.
\end{proposition}
\begin{proof}
We proceed step by step so that we may easily read off the inductive argument of the proof in the end.\\
We may assume that for each $i\geq 1$, $r_i>r_{i+1}$. Define $\rho_i:=-\log (r_i)$ (so that $\rho_i<\rho_{i+1}$ for each $i\geq 1$).\\
{\bf Step 1.}\\
Choose a representative $(x^{(1)}_\varepsilon)_\varepsilon$ of $x_1$.\\
{\bf Step 2.}\\
Due to Proposition \ref{lemmatique} (\ref{lemmatique1}) we may
choose a representative $(x^{(2)}_\varepsilon)_\varepsilon$ of $x_2$
and a net  $(C^{(1)}_\varepsilon)_\varepsilon$ of real numbers
satisfying condition (E) such that such that for each
$\varepsilon\in I$
\[
x_\varepsilon^{(2)}\in B_{\leq
\frac{C_\varepsilon^{(1)}\varepsilon^{\rho_1}}{2}}(x_\varepsilon^{(1)}).
\]
Denote by $\varepsilon_0^{(1)}\in I$ be the maximal $\varepsilon$
such that the inclusion relation
$B^{(2)}_{\varepsilon}\subseteq B^{(1)}_{\varepsilon}$  as in (cf.\ (\ref{lemmatique2}) of Proposition \ref{lemmatique}) holds.\\
{\bf Step 3.}\\
Similarly, take a representative $(\hat
x^{(3)}_\varepsilon)_\varepsilon$ of $x_3$ and a net  $(\hat
C^{(2)}_\varepsilon)_\varepsilon$ of real numbers satisfying
condition (E) such that such that for each $\varepsilon\in I$
\begin{equation}\label{inclpointball2}
\hat x_\varepsilon^{(3)}\in B_{\leq \frac{\hat
C_\varepsilon^{(2)}\varepsilon^{\rho_2}}{2}}(x_\varepsilon^{(2)}).
\end{equation}
We show now, how to adjust our choice of $\hat x_\varepsilon^{(3)},
\hat C_\varepsilon^{(2)}$ such that condition (E) as well as the
inclusion relation (\ref{inclpointball2}) is preserved, however, we
do this in a way such that we moreover achieve the inclusion
relation
\begin{equation}\label{inclrel2balls}
B^{(2)}_{\varepsilon}\subseteq B^{(1)}_{\varepsilon}
\end{equation}
for {\it each} $\varepsilon$ (for sufficiently small parameter this is guaranteed by the proceeding proposition).\\
For $\varepsilon \geq \varepsilon_0^{(1)}$ we leave the choice
unchanged, that is, we set
\[
 x_\varepsilon^{(3)}:=\hat x_\varepsilon^{(3)},\;C_\varepsilon^{(2)}:=\hat C_\varepsilon^{(2)};
\]
for $\varepsilon < \varepsilon_0^{(1)}$, however, we set
\begin{equation}\label{resetstep3}
 x_\varepsilon^{(3)}:=x_\varepsilon^{(2)},\;C_\varepsilon^{(2)}:=\min(\frac{C_\varepsilon^{(1)}}{2}\varepsilon^
 {\rho_1-\rho_2},  \hat C_\varepsilon^{(2)}).
\end{equation}
Therefore, $(C_\varepsilon^{(2)})_\varepsilon$ still satisfies
condition (E), since it is still positive and monotonically
increasing with $\varepsilon\rightarrow 0$, furthermore we have only
modified for big parameter $\varepsilon$, the asymptotic growth with
$\varepsilon\rightarrow 0$ therefore remains unchanged (and so does
the sharp norm of $(C_\varepsilon^{(2)})_\varepsilon$, which it is
identically $1$). Next, it is evident that
\[
x_\varepsilon^{(3)}\in B_{\leq
\frac{C_\varepsilon^{(2)}\varepsilon^{\rho_2}}{2}}(x_\varepsilon^{(2)}).
\]
still holds for each $\varepsilon\in I$. Finally, by
(\ref{resetstep3}) it follows that the inclusion relation
(\ref{inclrel2balls}) holds now for each $\varepsilon\in I$. For the
inductive proof of the statement one formally proceeds as in Step 3.
Let $k>1$. Assume we have representatives
\[
(x_{\varepsilon}^{(1)})_{\varepsilon},\dots,(x_{\varepsilon}^{(k+1)})_{\varepsilon}
\]
and nets of positive numbers
\[
(C_{\varepsilon}^{(j)})_{\varepsilon}, (1\leq j\leq k),
\]
satisfying condition (E), such that for each $\varepsilon\in I$ we
have:
\[
B_{\leq
C_{\varepsilon}^{(1)}\varepsilon^{\rho_1}}(x_{\varepsilon}^{(1)})\supseteq
B_{\leq
C_{\varepsilon}^{(2)}\varepsilon^{\rho_2}}(x_{\varepsilon}^{(2)})\supseteq\dots\supseteq
B_{\leq
C_{\varepsilon}^{(k-1)}\varepsilon^{\rho_{k-1}}}(x_{\varepsilon}^{(k-1)}).
\]
and for some $\varepsilon_0^{(k-1)}$ we have for each
$\varepsilon<\varepsilon_0^{(k-1)}$
\[
B_{\leq
C_{\varepsilon}^{(k-1)}\varepsilon^{\rho_{k-1}}}(x_{\varepsilon}^{(k-1)})\supseteq
B_{\leq
C_{\varepsilon}^{(k)}\varepsilon^{\rho_{k}}}(x_{\varepsilon}^{(k)}).
\]
Furthermore we suppose the following additional property is
satisfied: For each $\varepsilon\in I$ we have:
\[
x_{\varepsilon}^{(k+1)}\in B_{\leq
\frac{C_{\varepsilon}^{(k)}}{2}\varepsilon^{\rho_k}}(x_{\varepsilon}^{(k)}),
 \]
 where $\rho_k:=-\log r_k$. In the very same manner as above, we may now find a representative $(x_{\varepsilon}^{(k+2)})_{\varepsilon}$ of $x_{k+2}$ and
a net of numbers $(C_{\varepsilon}^{(k+1)})_{\varepsilon}$
satisfying condition (E) such that the above sequential construction
can be enlarged by one ($k\rightarrow k+1$).
\end{proof}
The preceding proposition is a key ingredient in the proof of our
main statement Theorem \ref{theoremsph}:
\begin{proof}
Let $(\widetilde B_i)_{i=1}^{\infty},\; B_i:=\widetilde B_{\leq
r_i}(x_i)\;(i\geq 1)$ be the given nested sequence of dressed balls;
due to Proposition \ref{propseq}, there exists a proper sequence of
associated euclidean models
\[
(B^{(i)}_{\varepsilon})_{\varepsilon}
\]
such that for representatives
$(x^{(i)}_{\varepsilon})_{\varepsilon}$ of $x_i$ ($i\geq 1$) the
above nets are given by
\[
B^{(i)}_{\varepsilon}:=B_{\leq
C_{\varepsilon}^{(i)}\varepsilon^{\rho_i}}(x^{(i)}_{\varepsilon}),\quad\rho_i:=-\log
r_i,\quad C_{\varepsilon}^{(i)}\in\mathbb R_+
\]
for each $(\varepsilon,i)\in I\times\mathbb N$. Since $\mathbb K$ is
locally compact, for each $\varepsilon\in I$ we may choose some
$x_{\varepsilon}\in\mathbb R$ such that
\[
x_{\varepsilon}\in \bigcap_{i=1}^{\infty}B^{(i)}_{\varepsilon}
\]
since for each $\varepsilon\in I$ we have
$B_{\varepsilon}^{(1)}\supseteq
B_{\varepsilon}^{(2)}\supseteq\dots$. Since the sequence of
euclidean models of the $\widetilde B_i$'s is proper, for each
$\varepsilon\in I$ further holds:
\[
\vert x_{\varepsilon}-x_{\varepsilon}^{(i)}\vert \leq
C_{\varepsilon}^{(i)}\varepsilon^{\rho_i}.
\]
This shows that not only the net $(x_{\varepsilon})_{\varepsilon}$
is moderate (use the triangle inequality), but also gives rise to a
generalized number $x:=(x_{\varepsilon})_{\varepsilon}+\mathcal
N(\mathbb K)$ with the property
\[
\vert x-x_i\vert_e\leq r_i
\]
for each $i$. This shows that
\[
x\in\bigcap_{i=1}^{\infty}\widetilde B_i\neq \emptyset
\]
which yields the claim: $\widetilde{\mathbb K}$ is spherically
complete.
\end{proof}
\subsection{A Hahn-Banach Theorem} Let $L$ be a subfield of
$\widetilde{\mathbb K}$ such that $\nu_e$ restricted to $L$ is
additive. Let $E$ be an ultra pseudo-normed $L$-linear space. We
call $\varphi$ an $L$- linear functional on $E$, if $\varphi$ is an
$L$- linear mapping on $E$ with values in $\widetilde{\mathbb K}$.
$\varphi$ is continuous if
\[
\|\varphi\|:=\sup_{0\neq x\in E}\frac{\vert
\varphi(x)\vert}{\|x\|}<\infty
\]
and the space of all continuous $L$-linear functionals on $E$ we
denote by $E'_L$.
\begin{remark}
Note that nontrivial subfields $L$ of $\widetilde{\mathbb K}$ exist.
For instance, one may choose $\mathbb K(\alpha)$ with
$\alpha=[(\varepsilon)_{\varepsilon}]\in\widetilde{\mathbb K}$ or
its completion with respect to $\vert \;\;\vert_e$-the Laurent series
over $\widetilde{\mathbb K}$.
\end{remark}
Having introduced these notions we show that following version of
the Hahn-Banach Theorem holds:
\begin{theorem}
Let $V$ be an $L$-linear subspace of $E$ and $\varphi\in V'_L$. Then
$\varphi$ can be extended to some $\psi\in E'_L$ such that
$\|\psi\|=\|\varphi\|$.
\end{theorem}
\begin{proof}
We follow the lines of the proof of Ingleton's theorem (cf.\
\cite{Ingleton}) in the fashion of (\cite{zAR}, pp.\ 194--195). To
start with, let $V$ be a strict $L$-linear subspace of $E$ and let
$a\in E\setminus V$. We first show that $\varphi\in V'_L$ can be
extended to $\psi\in (V+La)'_L$ under conservation of its norm. To
do this it is sufficient to prove that such $\psi$ satisfies for
each $x\in V$:
\begin{eqnarray}\label{ineqnorm}
\|\psi(x-a)\|&\leq&\|\psi\|\cdot\|x-a\| \\\nonumber
\|\varphi(x)-\psi(a)\|&\leq&\|\varphi\|\cdot\|x-a\|=:r_x.
\end{eqnarray}
To this end define for each $x$ in $V$ the dressed ball
\[
B_x:=B_{\leq r_x}(\varphi(x)).
\]
Next we claim that the family $\{B_x\mid x\in V\}$ of dressed balls
is nested. To see this, let $x,y \in V$. By the linearity of
$\varphi$ and the ultrametric (strong) triangle inequality we have
\[
\vert \varphi(x)-\varphi(y)\vert\leq \|
\varphi\|\cdot\|x-y\|\leq\|\varphi\|\max(\|x-a\|,\|y-a\|)=\max(r_x,r_y).
\]
Therefore we have $B_x\subseteq B_y$ or $B_y\subseteq B_x$ or vice
versa. According to Theorem \ref{theoremsph}, $\widetilde{\mathbb
K}$ is spherically complete, therefore we may choose
\[
\alpha\in \bigcap_{x\in V}B_x
\]
and further define $\psi(a):=\alpha$. Due to (\ref{ineqnorm}) and
the homogeneity of the sharp norm with respect to the field $L$ we
therefore have for each $z\in V$ and for each $\lambda \in L$,
\[
\vert\psi(z-\lambda a)\vert=\vert\lambda\vert\cdot\vert
\psi(z/\lambda-a)\vert\leq \vert \lambda\vert r_{z/\lambda}=\vert
\lambda\vert \|\varphi\|\cdot\|z/\lambda-a\|=\|\varphi\|\cdot
\|z-\lambda a\|
\]
which shows that $\psi$ is an extension of $\varphi$ onto $V+La$ and
$\|\psi\|=\|\varphi\|$.

The rest of the proof is the standard one-an application of Zorn's
Lemma.
\end{proof}
Let $E$ be a ultra pseudo-normed $\widetilde{\mathbb K}$ module and denote by
$E'$ all continuous linear functionals on $E$. We end this section
by posing the following conjecture:
\begin{conjecture}
Let $V$ be a submodule of $E$ and let $\varphi\in V'$. Then
$\varphi$ can be extended to some element $\psi\in E'$ such that
$\|\psi\|=\|\varphi\|$.
\end{conjecture}
\subsection*{Appendix} Finally, it is worth
mentioning that apart from the standard Fixed Point Theorem due to
Banach, a non-archimedean version is available in spherically
complete ultrametric spaces ( therefore, also on $\widetilde{\mathbb
K}$, cf.\ \cite{priess1}, and for a recent generalization cf.\
\cite{priess2}):
\begin{theorem}
Let ($M,d$) be a spherically complete ultrametric space and $f:
M\rightarrow M$ be a mapping having the property
\[
\forall x,y\in M: d(f(x),f(y))<d(x,y).
\]
Then $f$ has a unique fixed point in $M$.
\end{theorem}
\section[Scaling invariance]{Scaling invariance in algebras of generalized functions}\label{chapterconstant}

Recent research in the field of generalized functions increasingly
focuses on intrinsic problems in algebras of generalized functions.
This is emphasized by a number of scientific papers on algebraic (cf.\ \cite{A1}) and topological topics (cf.\ \cite{DHPV, Scarpalezos1, Garetto2, Garetto1}).\\

In this chapter we investigate scaling invariance of generalized
functions. We prove that a generalized function on the real line
which is invariant under positive standard scaling has to be a
constant. Also, we add a couple of further new characterizations of
locally constant generalized functions to the well known ones. Our
proof is partially based on the solution of the so-called "Lobster
problem". It was at the {\it International Conference on Generalized
functions 2000 (April, 17--21)} that Professor Michael
Oberguggenberger offered a lobster for the answer to the question:
"Are generalized functions which are invariant under standard
translations, merely the constants?" A (positive) answer to the
latter was first given by S.\ Pilipovic, D.\ Scarpalezos and V.\
Valmorin in \cite{Lobster}
 and an independent proof has recently been established by H.\ Vernaeve \cite{Vernaeve}.

Note that there is also an evident link between the present work and that of S.\ Konjik and M. Kunzinger dealing with group invariants in algebras of generalized functions (\cite{KKo1,KKo2}) which are also partially based on the solution of the Lobster problem.
\subsection{Preliminaries}
The setting of this chapter is the {\it special algebra} $\mathcal G(\mathbb R^d)$ of generalized functions (cf.\ the introduction).

To start with we shortly review the specific concepts resp.\ methods we are going to employ in the sequel: association and integration of generalized functions, generalized points and sharp topology as well as continuity issues with respect to the latter. For the sake of simplicity we set $d=1$.
For the generalized point value concept in algebras of generalized functions introduced by M.\ Kunzinger and M.\ Oberguggenberger in \cite{MO1}, we refer to the introduction. Next, let us recall the so-called sharp topology on the ring of generalized numbers:
\subsubsection{The sharp topology on $\widetilde{\mathbb R}$}
The - maybe most natural - topology on the ring of generalized numbers is the one which respects
the asymptotic growth by means of which they are defined.
Define a (real valued) valuation function $\nu$ on $\mathcal E_M(\mathbb R)$ in the following way:
\[
\nu((u_{\varepsilon})_{\varepsilon}):=\sup\,\{b\in\mathbb R\mid \vert u_{\varepsilon}\vert=O(\varepsilon^b)\;\;(\varepsilon\rightarrow 0)\}.
\]
This valuation can be carried over to the ring of generalized numbers in a well defined way, since
for two representatives of a generalized number, their valuations coincide (cf.\ \cite{Garetto1}, chapter 1).
We then may endow $\widetilde{\mathbb R}$ with an ultra-pseudo-norm ('pseudo' refers to non-multiplicativity) $\vert \;\; \vert_e$ in the following way:
$\vert 0\vert_e:=0$, and whenever $x\neq 0$, $\vert x\vert_e:=e^{-\nu(x)}$. With the metric $d_e$ induced by the above norm, $\widetilde{\mathbb R}$
turns out to be a non-discrete ultrametric space, with the following topological properties:
\begin{enumerate}
\item $(\widetilde{\mathbb R},d_e)$ is topologically complete (cf. \cite{Garetto1}),
\item $(\widetilde{\mathbb R},d_e)$ is not separable, since the restriction of $d_e$ onto $\mathbb R$ is discrete.
\end{enumerate}
The latter property holds, since on metric spaces second countability and separability are equivalent
and the well known fact that the property of second countability is inherited by subspaces (whereas separability
is not in general).
\subsubsection{Continuity issues}\label{contissue}
In (\cite{A2}) Aragona et al.\ develop a new concept of
differentiability of generalized functions $f$ viewed as maps
$\widetilde f:\widetilde{\mathbb R}_c\rightarrow\widetilde{\mathbb
R}$, a concept which is compatible with partial differentiation in
$\mathcal G(\mathbb R^d)$ and evaluation of functions at generalized
points. We need not recall this in detail; we only mention one notable
consequence which we will make use of subsequently:
\begin{fact}
If $\widetilde f: \widetilde{\mathbb
R}_c\rightarrow\widetilde{\mathbb R}$ is induced by a generalized
function, then $\widetilde f$ is continuous with respect to the
sharp topology on $\widetilde{\mathbb R}_c$.
\end{fact}
\subsubsection{Integration of generalized functions}
Generalized functions may be integrated over relatively compact
Lebesgue measurable sets. We recall an elementary statement (this is
Proposition 1.2.56 in \cite{Bible}):
\begin{fact}
Let $M$ be a Lebesgue-measurable set such that $\bar M\subset\subset \mathbb R$ and take $u\in\mathcal G(\mathbb R)$. Let $(u_\varepsilon)_\varepsilon$ be a representative of $u$. Then
\[
\int_Mu(x)dx:=\left(\int_Mu_\varepsilon(x)\,dx\right)_\varepsilon+\mathcal N
\]
is a well-defined element of $\widetilde{\mathbb R}$ called the integral of $u$ over $M$.
\end{fact}
Also, we are going to need the 'antiderivative' $F$ of a generalized function.
Let $f\in\mathcal G(\mathbb R)$. This we introduce
by
\[
F(x):=\int_0^xf(s)ds:=\left(\int_0^x f_{\varepsilon}(s)ds\right)_{\varepsilon }+\mathcal N(\mathbb R)\in\mathcal G(\mathbb R)
\]
where $(f_{\varepsilon})_{\varepsilon}$ is an arbitrary representative of $f$. Note that $F$ is the primitive of $f$ with point value $F(0)=0$ in $\widetilde{\mathbb R}$ (cf.\ Proposition 1.2.58 in \cite{Bible}).
\subsubsection{The concept of association}\label{scalinginvariance}
Finally we recall the concept of association in $\widetilde{\mathbb
R}$ and in $\mathcal G(\mathbb R^d)$. First, let $\alpha\in
\widetilde{\mathbb R}$. We write $\alpha\approx 0$ and we say
"$\alpha$ is associated to zero", if for some (hence any)
representative $(\alpha_\varepsilon)_\varepsilon$ we have
\[
\alpha_\varepsilon\rightarrow 0 \qquad \mbox{whenever}\qquad \varepsilon\rightarrow 0.
\]
Similarly, we say $u\in\mathcal G(\mathbb R^n)$ is associated with zero, if for each test function $\phi$
we have
\[
\int u_\varepsilon(x)\phi(x)\, dx^n\rightarrow 0\qquad \mbox{whenever}\qquad \varepsilon\rightarrow 0.
\]
The relation $\approx$ is an equivalence relation on $\widetilde{\mathbb R}$ resp.\ $\mathcal G(\mathbb R^d)$.
By slightly abusing the above terminology we write  $u\approx w$, $w\in\mathcal D'(\mathbb R^d)$ and say "$u$ is associated with $w$" (or , "$w$ is the distributional shadow of $u$"), if we have
\[
\int u_\varepsilon(x)\phi(x)\, dx^n\rightarrow \langle w,\phi\rangle\qquad \mbox{whenever}\qquad \varepsilon\rightarrow 0.
\]
It is a well known fact that a generalized function $u$ has at most one distributional shadow (cf.\ \cite{Bible}, Proposition 1.2.67).
\subsection[Functions supported at the origin]{Generalized functions supported at the origin}
To start with we establish a basic lemma:
\begin{lemma}\label{supportimursprungundscalinginvariance}
Let $f\in\mathcal G(\mathbb R)$ be a non-negative function with $\supp(f)\subseteq \{0\}$. If for some $a>0$
we have
\[
I(f)=\int_{[-a,a]} f(x)dx=0,
\]
then $f=0$.
\end{lemma}
\begin{proof}
We present two variants of the proof:\\
{\it First Proof.}\\ It has been shown recently (cf.\
\cite{VariationMO}) that if for a generalized function $f$ we have
for all $\varphi\in\mathcal G_c(\mathbb R)$ (the space of compactly
supported generalized functions)
\[
\int f(x)\varphi(x)\, dx=0,
\]
then $f=0$ in $\mathcal G(\mathbb R)$. This is the so-called
fundamental lemma of the calculus of variations in the generalized
context. Now we have by the non-negativity of $f$,
\[
\left|\int f(x)\varphi(x)\, dx\right|\leq \|\varphi\|_{\infty}\int f(x) dx=0,
\]
therefore by the above we have $f=0$ in $\mathcal G(\mathbb R)$ and we are done.\\
{\it Alternative Proof.}\\
This proof employs continuity arguments of generalized functions with respect to the sharp topology. In view of the
first proof this may also yield a link between the fundamental lemma of variational calculus (in the generalized setting) and (sharp) topological issues. For our (indirect) proof we proceed in three steps.\\
{\it Step 1.}\\
Since $f$ is non-negative and $K:=[-a,a]\subset\subset \mathbb R$ is a compact set, we may choose a representative $(f_{\varepsilon})_{\varepsilon}$ of $f$ which is non-negative on $K$, that is, $(f_\varepsilon)_\varepsilon$ satisfies:
\[
\forall\;x\in K\;\forall\; \varepsilon>0: f_{\varepsilon}\geq 0.
\]
Assume $f\neq 0$ in $\mathcal G(\mathbb R)$. Due to (cf.\ subsection \ref{pointissue}), there exists
a compactly supported generalized point $x_c\in\widetilde{\mathbb R}$ such that
$f(x_c)=c\neq 0$. From our assumption on the support of $f$ ($\supp(f)\subseteq \{0\}$)
it  is further evident that $x_c\approx 0$; this information, however, is not crucial for what follows).\\
{\it Step 2.}\\
Let $(x_{\varepsilon})_{\varepsilon}$ be a representative of $x_c$. We shall prove the following:
\begin{equation}\label{subclaimintegral}
\exists \;\varepsilon_k\rightarrow 0\;\exists\; m_0\;\exists\;\rho_0\;\forall\; k\;\forall\; y_k\in [x_{\varepsilon_k}-\varepsilon_k^{\rho_0},x_{\varepsilon_k}+\varepsilon_k^{\rho_0}]: f_{\varepsilon_k}(y_k)\geq\varepsilon_k^{m_0}.
\end{equation}
To see this, we first observe by means of {\it Step 1} that there exists a zero sequence $\varepsilon_k$
and a real number $m_0$ such that for each $k\geq 0$ we have $f_{\varepsilon_k}(x_{\varepsilon_k})\geq 2 \varepsilon_k^{m_0}$ (we shall take this zero sequence as the one of our claim). Next, we employ a continuity argument to prove (\ref{subclaimintegral}). Recall that $f$ viewed as a map $\widetilde f:\widetilde{\mathbb R}_c\rightarrow\widetilde{\mathbb R}$ is continuous with respect to the sharp topology (cf.\ subsection \ref{contissue}). Assume that (\ref{subclaimintegral}) is not true. Then for each $m$ and for each $\rho$ there exists a sequence $(y_k)_k$ with $y_k\in [x_{\varepsilon_k}-\varepsilon_k^{\rho},x_{\varepsilon_k}+\varepsilon_k^{\rho}]$ for each $k$ such that
\begin{equation}
0 \leq f_{\varepsilon_k}(y_k)<\varepsilon_k^{m}
\end{equation}
 (the first inequality holds because we may
assume without loss of generality that everything takes place inside $[-a,a]$, where we have found
a non-negative representative of $f$). Define a (compactly supported) generalized number $y:=(y_\varepsilon)_\varepsilon+\mathcal N$ via
\[
y_\varepsilon:=\begin{cases} y_k,\quad \mbox{if}\;\varepsilon=\varepsilon_k\\x_\varepsilon,\quad\mbox{otherwise}  \end{cases}
\]
Then we have for sufficiently small $m$
\[
\vert f_{\varepsilon_k}(x_{\varepsilon_k})-f_{\varepsilon_k}(y_{\varepsilon_k})\vert>2\varepsilon_k^{m_0}-\varepsilon_k^m>\varepsilon_k^{m_0},
\]
whereas for $\varepsilon\neq\varepsilon_k$ we have by the above construction that $f_\varepsilon(x_\varepsilon)-f_\varepsilon(y_\varepsilon)=0$ . In terms of the sharp norm $\vert\;\;\vert_e$ we therefore have:
\[
\vert f(x_c)-f(y)\vert_e\geq e^{-m_0};
\]
by our assumption, however, it follows that
\[
\vert x_c-y\vert_e\leq e^{-\rho_0}.
\]
The choice of $\rho$ was arbitrary, and $\rho\rightarrow 0$
violates the continuity of $f$ at $x_c$. Therefore we have established (\ref{subclaimintegral}). This we
apply in the third and final step:\\
{\it Step 3.}\\
For sufficiently large $k$ we obtain
\begin{equation}\label{intcont}
\int_{[-a,a]}f_{\varepsilon_k}(y)dy>\varepsilon_k^{m_0}(2\varepsilon_k^{\rho})=2\varepsilon_k^{\rho+m_0}.\end{equation}
Since $\left(\int_{[-a,a]}f_{\varepsilon}(y)dy\right)_{\varepsilon}$ is a representative of $I(f)$, inequality (\ref{intcont})
contradicts our assumption $I(f)=0$ (the representative not being a negligible net) and we are done.
\end{proof}
A further ingredient in the subsequent proof of our main result is
the elementary observation that generalized scaling invariant
functions $f\in\mathcal G(\mathbb R)$ with support contained in the
origin have to be identically zero. To motivate our proof, we first
analyze the-maybe- simplest non-trivial example: a generalized function $\hat \rho$
associated with a distribution supported at the origin, say
$\delta$. In this situation invariance under standard scaling is
absurd: Assume we are given a standard mollifier $\rho\in
C_c^{\infty}(\mathbb R)$, that is $\int_{\mathbb R} \rho(x) dx=1$.
Then
$\rho_{\varepsilon}:(\frac{1}{\varepsilon}\rho(\frac{x}{\varepsilon}))_{\varepsilon}$
gives rise to a generalized function
$\hat{\rho}:=[(\rho_{\varepsilon})_{\varepsilon}]\in\mathcal
G(\mathbb R)$ and, as it is well known, we have:
\[
\hat{\rho}\approx  \delta,\quad\mbox{that is,}\quad \forall \varphi \in\mathcal C_c^{\infty}(\mathbb R):\; \lim_{\varepsilon\rightarrow 0}\langle \rho_{\varepsilon},\varphi\rangle\rightarrow \langle \delta,\varphi\rangle=\varphi(0).
\]
Consider now, $h\neq 0,1$ and assume the identity $\hat
\rho(hx)=\hat \rho (x)$ holds in $\mathcal G(\mathbb R)$. This in
particular means that
\[
\hat{\rho}=[(\rho_{\varepsilon}(hx))_{\varepsilon}]\qquad \mbox{in}\qquad \mathcal G(\mathbb R)
\]
holds as well. But for each $\varphi \in\mathcal C_c^{\infty}(\mathbb R),\;\varphi(0)\neq 0$ we have
\[
\lim_{\varepsilon\rightarrow 0}\langle \rho_{\varepsilon}(hx),\varphi(x)\rangle\rightarrow \frac{1}{h}\varphi(0)\neq\langle \delta,\varphi\rangle
\]
therefore $\hat{\rho}$ has more than one distributional shadow, namely $\delta, h\delta,\;$ for arbitrary  $h\neq 0$ which is impossible!
\footnote{
Of course, $\delta$ is scaling invariant, however, in the sense that
\[
\langle \delta (*h),\varphi\rangle:=\langle \delta,\varphi(*h)\rangle.
\]
In terms of the model delta net above this refers to the following 'scaling':
\[
\rho_{\varepsilon}(x)\mapsto h\rho_{\varepsilon}(hx),\;h\neq 0
\]
and for each $h\neq 0$ the 'scaled' object is associated to $\delta$ as well; furthermore even the identity
\[
\hat{\rho}(x)=h\hat{\rho}(hx)
\]
holds in $\mathcal G(\mathbb R)$ (cf.\ the proof of Proposition \ref{prop}).}
We may now present the statement in full generality:
\begin{proposition}\label{prop}
Assume $f\in\mathcal G(\mathbb R)$ has the following properties:
\begin{enumerate}
\item $f$ is invariant under positive standard scaling.
\item $\supp(f)\subseteq\{0\}$.
\end{enumerate}
Then $f=0$ in $\mathcal G(\mathbb R)$.
\end{proposition}
\begin{proof}
Assume $f\in\mathcal G(\mathbb R)$ satisfies the assumption of the proposition and without loss of generality we further assume $f\geq 0$ (otherwise, take $f^2$ instead of $f$). Let $(f_{\varepsilon})_{\varepsilon}$ be a representative of $f$ and $a>0$. Since $f$ is supported at the origin, the integral
\[
I(f):=\int_{[-a,a]}f(x)dx:=\left(\int_{[-a,a]} f_{\varepsilon}(x)dx\right)_{\varepsilon} +\mathcal N\in\widetilde{\mathbb R}
\]
is well defined, that is, the value $I(f)$ is independent of the choice of $a>0$ resp.\ of the representative of $f$. Further, for each
$h\neq 0,1$ and each $\varepsilon>0$ we have:
\[
\int_{ [-a,a]}f_{\varepsilon}(xh)dx=\frac{1}{h}\int_{ [-ah,ah]}f_{\varepsilon}(s)ds.
\]
The scaling invariance of $f$, therefore, which in terms of representatives reads
\[
(f_{\varepsilon}(x))_{\varepsilon}-(f_{\varepsilon}(hx))_{\varepsilon}=(n_{\varepsilon}(x))_{\varepsilon}\in\mathcal N(\mathbb R),
\]
combined with the fact that $f$ is supported in the origin, yields
\[
I(f)=\frac{1}{h}I(f) \qquad \mbox{in}\qquad \widetilde{\mathbb R}.
\]
Since $h\neq 0, 1$ this implies $I(f)=0$. Now we may apply Lemma \ref{supportimursprungundscalinginvariance} to the non-negative function $f$, and we obtain $f=0$.
\end{proof}
\subsection{The main theorem}
We are now ready to state the main theorem:
\begin{theorem}\label{th1}
Let $f\in\mathcal G(\mathbb R)$. The following are equivalent:
\begin{enumerate}
\item \label{const1} $f$ is constant, that is, there exists an $a\in\widetilde{\mathbb R}$ such that $f=a$ holds in $\mathcal G(\mathbb R)$.
\item \label{const2} $\widetilde f$ is constant.
\item \label{const3} $\widetilde f$ is locally constant.
\item \label{const4} $f$ is translation invariant, that is $\forall h\in\mathbb R: f(x+h)=f(x)$ holds in $\mathcal G(\mathbb R)$.
\item \label{const8} $f$ is invariant under positive standard scaling, that is,
\[
\forall h\in\mathbb R^+: f(hx)=f(x).
\]
\item \label{const5} $F$ is additive, that is, $\forall h\in\mathbb R: F(x+h)=F(x)+F(h)$ holds in $\mathcal G(\mathbb R)$.
\item \label{const6} $\widetilde F$ is additive, that is,
\[
\forall x_c, h_c\in\widetilde{\mathbb R}: \widetilde F(x_c+h_c)=\widetilde F(x_c)+\widetilde F(h_c)
\]
holds in $\widetilde{\mathbb R}$.
\item \label{const7} $F$ has the following property: There exists $\gamma\in (0,1)$ such that the identity:
\begin{equation}\label{cond6}
\forall h\in \mathbb R: F(\gamma x+(1-\gamma) h)=\gamma F(x)+(1-\gamma)F(h)
\end{equation}
holds in $\mathcal G(\mathbb R)$.
\end{enumerate}
\end{theorem}
\begin{proof}
We establish the implications (\ref{const3})$\Rightarrow$(\ref{const1})$\Rightarrow$(\ref{const2})$\Rightarrow$(\ref{const3})
as well as\\  (\ref{const7})$\Rightarrow$(\ref{const4})$\Rightarrow$(\ref{const1})$\Rightarrow$(\ref{const7}) and the equivalences (\ref{const1})$\Leftrightarrow$(\ref{const5}), (\ref{const1})$\Leftrightarrow$(\ref{const8}),
To begin with, assume (\ref{const3}), that is $f$ is locally constant. We show the implication by applying the generalized differential calculus
for Colombeau generalized functions evaluated on generalized points as has been developed by Aragona et al.\ in (\cite{A2}). Let $\kappa:\mathcal G(\mathbb R)\rightarrow \widetilde{\mathbb R}^{\widetilde{\mathbb R}_c}$ be the linear embedding of generalized functions into mappings on compactly supported points due to \cite{MO1}. Due to (\cite{A2}, Theorem 4.1)
differentiation in $\mathcal G(\mathbb R)$ resp.\ in $\widetilde{\mathbb R}^{\widetilde{\mathbb R}_c}$ commute with $\kappa$. Clearly
$\kappa(f)$ is differentiable with derivative $\kappa(f)'\equiv 0$, and as just mentioned, $\kappa(f')=\kappa(f)'=0$, therefore, due to the generalized point characterization in $\mathcal G(\mathbb R)$ (\cite{MO1}) we have $f'=0$ in $\mathcal G(\mathbb R)$ and
integrating yields (\ref{const1}) that is, $f$ is constant as a generalized function. The latter immediately implies (\ref{const2})
by evaluating $f$ on compactly supported generalized numbers and the implication (\ref{const2})$\Rightarrow$(\ref{const3}) is trivial.\\ Next, let $\gamma\in(0,1)$ and assume (\ref{cond6}) holds for $F$. Differentiating yields
\[
\forall h\in \mathbb R: f(\gamma x+(1-\gamma)h)=f(x) \quad \mbox{holds in}\quad \mathcal G(\mathbb R).
\]
This is equivalent to
\begin{equation}\label{cond7}
\forall h\in \mathbb R: f(x+h)=f(\gamma^{-1}x) \quad \mbox{holds in} \quad \mathcal G(\mathbb R).
\end{equation}
Setting $h=0$ shows that $f(x)=f(\gamma^{-1}x)$ in $\mathcal G(\mathbb R)$ which further implies
\begin{equation}\label{cond7}
\forall h\in \mathbb R: f(x+h)=f(x) \quad \mbox{holds in} \quad \mathcal G(\mathbb R),
\end{equation}
i.\ e., $f$ is translation invariant. This proves (\ref{const4}). The implication (\ref{const4})$\Rightarrow$(\ref{const1}) is proven in (\cite{Lobster}, Theorem 6); for an alternative proof
cf.\ the appendix to \cite{Vernaeve}. Since $f=a$ implies $F=ax$ in $\mathcal G(\mathbb R)$, the implication (\ref{const1})$\Rightarrow$(\ref{const7}) holds.\\ Further we establish the equivalence (\ref{const1})$\Leftrightarrow$(\ref{const5}). Again  (\ref{const1}) implies
that $F$ is of the form $F=ax$ with some generalized number $a$, therefore (\ref{const5}) holds. Conversely, assume that f satisfies
\[
F(x+h)=F(x)+F(h)\quad\mbox{holds in}\quad\mathcal G(\mathbb R)
\]
for each $h\in\mathbb R$. Differentiation yields
\[
\forall h\in\mathbb R: f(x+h)=f(x)\quad\mbox{holds in}\quad\mathcal G(\mathbb R)
\]
and by the above this implies (\ref{const1}). Finally we establish the equivalence (\ref{const1})$\Leftrightarrow$(\ref{const8}).
Since (\ref{const1})$\Rightarrow$(\ref{const8}) is trivial, we only need to show
(\ref{const1})$\Leftarrow$(\ref{const8}):\\
Note that without loss of generality we may assume that $f$ is symmetric (or equivalently, $f$ is invariant under any non-zero standard scaling). Indeed, if $f$ is not, we may introduce the two functions $g_{\pm}$ resp.\ $f_{\pm}$ given via
\[
g_+(x):=f_+^2(x):=(f(x)+f(-x))^2,\quad g_-(x):=f_-^2(x):=(f(x)-f(-x))^2.
\]
If for generalized constants $c_1,c_2$ we would have $g_+=c_1,g_-=c_2$, then for some generalized constants
$d_1,d_2$ we would have $f_+=d_1,f_-=d_2$, therefore
\[
f(x):=\frac{f_+(x)+f_-(x)}{2}=\frac{d_1+d_2}{2},
\]
that is, $f$ is a constant, and we would be done.\\
We may proceed now in two different ways: the first is a variant of H.\ Vernaeve's (\cite{Vernaeve})
proof of the Lobster problem.\\
{\it First proof.}\\
We distinguish the two possible cases, '$f$ is constant in a neighborhood of $1$' or not.
\\{\it Case 1}\\ Assume first, there exist a neighborhood $\Omega:=(1-\delta,1+\delta)$ of $1$, $\delta>0$ and $c\in\widetilde{\mathbb R}$ such that $f=c$ on $\Omega$. Since $f$ is invariant under positive standard scaling and symmetric, it follows that
\begin{enumerate}
\item For each $h>0$, $f=c$ on $(h-h\delta,h+h\delta)$.
\item $f(x)=f(-x)$ in $\mathcal G(\mathbb R)$.
\end{enumerate}
Since $\mathcal G(\mathbb R)$ is a sheaf, $f=c$ on $\mathbb R\setminus \{0\}$ and we have obtained a scaling invariant generalized function $g:=f-c$ with $\supp (g)\subseteq\{0\}$. Applying Proposition \ref{prop} yields $g=0$, that is, $f$ is a constant and we are done with the first case.
\\{\it Case 2}\\ If $f|_{\Omega}\in\mathcal G(\Omega)$ is non-constant
on every standard neighborhood $\Omega=(1-\delta,1+\delta)$ ($\delta>0$) of $1$, then we have for any representative $(f_{\varepsilon})_{\varepsilon}$ of $f$:
\[
(f_{\varepsilon}|_{\Omega}-f_{\varepsilon}(1))\notin \mathcal N(\Omega).
\]
Thus there exists a representative $(f_{\varepsilon})_{\varepsilon}$ of $f$ along with a zero sequence $(\varepsilon_k)_k$, a sequence $(a_k)_k\in[\frac{1}{2},\frac{3}{2}]^{\mathbb N}$ and an $N$ such that for all sufficiently large $k$ we have
\begin{equation}\label{Eq:assumption on f}
\vert
f_{\varepsilon_k}(a_k)-f_{\varepsilon_k}(1)\vert>\varepsilon_k^N.
\end{equation}
We now follow the basic idea of H.\ Vernaeve in (theorem 7 in \cite{Vernaeve}). Let $g_k(x):=f_{\varepsilon_k}(x)-f_{\varepsilon_k}(1)$
for each $k\geq 1$. We define
\[
        A_k:=\{x\in\mathbb R:|g_k(x)|<\frac 1 {3}\varepsilon_k^N \},\quad B_k:=\bigcap_{m\geq k}A_m
\]
It is evident that for all $k\in\mathbb N$ $g_k(1)=0$, therefore $1\in B_1$. Furthermore for each $x\in\mathbb R^{\ast}$ there exists $(n_\varepsilon)_{\varepsilon}\in\mathcal N$ such that for each $\varepsilon\in I$ we have
\[
        f_\varepsilon(x)=f_\varepsilon(1) + n_\varepsilon.
\]
In particular
\[
        g_k(x)=f_{\varepsilon_k}(x)-f_{\varepsilon_k}(1)=n_{\varepsilon_k}.
\]
As a consequence $\forall x\in\mathbb R^{\ast}\;\exists\; k_0\forall\; k\geq k_0: x\in A_k$. This clearly implies that
for each $x\in\mathbb R^{\ast}$ there exists a $k\geq 1$ such that $x\in B_k$, therefore we obtain
\begin{equation}\label{unionB}
\mathbb R^{\ast}\subseteq(\bigcup_{k=1}^{\infty} B_k)\subseteq\mathbb R.
\end{equation}
In a similar way
as $A_k,\; B_k$ we introduce the sets:
\[
        C_k:=\{x\in\mathbb R:|g_k(xa_k)-g_k(a_k)|<\frac 1 {3}\varepsilon_k^N\},\quad D_k:=\bigcap_{m\geq k}C_m.
\]
Again for each $x\in\mathbb R^{\ast}$, $x\in D_k$ for some $k$, since by the assumption of scaling invariance there exists an $(n_\varepsilon(y))_\varepsilon\in \mathcal N(\mathbb R)$ such that
\begin{align*}
        g_k(xa_k)-g_k(a_k)= & f_{\varepsilon_k}(xa_k)-f_{\varepsilon_k}(1)-f_{\varepsilon_k}(a_k)+f_{\varepsilon_k}(1) \\
        = & f_{\varepsilon_k}(xa_k) - f_{\varepsilon_k}(a_k) \\
        = & n_{\varepsilon_k}(a_k).
\end{align*}
Therefore we have
\begin{equation}\label{unionD}
\mathbb R^{\ast}\subseteq(\bigcup_{k=1}^{\infty} D_k)\subseteq\mathbb R.
\end{equation}
$B_k$ and $D_k$ are increasing sequences of Lebesgue measurable subsets of $\mathbb R$. Let $\mu$ be the Lebesgue measure on $\mathbb R$ and let $B_r(x)$ denote the open ball with radius $r$ and center $x$. For each $\rho>0$ we have due to (\ref{unionB}) and (\ref{unionD})
\begin{equation}\label{BD}
        \mu(B_{\rho}(0)\backslash B_k)\to 0,\; \mu(B_{\rho}(0)\backslash D_k)\to 0,\;(k\to\infty).
\end{equation}
Moreover by construction $B_k\subseteq A_k$ and $D_k\subseteq C_k$, therefore for each $\rho>0$ we also have
\begin{equation}\label{AC}
        \mu(B_{\rho}(0)\backslash A_k)\to 0,\; \mu(B_{\rho}(0)\backslash C_k)\to 0,\quad (k\to\infty).
\end{equation}
Finally we define
\[
        E_k:=\{x\in\mathbb R:|g_k(x)-g_k(a_k)|<\frac 1 {3}\varepsilon_k^N\}=a_k C_k.
\]
By the above we obtain for each $\rho>0$
\begin{align}\label{E}
        \mu(B_{\rho}(0)\backslash E_k) = & \mu(B_{\rho}(0)\backslash a_k C_k) \nonumber\\
        = & \mu\left(a_k(B_{\frac {\rho}{\vert a_k\vert}}(0)\backslash C_k)\right) \nonumber\\
        = & |a_k|\mu(B_{\frac{\rho}{\vert a_k\vert}}(0)\backslash C_k) \nonumber\\
        \leq & \frac{3}{2} \mu(B_{2\rho}(0)\backslash C_k)\to 0,
\end{align}
whenever $k\to\infty$, since $\frac 1 2\leq |a_k|\leq \frac 3 2$ and due to (\ref{AC}). A consequence of (\ref{AC}) and (\ref{E})
is the following:
\[
        \mu(B_{\rho}(0)\backslash (A_k\cap E_k))\leq \mu(B_{\rho}(0)\backslash A_k)+\mu(B_{\rho}(0)\backslash E_k) \to 0,
\]
that is, for sufficiently large $k$ the intersection of $A_k$ and $E_k$ is not empty, i.\ e.,
\[
        \exists k_0:\forall k\geq k_0\,\exists y_k\in A_k\cap E_k.
\]
Hence $|g_k(y_k)| < \frac 1 {3k^N}$ and $|g_k(y_k)-g_k(a_k)|<\frac 1
{3}\varepsilon_k^N$ for all $k\geq k_0$. The triangle inequality
yields for each $k\geq k_0$
\[
        |g_k(a_k)|=|f_{\varepsilon_k}(a_k)-f_{\varepsilon_k}(1)| < \frac 2 {3}\varepsilon_k^N.
\]
This contradicts line \eqref{Eq:assumption on f} and we are done.\\
{\it Alternative proof.}\\ First we consider the problem for $f\in
\mathcal G(\mathbb R^+)$ i.\ e.,
\[
\forall \lambda>0: f(\lambda x)=f(x)\quad \mbox{in}\quad \mathcal G(\mathbb R^+).
\]
This is equivalent to the problem
\[
\forall h\in\mathbb R: g(x+h)=g(x)\quad \mbox{in}\quad \mathcal G(\mathbb R),
\]
where $g:=f\circ\exp$. Therefore, by (\cite{Lobster}, Theorem 6) it follows that $g=$const. We are going to show that $f$ is constant on $\mathbb R^+$ as well.
To this end, note that the logarithm on $\widetilde{\mathbb R}_c^+$ is a well defined mapping since it stems from evaluation of
$\log\in\mathcal G(\mathbb R^+)$. Assume that $f$ is non-constant on the positive real numbers, that is,
there exist $x_c^+,y_c^+\in\widetilde{\mathbb R}_c^+$ such that $\widetilde f(x_c^+)\neq\widetilde f(y_c^+)$. This is equivalent to the fact that $f\circ\exp(x_c)\neq f\circ\exp(y_c) $, where $x_c:=\log x_c^+,y_c:=\log y_c^+$, a contradiction. By the symmetry of $f$ we have $f=c=const$ on $\mathbb R\setminus \{0\}$. Now we proceed as in {\it Case 1} of the first variant of the proof and we are done.
\end{proof}
\subsection{Scaling invariance in space.}
In the preceding section we established that any generalized function on the real
 line, which is invariant under positive standard scaling, is a constant.
  An important information we used was that without loss of generality
   we may assume that $f$ is symmetric. This helped us to overcome the
   obstacle    that $\mathbb R\setminus\{0\}$ is not connected, and we were able to reduce
     the problem to scaling invariance of generalized functions supported at the origin.
     The analogous question in higher space dimensions may be reduced to the one dimensional case. In the following,
     $d$ is an arbitrary positive integer.
         
\begin{theorem}\label{scinvth}
Any generalized function $f$ in $\mathbb R^d$ which is invariant under standard scaling is constant.
\end{theorem}
\begin{proof}
Let $f\in\Gen(\R^d)$ be invariant under positive (standard) scaling, that is,
$\forall\lambda\in\R$, $\lambda >0$ we have:
\[
f(\lambda x)=f(x).
\]
Fix a net $(a_\eps)_\eps$  such that $a_\eps\in L\csub\R^d$ for all $\eps>0$.
Then the net $(g_\varepsilon)_\varepsilon:=(f_\eps(a_\eps t))_\eps$ defines a generalized function
$g:=[(g_\varepsilon)_\varepsilon]\in\Gen(\R)$.
%
Now the scaling invariance for a fixed $\lambda$
 \[
 \forall\, L\csub\R^d\, \forall\, b\in\R:\sup_{x\in L}\abs{f_\eps(\lambda x) -
f_\eps(x)}= \Ord(\eps^b), \text{ as }\eps\to 0
\]
implies the scaling invariance for the same $\lambda$ of $g$
\[
\forall K\csub\R\,\forall\, b\in\R:\sup_{t\in K}\abs{f_\eps(\lambda a_\eps t) -
f_\eps(a_\eps t)}= \Ord(\eps^b), \text{ as }\eps\to 0.
\]
So the one-dimensional statement (Theorem \ref{th1}) implies that $g$ is a generalized constant, that is,

\[
\forall K\csub\R\,\forall\, b\in\R:\sup_{t\in K}\abs{f_\eps(a_\eps t) -
f_\eps(0)}= \Ord(\eps^b), \text{ as }\eps\to 0.
\]

By setting $t=1$ and $a:=(a_\varepsilon)_\varepsilon+\mathcal N(\mathbb R^d)$ we therefore have $f(a)=f(0)$ in $\widetilde{\mathbb R}$. Since the net $(a_\varepsilon)_\varepsilon$
was arbitrary it follows from Theorem \ref{distingprop} that $f=f(0)$ in $\mathcal G(\mathbb R^d)$ and we are done.

\end{proof}

\backmatter


\begin{thebibliography}{10}

\bibitem{Adams}
{\sc R.~A. Adams}, {\em Sobolev spaces}, Academic Press, New
York-London, 1975.
\newblock Pure and Applied Mathematics, Vol. 65.

\bibitem{AKS2}
{\sc S.~Albeverio, A.~Y. Khrennikov, and V.~M. Shelkovich}, {\em
Nonlinear
  singular problems of {$p$}-adic analysis: associative algebras of {$p$}-adic
  distributions}, Izv. Ross. Akad. Nauk Ser. Mat., 69 (2005), pp.~3--44.

\bibitem{AKS}
\leavevmode\vrule height 2pt depth -1.6pt width 23pt, {\em
{$p$}-adic
  {C}olombeau-{E}gorov type theory of generalized functions}, Math. Nachr., 278
  (2005), pp.~3--16.

\bibitem{A2}
{\sc J.~Aragona, R.~Fernandez, and S.~O. Juriaans}, {\em A
discontinuous
  {C}olombeau differential calculus}, Monatsh. Math., 144 (2005), pp.~13--29.

\bibitem{A1}
{\sc J.~Aragona and S.~O. Juriaans}, {\em Some structural properties
of the
  topological ring of {C}olombeau's generalized numbers}, Comm. Algebra, 29
  (2001), pp.~2201--2230.

\bibitem{RB}
{\sc R.~Beig}, {\em Lecture notes on special and general relativity,
  unpublished}, University of Vienna, Physics Institute,  (2004).

\bibitem{Clarke}
{\sc C.~J.~S. Clarke}, {\em Generalized hyperbolicity in singular
spacetimes},
  Class. Quantum Grav., 15 (1998), pp.~975--984.

\bibitem{CVW}
{\sc C.~J.~S. Clarke, J.~A. Vickers, and J.~P. Wilson}, {\em
Generalized
  functions and distributional curvature of cosmic strings}, Class. Quantum
  Grav., 13 (1996), pp.~2485--2498.

\bibitem{Colombeau}
{\sc J.-F. Colombeau}, {\em New generalized functions and
multiplication of
  distributions}, vol.~84 of North-Holland Mathematics Studies, North-Holland
  Publishing Co., Amsterdam, 1984.
\newblock Notas de Matem\'atica [Mathematical Notes], 90.

\bibitem{C}
\leavevmode\vrule height 2pt depth -1.6pt width 23pt, {\em
Elementary
  introduction to new generalized functions}, vol.~113 of North-Holland
  Mathematics Studies, North-Holland Publishing Co., Amsterdam, 1985.
\newblock Notes on Pure Mathematics, 103.

\bibitem{DHPV}
{\sc A.~Delcroix, M.~F. Hasler, S.~Pilipovi{\'c}, and V.~Valmorin},
{\em
  Generalized function algebras as sequence space algebras}, Proc. Amer. Math.
  Soc., 132 (2004), pp.~2031--2038 (electronic).

\bibitem{Scarpalezos1}
{\sc A.~Delcroix and D.~Scarpalezos}, {\em Sharp topologies on
($\widetilde
  {C}$,$\widetilde {E}$,$\widetilde {P}$)-algebras}, in Nonlinear theory of
  generalized functions (Vienna, 1997), vol.~401 of Chapman \& Hall/CRC Res.
  Notes Math., Chapman \& Hall/CRC, Boca Raton, FL, 1999, pp.~165--173.

\bibitem{GF}
{\sc G.~Fischer}, {\em Lineare {A}lgebra}, vol.~17 of Grundkurs
Mathematik,
  Friedr. Vieweg \& Sohn, Braunschweig, fifth~ed., 1979.
\newblock In collaboration with Richard Schimpl.

\bibitem{FL1}
{\sc F.~G. Friedlander}, {\em The wave equation on a curved
space-time},
  Cambridge University Press, Cambridge, 1975.
\newblock Cambridge Monographs on Mathematical Physics, No. 2.

\bibitem{Garetto2}
{\sc C.~Garetto}, {\em Topological structures in {C}olombeau
algebras:
  investigation of the duals of {$\mathcal G_c(\Omega),\, \mathcal G(\Omega)$
  and $\mathcal G_{\mathcal S}(\mathbb R^n)$}}, Monatsh. Math., 146 (2005),
  pp.~203--226.

\bibitem{Garetto1}
\leavevmode\vrule height 2pt depth -1.6pt width 23pt, {\em
Topological
  structures in {C}olombeau algebras: topological {$\widetilde{\mathbb
  C}$}-modules and duality theory}, Acta Appl. Math., 88 (2005), pp.~81--123.

\bibitem{VickersESI}
{\sc M.~Grosser, G.~H{\"o}rmann, M.~Kunzinger, and
M.~Oberguggenberger}, eds.,
  {\em Nonlinear theory of generalized functions}, vol.~401 of Chapman \&
  Hall/CRC Research Notes in Mathematics, Chapman \& Hall/CRC, Boca Raton, FL,
  1999.

\bibitem{Bible}
{\sc M.~Grosser, M.~Kunzinger, M.~Oberguggenberger, and
R.~Steinbauer}, {\em
  Geometric theory of generalized functions with applications to general
  relativity}, vol.~537 of Mathematics and its Applications, Kluwer Academic
  Publishers, Dordrecht, 2001.

\bibitem{fivepeople}
{\sc M.~Grosser, M.~Kunzinger, R.~Steinbauer, H.~Urbantke, and J.~A.
Vickers},
  {\em Diffeomorphism invariant construction of nonlinear generalised
  functions}, Acta Appl. Math., 80 (2004), pp.~221--241.

\bibitem{GlobTh}
{\sc M.~Grosser, M.~Kunzinger, R.~Steinbauer, and J.~A. Vickers},
{\em A global
  theory of algebras of generalized functions}, Adv. Math., 166 (2002),
  pp.~50--72.

\bibitem{HE}
{\sc S.~W. Hawking and G.~F.~R. Ellis}, {\em The large scale
structure of
  space-time}, Cambridge University Press, London, 1973.
\newblock Cambridge Monographs on Mathematical Physics, No. 1.

\bibitem{MOHor}
{\sc G.~H{\"o}rmann and M.~Oberguggenberger}, {\em Elliptic
regularity and
  solvability for partial differential equations with {C}olombeau
  coefficients}, Electr. Jour. Diff. Equ.,,  (2004), pp.~1--30.

\bibitem{Hungerford}
{\sc T.~W. Hungerford}, {\em Algebra}, Hgolt, Rinehart and Winston,
Inc., New
  York, 1974.

\bibitem{Ingleton}
{\sc W.~Ingleton}, {\em The {H}ahn-{B}anach theorem for
non-archimedean valued
  fields}, Proc. Cambridge Phil. Soc., 48 (1952), pp.~41--45.

\bibitem{KKo1}
{\sc S.~Konjik and M.~Kunzinger}, {\em Generalized group actions in
a global
  setting, to appear.}, J. Math. Anal. Appl.

\bibitem{KKo2}
\leavevmode\vrule height 2pt depth -1.6pt width 23pt, {\em Group
invariants in
  algebras of generalized functions, to appear.}, Integral Transforms Spec.
  Funct.

\bibitem{KU2}
{\sc M.~Kunzinger}, {\em Generalized functions valued in a smooth
manifold},
  Monatsh. Math., 137 (2002), pp.~31--49.

\bibitem{K}
\leavevmode\vrule height 2pt depth -1.6pt width 23pt, {\em Nonsmooth
  differential geometry and algebras of generalized functions}, J. Math. Anal.
  Appl., 297 (2004), pp.~456--471.
\newblock Special issue dedicated to John Horv\'ath.

\bibitem{3MikesV}
{\sc M.~Kunzinger, M.~Oberguggenberger, R.~Steinbauer, and J.~A.
Vickers}, {\em
  Generalized flows and singular {ODE}s on differentiable manifolds}, Acta
  Appl. Math., 80 (2004), pp.~221--241.

\bibitem{KS}
{\sc M.~Kunzinger and R.~Steinbauer}, {\em Foundations of a
nonlinear
  distributional geometry}, Acta Appl. Math., 71 (2002), pp.~179--206.

\bibitem{KS1}
\leavevmode\vrule height 2pt depth -1.6pt width 23pt, {\em
Generalized
  pseudo-{R}iemannian geometry}, Trans. Amer. Math. Soc., 354 (2002),
  pp.~4179--4199 (electronic).

\bibitem{KSV}
{\sc M.~Kunzinger, R.~Steinbauer, and J.~A. Vickers}, {\em Intrinsic
  characterization of manifold-valued generalized functions}, Proc. London
  Math. Soc. (3), 87 (2003), pp.~451--470.

\bibitem{GenConKSV}
\leavevmode\vrule height 2pt depth -1.6pt width 23pt, {\em
Generalised
  connections and curvature}, Math. Proc. Cambridge Philos. Soc., 139 (2005),
  pp.~497--521.

\bibitem{Eyb4}
{\sc E.~Mayerhofer}, {\em On the characterization of p-adic
  {C}olombeau-{E}gorov generalized functions by their point values,}, to appear
  in Mathematische Nachrichten,  (2006).

\bibitem{PSMO}
{\sc M.~Oberguggeberger, S.~Pilipovi\'c, and D.~Scarpalézos}, {\em
Positivity
  and positive definiteness in generalized function algebras, to appear}, J.
  Math. Anal. Appl.,  (2003).

\bibitem{MObook}
{\sc M.~Oberguggenberger}, {\em Multiplication of distributions and
  applications to partial differential equations}, vol.~259 of Pitman Research
  Notes in Mathematics Series, Longman Scientific \& Technical, Harlow, 1992.

\bibitem{VariationMO}
\leavevmode\vrule height 2pt depth -1.6pt width 23pt, {\em Calculus
of
  variations in {C}olombeau algebras}, unpublished manuscript,  (1995).

\bibitem{MO1}
{\sc M.~Oberguggenberger and M.~Kunzinger}, {\em Characterization of
  {C}olombeau generalized functions by their pointvalues}, Math. Nachr., 203
  (1999), pp.~147--157.

\bibitem{ON}
{\sc B.~O'Neill}, {\em Semi-{R}iemannian geometry}, vol.~103 of Pure
and
  Applied Mathematics, Academic Press Inc., New York, 1983.
\newblock With applications to relativity.

\bibitem{Lobster}
{\sc S.~Pilipovi{\'c}, D.~Scarpalezos, and V.~Valmorin}, {\em
Equalities in
  algebras of generalized functions}, Forum Math., 1 (2006), pp.~1--2.

\bibitem{priess1}
{\sc S.~Prie{\ss}-Crampe}, {\em Der {B}anachsche {F}ixpunktsatz
f\"ur
  ultrametrische {R}\"aume}, Results Math., 18 (1990), pp.~178--186.

\bibitem{priess2}
{\sc S.~Priess-Crampe and P.~Ribenboim}, {\em Fixed point and
attractor
  theorems for ultrametric spaces}, Forum Math., 12 (2000), pp.~53--64.

\bibitem{zAR}
{\sc A.~M. Robert}, {\em A course in {$p$}-adic analysis}, vol.~198
of Graduate
  Texts in Mathematics, Springer-Verlag, New York, 2000.

\bibitem{Schw1}
{\sc L.~Schwartz}, {\em Sur l'impossibilit\'e de la multiplication
des
  distributions}, C. R. Acad. Sci. Paris, 239 (1954), pp.~847--848.

\bibitem{Senovilla}
{\sc J.~M.~M. Senovilla}, {\em Super-energy tensors}, Class. Quantum
Grav., 17
  (2000), pp.~2799--2841.

\bibitem{SV}
{\sc R.~Steinbauer and J.~Vickers}, {\em The use of generalized
functions and
  distributions in general relativity}, Class. Quantum Grav.

\bibitem{SJ}
{\sc G.~W. Stewart and J.~G. Sun}, {\em Matrix perturbation theory},
Computer
  Science and Scientific Computing, Academic Press Inc., Boston, MA, 1990.

\bibitem{Vernaeve}
{\sc H.~Vernaeve}, {\em Group invariant colombeau generalized
functions},
  Retrieve from http://arxiv.org/math.FA/0512219,  (2005).

\bibitem{VW}
{\sc J.~A. Vickers and J.~P. Wilson}, {\em Generalized hyperbolicity
in conical
  spacetimes}, Class. Quantum Grav., 17 (2000), pp.~1333--1260.

\bibitem{Wald}
{\sc R.~M. Wald}, {\em General relativity}, University of Chicago
Press,
  Chicago, IL, 1984.

\end{thebibliography}
\end{document}